\documentclass[12pt]{amsart}
\usepackage{amssymb}
\usepackage{amsbsy}
\usepackage{amscd}
\newcommand{\linelabel}[1]{}

\usepackage[mathscr]{eucal}
\usepackage{nicefrac}

\DeclareSymbolFont{rsfs}{U}{rsfs}{m}{n}
\DeclareSymbolFontAlphabet{\mathrsfs}{rsfs}

\usepackage{a4wide}
\usepackage[all]{xy}
\usepackage{tikz}
\usetikzlibrary{arrows,decorations.pathmorphing,backgrounds,positioning,fit,petri}
\usetikzlibrary{decorations.markings}

\usepackage{verbatim}
\usepackage{version}
\usepackage{color}
\definecolor{darkspringgreen}{rgb}{0.09, 0.45, 0.27}
\definecolor{deepjunglegreen}{rgb}{0.0, 0.29, 0.29}
\newenvironment{NB}{
\color{red}{\bf NB}. \footnotesize
}{}
\newenvironment{NB2}{
\color{blue}{\bf NB2}. \footnotesize
}{}
\excludeversion{NB}
\excludeversion{NB2}
\makeatletter

\usepackage{url}
\usepackage{ifmtarg}

\usepackage{xr-hyper}
\usepackage[
bookmarks=true,
colorlinks=true,
citecolor=deepjunglegreen,
debug=true,
pdfnewwindow=true]{hyperref}
\usepackage{cleveref}
\let\RR\relax
\let\bN\relax
\let\SL\relax
\let\PGL\relax
\usepackage{xspace}

\AtBeginDocument{%
\makeatletter
\let\oldref=\ref
\renewcommand{\ref}[1]{%
  \def\@mystring{blowup_pre-#1}%
  \@ifundefined{r@\@mystring}{%
    \def\@mystring{Coulomb2-#1}%
    \@ifundefined{r@\@mystring}{%
      \cref{#1}}{%
      \namecref{Coulomb2-#1}II.\oldref{Coulomb2-#1}%
    }%
  }%
  {%
    \namecref{blowup_pre-#1}Q.\oldref{blowup_pre-#1}%
  }%
}
\renewcommand{\eqref}[1]{%
  \def\@mystring{blowup_pre-#1}%
  \@ifundefined{r@\@mystring}{%
    \def\@mystring{Coulomb2-#1}%
    \@ifundefined{r@\@mystring}{%
      \textup{\tagform@{\oldref{#1}}}}{%
      \text{(II.\oldref{Coulomb2-#1})}%
    }%
  }%
  {%
    \text{(Q.\oldref{blowup_pre-#1})}%
}%
}
\makeatother
}

\crefname{Theorem}{Theorem\xspace}{Theorems}
\crefformat{Theorem}{Theorem~#2#1#3}
\crefname{section}{\S}{\S\S}
\crefformat{section}{\S#2#1#3} 
\crefmultiformat{section}{\S\S#2#1#3}{, #2#1#3}{, #2#1#3}{, #2#1#3}
\crefname{Lemma}{Lemma\xspace}{Lemmas\xspace}
\crefformat{Lemma}{Lemma~#2#1#3}
\crefname{Proposition}{Proposition\xspace}{Propositions\xspace}
\crefformat{Proposition}{Proposition~#2#1#3}
\crefname{Corollary}{Corollary\xspace}{Corollaries\xspace}
\crefformat{Corollary}{Corollary~#2#1#3}
\crefname{Definition}{Definition}{Definitions}
\crefformat{Definition}{Definition~#2#1#3}
\crefname{Remark}{Remark\xspace}{Remarks\xspace}
\crefformat{Remark}{Remark~#2#1#3}
\crefname{Remarks}{Remark\xspace}{Remarks\xspace}
\crefformat{Remarks}{Remark~#2#1#3}
\crefname{Conjecture}{Conjecture\xspace}{Conjectures\xspace}
\crefformat{Conjecture}{Conjecture~#2#1#3}
\crefname{figure}{Figure\xspace}{Figure\xspace}
\crefformat{figure}{Figure~#2#1#3}

\crefname{appendix}{Appendix\xspace}{Appendices\xspace}
\crefformat{appendix}{\S#2#1#3}
\crefmultiformat{appendix}{\S\S#2#1#3}{, #2#1#3}{, #2#1#3}{, #2#1#3}
\crefformat{enumi}{#2#1#3}

\crefname{equation}{}{}
\crefformat{equation}{(#2#1#3)}

\usepackage{stmaryrd}
\SetSymbolFont{stmry}{bold}{U}{stmry}{m}{n}
\usepackage{pigpen} 
\usepackage{enumitem}
\usepackage{calc}

\usepackage{aliascnt}
\usepackage{textgreek}
\usepackage{mathabx}

\usepackage{scalerel}

\renewcommand{\thesubsection}{\thesection(\@roman\c@subsection)}
\newcounter{number}
\makeatother
\newtheorem{Theorem}[equation]{Theorem}
\newtheorem{Corollary}[equation]{Corollary}
\newtheorem{Lemma}[equation]{Lemma}
\newtheorem{Proposition}[equation]{Proposition}

\theoremstyle{definition}
\newtheorem{Definition}[equation]{Definition}
\newtheorem{Example}[equation]{Example}

\newtheorem{Conjecture}[equation]{Conjecture}

\theoremstyle{remark}
\newtheorem{Remark}[equation]{Remark}
\newtheorem{Remarks}[equation]{Remarks}
\newtheorem*{Claim}{Claim}

\numberwithin{equation}{section}

\newcommand{\thmref}[1]{\ref{#1}}
\newcommand{\secref}[1]{\ref{#1}}
\newcommand{\lemref}[1]{\ref{#1}}
\newcommand{\propref}[1]{\ref{#1}}
\newcommand{\subsecref}[1]{\ref{#1}}
\newcommand{\remref}[1]{\ref{#1}}

\newcommand{\defeq}{\overset{\operatorname{\scriptstyle def.}}{=}}
\newcommand{\CC}{{\mathbb C}}
\newcommand{\ZZ}{{\mathbb Z}}

\newcommand{\RR}{{\mathbb R}}
\newcommand{\proj}{{\mathbb P}}
\newcommand{\CP}{\proj}

\newcommand{\SL}{\operatorname{\rm SL}}
\newcommand{\SU}{\operatorname{\rm SU}}
\newcommand{\GL}{\operatorname{GL}}
\newcommand{\PGL}{\operatorname{PGL}}

\newcommand{\SO}{\operatorname{\rm SO}}
\newcommand{\grpSp}{\operatorname{\rm Sp}}

\newcommand{\algsl}{\operatorname{\mathfrak{sl}}} 

\newcommand{\gl}{\operatorname{\mathfrak{gl}}}

\newcommand{\g}{{\mathfrak g}}

\newcommand{\Spec}{\operatorname{Spec}\nolimits}
\newcommand{\Proj}{\operatorname{Proj}\nolimits}
\newcommand{\End}{\operatorname{End}}
\newcommand{\Hom}{\operatorname{Hom}}
\newcommand{\Ext}{\operatorname{Ext}}

\newcommand{\rank}{\operatorname{rank}}

\newcommand{\id}{\operatorname{id}}

\renewcommand{\MR}[1]{}

\newcommand{\dslash}{/\!\!/}
\newcommand{\vin}[1]{\operatorname{i}(#1)} 
\newcommand{\vout}[1]{\operatorname{o}(#1)} 
\newcommand{\bM}{\mathbf M}
\newcommand{\bN}{\mathbf N}

\newcommand{\shfO}{\mathcal O}
\newcommand{\tslash}{/\!\!/\!\!/}
\newcommand{\tslabar}{\mathbin{
\setbox0=\hbox{/\!\!/\!\!/}\rule[0.4\ht0]{\wd0}{.3\dp0}\kern-\wd0\box0}}
\newcommand{\Hyp}{\operatorname{Hyp}}

\newcommand{\la}{\lambda}

\newcommand{\Gr}{\mathrm{Gr}}
\newcommand{\Fl}{\mathsf{Fl}}
\newcommand{\cC}{\mathcal C}
\newcommand{\cE}{\mathcal E}
\newcommand{\cF}{\mathcal F}
\newcommand{\cG}{\mathcal G}

\newcommand{\cK}{\mathcal K}
\newcommand{\cO}{\mathcal O}
\newcommand{\cR}{\mathcal R}
\newcommand{\cT}{\mathcal T}
\newcommand{\cW}{\mathcal W}
\newcommand{\cX}{\mathcal X}
\newcommand{\cY}{\mathcal Y}

\newcommand{\scP}{\mathscr P}
\newcommand{\scA}{\mathscr A}
\newcommand{\scB}{\mathscr B}

\makeatletter
\newcommand{\cA}[1][{}]{%
  \@ifmtarg{#1}%
  {\mathcal A}
  {\mathcal A(#1)}
}
\newcommand{\cAh}[1][{}]{%
  \@ifmtarg{#1}%
  {\mathcal A_\hbar}
  {\mathcal A_\hbar(#1)}
}
\makeatother

\newcommand{\leftmapsto}{\mapsfrom}
\newcommand{\DD}{\mathbb D}
\newcommand{\DC}{\boldsymbol\omega}

\newcommand{\ft}{\mathfrak t}
\newcommand{\gr}{\operatorname{gr}}

\newcommand{\Iw}{\tilde G_\cK^\cO}
\newcommand{\Res}{\operatorname{Res}}

\newcommand{\For}{\operatorname{For}}

\newcommand{\laF}{{\lambda_F}}
\newcommand{\muF}{{\mu_F}}

\newcommand{\scAres}{\tilscA^{\mathrm{res}}}
\newcommand{\scAfor}{\scA^{\mathrm{for}}}

\newcommand{\cL}{\mathcal L}
\newcommand{\Aut}{\operatorname{Aut}}

\newcommand{\HHom}{\mathcal Hom}
\newcommand{\RHHom}{{\mathcal{RH}om}}
\newcommand{\triv}{\mathrm{triv}}
\newcommand{\po}{\ar@{}[dr]|{\text{\pigpenfont R}}}
\newcommand{\pb}{\ar@{}[dr]|{\text{\pigpenfont J}}}
\newcommand{\pp}{\ar@{}[dr]|{\text{\pigpenfont P}}}

\newcommand{\cM}{\mathcal M}

\newcommand{\reg}{{\operatorname{reg}}}
\newcommand{\BA}{{\mathbb{A}}}
\newcommand{\BC}{{\mathbb{C}}}
\newcommand{\BD}{{\mathbb{D}}}

\newcommand{\BG}{{\mathbb{G}}}
\newcommand{\BN}{{\mathbb{N}}}
\newcommand{\BO}{{\mathbb{O}}}

\newcommand{\BP}{{\mathbb{P}}}

\newcommand{\BZ}{{\mathbb{Z}}}

\newcommand{\bD}{{\mathbf{D}}}

\newcommand{\bq}{{\mathbf{q}}}

\newcommand{\CB}{{\mathcal{B}}}

\newcommand{\CF}{{\mathcal{F}}}
\newcommand{\CG}{{\mathscr{G}}}

\newcommand{\CJ}{{\mathcal{J}}}

\newcommand{\CL}{{\mathcal{L}}}
\newcommand{\CM}{{\mathcal{M}}}
\newcommand{\CN}{{\mathcal{N}}}
\newcommand{\CO}{{\mathcal{O}}}

\newcommand{\CR}{{\mathcal{R}}}
\newcommand{\cS}{{\mathcal{S}}}

\newcommand{\CW}{{\mathcal{W}}}
\newcommand{\oW}{\overline{\mathcal{W}}{}}

\newcommand{\iso}{\overset{\sim}{\longrightarrow}}

\newcommand{\fb}{{\mathfrak{b}}}
\newcommand{\fg}{{\mathfrak{g}}}
\newcommand{\fh}{{\mathfrak{h}}}
\newcommand{\fri}{{\mathfrak{i}}}

\newcommand{\fl}{\mathfrak l}
\newcommand{\fz}{{\mathfrak{z}}}

\newcommand{\fu}{{\mathfrak{u}}}

\newcommand{\sfm}{{\mathsf{m}}}

\newcommand{\ol}{\overline}
\newcommand{\on}{\operatorname}
\newcommand{\wit}{\widetilde}
\newcommand{\oZ}{\mathring{Z}}

\newcommand{\Areg}{\scA_R}
\newcommand{\Weyl}{{\mathbb W}}
\newcommand{\gvee}{{\mathfrak g}^{\!\scriptscriptstyle\vee}}

\newcommand{\tilscA}{\scA}

\newcommand{\fr}{\mathrm{fr}}
\newcommand{\bz}{\mathbf z}

\newcommand{\iialp}{\mathsf n}

\newcommand{\stslash}{%
  \mathord{%
    \begin{tikzpicture}[baseline=.3ex]
      \draw[densely dotted] (0,0) -- (.7ex,2.3ex);
      \draw (.4ex,0) -- (1.1ex,2.3ex);
    \end{tikzpicture}\kern.15ex
  }
}

\makeatletter
\let\@wraptoccontribs\wraptoccontribs
\def\@setcontribs{%
  \@xcontribs
  \textsf{\xcontribs}
}
\makeatother

\makeatletter 

\@ifundefined{addto}{\def\addto#1#2{%
  \ifx#1\@undefined
    \def#1{#2}%
  \else
    \ifx#1\relax
      \def#1{#2}%
    \else
      {\toks@\expandafter{#1#2}%
        \xdef#1{\the\toks@}}%
    \fi
  \fi
}}{}

\def\addto@every@math{%
  \expandafter\addto\csname \expandafter\ifx
  \csname mathoptions@on\endcsname\relax 
  check@mathfonts\else mathoptions@on\fi\endcsname
}

\def\active@def#1{%
  \begingroup\lccode`\~=`#1\relax\lowercase{\endgroup\def~}%
}

\def\fixmathspacing#1#2{%
  \addto@every@math{%
    \catcode`#1=12 \mathcode`#1="8000
    \active@def#1{#2}%
  }%
}

\makeatother 

\fixmathspacing{/}{\mathclose{}\mathchar"013D\mathopen{}}

\def\backslash{\delimiter"526E30F\mathopen{}}

\begin{document}

\title[Ring objects from Coulomb branches]{Ring objects in the equivariant derived Satake category arising from Coulomb branches
}
\author[A.~Braverman]{Alexander Braverman}
\address{
Department of Mathematics,
University of Toronto and Perimeter Institute of Theoretical Physics,
Waterloo, Ontario, Canada, N2L 2Y5
}
\email{braval@math.toronto.edu}
\author[M.~Finkelberg]{Michael Finkelberg}
\address{National Research University Higher School of Economics,
Russian Federation, Department of Mathematics, 6 Usacheva st, Moscow 119048;
Skolkovo Institute of Science and Technology}
\email{fnklberg@gmail.com}
\author[H.~Nakajima]{Hiraku Nakajima}
\address{Research Institute for Mathematical Sciences,
Kyoto University, Kyoto 606-8502,
Japan}
\email{nakajima@kurims.kyoto-u.ac.jp}
\contrib[Appendix B by]{Gus Lonergan
}

\subjclass[2000]{}
\begin{abstract}
  This is the second companion paper of \cite{main}. We consider the
  morphism from the variety of triples introduced in \cite{main} to
  the affine Grassmannian. The direct image of the dualizing complex
  is a ring object in the equivariant derived category on the affine
  Grassmannian (equivariant derived Satake category). We show that
  various constructions in \cite{main} work for an arbitrary
  commutative ring object.

  The second purpose of this paper is to study Coulomb branches
  associated with star shaped quivers, which are expected to be
  conjectural Higgs branches of $3d$ Sicilian theories in type $A$ by
  \cite{Benini:2010uu}.
\end{abstract}

\maketitle

\setcounter{tocdepth}{2}

\section{Introduction}

This is the second companion paper of \cite{main}, where
we give a mathematical definition of the Coulomb branch $\cM_C$ of a
$3d$ SUSY gauge theory associated with a complex reductive group $G$
and its symplectic representation $\bM$ of a form
$\bN\oplus\bN^*$. Recall that $\cM_C$ is defined as an affine
algebraic variety whose coordinate ring is the equivariant Borel Moore
homology group $H^{G_\cO}_*(\cR)$ of a certain space $\cR$, called the
variety of triples. The product is given by the convolution. Here
$G_\cO$ is the $\CC[[z]]$-valued points of $G$.

By its definition, we have a projection $\pi\colon \cR\to\Gr_G$, where
$\Gr_G$ is the affine Grassmannian for $G$.
Therefore we have a natural object $\scA$ in an approproate Ind-completion 
$D_G(\Gr_G)$ of the derived $G_\cO$-equivariant
constructible category on $\Gr_G$ defined by\linebreak[4]
$\pi_*\DC_{\cR}[-2\dim\bN_\cO]$, where $\DC_{\cR}$ is the dualizing
complex on $\cR$.
We can recover $H^{G_\cO}_*(\cR)$ as
$H^*_{G_\cO}(\Gr_G,\scA)$. Moreover the construction of the
convolution product gives us a homomorphism
$\mathsf m\colon\scA\star\scA\to \scA$, where $\star$ is the
convolution product on $D_{G}(\Gr_G)$. It is an associative multiplication on
$\scA$. Then we have an induced multiplication on
$H^*_{G_\cO}(\Gr_G,\scA)$ from $\mathsf m$, which is the same as the
product on $H^{G_\cO}_*(\cR)$ defined in \cite{main}. We also 
prove that it is a commutative object in $D_G(\Gr_G)$, and hence
the induced multiplication on $H^*_{G_\cO}(\Gr_G,\scA)$ is
commutative. It is the second proof of the commutativity of the
product on $H^{G_\cO}_*(\cR)$, which is more conceptual than the first
computational proof in \cite{main}.

In view of the original proposal in \cite{2015arXiv150303676N}, we
expect that this construction can be generalized to the case when
$\bM$ is not necessarily of the form $\bN\oplus\bN^*$.

Anyhow if we have a commutative ring object $\scA$ in $D_G(\Gr_G)$, we
get a commutative ring structure on $H_{G_\cO}^*(\Gr_G,\scA)$, and
hence the `Coulomb branch' as its spectrum.

Our reformulation of the definition of the Coulomb branch via
$(\scA,\mathsf m)$ reminds us a construction of the nilpotent cone and
its Springer resolution via a perverse sheaf $\Areg$
\cite{MR2053952}. Here $\Areg$ is a perverse sheaf corresponding to the
regular representation $\CC[G^\vee]$ of the Langlands dual group
$G^\vee$ under the geometric Satake correspondence, and hence is a
commutative ring object in $\operatorname{Perv}_{G_\cO}(\Gr_G)$.
Let us call it the \emph{regular sheaf}. It is given by
$\bigoplus_\la (V_{G^\vee}^\la)^\vee\otimes_\CC
\operatorname{IC}(\overline{\Gr}_{G}^{\la})$,
where $(V_{G^\vee}^\la)^\vee$ is the dual of the irreducible
representation of $G^\vee$ with the highest weight $\la$ and
$\overline{\Gr}_{G}^{\la}$ is the closure of the $G_\cO$-orbit of
$z^\la$ in $\Gr_G$.
We prove that $\Areg$ is realised as a variant of the above $\scA$ for
a quiver gauge theory in type $A$. (We consider the framed quiver
gauge theory of type $A_{N-1}$ with $\dim V = (N-1,N-2,\dots,1)$,
$\dim W = (N,0,\dots,0)$ and consider the pushforward to
$\Gr_{\PGL(N)}$. See \ref{subsec:ABG} for more detail.)
This constuction might be generalized to type $BCD$, once we can
generalize our definition to the case when $\bM$ is not necessarily of
a form $\bN\oplus\bN^*$ (cotangent type). However we do \emph{not}
expect $\Areg$ arises in a similar way for exceptional types. Hence we
have more examples of commutative ring objects in $D_{G}(\Gr_G)$ than
our construction.

Once we have a collection $\{ \scA_i\}$ of commutative ring objects in
$D_G(\Gr_G)$, we can construct a new commutative ring object as
$i_\Delta^!(\boxtimes\scA_i)$, where
$i_\Delta\colon\Gr_G\to\prod_i \Gr_G$ is the diagonal embedding. We
call this the \emph{gluing construction}. It is motivated by
\cite{Cremonesi:2014kwa}. (See \cite[5(i)]{2015arXiv150303676N} for a
quick review and links to other physics literature.)

The second purpose of this paper is to study Coulomb branches
associated with a star shaped quiver. It is regarded as an example of
the gluing construction of a ring objects from those for legs.
It is expected that Coulomb branches of star shaped quiver gauge
theories are conjectural Higgs branches of $3d$ Sicilian theories in
type $A$ \cite{Benini:2010uu}. (See \cite[3(iii)]{2015arXiv150303676N}
for a review for a mathematician.) Expected properties of these Higgs
branches are listed in \cite{MR2985331}. 
Recently Ginzburg-Kazhdan \cite{GK} construct holomorphic symplectic
varieties satisfying (most of) these properties for any type. The
construction of $\Areg$ as $\scA$ implies that Coulomb branches of
star shaped quiver gauge theories are isomorphic to Ginzburg-Kazhdan
varieties in type $A$ via \cite{MR3366026}.
%
%
We check two among the remaining properties, which identify
Ginzburg-Kazhdan varieties of type $A_1$, $A_2$ with
$\CC^2\otimes\CC^2\otimes\CC^2$ and the minimal nilpotent orbit of
$E_6$ respectively.

We do \emph{not} expect Ginzburg-Kazhdan varieties for exceptional
groups are Coulomb branches of gauge theories. This is compatible with
physicists' expectation that $3d$ Sicilian theories are \emph{not}
lagrangian theories. Nevertheless $3d$ Sicilian theories are accepted
as well-defined quantum field theories. And there are many such
examples. It is compatible with our observation that
\begin{enumerate}
\item We have examples of ring objects on $D_G(\Gr_G)$,
  which may not arise from any pair $(G,\bN)$.
\item We have manipulations on ring objects, such as the gluing
  construction and hamiltonian reduction (see \ref{ham} for the latter).
\end{enumerate}
We thus hope that ring objects are useful to study non-lagrangian
theories.

There is an appendix \ref{sec:group_action}, which discusses a result
of independent interest. We construct a complex reductive group
hamiltonian action on the Coulomb branch of a framed quiver gauge
theory by integrating hamiltonian vector fields of functions
introduced in \cite[Appendix~B]{blowup}. This extends a torus action
constructed in \cite[\S3(v)]{main} by grading on
$H^{G_\cO}_*(\cR)$. The regular sheaf $\Areg$ has the $G^\vee$-action,
which is identified with this group action for the framed quiver gauge
theory mentioned above.

\begin{NB}
The other appendices \ref{hilbert,det slice} discuss (partial)
desingularization of Coulomb branches and line bundles over them for
Hilbert schemes and generalized slices respectively. They are
applications of study in \S\S\oldref{line Klein}$\sim$\oldref{Klein
  via} where the corresponding result is proved for Kleinian
surfaces.
\end{NB}%

The other parts of the paper are organized as follows. In
\ref{sec:sheav-affine-grassm} we show that
$\pi_*\DC_{\cR}[-2\dim\bN_\cO]$ and its cousin for gauge theory with a
flavor symmetry group are ring objects. We observe that
$\Ext^*_{D_G(\Gr_G)}(\mathbf 1_{\Gr_{G}},\scA)$ is a commutative ring
for a commutative ring object $\scA$ in $D_G(\Gr_G)$, where
$\mathbf 1_{\Gr_{G}}$ is the skyscraper sheaf at the base point in
$\Gr_G$.
Considering skyscraper sheaves at other points, we construct line
bundles over a partial resolution of
$\Spec \Ext^*_{D_G(\Gr_G)}(\mathbf 1_{\Gr_{G}},\scA)$. We follow
\cite{MR2053952} for these constructions.
The gluing construction is explained in \ref{subsec:glue}.
In \ref{sec:commute} we give a proof of commutativity of $\mathsf m$.
The idea is well-known: we use
Beilinson-Drinfeld Grassmannian to deform a situation where the
product is manifestly symmetric. Then we use nearby cycle functors and
dual specialization homomorphisms.
In \ref{monopole} we show that the regular sheaf $\Areg$ arises as a
pushforward in a framed quiver gauge theory in type $A$.
In \ref{sec:Sicilian} we study Coulomb branches associated with star
shaped quivers.
Since \ref{monopole,sec:Sicilian} depend crucially on the construction
in \ref{sec:group_action}, the authors recommend the reader to go to
\ref{sec:group_action} before visiting \ref{monopole,sec:Sicilian}.

In \ref{gusgusgus} written by Gus Lonergan, we give another proof of
the commutativity of the convolution product. This proof is more
direct than the proof in the main text. A key ingredient is a global
version of the convolution diagram for the variety of triples $\cR$.

\subsection*{Notation}

We basically follow the notation in \cite{main} and \cite{blowup}.
The Weyl group is denoted by $\Weyl$ in order to distinugish a vector
space $W$ used for a quiver.

Sections, equations, Theorems, etc in \cite{main} (resp.\
\cite{blowup}) will be referred with `II.' (resp.\ `Q.') plus the
numbering, such as \ref{prop:flat} (resp.\ \ref{Coulomb_quivar}).

\subsection*{Acknowledgments}

We thank
S.~Arkhipov,
R.~Bezrukavnikov,
D.~Gaiotto,
D.~Gaitsgory,
V.~Ginzburg,
A.~Hanany,
J.~Kamnitzer,
Y.~Namikawa,
and
Y.~Tachikawa
for the useful discussions.
We also thank V.~Gorin for his reply to our question in mathoverflow, and
I.~Losev for providing us a proof of \ref{Prop:iso}.
Last but not least, we thank G.~Lonergan for writing an appendix.

A.B.\  was partially supported by the NSF grant DMS-1501047.
M.F.\ was partially supported by
a subsidy granted to the HSE by
the Government of the Russian Federation for the implementation of the Global
Competitiveness Program.
The research of H.N.\ is supported by JSPS Kakenhi Grant Numbers
24224001, 
25220701, 
16H06335. 

\section{Complexes on the affine Grassmannian}\label{sec:sheav-affine-grassm}

In this section 
we interpret the convolution product $\ast$ in terms of a complex on
the affine Grassmannian. Our goal is to construct a commutative ring
object in $D_{G}(\Gr_G)$, an appropriate Ind-completion of the
$G_\cO$-equivariant derived constructible category on $\Gr_G$. Here the
multiplication is given by
the product $\star$ appearing in geometric Satake correspondence
\cite{MV2}.

The construction of this section, except \subsecref{subsec:glue}, is
motivated by the work of Arkhipov, Bezru\-kavnikov and Ginzburg
\cite{MR2053952}, where the nilpotent cone $\CN$ of the Langlands dual
group is constructed from the regular sheaf $\Areg$ on $\Gr_G$.

The construction of \subsecref{subsec:glue} is motivated by
\cite{Cremonesi:2014kwa}, as we have mentioned already in
Introduction.
\begin{NB}
It allows us to construct a complicated commutative ring
object from simpler ones. See \cite[\S5(i)]{2015arXiv150303676N} for a
review.
\end{NB}%

\subsection{Categorical generalities}Let $X$ be a scheme of finite type over $\mathbb C$. Then we denote by $D(X)$ the ind-completion of the bounded derived category of constructible sheaves on $X$; same definition applies to the equivariant derived category $D_G(X)$ where $G$ is a (pro)algebraic group acting on $X$. It is obvious that for a $G$-equivariant morphism $f\colon X\to Y$ the derived direct images $f_!,f_*\colon D_G(X)\to D_G(Y)$ are well-defined. The same thing is true for the inverse images $f^!,f^*\colon D_G(Y)\to D_G(X)$. 

Assume that $G$ has finitely many orbits on $X$. Then a morphism $\mathcal F \to \mathcal G$ in $D_G(X)$ is an isomorphism if and only if it is an isomorphism on all $!$-stalks (the assumption that $G$ acts with finitely many orbits is needed in order to guarantee that there is an open dense subset of $X$ on which both $\mathcal F$ and $\mathcal G$ are locally constant).

Let now $X$ be an ind-scheme which is a filtered inductive limit of schemes of finite type over $\mathbb C$ with respect to closed embeddings.
For simplicity we shall assume that $X$ is just the union of closed subschemes $X_0\subset X_1\subset \cdots$ where each  $X_i$ is a scheme of finite type over $\mathbb C$ and each inclusion $X_i\subset X_{i+1}$ is a closed embedding; we denote this embedding by $\sigma_i$. We shall call such ind-schemes {\em good}. Assume that a (pro)algebraic group $G$ acts on each $X_i$ and this action commutes with $\sigma_i$'s. Then we shall say that $X$ is a good $G$-scheme.

For a good $G$-ind-scheme $X$ we define the category $D_G(X)$ whose objects are systems $(\mathcal F_i,\kappa_i)_{i=0}^{\infty}$ where

$\bullet$ $\mathcal F_i\in D_G(X_i)$

$\bullet$ $\kappa_i\colon \sigma_i^!\mathcal F_{i+1}\to \mathcal F_i$ is an isomorphism.

\noindent
A morphism $\alpha\colon  (\mathcal F_i,\kappa_i)\to (\mathcal  F_i',\kappa_i')$ is collection of morphisms $\mathcal F_i\to \mathcal F_i'$ for each $i$ which commute with the $\kappa_i$'s. It is easy to see that $D_G(X)$ is a triangulated category. Assume that $G$ acts with finitely many orbits on each $X_i$; in this case we shall say that $X$ is a very good $G$-ind-scheme.  Then again  a morphism $\mathcal F \to \mathcal G$ in $D_G(X)$ is an isomorphism if and only if it is an isomorphism on all $!$-stalks.

Let $X,Y$ be two good $G$-ind-schemes and let $f\colon X\to Y$ be a $G$-equivariant morphism. Then we can define the functor $f_*\colon D_G(X)\to D_G(Y)$ (but a priori not the functor $f_!$). It is defined in the following way. Given an object $(\mathcal F_i,\kappa_i)$ of $D_G(X)$ we need to define an  object $(\mathcal G_j,\eta_j)$ of $D_G(Y)$. Let $Z_j=f^{-1}(Y_j)$. This is again a good $G$-ind-scheme -- it is the inductive limit of $Z_{i,j}=X_i\cap f^{-1}(Y_j)$. Let $\mathcal F_{i,j}$ denote the $!$-restriction of $\mathcal F_i$ to $Z_{i,j}$. Let also $f_{i,j}\colon  Z_{i,j}\to Y_j$ denote the natural morphism. Since $(\sigma_i^!)$ is right adjoint to $(\sigma_i)_!$, the isomorphism $\kappa_i$ gives rise to a map $(\sigma_i)_!\mathcal F_i=(\sigma_i)_*\mathcal F_i\to \mathcal F_{i+1}$; $!$-restricting this to $Z_j$ we get a morphism $(\sigma_i)_*\mathcal F_{i,j}\to \mathcal F_{i+1,j}$ which gives rise to a natural map $(f_{i,j})_*\mathcal F_{i,j}\to (f_{i+1,j})_*\mathcal F_{i+1,j}$. Hence the inductive limit of $(f_{i,j})_*\mathcal F_{i,j}$'s (with respect to $i$) makes sense and we denote it by $\mathcal G_j$. The construction of isomorphisms $\eta_j$ between the $!$-restriction of $\mathcal G_{j+1}$ and $\mathcal G_j$ is immediate from the usual base change.

In what follows we are going apply it for example to $X$ being $\Gr_G$ for some reductive group $G$. In this case we can talk about the equivariant derived category $D_{G_{\mathcal O}}(\Gr_G)$ which as before we shall simply denote by $D_G(\Gr_G)$ (a priori it depends on a choice of $X_i$'s above; to simplify the discussion we are going to make this choice, although it is not difficult to see that the resulting category is independent of that choice); it is clear that (for any choice of $X_i$'s) $\Gr_G$ is a very good $G_{\mathcal O}$-ind-scheme. The above general discussion also shows that given two objects $\mathcal F,\mathcal G\in D_{G}(\Gr_G)$ we can define their convolution $\mathcal F\star \mathcal G\in D_{G}(\Gr_G)$.

\subsection{Pushforward to the affine Grassmannian}\label{subsec:pushf-affine-grassm}


Let $\bN$ be a finite dimensional representation of a complex
reductive group $G$. Let $\cR$ be the variety of triples as in
\cite{main}, and $\DC_\cR$ its dualizing complex.

\begin{Proposition}\label{prop:sheav-affine-grassm}
    Let $\pi\colon \cR\to \Gr_G$ be the projection and $\scA \defeq
    \pi_* \DC_\cR[-2\dim\bN_\cO]\in D_{G}(\Gr_G)$.

    \textup{(1)} There exists a natural multiplication homomorphism
    \begin{equation*}
        \mathsf m\colon
        \scA\star\scA
        \to
        \scA,
    \end{equation*}
    where the left hand side is the convolution product of $\scA$ with
    itself given by the diagram \eqref{eq:1}.

    \textup{(2)} Let $\mathbf 1_{\Gr_G}$ denote the skyscraper sheaf
    at the base point in $\Gr_G$. Recall that it is the unit element
    in $D_G(\Gr_G)$, i.e., we have natural isomorphisms $\mathbf
    1_{\Gr_G}\star\scA \cong \scA \cong \scA\star \mathbf 1_{\Gr_G}$. We
    have a homomorphism $1\colon \mathbf 1_{\Gr_G}\to \scA$ such that
    \begin{equation*}
            \scA \cong \scA\star\mathbf 1_{\Gr_G}
            \xrightarrow{\id\star 1}
            \scA\star\scA\xrightarrow{\mathsf m} \scA,
            \qquad
            \scA \cong \mathbf 1_{\Gr_G}\star\scA
            \xrightarrow{1\star \id}
            \scA\star\scA\xrightarrow{\mathsf m} \scA
    \end{equation*}
    are both $\id_{\scA}$.

    \textup{(3)} Under the natural associativity isomorphism
    $\scA\star(\scA\star\scA)\cong (\scA\star\scA)\star\scA$, we have
    \begin{equation*}
        {\mathsf m\circ (\mathsf m\star\id)} =
        {\mathsf m\circ (\id\star\mathsf m)}.
    \end{equation*}

    \textup{(4)} The product on $H^*_{G_\cO}(\Gr_G, \scA) \cong
    H^{G_\cO}_*(\cR)$ induced by $\mathsf m$ is the same as the
    convolution product $\ast$.

    \textup{(5)} \textup{(1)}$\sim$\textup{(4)} remain
    true for the $G_\cO\rtimes\CC^\times$-equivariant setting.
\end{Proposition}

The product in (4) is defined as follows:
\begin{NB}
    let $\scA= \pi_* \DC_\cR[-2\dim\bN_\cO]$. We consider
    \begin{equation*}
        H^*_{G_\cO\times G_\cO}(\scA\boxtimes \scA)
        \xrightarrow{\bar p^*} H^*_{G_\cO\times G_\cO}(\bar p^*(\scA\boxtimes
        \scA))
        \xleftarrow[\cong]{\bar q^*}
        H^*_{G_\cO}((\bar q^*)^{-1}\bar p^*(\scA\boxtimes \scA))
        = H^*_{G_\cO}(\bar m_* (\bar q^*)^{-1}\bar p^*(\scA\boxtimes \scA))
        = H^*_{G_\cO}(\scA\star\scA).
    \end{equation*}
    Then we further send $\mathsf m\colon H^*_{G_\cO}(\scA\star\scA)\to
    H^*_{G_\cO}(\scA)$.

    Alternatively, .....
\end{NB}%
Let $x$, $y\in H^*_{G_\cO}(\scA) = \Ext^*_{D_{G}(\Gr_G)}(\CC_{\Gr_G},
\scA)$.
Then $x\star y\in \Ext^*_{D_{G}(\Gr_G)}(\CC_{\Gr_G}\star
\CC_{\Gr_G}, \scA\star \scA)$. We have a natural homomorphism
$\CC_{\Gr_G}\to \CC_{\Gr_G}\star \CC_{\Gr_G}$ from the adjunction
homomorphism $\CC_{\Gr_G}\to m_*m^* \CC_{\Gr_G}$. Therefore
we combine it with $\mathsf m\colon \scA\star \scA\to \scA$, we get
$x\star y\in \Ext^*_{D_{G}(\Gr_G)}(\CC_{\Gr_G}, \scA)$.

\begin{NB}
\begin{center}
    E-mail on Mar. 23, 2015.
\end{center}

Suppose we have a flavor symmetry group $G_F$ (I assume a larger group
is just a product $G\times G_F$ for brevity, I am not sure how this is
essential) and consider the affine Grassmannian
$\operatorname{Gr}_{G_F}$ of $G_F$. Then we first consider the variety
of triples $\mathcal R$ over $\operatorname{Gr}_G\times
\operatorname{Gr}_{G_F}$. In stead of considering the absolute
homology $H^{G\times G_F}_*(\mathcal R)$, we consider the pushforward
$\pi_*(\DC_{\mathcal R})$, where $ \pi \colon\mathcal R\to
\operatorname{Gr}_{G_F}$. Then this pushforward should be the analog
of the regular perverse sheaf $\mathcal R$
\begin{NB2}
    considered in [ABG],
\end{NB2}
(unfortunately we have a conflict of notation). Let us denote it by
$\mathcal R'$. This is more or less you wrote in a message sometime
ago.  Then I guess all the following questions should have affirmative
answers:

Q1. Does $\mathcal R'$ have a morphism $\mathbf m\colon
\mathcal R'\ast\mathcal R'\to\mathcal R'$ ? Here $\ast$ is the
convolution product on $\operatorname{Gr}_{G_F}$.

Q2. Consider $
\operatorname{Ext}^\bullet_{D^b(\operatorname{Gr}_{G_F})}(\mathbf
1_{\operatorname{Gr}}, \mathcal R')$, where $\mathbf
1_{\operatorname{Gr}}$ is the sky-scraper sheaf at the base point of $
\operatorname{Gr}_{G_F}$ as in [ABG]. If Q1 is yes, it has a structure of
an algebra as in [ABG]. Is it true that the algebra is isomorphic to our
Coulomb branch ? (Here the Coulomb branch is defined in terms of $
\operatorname{Gr}_G$, but this should be compatible with that $
\mathbf 1_{\operatorname{Gr}}$ is the sky-scraper sheaf at the base point
of $\operatorname{Gr}_{G_F}$(, not of $\operatorname{Gr}_G$).
This Ext group is isomorphic to the BM homology of the original (smaller)
space $\mathcal R$ of triples for $\operatorname{Gr}_G$.
Therefore the actual question is : are two multiplications the same ?

Q2'. As in Theorem 7.3.1 in [ABG], it may be also possible to consider the
Ext-algebra of $\mathcal R'$. What the algebra is it ? (This
question, I do not have a guess.)

Q3. Now let $\mathcal W_\lambda$ (for a dominant coweight $\lambda$) be the Wakimoto sheaf. (I do not read [ABG] carefully, but it
should be an Iwahori equivariant perverse sheaf on $
\operatorname{Gr}_{G_F}$, which Sergei explained to me.) Then $
\operatorname{Ext}^\bullet_{D^b(\operatorname{Gr}_{G_F})}(\mathbf
1_{\operatorname{Gr}}, \mathcal W_\lambda\ast\mathcal R')$ is a module of
the algebra in Q2 by [ABG].

Is its graded dimension given by the generalized monopole formula ?
(Generalized in the sense of 5(i) of Part 1.)

As far as I understand [ABG] in section 8.4 constructed the multiplication

\begin{equation}
    \label{eq:8}
\operatorname{Ext}^\bullet_{D^b(\operatorname{Gr}_{G_F})}(\mathbf 1_{\operatorname{Gr}}, \mathcal
W_\lambda\ast\mathcal R')\otimes
\operatorname{Ext}^\bullet_{D^b(\operatorname{Gr}_{G_F})}(\mathbf
1_{\operatorname{Gr}}, \mathcal W_\mu\ast\mathcal R')\to
\operatorname{Ext}^\bullet_{D^b(\operatorname{Gr}_{G_F})}(\mathbf 1_{\operatorname{Gr}}, \mathcal
W_{\lambda+\mu}\ast\mathcal R').
\end{equation}

This is basically what I wanted. Considering $\{ n\lambda \mid n\ge
0\}$, we get a graded ring whose Proj is the partial resolution of the
Coulomb branch, which I want. The above module is the space of sections of
a line bundle by tautology. For generic $\lambda$, I want to regard
modules from another $\lambda'$ also as a line bundle, but probably
it should be automatic....
\end{NB}

\begin{NB}
    More precisely, this $\pi$ should be considered as the Borel
    construction for $G$, so that fiber at $\mathbf
    1_{\Gr}\in\Gr_{G_F}$ is $H^G_*(\text{original $\cR$})$. Since
    $G_F$ acts on the original $\cR$, we can also consider $H^{G\times
      G_F}_*(\text{original $\cR$})$. By a general expectation, this
    gives a deformation of the Coulomb branch $\mathcal M_C$. In
    [ABG], the Coulomb branch is the nilpotent cone $\mathcal N$ of
    $G_F$. On the other hand $G_F$-equivariant one is $\g_F^*$, the
    dual of the Lie algebra $\g_F$ of $G_F$. Thus it fits with a
    general expectation. There is also a $G_F$-equivariant version of
    the construction \eqref{eq:8}. See [8.7.9, ABG], where the
    cotangent bundle is replace by the Grothendieck-Springer
    simultaneous resolution of $\g_F^*$.
\end{NB}%

\begin{NB}
    Let us suppose $G_F = 1$ and consider $\pi\colon\cT\to\Gr_G$
    for brevity. We consider the diagram
    \begin{equation*}
        \begin{CD}
            \cT\times\cT \times \cT\times\cT
            @<<p_{12}\times p_{23}<
            \cT\times\cT\times\cT
            @>>p_{13}> \cT\times\cT
            \\
            @A{i\times i}AA @AA{i'}A @AA{i}A
            \\
            \cT\times_{\bN_\cK}\cT \times \cT\times_{\bN_\cK}\cT
            @<\pi_{12}\times \pi_{23}<<
            \cT\times_{\bN_\cK}\cT\times_{\bN_\cK}\cT
            @>\pi_{13}>> \cT\times_{\bN_\cK}\cT
            \\
            @V{\pi\times\pi\times\pi\times\pi}VV
            @VV{\pi\times\pi\times\pi}V
            @VV{\pi\times\pi}V
            \\
            \Gr_G\times\Gr_G\times\Gr_G\times\Gr_G
            @<<q_{12}\times q_{23}<
            \Gr_G\times\Gr_G\times\Gr_G
            @>>q_{13}> \Gr_G\times\Gr_G,
        \end{CD}
    \end{equation*}
    where $i$, $i'$ are inclusions, $p_{ij}$, $\pi_{ij}$, $q_{ij}$ are
    projections to the product of $i^{\mathrm{th}}$ and
    $j^{\mathrm{th}}$ factors. Our goal is to construct a homomorphism
    \begin{equation}
        \label{eq:10}
        q_{13*}(q_{12}\times q_{23})^*\left(
        (\pi\times\pi)_*
        \DC_{\cT\times_{\bN_\cK}\cT}
        \boxtimes
        (\pi\times\pi)_*
        \DC_{\cT\times_{\bN_\cK}\cT}\right)
        \to
        (\pi\times\pi)_* \DC_{\cT\times_{\bN_\cK}\cT}.
    \end{equation}

    We start with the adjunction homomorphism
    \begin{equation*}
        \begin{split}
        \DC_{\cT\times_{\bN_\cK}\cT}
        \boxtimes
        \DC_{\cT\times_{\bN_\cK}\cT}
      =
        (i\times i)^! \DC_{\cT\times\cT\times\cT\times\cT}
        \to & (i\times i)^! (p_{12}\times p_{23})_* (p_{12}\times p_{23})^*
        \DC_{\cT\times\cT\times\cT\times\cT}
\\
      & = (\pi_{12}\times\pi_{23})_* i^{\prime!}
      (p_{12}\times p_{23})^* \DC_{\cT\times\cT\times\cT\times\cT},
      \end{split}
    \end{equation*}
    where we have used the base change from the first line to the
    second. If $\cT$ would be smooth, we have
    \begin{equation*}
        i^{\prime!}
      (p_{12}\times p_{23})^* \DC_{\cT\times\cT\times\cT\times\cT}
      = \DC_{\cT\times_{\bN_\cK}\cT\times_{\bN_\cK}\cT}[2\dim \cT].
    \end{equation*}
    Substituting this into the right hand side, we apply
    $(q_{12}\times q_{23})^*
    (\pi\times\pi\times\pi\times\pi)_*$.
    Using the commutativity of the diagram, we replace it as
    \begin{equation*}
        (q_{12}\times q_{23})^* (q_{12}\times q_{23})_*
        (\pi\times\pi\times\pi)_*
        \DC_{\cT\times_{\bN_\cK}\cT\times_{\bN_\cK}\cT}[2\dim \cT]
    \end{equation*}
    Using the adjunction homomorphism $(q_{12}\times q_{23})^*
    (q_{12}\times q_{23})_* \to\operatorname{Id}$, we get a
    homomorphism
    \begin{equation*}
        (q_{12}\times q_{23})^*
    (\pi\times\pi\times\pi\times\pi)_*
    \DC_{\cT\times_{\bN_\cK}\cT\times\cT\times_{\bN_\cK}\cT}
    \to
    (\pi\times\pi\times\pi)_*
    \DC_{\cT\times_{\bN_\cK}\cT\times_{\bN_\cK}\cT}[2\dim \cT].
    \end{equation*}
    Next we apply $q_{13*}$. By the commutativity, we have
    \begin{equation*}
        \begin{split}
            & q_{13*}(q_{12}\times q_{23})^*
    (\pi\times\pi\times\pi\times\pi)_*
    \DC_{\cT\times_{\bN_\cK}\cT\times\cT\times_{\bN_\cK}\cT}
\\
    \to \; &
    (\pi\times\pi)_* \pi_{13*}
        \DC_{\cT\times_{\bN_\cK}\cT\times_{\bN_\cK}\cT}[2\dim \cT].
        \end{split}
    \end{equation*}
Since $\pi_{13}$ is proper, we have $\pi_{13*} = \pi_{13!}$. Moreover
$\DC_{\cT\times_{\bN_\cK}\cT\times_{\bN_\cK}\cT}
= \pi_{13}^! \DC_{\cT\times_{\bN_\cK}\cT}$. Combining with the adjunction
homomorphism $\pi_{13!}\pi_{13}^! \to\operatorname{Id}$, we finally get
the desired homomorphism \eqref{eq:10}.
\end{NB}%

\begin{proof}
Let us combine two diagrams \eqref{eq:1} and \eqref{eq:12}:
\begin{equation}\label{eq:40}
    \begin{CD}
        \cR \times\cR @<\tilde p<< p^{-1}(\cR\times\cR)
        @>\tilde q>> q(p^{-1}(\cR\times\cR))
        @>\tilde m>> \cR
        \\
        @V{i\times\id_\cR}VV @V{i'}VV @VV{\bar i}V @VV{i}V
        \\
        \cT\times \cR @<p<< G_\cK\times\cR @>q>>
        G_\cK\times_{G_\cO}\cR @>m>> \cT
        \\
        @V{\pi\times\pi}VV @V{\id_{G_\cK}\times\pi}VV @VV{\bar\pi}V @VV{\pi}V
        \\
        \Gr_G\times\Gr_G @<<\bar p<
        G_\cK\times\Gr_G 
        @>>\bar q> \Gr_G\tilde\times\Gr_G @>>\bar m> \Gr_G,
    \end{CD}
\end{equation}
where we have changed the notation for morphisms in the bottom row
putting `bar'. 
We also denote $\pi\circ i$ simply by $\pi$ for brevity.

The restriction with support homomorphism \eqref{eq:7} induces
\begin{multline*}
    \scA\boxtimes\scA = (\pi\times\pi)_*(\DC_{\cR\times\cR})[-4\dim\bN_\cO]
\\
    \to (\pi\times\pi)_* \tilde p_*(\DC_{p^{-1}(\cR\times\cR)}
    [-2\dim\bN_\cO-2\dim G_\cO])
\\
   \cong \bar p_*(\id_{G_\cK}\times\pi)_* i'_*
   \DC_{p^{-1}(\cR\times\cR)} [-2\dim\bN_\cO-2\dim G_\cO]).
\end{multline*}
By adjunction, we get
\begin{equation*}
    \bar p^* (\scA\boxtimes\scA)
    \to (\id_{G_\cK}\times\pi)_* i'_*
   \DC_{p^{-1}(\cR\times\cR)} [-2\dim\bN_\cO-2\dim G_\cO]).
\end{equation*}
Since $\tilde q$ is the quotient by $G_\cO$, the right hand side is
\begin{equation*}
   (\id_{G_\cK}\times\pi)_* i'_*
   \tilde q^!
   \DC_{q(p^{-1}(\cR\times\cR))} [-2\dim\bN_\cO-2\dim G_\cO])
   \cong \bar q^* \bar \pi_* \bar i_*
      \DC_{q(p^{-1}(\cR\times\cR))} [-2\dim\bN_\cO].
\end{equation*}
Applying $(\bar q^*)^{-1}$, we get a homomorphism
\begin{equation}\label{eq:58}
    \scA\tilde\boxtimes\scA =
    (\bar q^*)^{-1} \bar p^* (\scA\boxtimes\scA)
    \to \bar\pi_* \bar i_*
    \DC_{q(p^{-1}(\cR\times\cR))} [-2\dim\bN_\cO].
\end{equation}
We further apply $\bar m_*$:
\begin{equation*}
    \scA\star\scA = \bar m_*(\scA\tilde\boxtimes\scA)
    \to
    \pi_* i_* \tilde m_* \DC_{q(p^{-1}(\cR\times\cR))} [2\dim\bN_\cO].
\end{equation*}
The left hand side is nothing but the convolution product
$\scA\star\scA$ defined by the diagram \eqref{eq:1}.

Since $\tilde m$ is proper, we have a natural homomorphism $\tilde
m_*\DC_{q(p^{-1}(\cR\times\cR))} [2\dim\bN_\cO] \to
\DC_\cR[2\dim\bN_\cO]$. Thus we obtain the homomorphism in (1).

Proofs of (2),(3) are already given in the proof of
\thmref{thm:convolution}. Note that the associativity isomorphism is
given by the $\Gr_G$-version of the big square diagram appearing in the
proof of \thmref{thm:convolution}. See \cite[Prop.~4.6]{MV2}.

Taking hypercohomology groups, one can check (4). We omit the detail.
\begin{NB}
    Let us give a detail of the proof of (4).

    Let us omit shifts for brevity. Recall we have constructed
    homomorphisms
    \begin{align}
        \label{eq:29}
        & \bar p^* (\pi\times\pi)_*\DC_{\cR\times\cR} \to
        (\id_{G_\cK}\times\pi)_* i'_* \DC_{p^{-1}(\cR\times\cR)},
        \\
        \label{eq:31}
        & (\bar q^*)^{-1} \bar p^* (\pi\times\pi)_*\DC_{\cR\times\cR}
        \to \bar\pi_*\bar i_*\DC_{q(p^{-1}(\cR\times\cR))},
        \\\label{eq:32} & \scA\star\scA =
        \bar m_*(\bar q^*)^{-1} \bar p^*
        (\pi\times\pi)_*\DC_{\cR\times\cR} \to \pi_* i_* \tilde m_*
        \DC_{q(p^{-1}(\cR\times\cR))},
        \\\label{eq:30} & \tilde m_*\DC_{q(p^{-1}(\cR\times\cR))} \to
        \DC_{\cR}.
    \end{align}
    We have an obvious commutative diagram:
    \begin{equation*}
    \begin{CD}
        \Ext^*
        (
        \CC_{\Gr_G\times\Gr_G},
        (\pi\times\pi)_* \DC_{\cR\times\cR}
        )
        @>{(\pi\times\pi)_*\eqref{eq:7}}>>
        \Ext^*
        (       \CC_{\Gr_G\times\Gr_G},
        \bar p_*
        (\id\times\pi)_*i'_*\DC_{p^{-1}(\cR\times\cR)}
        )
        \\
        @V{\bar p^*}VV @VV{\cong}V \\
        \Ext^*
        (
        \CC_{G_\cK\times\Gr_G},
        \bar p^*(\pi\times\pi)_* (\DC_{\cR\times\cR})
        )
        @>>{\eqref{eq:29}}>
        \Ext^*
        (
        \CC_{G_\cK\times\Gr_G},
        (\id\times\pi)_*i'_*\DC_{p^{-1}(\cR\times\cR)}
        )
    \end{CD}
    \end{equation*}
    from the adjunction. We suppress the derived categories
    $D_G(\Gr_G)$, etc from the notation, as they are clear from the
    context. Next note that \eqref{eq:31} is obtained from
    \eqref{eq:29} by applying $(\bar q^*)^{-1}$. Therefore we have a
    commutative diagram
    \begin{equation*}
        \begin{CD}
            \Ext^*
            ( \CC_{G_\cK\times\Gr_G}, \bar p^*(\pi\times\pi)_*
            (\DC_{\cR\times\cR})) @>\eqref{eq:29}>>
            \Ext^*
            ( \CC_{G_\cK\times\Gr_G},
            (\id\times\pi)_*i'_*\DC_{p^{-1}(\cR\times\cR)} )
            \\
            @A{\bar q^*}A{\cong}A @A{\bar q^*}A{\cong}A
            \\
            \Ext^*
            (
            \CC_{\Gr_G\tilde\times\Gr_G},
            (\bar q^*)^{-1} \bar p^*(\pi\times\pi)_*
            (\DC_{\cR\times\cR})) @>>{\eqref{eq:31}}>
            \Ext^*
            (
            \CC_{\Gr_G\tilde\times\Gr_G},
            \bar\pi_*\bar
            i_*\DC_{q(p^{-1}(\cR\times\cR))}).
        \end{CD}
    \end{equation*}
    Note that $\CC_{\Gr_G\tilde\times\Gr_G} = (\bar q^*)^{-1}\bar p^*
    \CC_{\Gr_G\times\Gr_G}$.
    Since \eqref{eq:32} is $\bar m_*\eqref{eq:31}$, we further have a
    commutative diagram
    \begin{equation*}
        \begin{CD}
            \Ext^*
            (
            \CC_{\Gr_G\tilde\times\Gr_G},
            (\bar q^*)^{-1} \bar p^*(\pi\times\pi)_*
            (\DC_{\cR\times\cR})) @>{\eqref{eq:31}}>>
            \Ext^*
            (
            \CC_{\Gr_G\tilde\times\Gr_G},
            \bar\pi_*\bar
            i_*\DC_{q(p^{-1}(\cR\times\cR))}).
            \\
            @V{\bar m_*}VV @VV{\bar m_*}V
            \\
            \Ext^*
            (
            \bar m_*\CC_{\Gr_G\tilde\times\Gr_G},
            \bar m_*(\bar q^*)^{-1} \bar
            p^*(\pi\times\pi)_* (\DC_{\cR\times\cR}))
            @>>\eqref{eq:32}>
            \Ext^*
            (
            \bar m_*\CC_{\Gr_G\tilde\times\Gr_G},
            \pi_* i_* \tilde m_*
            \DC_{q(p^{-1}(\cR\times\cR))}).
            \\
            @| @|
            \\
            \Ext^*
            (
            \CC_{\Gr_G}\star\CC_{\Gr_G},
            \scA\star\scA)
            @>>\eqref{eq:32}>
            \Ext^*
            (
            \bar m_*\CC_{\Gr_G\tilde\times\Gr_G},
            \pi_* i_* \tilde m_*
            \DC_{q(p^{-1}(\cR\times\cR))}).
        \end{CD}
    \end{equation*}
    Next use $\CC_{\Gr_G}\to\CC_{\Gr_G}\star\CC_{\Gr_G}$:
    \begin{equation*}
        \begin{CD}
            \Ext^*
            (
            \CC_{\Gr_G}\star\CC_{\Gr_G},
            \scA\star\scA)
            @>\eqref{eq:32}>>
            \Ext^*
            (
            \bar m_*\CC_{\Gr_G\tilde\times\Gr_G},
            \pi_* i_* \tilde m_*
            \DC_{q(p^{-1}(\cR\times\cR))})
            \\
            @VVV @VVV
            \\
            \Ext^*
            (
            \CC_{\Gr_G},
            \scA\star\scA)
            @>>\eqref{eq:32}>
            \Ext^*
            (
            \CC_{\Gr_G},
            \pi_* i_* \tilde m_*
            \DC_{q(p^{-1}(\cR\times\cR))})
        \end{CD}
    \end{equation*}
    Finally \eqref{eq:30} induces the proper pushforward
    \begin{equation*}
     H^{G_\cO}_*(q(p^{-1}(\cR\times\cR))) =
     \Ext^*
     (\CC_{\Gr_G},
     \pi_* i_* \tilde m_*
     \DC_{q(p^{-1}(\cR\times\cR))})
     \to
     \Ext^*
     (\CC_{\Gr_G},
     \pi_*i_* \DC_\cR = \scA)
     = H^{G_\cO}_*(\cR).
    \end{equation*}
    Spaces in the right column are identified with equivariant
    Borel-Moore homology groups appearing in the original definition
    of $\ast$, e.g., $H^{G_\cO\times G_\cO}_*(p^{-1}(\cR\times\cR))$,
    $H^{G_\cO}_*(q(p^{-1}(\cR\times\cR)))$, etc. Therefore arrows in
    the right column give the definition of $\ast$. An exception is
    $
    \Ext^*
    ( \bar m_*\CC_{\Gr_G\tilde\times\Gr_G}, \pi_* i_* \tilde m_*
    \DC_{q(p^{-1}(\cR\times\cR))})
    $. It appears as in the middle of the upper row:
    \begin{equation*}
        \xymatrix{
          \Ext^*
    (\CC_{\Gr_G\tilde\times\Gr_G},
    \mathscr B
    ) \ar[r]^-{\bar m_*} \ar@{=}[d] &
    \Ext^*
    ( \bar m_*\CC_{\Gr_G\tilde\times\Gr_G},
    \bar m_* \mathscr B
    )
    \ar[r]^-{\id\to \bar m_*\bar m^*} &
    \Ext^*
    (
    \CC_{\Gr_G},
    \bar m_* \mathscr B
    ) \ar@{=}[d]
    \\
    H^{G_\cO}_*(q(p^{-1}(\cR\times\cR))) \ar@{=}[rr] &&
    H^{G_\cO}_*(q(p^{-1}(\cR\times\cR))),
  }
\end{equation*}
where $\mathscr B = \bar\pi_* \bar i_*
\DC_{q(p^{-1}(\cR\times\cR))}$. We need to check that this is
commutative. It boils down to check that the composition of
\begin{equation*}
    \CC_{\Gr_G\tilde\times\Gr_G} \cong
    \bar m^* \CC_{\Gr_G} \xrightarrow{\id\to\bar m_*\bar m^*}
    \bar m^* \bar m_* \bar m^* \CC_{\Gr_G}
    \xrightarrow{\bar m^*\bar m_*\to\id}
    \bar m^* \CC_{\Gr_G} \cong    \CC_{\Gr_G\tilde\times\Gr_G}
\end{equation*}
is the identity. But this is obvious. Now the assertion is clear.
\end{NB}%
\end{proof}

\begin{Remarks}\label{rem:ABG-1}
    \textup{(1)} By \cite[Theorem~5]{MR2422266}, $\scA\in
    D_{G_\cO\rtimes\CC^\times}(\Gr_G)$ corresponds to a certain
    differential graded Harish-Chandra bimodule of $G^\vee$. We do not
    know anything about it except the example just below.

    \textup{(2)} Let us denote by $\Areg$ the regular sheaf, i.e.,
    the perverse sheaf corresponding to the regular representation
    $\CC[G^\vee]$ of the Langlands dual group $G^\vee$ under the
    geometric Satake correspondence. It was denoted by $\cR$ in
    \cite{MR2053952}, but it conflicts with our notation for the space
    $\cR$. It is endowed with a natural morphism
    $\mathbf m\colon \Areg\star \Areg\to \Areg$ with properties listed
    in \propref{prop:sheav-affine-grassm}.
    The nilpotent cone $\mathcal N$ of $G^\vee$ and its Springer resolution
    $\tilde{\mathcal N}$ were constructed from $\Areg$ in
    \cite{MR2053952}. Since it is more natural to compare $\Areg$ with
    $\tilscA$ arising in the framework of a flavor symmetry group,
    more detail will be given \subsecref{subsec:ABG}.
    Finally, the dg-Harish-Chandra bimodule corresponding to $\Areg$ is
    the ring $U_\hbar^{[]}\ltimes\BC[G^\vee]$ of $\hbar$-differential operators
    on $G^\vee$.
\end{Remarks}

The construction in this and the subsequent subsections shows that it
is enough to have $\scA$ with $\mathsf m\colon \scA\star\scA\to\scA$,
i.e., a ring object in $D_G(\Gr_G)$ to define the Coulomb branch
$\mathcal M_C$.
For example, $\Areg$. Since $\Areg$ for an
exceptional group is unlikely to arise from any gauge theory
$(G,\bN)$,
it is interesting to find other recipes to construct such an $(\scA,
\mathsf m)$.
We give one example of such a recipe in \subsecref{subsec:glue} below.

\subsection{Commutativity}

In this subsection we forget the loop rotation.

Let $\Theta\colon \scA\star\scA\to\scA\star\scA$ be the commutativity constraint of the convolution product. Its construction, following \cite[\S5]{MV2} and also \cite{MR1826370},
\begin{NB}
Note that \cite[5.3.9]{Beilinson-Drinfeld} is \emph{incomplete}.....
\end{NB}%
will be recalled in \subsecref{subsec:constraint}.

\begin{Theorem}\label{thm:commutative}
    We have $\mathsf m\circ\Theta \cong \mathsf m$ as homomorphism
    $\scA\star\scA\to\scA$.
\end{Theorem}

It means that $(\scA,\mathsf m)$ is a \emph{commutative} ring object
in $(D_G(\Gr_G),\star)$.
We give a proof in \secref{sec:commute}.

Our proof is indirect. We construct another multiplication
$\mathsf m^\psi\colon\scA\star\scA\to \scA$ using nearby cycle
functors and dual specialization. We have
$\mathsf m^\psi\circ\Theta\cong \mathsf m^\psi$. Therefore
$(\scA,\mathsf m^\psi)$ is {commutative}, but we cannot check
$\mathsf m^\psi$ is associative directly.

Next we show $\mathsf m^\psi = \mathsf m$ for $\bN=0$. This implies
that $\mathsf m^\psi = \mathsf m$ holds after the fixed point
localization for general $\bN$. We do not have torsion where
$\mathsf m$ and $\mathsf m^\psi$ live, hence this is enough.

\subsection{A complex on the affine Grassmannian of the flavor
  symmetry group}
\label{subsec:affG_flavor}

\begin{NB}
If we compare $\scA$ with $\Areg$, it is too small, as
$\Ext^*_{D_G(\Gr_G)}(\mathbf 1_{\Gr_G}, \Areg)$, instead
of $H^*_{G_\cO}(\cA) = \Ext^*_{D_G(\Gr_G)}(\CC_{\Gr_G},\scA)$, gives
the nilpotent cone. Here $\mathbf 1_{\Gr_G}$ is the skyscraper sheaf
at the base point in $\Gr_G$.
\begin{NB2}
    \cite[Lemma~7.6.3]{MR2053952} says
    $H^*_{G_\cO}(\Areg)$ gives $\mathfrak t/W\times G^\vee$.
\end{NB2}%
In our case, $\Ext^*_{D_G(\Gr_G)}(\mathbf 1_{\Gr_G}, \scA)$ is an
algebra too (see below), but it just give a trivial one.

Therefore we give a different construction of a ring object in
$D_G(\Gr_G)$.

For this purpose, w
\end{NB}%
We suppose that $\bN$ is a representation of a
larger group $\tilde G$ containing $G$ as a normal subgroup as in
\ref{subsec:flavor}, \ref{subsec:flav-symm-group2}. Let $G_F =
\tilde G/G$. We are going to construct a ring object in $D_{G_F}(\Gr_{G_F})$.

Let us denote $\cT_{\tilde G,\bN}$, $\cR_{\tilde G,\bN}$ by
$\tilde\cT$, $\tilde\cR$ respectively for short as before.
Composing $\tilde\cT\to\Gr_{\tilde G}$ or $\tilde\cR\to\Gr_{\tilde G}$
with the morphism $\Gr_{\tilde G}\to \Gr_{G_F}$, we have
\begin{equation}\label{eq:52}
    \tilde\pi\colon \text{$\tilde\cT$ or $\tilde\cR$}\to \Gr_{G_F}.
\end{equation}
Let us denote the fiber over $\lambda\in\Gr_{G_F}$ by $\tilde\cR^\lambda$,
where $\lambda$ is a coweight of $G_F$ regarded as a point in
$\Gr_{G_F}$. (In \ref{subsec:flav-symm-group2} it was denoted by
$\lambda_F$.)

\begin{NB}
  As an analog of \eqref{eq:4} (this equation number is in {\bf NB} of
  Coulomb2.tex, hence is undefined), we have
\begin{equation*}
    \cT\xleftarrow{p_1}
    G_\cK \times_{G_\cO} \tilde\cR \cong
    \cT\times_{\bN_\cK}\tilde \cT \xrightarrow{p_2}\tilde \cT.
\end{equation*}
This restricts to
\(
  \cT\leftarrow G_\cK \times_{G_\cO} \tilde\cR_x \cong
    \cT\times_{\bN_\cK}\tilde \cT_x \xrightarrow{p_2}\tilde \cT_x.
\)
Hence we have the convolution product
\begin{equation*}
    H_*^{G_\shfO}(\cR)\otimes H_*^{G_\shfO}(\tilde\cR_x)
    \to H_*^{G_\shfO}(\tilde\cR_x),
\end{equation*}
which makes $H_*^{G_\shfO}(\tilde\cR_x)$ as a
$H_*^{G_\shfO}(\cR)$-module.
\end{NB}%

\begin{NB}
    \begin{NB2}
        The following construction contains a \emph{mistake}, and
        hence does not work.
    \end{NB2}%

Let $\Gr_{G_F} = \bigcup_{\la_F} \overline{\Gr}_{G_F}^{\la_F}$ be the
stratification as in \subsecref{triples}, where $\la_F$ is a dominant
coweight of $G_F$. We set $\tilde\cR_{\le\la_F} =
\tilde\pi^{-1}(\overline{\Gr}_{G_F}^{\la_F})$. As in
\secref{sec:degeneration}, we have
\begin{equation*}
    H^{G_\cO}_*(\tilde\cR_{\le\la_F})\ast H^{G_\cO}_*(\tilde\cR_{\le\mu_F})
    \subset H^{G_\cO}_*(\tilde\cR_{\le \la_F+\mu_F}).
\end{equation*}
\begin{NB2}
    This seems \emph{wrong}, as we need to use a descent for a $\tilde
    G_\cO$-bundle. Therefore we need to use $H^{\tilde
      G_\cO}(\tilde\cR)$.
\end{NB2}%
Therefore the associated graded $\gr H^{G_\cO}_*(\tilde\cR)$ is an
algebra graded by dominant coweights of $G_F$. The associated graded
is identified with
\(
    \bigoplus_{\la_F} H^{G_\cO}_*(\tilde\cR_{\la_F}).
\)
In particular, the degree $0$ part $H^{G_\cO}_*(\tilde\cR_0)$ is the
inverse image $\tilde\pi^{-1}([1]_{G_F})$ of the base point $[1]_{G_F}$
in $\Gr_{G_F}$. It is nothing but the original $\cR$.

\begin{Theorem}\label{thm:line}
    $H^{G_\cO}_*(\tilde\cR)$ is a filtered algebra with respect to the
    filtration $H^{G_\cO}_*(\tilde\cR) = \bigcup_{\la_F}
    H^{G_{\cO}}_*(\tilde\cR_{\le\la_F})$. The degree $0$ part is $\cA =
    H^{G_\cO}_*(\cR)$. The same are true for $\tilde G_\cO$,
    $G_\cO\rtimes\CC^\times$ and $\tilde
    G_\cO\rtimes\CC^\times$-equvariant homology gruops.
\end{Theorem}

For a dominant coweight $\la_F$, we consider the direct sum of parts
with degrees in $\ZZ_{\ge 0}\la_F$. It is an algebra graded by
$\ZZ_{\ge 0}$. Its $\Proj(\bigoplus_{n\ge 0}
H^{G_\cO}_*(\tilde\cR_{n\la_F}))$ has a natural projective morphism to
$\cA$.  For $\tilde G_\cO$-equivariant homology groups, the same
construction gives a scheme $\Proj(\bigoplus_{n\ge 0} H^{\tilde
  G_\cO}_*(\tilde\cR_{n\la_F}))$ projective over the deformation of
$\cA$ parametrized by $H^*_{G_F}(\mathrm{pt})$.
\end{NB}

\begin{NB}
    In the quantized case, people usually constructed noncommutative
    sheaves on a (partial) resolution. Do we have such a construction
    ?
\end{NB}%


As in \propref{prop:sheav-affine-grassm}, we consider a pushforward of
the dualizing sheaf $\DC_{\tilde\cR}$. Here we consider the dualizing
sheaf of the larger space $\tilde\cR$, and take the pushforward
$\tilscA\defeq \tilde\pi_*\DC_{\tilde\cR}[-2\dim\bN_\cO]$ to
$\Gr_{G_F}$. We consider it as an object in $D_{\tilde G}(\Gr_{G_F})$,
an appropriate Ind-completion of the $\tilde G_\cO$-equivariant derived
constructible category of $\Gr_{G_F}$. We also have
$Q_{\tilde\pi*}\DC_{\tilde\cR}[-2\dim\bN_\cO] = Q_{\id*}\tilscA$,
which is a $(G_F)_\cO$-equivariant object on $\Gr_{G_F}$. Here `$\id$'
is the identity of $\Gr_{G_F}$ and the general pushforward functor
$Q_{\id*},Q_{\tilde\pi*}$ changes the equivariance group from $\tilde G_\cO$ to
$(G_F)_\cO$. See \cite[\S6]{BL}.
\begin{NB}
    In practice, $\tilde\pi_*\DC_{\tilde\cR}$ and
    $Q_{\tilde\pi*}\DC_{\tilde\cR}$ are not so much different. The
    group homomorphism $\tilde G\to G_F$ induces $B\tilde G\to
    BG_F$. Therefore we have the pull-back homomorphism
    $H^*_{G_F}(\mathrm{pt})\to H^*_{\tilde G}(\mathrm{pt})$, which is
    nothing but $\CC[\g_F]\to \CC[\tilde\g]$. Thus $H^{\tilde
      G_\cO}_*(\cR)$ can be considered as an
    $H_{G_F}^*(\mathrm{pt})$-module. $H^*_{G_F}(Q_{\tilde\pi*}\DC_{\tilde\cR})$
    is $H^*_{\tilde G}(\tilde\pi_*\DC_{\tilde\cR}) = H^{\tilde
      G}_*(\cR)$, considered as an
    $H_{G_F}^*(\mathrm{pt})$-module. See \cite[\S13.5]{BL}.
\end{NB}%

In the same way as in \propref{prop:sheav-affine-grassm}, we have
natural homomorphisms
\begin{equation}\label{eq:35}
        \mathsf m\colon \tilscA \star\tilscA\to\tilscA,
        \qquad
        \mathsf m\colon Q_{\id*}\tilscA\star Q_{\id*}\tilscA\to
        Q_{\id*}\tilscA,
\end{equation}
that satisfy the unit and associativity properties. It also satisfies
the commutativity.

Let us give a small remark for the construction of the homomorphisms:
When we define the convolution product $\tilscA\star\tilscA$, we
use $(q^*)^{-1}$ for $\Gr_{G_F}$. For this, we only need the
$(G_F)_\cO$-equivariant structure, therefore we can replace the second
factor $\tilscA$ by $Q_{\id*}\tilscA$. However in the definition
of the first homomorphism $\mathsf m$, we need to go back to the space
$\tilde\cR$, hence we need the $\tilde G_\cO$-equivariant structure.
The second homomorphism $\mathsf m$ is induced from the first by
applying $Q_{\id*}$ and using the smooth base change.

\begin{NB}
Instead of repeating the argument in
\subsecref{sec:sheav-affine-grassm}, we can also construct
homomorphisms \eqref{eq:35}, using $\pi'\colon\Gr_{\tilde
  G}\to\Gr_{G_F}$. Note that $\tilde\pi = \pi'\circ\pi_{\tilde G}$,
where $\pi_{\tilde G}\colon\tilde\cR\to\Gr_{\tilde G}$ is the
projection for $\tilde G$. Then the construction in
\subsecref{sec:sheav-affine-grassm} applied for $\cR$ gives $\mathsf
m\colon \pi_{\tilde G*}\DC_{\tilde\cR}[-2\dim\bN_\cO]\star \pi_{\tilde
  G*}\DC_{\tilde\cR}[-2\dim\bN_\cO]\to \pi_{\tilde
  G*}\DC_{\tilde\cR}[-2\dim\bN_\cO]$.

Now we use a functoriality of the convolution product under the group
homomorphism:
Let $D_{\tilde G}(\Gr_{G_F})$ (resp.\ $D_{G_F}(\Gr_{G_F})$ denote an approproiate
Ind-completion of the
$\tilde G_\cO$ (resp.\ $(G_F)_\cO$) equivariant constructible derived category.
Let $A$, $B\in D_{\tilde G}(\Gr_{G_F})$. We have $(q^*)^{-1}
p^*(\pi^{\prime}\times\pi^{\prime})^*(A\boxtimes B) \cong \Pi^*
(q^*)^{-1} p^*(A\boxtimes B)$, where $\Pi\colon \tilde
G_\cK\times_{\tilde G_\cO} \Gr_{\tilde G} \to
(G_F)_\cK\times_{(G_F)_\cO}\Gr_{G_F}$, and we have used $p$, $q$ for
both $\Gr_{G_F}$ and $\Gr_{\tilde G}$ for brevity. Then $\Pi$ and $m$
factor through the fiber product $X$ of $\Gr_{\tilde G}$ and
$(G_F)_\cK\times_{(G_F)_\cO}\Gr_{G_F}$:
\begin{equation*}
    \xymatrix{
      \tilde G_\cK\times_{\tilde G_\cO} \Gr_{\tilde G}
      \ar@/_/[ddr]_{\Pi} \ar@/^/[drr]^m
      \ar[dr]^{\Pi''} \\
      & X \ar[d]^{\Pi'} \ar[r]_{m'}
      & \Gr_{\tilde G} \ar[d]^\pi \\
      & \Gr_{G_F}\times_{(G_F)_\cO} \Gr_{G_F} \ar[r]_-{m} & \Gr_{G_F} }
\end{equation*}
Therefore
\begin{equation*}
    m_* \Pi^* (q^*)^{-1} p^* (A\boxtimes
    B)\cong m'_* \Pi''_* \Pi^{\prime\prime*}\Pi^{\prime*}(q^*)^{-1} p^*
    (A\boxtimes B).
\end{equation*}
Using $\id\to \Pi''_*\Pi^{\prime\prime*}$ and the base change,
\begin{NB2}
    (note $m'$, $m$ are proper, hence $m'_* = m'_!$, $m_* = m_!$)
\end{NB2}%
we get a homomorphism
\begin{equation}\label{eq:34}
    \pi^{\prime*}(A\star B)
    \to
    (\pi^{\prime*}A)\star (\pi^{\prime*}B).
\end{equation}
Now suppose $A = \pi'_* \tilde A$, $B = \pi'_*\tilde B$,
$C=\pi'_*\tilde C$. The natural morphism $\pi^{\prime*}\pi'_*\to\id$
gives us
\begin{equation*}
    \pi^{\prime*} (\pi'_*\tilde A \star\pi'_*\tilde B)
    \xrightarrow{\eqref{eq:34}}
    (\pi^{\prime*}\pi'_*\tilde A)\star (\pi^{\prime*}\pi'_*\tilde B)
    \to
    \tilde A\star \tilde B.
\end{equation*}
Thus a multiplication $\tilde A\star\tilde B\to \tilde C$ induces
$\pi'_*\tilde A \star\pi'_*\tilde B\to \pi'_*\tilde C$. Taking
$\tilde A = \tilde B = \tilde C = \pi_{\tilde
  G*}\DC_{\tilde\cR}[-2\dim\bN_\cO]$, and hence
$A=B=C=\pi'_*\DC_{\tilde\cR}[-2\dim\bN_\cO]$, we obtain the multiplication \eqref{eq:35}.
\end{NB}%

Let $\mathbf 1_{\Gr_{G_F}}$ be the skyscraper sheaf at the base point
in $\Gr_{G_F}$. As in \cite[\S7.2]{MR2053952} we have an algebra
structure on $\Ext^*_{D_{\tilde G}(\Gr_{G_F})}(\mathbf 1_{\Gr_{G_F}},
\tilscA)$: Let $x\in \Ext^i_{D_{\tilde G}(\Gr_{G_F})}(\mathbf
1_{\Gr_{G_F}}, \tilscA)$, $y\in\Ext^j_{D_{\tilde
    G}(\Gr_{G_F})}(\mathbf 1_{\Gr_{G_F}}, \tilscA)$.
\begin{NB}
    We consider the `derived morphism'
    $\star x\colon \tilscA = \tilscA\star \mathbf
    1_{\Gr_{G_F}}\to \tilscA\star\tilscA$. Then we define
    $y\cdot x$ as the composite $\mathsf m \circ (\star x)\circ y
    \colon \mathbf 1_{\Gr_{G_F}}\to \tilscA[i+j]$. But why the derived
    morphism $\star x$ is defined ? We need to check that
\[
p^*(\id_{\tilscA}\boxtimes x)\colon p^*(\tilscA\boxtimes\mathbf 1_{\Gr_{G_F}})\to
p^*(\tilscA\boxtimes \tilscA)[i]
\]
descends to
\[
(q^*)^{-1}p^*(A\boxtimes\mathbf 1_{\Gr_{G_F}}) \to (q^*)^{-1}p^*
(\tilscA\boxtimes\tilscA)[i].
\]
This is fine for $(G_F)_\cO$-equivariant complexes, but not clear for
non-equivariant complexes.

Alternatively ......
\end{NB}%
We consider $x\star y\in \Ext^{i+j}_{D_{\tilde G}(\Gr_{G_F})}(\mathbf
1_{\Gr_{G_F}}\star \mathbf 1_{\Gr_{G_F}}, \tilscA\star
\tilscA)$. We compose $\mathbf 1\colon \mathbf 1_{\Gr_{G_F}}\cong
\mathbf 1_{\Gr_{G_F}}\star \mathbf 1_{\Gr_{G_F}}$ and $\mathsf m\colon
\tilscA\star\tilscA\to\tilscA$, we get $\mathsf m(x\star
y)\mathbf 1\in \Ext^{i+j}_{D_{\tilde G}(\Gr_{G_F})}(\mathbf
1_{\Gr_{G_F}},\tilscA)$.

Note that ext-groups in $D_{G_F}$ and $D_{\tilde G_F}$ are isomorphic:
\[
   \Ext^*_{D_{G_F}(\Gr_{G_F})}(\mathbf 1_{\Gr_{G_F}},
Q_{\id*}\tilscA)
\cong
\Ext^*_{D_{\tilde G}(\Gr_{G_F})}(\mathbf 1_{\Gr_{G_F}},
\tilscA),
\]
where the right hand side is regarded as
$H^*_{G_F}(\mathrm{pt})$-module via $H^*_{G_F}(\mathrm{pt})\to
H^*_{\tilde G}(\mathrm{pt})$. See \cite[\S13.5]{BL}. Thus the
difference between $\tilscA$ and $Q_{\id*}\tilscA$ is not
essential, we omit $Q_{\id*}$ hereafter.

Since the fiber of $\tilde\pi\colon\tilde\cR\to\Gr_{G_F}$ at the base
point is our original $\cR$, we have a natural isomorphism
\begin{equation}
    \label{eq:33}
    \Ext^*_{D_{\tilde G}(\Gr_{G_F})}(\mathbf 1_{\Gr_{G_F}}, \tilscA)
    \cong H^{\tilde G_\cO}_*(\cR)
\end{equation}
of $H_{\tilde G}^*(\mathrm{pt})$-modules.

The definition of the multiplication on $\Ext^*_{D_{\tilde
    G}(\Gr_{G_F})}(\mathbf 1_{\Gr_{G_F}}, \tilscA)$ uses $\tilde G$
(or $G_F$) equivariance, as we use the descent $(q^*)^{-1}$.\footnote{We thank Roman
  Bezrukavnikov for a clarification of this point.
\begin{NB}Misha, please
  check.
\end{NB}%
} On the other hand, the multiplication on the right hand side
given in \propref{prop:deformation} descends to $H^{G_\cO}_*(\cR)$. In
fact, we will see that a simple modification of the definition gives a
multiplication on the left hand side with the group changed from $\tilde G$
to $G$ in \ref{subsec:line-bundles-via}.

\begin{Lemma}\label{lem:isom}
    The isomorphism \eqref{eq:33} respects the multiplication. The
    same is true for $\tilde G_\cO\rtimes\CC^\times$-equivariant
    groups.
\end{Lemma}

\begin{proof}
Let us consider a modification of the commutative diagram \eqref{eq:40}:
\begin{equation}\label{eq:44}
    \begin{CD}
        \tilde\cT\times\tilde\cR @<p<< \tilde G_\cK\times\tilde\cR
        @>q>> \tilde G_\cK\times_{\tilde G_\cO}\tilde\cR @>m>>
        \tilde\cT
        \\
        @V{\tilde\pi\times\tilde\pi}VV @V{\xi\times\tilde\pi}VV
        @VV{\bar\pi}V @VV{\tilde\pi}V
        \\
        \Gr_{G_F}\times\Gr_{G_F} @<<\bar p<
        (G_F)_\cK\times\Gr_{G_F} 
        @>>\bar q> \Gr_{G_F}\tilde\times\Gr_{G_F} @>>\bar m> \Gr_{G_F},
    \end{CD}
\end{equation}
where $\xi\colon\tilde G_\cK\to (G_F)_\cK$ is a morphism induced from
$\tilde G\to G_F$, and all other maps are given by replacing $G$,
$\cR$, ... by $\tilde G$, $\tilde\cR$, ..., and composing $\Gr_{\tilde
  G}\to\Gr_{G_F}$, etc. We omit the first row for brevity.

Let $[1_{G_F}]$ denote the base point in $\Gr_{G_F}$. We take the
inverse images of $[1_{G_F}]\times[1_{G_F}]$,
$(G_F)_\cO\times[1_{G_F}]$, $[1_{G_F}]\star[1_{G_F}]$, $[1_{G_F}]$ in
the first row. They are $\cT\times\cR$, $\Iw\times\cR$,
$\Iw\times_{\tilde G_\cO}\cR$, $\cT$ respectively. Here $\Iw =
\xi^{-1}((G_F)_\cO)$ is the group introduced in
\subsecref{subsec:flavor}. Thus we recover the diagram \eqref{eq:43}. Now the assertion is easy to check, and hence we omit the detail.
\begin{NB}
Let us omit shifts for brevity. In the same way as in the proof of
\propref{prop:sheav-affine-grassm}, we have homomorphisms
\begin{align}
    \label{eq:36}
    & \bar p^* (\tilde\pi\times\tilde\pi)_*\DC_{\tilde\cR\times\tilde\cR} \to
    (\xi\times\tilde\pi)_* i'_* \DC_{p^{-1}(\tilde\cR\times\tilde\cR)},
    \\
    \label{eq:37}
    & (\bar q^*)^{-1} \bar p^* (\tilde\pi\times\tilde\pi)_*
    \DC_{\tilde\cR\times\tilde\cR} \to
    \bar\pi_*\bar i_*\DC_{q(p^{-1}(\tilde\cR\times\tilde\cR))}, \\
    \label{eq:38} &
    \tilscA\star\tilscA = \bar m_*(\bar q^*)^{-1} \bar p^*
    (\tilde\pi\times\tilde\pi)_*\DC_{\tilde\cR\times\tilde\cR} \to
    \tilde\pi_* i_* \tilde m_*
    \DC_{q(p^{-1}(\tilde\cR\times\tilde\cR))},
    \\
    \label{eq:39} & \tilde m_*\DC_{q(p^{-1}(\tilde\cR\times\tilde\cR))} \to
    \DC_{\tilde\cR}.
\end{align}
We have an obvious commutative diagram
\begin{equation*}
    \begin{CD}
        \Ext^*
        (\CC_{[1_{{G_F}}]\times[1_{{G_F}}]},
        (\tilde\pi\times\tilde\pi)_* \DC_{\tilde\cR\times\tilde\cR} )
        @>{(\pi\times\pi)_*\eqref{eq:7}}>>
        \Ext^*
        (\CC_{[1_{{G_F}}]\times[1_{{G_F}}]},
        \bar p_* (\xi\times\tilde\pi)_*i'_*
        \DC_{p^{-1}(\tilde\cR\times\tilde\cR)} )
        \\
        @V{\bar p^*}VV @VV{\cong}V \\
        \Ext^*
        ( \CC_{(G_F)_\cO\times[1_{{G_F}}]},
        \bar p^*(\pi\times\pi)_* (\DC_{\tilde\cR\times\tilde\cR}) )
        @>>{\eqref{eq:36}}>
        \Ext^*
        (\CC_{(G_F)_\cO\times[1_{{G_F}}]},
        (\xi\times\tilde\pi)_*i'_*\DC_{p^{-1}(\tilde\cR\times\tilde\cR)} )
    \end{CD}
\end{equation*}
from the adjunction.
\begin{NB2}
    $\bar p^*(\CC_{[1_{{G_F}}]\times[1_{{G_F}}]})
    = \CC_{(G_F)_\cO\times[1_{{G_F}}]}$.
\end{NB2}%
Next note that \eqref{eq:37} is obtained from
\eqref{eq:36} by applying $(\bar q^*)^{-1}$. Therefore we have a
commutative diagram
\begin{equation*}
    \begin{CD}
        \Ext^*
        ( \CC_{(G_F)_\cO}\boxtimes \mathbf 1_{\Gr_{G_F}},
        \bar p^*(\tilde\pi\times\tilde\pi)_* (\DC_{\tilde\cR\times\tilde\cR}) )
        @>\eqref{eq:36}>>
        \Ext^*
        (\CC_{(G_F)_\cO}\boxtimes \mathbf 1_{\Gr_{G_F}},
        (\xi\times\tilde\pi)_*i'_*\DC_{p^{-1}(\tilde\cR\times\tilde\cR)})
        \\
        @A{\bar q^*}A{\cong}A @A{\bar q^*}A{\cong}A
        \\
        \Ext^*
        ( \CC_{[1_{G_F}]\star[1_{G_F}]},
        (\bar q^*)^{-1}\bar p^*(\pi\times\pi)_* (\DC_{\tilde\cR\times\tilde\cR}) )
        @>>{\eqref{eq:37}}>
        \Ext^*
        ( \CC_{[1_{G_F}]\star[1_{G_F}]},
        \bar\pi_*\bar i_*\DC_{q(p^{-1}(\tilde\cR\times\tilde\cR))}).
    \end{CD}
\end{equation*}
Since \eqref{eq:38} is $\bar m_*\eqref{eq:37}$, we further have a
commutative diagram
\begin{equation*}
    \begin{CD}
        \Ext^*
        ( \CC_{[1_{G_F}]\star[1_{G_F}]}, (\bar q^*)^{-1}\bar
        p^*(\pi\times\pi)_* (\DC_{\tilde\cR\times\tilde\cR}) )
        @>{\eqref{eq:37}}>> \Ext^*
        ( \CC_{[1_{G_F}]\star[1_{G_F}]}, \bar\pi_*\bar
        i_*\DC_{q(p^{-1}(\tilde\cR\times\tilde\cR))}).
        \\
        @V{\bar m_*}VV @VV{\bar m_*}V
        \\
        \Ext^*
        (\mathbf 1_{\Gr_{G_F}},
        \tilscA\star\tilscA
        ) @>>\eqref{eq:38}>
        \Ext^*
        (\mathbf 1_{\Gr_{G_F}},
        \pi_* i_*
        \tilde m_* \DC_{q(p^{-1}(\tilde\cR\times\tilde\cR))}).
    \end{CD}
\end{equation*}
Here we have used $\bar m_*(\CC_{[1_{G_F}]\star [1_{G_F}]}) =
\mathbf 1_{\Gr_{G_F}},\star \mathbf 1_{\Gr_{G_F}} = \mathbf 1_{\Gr_{G_F}}$.

Finally \eqref{eq:39} induces the pushforward
\begin{equation*}
    \Ext^*
    (\mathbf 1_{\Gr_{G_F}},
    \pi_* i_* \tilde m_*
    \DC_{q(p^{-1}(\tilde\cR\times\tilde\cR))})
    \to
    \Ext^*
    (\mathbf 1_{\Gr_{G_F}},
    \tilde\pi_*i_* \DC_{\tilde\cR}=\tilscA).
\end{equation*}

Let us recall how \eqref{eq:33} is given.
Let $i_{1_{G_F}}\colon \{ [1_{G_F}]\}\to \Gr_{G_F}$ be the inclusion.
Then $\mathbf 1_{\Gr_{G_F}} = i_{1_{G_F}!}(\CC_{[1_{G_F}]})$.
We have
\(
    \Ext^*(\mathbf 1_{Gr_{G_F}},\tilde\pi_* i_*\DC_{\tilde\cR})
    = \Ext^*(\CC_{[1_{G_F}]}, i_{1_{G_F}}^!\tilde\pi_* i_*\DC_{\tilde\cR}).
\)
We have $i_{1_{G_F}}^!\tilde\pi_* i_*\DC_{\tilde\cR} \cong \tilde\pi_* i_*\DC_{\cR}$ by the base change and the observation
$(\tilde\pi\circ i)^{-1}([1_{G_F}]) = \cR$. Thus we have
\(
    \Ext^*(\mathbf 1_{Gr_{G_F}},\tilde\pi_* i_*\DC_{\tilde\cR})
    \cong H^{\tilde G_\cO}_*(\cR).
\)

This argument applies also Ext-groups in the right columns, and we get equivariant Borel-Moore homology groups appearing in the definition of the convolution product on $H^{\tilde G_\cO}_*(\cR)$ in \subsecref{subsec:flavor}. Now the assertion is clear.
\end{NB}%
\end{proof}

\begin{NB}
\begin{Remark}
    For a quiver gauge theory of type $A_{N-1}$ with $\dim V =
    (N-1,N-2,\dots,1)$, $\dim W = (N,0,\dots,0)$ with $G_F = \GL(W) =
    \GL(N)$, $\tilscA$ is a complex on $\Gr_{\GL(N)}$. We also know
    that the Coulomb branch $\mathcal M_C$ is the nilpotent cone in
    $\gl(N)$.
    (We know that $\mathcal M_C$ is a transversal slice in the affine
    Grassmannian by \secref{QGT} for a quiver gauge theory of type
    ADE. And in this case the transversal slice in the affine
    Grassmannian is the nilpotent cone by \cite{Lu-Green}. See also
    \cite{MR1968260}.)
    As $\Ext^*_{D(\Gr_{\GL(N)})}(\mathbf 1_{\Gr_{GL(N)}},
    \Areg)$ gives also the nilpotent cone
    \cite[7.3.1]{MR2053952}, it is natural to conjecture that
    $\Areg = \tilscA$, i.e., $\tilscA$ is the
    $\GL(N)_\cO$-equivariant perverse sheaf on $\Gr_{\GL(N)}$
    corresponding to the regular representation $\CC[\GL(N)]$ under
    the geometric Satake equivalence.
\end{Remark}
\end{NB}

\subsection{An alternative construction of a regular sheaf}\label{subsec:ABG}

Consider a quiver gauge theory of type $A_{N-1}$ with
$\dim V = (N-1,N-2,\dots,1)$, $\dim W = (N,0,\dots,0)$ with
$G = \GL(V) = \prod_{i=1}^{N-1} \GL(i)$,
$\tilde G = \left(\GL(V)\times \GL(W)\right)/Z$, where
$Z\cong \CC^\times$ is the diagonal central subgroup. We have
$G_F = \PGL(W) = \PGL(N)$ and apply the above construction to define
$\tilscA$. It is a complex on $\Gr_{\PGL(N)}$. We also know that the
Coulomb branch $\cM_C$ of this quiver gauge theory is the nilpotent
cone in $\algsl(N)$.
(We know that $\mathcal M_C$ is a transversal slice in the affine
Grassmannian by \secref{QGT} for a quiver gauge theory of type
ADE. And in this case the transversal slice in the affine Grassmannian
is the nilpotent cone by \cite{Lu-Green}. See also \cite{MR1968260}.)
Recall $\Areg$ in \cref{rem:ABG-1}(2). We take $G=\PGL(N)$.
Then $\Ext^*_{D(\Gr_{\PGL(N)})}(\mathbf 1_{\Gr_{PGL(N)}}, \Areg)$ gives
also the nilpotent cone \cite[7.3.1]{MR2053952}.
This is not a coincidence. We have
\begin{Theorem}\label{thm:ABG}
  $\Areg$ and $\tilscA$ are isomorphic as ring objects in
  $D_{\PGL(N)}(\Gr_{\PGL(N)})$.
\end{Theorem}

The proof will be given in \ref{monopole}.

\subsection{Line bundles via homology groups of fibers}\label{subsec:line-bundles-via}

We now return back to a general situation: we are given a commutative
ring object in $D_G(\Gr_G)$, i.e., we are given $\scA\in D_G(\Gr_G)$
with $1\colon \mathbf 1_{\Gr_G}\to \scA$,
$\mathsf m\colon \scA\star\scA\to\scA$ satisfying the unit and
associativity properties in \propref{prop:sheav-affine-grassm}(2) and
the commutativity as in \cref{thm:commutative}.
The object constructed in \subsecref{subsec:pushf-affine-grassm}, as
well as the object $\tilscA$ or $Q_{\id*}\tilscA$ in
\ref{subsec:affG_flavor} is an example when we regard $G_F$ as $G$. In
fact, the latter is our primary example.

Let $D(\Gr_{G})$ denote an appropriate Ind-completion of the constructible
derived category on $\Gr_{G}$ (without $G_\cO$-equivariance structure).
Let $\For\colon D_G(\Gr_G)\to D(\Gr_G)$ be the forgetful functor.

\begin{Remark}
In the setting of \ref{subsec:affG_flavor}, we could consider
$D_G(\Gr_{G_F})$, an appropriate Ind-completion of the $G_\cO$-equivariant
constructible derived category on $\Gr_{G_F}$. Note that $G_\cO$ acts
trivially on $\Gr_{G_F}$.
Let $\Res_{G_\cO,\tilde G_\cO}$ be the restriction functor $D_{\tilde
  G}(\Gr_{G_F})\to D_{G}(\Gr_{G_F})$ restricting the group action from
$\tilde G_\cO$ to $G_\cO$.
Then we could consider $\scAres = \Res_{G_\cO,\tilde G_\cO}
\tilscA\in D_G(\Gr_{G_F})$.
This allows us to consider $\Ext^*_{D_{G}(\Gr_{G_F})}(\mathbf
1_{\Gr_{G_F}},\scAres)$, but the difference between this Ext group and
$\Ext^*_{D(\Gr_{G_F})}(\mathbf 1_{\Gr_{G_F}},\For Q_{\id*}\tilscA)$
is not essential as we have remarked above. Therefore we do not keep
two groups $G$, $G_F$, and just consider the above situation for
brevity of the notation.
\end{Remark}

Let $\scAfor \defeq \For \scA$.
Note that $\scAfor\star \scAfor$ is not defined as we do not have
$(q^*)^{-1}$ for non $G_\cO$-equivariant objects. However we still
have $\For\mathsf m\colon \For(\scA\star \scA)\to \scAfor = \For\scA$.

Viewing a coweight $\la$ of $G$ as a point in $\Gr_{G}$, we denote the
embedding by $i_{\la}\colon \{\la\}\to \Gr_{G}$.
\begin{NB}
    Note that $i_{\la}$ is not $G_\cO$-equivariant (unless $G$
    is abelian), therefore $i_{\la}^! \scA$ is \emph{not} defined.
\end{NB}%
\begin{NB}
Viewing a coweight $\laF$ of $G_F$ as a point in $\Gr_{G_F}$, we
denote the embedding by $i_{\laF}\colon \{\laF\}\to \Gr_{G_F}$.
\begin{NB2}
    Note that $i_{\laF}$ is not $(G_F)_\cO$-equivariant (unless $G_F$
    is abelian), therefore $i_{\laF}^! \scA$ is \emph{not} defined.
\end{NB2}%
Let $\tilde\cT^{\laF}$ or $\tilde\cR^{\laF} = \tilde\pi^{-1}(\laF)$
be a fiber of \eqref{eq:52} at a coweight $\laF$ of $G_F$.
It is preserved under the action of $G_\cO$.
\begin{NB2}
    $\tilde\pi$ is equivariant under $\tilde G_\cO$ where $\tilde
    G_{\cO}$ acts on $\Gr_{G_F}$ through the homomorphism $\tilde
    G_{\cO}\to (G_F)_{\cO}$. Since $G_\cO$ is in the kernel of this
    homomorphism, it acts on a fiber.
\end{NB2}%
We have $H^*(i_{\laF}^! \scAres) \cong H^{G_\cO}_*(\tilde\cR^{\laF})$.
\end{NB}%

Recall $m\colon \Gr_{G}\star \Gr_{G}\to \Gr_{G}$. For a
coweight $\chi$, let $\Gr^2_{\chi} \defeq m^{-1}(\chi)$ and denote
the embedding $\Gr^2_{\chi}\to \Gr_{G}\star \Gr_{G}$ by
$j_{\chi}$.
We have the base change $i_\chi^! m_* = m_* j_\chi^!$.
\begin{NB}
Recall $\bar m\colon \Gr_{G_F}\star \Gr_{G_F}\to \Gr_{G_F}$. For a
coweight $\chi$, let $\Gr^2_{\chi} \defeq\bar m^{-1}(\chi)$ and denote
the embedding $\Gr^2_{\chi}\to \Gr_{G_F}\star \Gr_{G_F}$ by
$j_{\chi}$.
We have the base change $i_\chi^! \bar m_* = \bar m_* j_\chi^!$.
\end{NB}%

Recall $\scA\star\scA = m_* (q^*)^{-1} p^* (\scA\boxtimes\scA)$. Let
us set $\scA\tilde\boxtimes\scA = (q^*)^{-1}p^* (\scA\boxtimes\scA)$.
As the forgetful functor commutes with $m_*$, we have
$\For(\scA\star \scA) = m_* \For
(\scA\tilde\boxtimes\scA)$. We have
\begin{equation}\label{eq:45}
    m_* j_\chi^! \For
    (\scA\tilde\boxtimes\scA)
    = i_\chi^! m_* \For
    (\scA\tilde\boxtimes\scA)
    = i_\chi^! \For (\scA\star\scA)
    \xrightarrow{i_\chi^! \For\mathsf m}
    i_\chi^! \scAfor.
\end{equation}
\begin{NB}
Recall $\tilscA\star\tilscA = \bar m_* (\bar q^*)^{-1}\bar p^*
(\tilscA\boxtimes\tilscA)$. Let us set
$\tilscA\overline\boxtimes\tilscA = (\bar q^*)^{-1}\bar p^*
(\tilscA\boxtimes\tilscA)$.
As the restriction functor commutes with $\bar m_*$, we have
$\Res_{G_\cO,\tilde G_\cO}(\tilscA\star \tilscA) = \bar m_*
\Res_{G_\cO,\tilde G_\cO}
(\tilscA\overline\boxtimes\tilscA)$. We have
\begin{equation}\label{eq:45NB}
    \bar m_* j_\chi^! \Res_{G_\cO,\tilde G_\cO}
    (\tilscA\overline\boxtimes\tilscA)
    = i_\chi^! \bar m_* \Res_{G_\cO,\tilde G_\cO}
    (\tilscA\overline\boxtimes\tilscA)
    \xrightarrow{i_\chi^! \Res_{G_\cO,\tilde G_\cO}\mathsf m}
    i_\chi^! \scAres.
\end{equation}
\end{NB}%

\begin{Claim}
    The embedding $\{ \la\}\times \{\mu\} \to \Gr^2_{\la+\mu}$
    induces a natural homomorphism
    \begin{equation}\label{eq:46}
        H^*(i_\la^! \scAfor)\otimes
        H^*(i_\mu^! \scAfor)
        \to H^*(j_{\la+\mu}^!
        \For(\scA\tilde\boxtimes\scA)).
    \end{equation}
\end{Claim}

\begin{NB}
    The restriction functor commutes with $!$-pull back
    \cite[\S3.5]{BL}, but $j_{\la+\mu}$ is not
    $G_\cO$-equivariant. Hence we do not have
    \(
    j_{\la+\mu}^!\For
    (\scA\tilde\boxtimes\scA)
    = \For j_{\la+\mu}^!
    (\scA\tilde\boxtimes\scA).
  \)
\end{NB}%

\begin{NB}
\begin{Claim}
    The embedding $\{ \laF\}\times \{\muF\} \to \Gr^2_{\laF+\muF}$
    induces a natural homomorphism
    \begin{equation}\label{eq:46NB}
        H^*(i_\laF^! \scAres)\otimes
        H^*(i_\muF^! \scAres)
        \to H^*(j_{\laF+\muF}^!
        \Res_{G_\cO,\tilde G_\cO}(\tilscA\overline\boxtimes\tilscA)).
    \end{equation}
\end{Claim}

\begin{NB2}
    The restriction functor commutes with $!$-pull back
    \cite[\S3.5]{BL}, but $j_{\laF+\muF}$ is not
    $(G_F)_\cO$-equivariant. Hence we do not have
    \(
    j_{\laF+\muF}^!\Res_{G_\cO,\tilde G_\cO}
    (\tilscA\overline\boxtimes\tilscA)
    = \Res_{G_\cO,\tilde G_\cO}j_{\laF+\muF}^!
    (\tilscA\overline\boxtimes\tilscA).
  \)
\end{NB2}%
\end{NB}%

\begin{proof}
Let us regard $\la$ as an element in $G_\cK$ and denote the embedding
$\{\la\}\to G_\cK$ by $\tilde i_\la$. The morphism
\(
   q (\tilde i_\la \times i_\mu) \colon \{ \la\}\times \{\mu\}
 \to \Gr_G\tilde\times\Gr_G
\)
factors through $\Gr^2_{\la+\mu}$. Let us write the embedding
$k_{\la,\mu}\colon \{ \la\}\times \{\mu\} \to \Gr^2_{\la+\mu}$.

We note $k_{\la,\mu}^! j^!_{\la+\mu} = (\tilde i_\la\times
i_\mu)^! q^!  = (\tilde i_\la\times i_\mu)^! q^* [2\dim
G_\cO]$.  Since the forgetful functor commutes with pull back
homomorphisms \cite[\S3.4]{BL}, we get
\begin{equation}\label{eq:48}
    \begin{split}
        & k_{\la,\mu}^! j^!_{\la+\mu} \For 
        (\scA\tilde\boxtimes\scA)
        \\
        =\; & (\tilde i_\la\times i_\mu)^! p^*
        (\scAfor\boxtimes \scAfor) [2\dim G_\cO]
        \\
        =\; & (\tilde i_\la\times i_\mu)^! p^!
        (\scAfor\boxtimes \scAfor) =
        (i_\la\times i_\mu)^!(\scAfor\boxtimes
        \scAfor).
    \end{split}
\end{equation}
Since $k_{\la,\mu}$ is proper, we have a homomorphism
$k_{\la,\mu*}k_{\la,\mu}^!  = k_{\la,\mu!}k_{\la,\mu}^! \to
\id$. Now the assertion is clear.
\end{proof}

Combining \eqref{eq:45} with $\chi = \la+\mu$ and \eqref{eq:46}, we
obtain a multiplication
\begin{equation}\label{eq:47}
    H^*(i^!_\la\scAfor)\otimes H^*(i^!_\mu\scAfor)
    \to H^*(i^!_{\la+\mu}\scAfor)
    .
\end{equation}

\begin{NB}
\begin{proof}
Let us regard $\laF$ as an element in $(G_F)_\cK$ and denote the embedding
$\{\la_F\}\to (G_F)_\cK$ by $\tilde i_\laF$. The morphism
\(
   q (\tilde i_\laF \times i_\muF) \colon \{ \laF\}\times \{\muF\}
 \to \Gr_G\tilde\times\Gr_G
\)
factors through $\Gr^2_{\laF+\muF}$. Let us write the embedding
$k_{\laF,\muF}\colon \{ \laF\}\times \{\muF\} \to \Gr^2_{\laF+\muF}$.

We note $k_{\laF,\muF}^! j^!_{\laF+\muF} = (\tilde i_\laF\times
i_\muF)^! q^!  = (\tilde i_\laF\times i_\muF)^! q^* [2\dim
(G_F)_\cO]$.  Since the restriction functor commutes with pull back
homomorphisms \cite[\S3.4]{BL}, we get
\begin{equation}\label{eq:48NB}
    \begin{split}
        & k_{\laF,\muF}^! j^!_{\laF+\muF} \Res_{G_\cO,\tilde G_\cO}
        (\tilscA\overline\boxtimes\tilscA)
        \\
        =\; & (\tilde i_\laF\times i_\muF)^! \bar p^*
        (\scAres\boxtimes \scAres) [2\dim (G_F)_\cO]
        \\
        =\; & (\tilde i_\laF\times i_\muF)^!\bar p^!
        (\scAres\boxtimes \scAres) =
        (i_\laF\times i_\muF)^!(\scAres\boxtimes
        \scAres).
    \end{split}
\end{equation}
Since $k_{\laF,\muF}$ is proper, we have a homomorphism
$k_{\laF,\muF*}k_{\laF,\muF}^!  = k_{\laF,\muF!}k_{\laF,\muF}^! \to
\id$. Now the assertion is clear.
\end{proof}

Combining \eqref{eq:45NB} with $\chi = \laF+\muF$ and \eqref{eq:46NB}, we
obtain a multiplication
\begin{equation}\label{eq:47NB}
    H^*(i^!_\laF\scAres)\otimes H^*(i^!_\muF\scAres)
    \to H^*(i^!_{\laF+\muF}\scAres)
    .
\end{equation}
\end{NB}%

\begin{Remarks}\label{rem:abelianGrassmann}
    (1) Note that the embedding $i_\la$ is
    $T_\cO$-equivariant. Therefore we can use the restriction functor
    $\Res_{T_\cO,G_\cO}$ from $G_\cO$ to $T_\cO$ instead of the
    forgetful functor $\For$. Then the same construction gives a
    multiplication
\begin{equation}\label{eq:54}
    H^*_{T_\cO}(i^!_\la\Res_{T_\cO,G_\cO}\scA)\otimes
    H^*_{T_\cO}(i^!_\mu \Res_{T_\cO,G_\cO}\scA)
    \to H^*_{T_\cO}(i^!_{\la+\mu}\Res_{T_\cO,G_\cO}\scA).
\end{equation}

(2) Suppose $G = T$. Then $\Gr_T = \bigsqcup_{\la\in Y} \{\la\}$,
hence $H^*_{T_\cO}(\Gr_T,\scA) = \bigoplus_{\la\in Y}
H^*_{T_\cO}(i_\la^!\scA)$. The multiplication explained after
\propref{prop:sheav-affine-grassm} is $Y$-graded, hence gives
$H^*_{T_\cO}(i_\la^!\scA)\otimes H^*_{T_\cO}(i_\mu^!\scA) \to
H^*_{T_\cO}(i_{\la+\mu}^!\scA)$. It is clear that this
multiplication is same as \eqref{eq:54}.
\begin{NB}
    We have $\scA = \bigoplus_{\la\in Y} \scA_\la$, and $\mathsf m$
    gives $\scA_\la\star\scA_\mu\to \scA_{\la+\mu}$. We also have
    $\Gr_T\star \Gr_T\cong \Gr_T\times \Gr_T$ by $[g_1, [g_2]]\mapsto
    ([g_1],[g_2])$. Then
    \eqref{eq:46} is just the embedding $H^*(\scA_\la)\otimes
    H^*(\scA_\mu)\to
    \bigoplus_{\la'+\mu'=\la+\mu}H^*(\scA_{\la'})\otimes
    H^*(\scA_{\mu'})$.
\end{NB}%
\end{Remarks}

Suppose $\la=\mu=0$. We have a commutative diagram
\begin{equation}\label{eq:49}
    \begin{CD}
        \Ext^*_{D_{G}(\Gr_{G})}(\mathbf 1_{\Gr_{G}},\scA)
        \otimes
        \Ext^*_{D_{G}(\Gr_{G})}(\mathbf 1_{\Gr_{G}},\scA)
        @>{\mathsf m(\bullet\star\bullet)\mathbf 1}>>
        \Ext^*_{D_{G}(\Gr_{G})}(\mathbf 1_{\Gr_{G}},\scA)
\\
        @V{\For}VV @VV{\For}V
\\
        H^*(i_0^!\scAfor)
        \otimes
        H^*(i_0^!\scAfor)
        @>>\eqref{eq:47}>
        H^*(i_0^!\scAfor)
    \end{CD}
\end{equation}
via the isomorphism $\Ext_{D(\Gr_{G})}(\mathbf
1_{\Gr_{G}},\scAfor)\cong H^*(i_0^!\scAfor)$.

\begin{NB}
    $i_\la$ is $G_\cO$-equivariant if $\la=0$, hence
    $i_\la^! \scAfor = \For i_\la^!
    \scA$.
\end{NB}%

In fact, the only place we need to check is the commutativity of
\begin{equation*}
    \begin{CD}
        \Ext^*_{D_{G}(\Gr_{G})}(\mathbf 1_{\Gr_{G}},\scA)
        \otimes
        \Ext^*_{D_{G}(\Gr_{G})}(\mathbf 1_{\Gr_{G}},\scA)
        @>{(q^*)^{-1}p^*}>>
        \Ext^*_{D_{G}(\Gr_{G}\tilde\times\Gr_{G})}
        (\CC_{[1_{G}]\star[1_{G}]},
        \scA\tilde\boxtimes\scA)
\\
        @V{\For}VV @VV{\For}V
\\
        H^*(i_0^!\scAfor)
        \otimes
        H^*(i_0^!\scAfor)
        @>>\eqref{eq:47}>
        H^*(j_0^!\For(\scA\tilde\boxtimes\scA)),
    \end{CD}
\end{equation*}
where the right vertical arrow is defined as the embedding of
$[1_{G}]\star[1_{G}]$ into $\Gr_G\tilde\times\Gr_G$ factors
through $\Gr^2_0$.
\begin{NB}
    $\Gr^2_0 = \bar m^{-1}(0) = \{ [g_1, [g_2]]\in \Gr_G\tilde\times\Gr_G
    \mid g_1 g_2\in G_\cO\}$.
\end{NB}%
This commutativity is clear from \eqref{eq:48}.

In the setting of the previous subsection,
the upper row of \eqref{eq:49} is the same as the multiplication
on $H^{\tilde G_\cO}_*(\cR)$ by \lemref{lem:isom},
hence
the lower row is also the same as $\ast$ on $H^{G_\cO}_*(\cR)$. In
this sense the multiplication in \eqref{eq:47} is a generalization of
$\ast$.

\begin{NB}
When $\laF=\muF=0$, we have $\tilde\cR^0=\cR$. We have a commutative
diagram
\begin{equation}\label{eq:49NB}
    \begin{CD}
        \Ext^*_{D_{\tilde G}(\Gr_{G_F})}(\mathbf 1_{\Gr_{G_F}},\tilscA)
        \otimes
        \Ext^*_{D_{\tilde G}(\Gr_{G_F})}(\mathbf 1_{\Gr_{G_F}},\tilscA)
        @>{\mathsf m(\bullet\star\bullet)\mathbf 1}>>
        \Ext^*_{D_{\tilde G}(\Gr_{G_F})}(\mathbf 1_{\Gr_{G_F}},\tilscA)
\\
        @V{\Res_{G_\cO,\tilde G_\cO}}VV @VV{\Res_{G_\cO,\tilde G_\cO}}V
\\
        H^*(i_0^!\scAres)
        \otimes
        H^*(i_0^!\scAres)
        @>>\eqref{eq:47}>
        H^*(i_0^!\scAres)
    \end{CD}
\end{equation}
via the isomorphism $\Ext_{D_G(\Gr_{G_F})}(\mathbf
1_{\Gr_{G_F}},\scAres)\cong H^*(i_0^!\scAres)$.

\begin{NB2}
    $i_\laF$ is $\tilde G_\cO$-equivariant if $\laF=0$, hence
    $i_\laF^! \scAres = \Res_{G_\cO,\tilde G_\cO} i_\laF^!
    \tilscA$.
\end{NB2}%

In fact, the only place we need to check is the commutativity of
\begin{equation*}
    \begin{CD}
        \Ext^*_{D_{\tilde G}(\Gr_{G_F})}(\mathbf 1_{\Gr_{G_F}},\tilscA)
        \otimes
        \Ext^*_{D_{\tilde G}(\Gr_{G_F})}(\mathbf 1_{\Gr_{G_F}},\tilscA)
        @>{(\bar q^*)^{-1}\bar p^*}>>
        \Ext^*_{D_{\tilde G}(\Gr_{G_F}\tilde\times\Gr_{G_F})}
        (\CC_{[1_{G_F}]\star[1_{G_F}]},
        \tilscA\overline\boxtimes\tilscA)
\\
        @V{\Res_{G_\cO,\tilde G_\cO}}VV @VV{\Res_{G_\cO,\tilde G_\cO}}V
\\
        H^*(i_0^!\scAres)
        \otimes
        H^*(i_0^!\scAres)
        @>>\eqref{eq:47NB}>
        H^*(j_0^!\Res_{G_\cO,\tilde G_\cO}(\tilscA\overline\boxtimes
        \tilscA)),
    \end{CD}
\end{equation*}
where the right vertical arrow is defined as the embedding of
$[1_{G_F}]\star[1_{G_F}]$ into $\Gr_{G_F}\tilde\times\Gr_{G_F}$ factors
through $\Gr^2_0$.
\begin{NB2}
    $\Gr^2_0 = \bar m^{-1}(0) = \{ [g_1, [g_2]]\in \Gr_{G_F}\tilde\times\Gr_{G_F}
    \mid g_1 g_2\in (G_F)_\cO\}$.
\end{NB2}%
This commutativity is clear from \eqref{eq:48NB}.

Since the upper row of \eqref{eq:49NB} is the same as the multiplication
on $H^{\tilde G_\cO}_*(\cR)$ by \lemref{lem:isom}, the lower row is
also the same as $\ast$ on $H^{G_\cO}_*(\cR)$. In this sense the
multiplication in \eqref{eq:47NB} is a generalization of $\ast$.

\end{NB}%

Thus $\bigoplus H^*(i_\la^!\scAfor)$ is an algebra graded by the
coweight lattice of $G$.
For $\la = 0$, we have a subalgebra $H^*(i_0^!\scAfor)$, which is
isomorphic to $H^{G_\cO}_*(\cR)$ in the setting of the previous subsection.
One can also take a direct sum over \emph{dominant} coweights $\la$
of $G$.

For a fixed coweight $\la$, we consider the direct sum of
$H^*(i_{n\la}^!\scAfor)$ with degrees $n\la$ ($n\in \ZZ_{\ge 0}$). It is an
algebra graded by $\ZZ_{\ge 0}$. Its $\Proj(\bigoplus_{n\ge 0}
H^*(i_{n\la}^!\scAfor))$ has a natural projective morphism to $\Spec
(H^*(i_0^!\scAfor))$. We have a natural line bundle $\shfO(1)$ on
$\Proj(\bigoplus_{n\ge 0} H^*(i_{n\la}^!\scAfor))$ such that
$H^*(i_{n\la}^!\scAfor)$ is identified with the space of sections of
$\shfO(n) = \shfO(1)^{\otimes n}$.
Under some circumstances we expect $\Proj(\bigoplus_{n\ge 0}
H^*(i_{n\la}^!\scAfor))$ is a 
(partial) resolution of $\Spec (H^*(i_0^!\scAfor))$.

In the example in \remref{rem:ABG-1}, $\Areg$ gives the
Springer resolution of the nilpotent cone $\mathcal N$ of $G^\vee$,
the Langlands dual group of $G$. See \cite[8.5.2]{MR2053952}.

\begin{NB}
Thus $\bigoplus H^{G_\cO}_*(\tilde\cR^{\laF})$ is an algebra graded by
the coweight lattice of $G_F$.
For $\laF = 0$, we have a subalgebra $H^{G_\cO}_*(\tilde\cR^0)$,
which is isomorphic to $H^{\tilde G_\cO}_*(\cR)$.
One can also take a direct sum over \emph{dominant} coweights $\laF$
of $G_F$.

For a coweight $\laF$, we consider the direct sum of
$H^{G_\cO}_*(\tilde\cR^{\muF})$ with degrees $\muF\in \ZZ_{\ge
  0}\laF$. It is an algebra graded by $\ZZ_{\ge 0}$. Its
$\Proj(\bigoplus_{n\ge 0} H^{G_\cO}_*(\tilde\cR^{n\laF}))$ has a
natural projective morphism to $\Spec H^{\tilde G_\cO}_*(\cR)$,
deformation of $\mathcal M_c$. We have a natural line bundle
$\shfO(1)$ on $\Proj(\bigoplus_{n\ge 0}
H^{G_\cO}_*(\tilde\cR^{n\laF}))$ such that
$H^{G_\cO}_*(\tilde\cR^{n\laF})$ is identified with the space of
sections of $\shfO(n) = \shfO(1)^{\otimes n}$.
Under some circumstances we expect $\Proj(\bigoplus_{n\ge 0}
H^{G_\cO}_*(\tilde\cR^{n\laF}))$ is a 
(partial) resolution
of 
$\mathcal M_C$.

In the example in \remref{rem:ABG-1}(2), $\Areg$ gives the
Springer resolution of the nilpotent cone $\mathcal N$ of $\mathfrak
g_F^\vee$, the Langlands dual Lie algebra of $\mathfrak g_F$. See
\cite[8.5.2]{MR2053952}.
\end{NB}%

See \cite[\S5.1]{2015arXiv150303676N} (and also \ref{rem:HLmonopole})
for a physical origin of this construction.

\begin{Remark}
In view of \remref{rem:abelianGrassmann}(2), the construction in
\subsecref{subsec:flav-symm-group2} and the above construction is the
same for $\tilscA$ in \subsecref{subsec:affG_flavor}.
Here the construction in \subsecref{subsec:flav-symm-group2} is as
follows: Let us suppose $G\vartriangleleft \tilde G$ as in
\subsecref{subsec:affG_flavor} and further assume $G_F = \tilde G/G$
is a torus. Let us write $T_F = G_F$. The Coulomb branch
$\cM_C(\tilde G,\bN)$ for the larger group $\tilde G$ has an action of
$\pi_1(T_F)^\wedge = T_F^\vee$, and \ref{prop:reduction} says that
$\cM(G,\bN)$ is the Hamiltonian reduction of $\cM(\tilde G,\bN)$ by
$T_F^\vee$. Let us denote the moment map by $\mu_{T_F^\vee}$. The
hamiltoian reduction more precisely means the affine algebro-geometric
quotient $\mu_{T_F^\vee}^{-1}(0)\dslash T_F^\vee$.
If we have a cocharacter $\lambda_F$ of $T_F$, we view it
as a character of $T_F^\vee$ and consider the GIT quotient
$\mu_{T_F^\vee}^{-1}(0)\dslash_{\lambda_F} T_F^\vee$.
\end{Remark}

\subsection{Wakimoto sheaves}

The original definition of the multiplication \eqref{eq:47} in
\cite{MR2053952} was given by Wakimoto sheaves, and the above
definition is taken from the proof of
\cite[Th.~8.5.2]{MR2053952}. Although it is unnecessary, let us review
the construction for the sake of the reader.

Let $I$ be the Iwahori subgroup of $G_\cK$ and let $\Fl_{G} = G_\cK/I$
be the affine flag variety.
We have a smooth proper morphism
$\varpi\colon \Fl_{G}\twoheadrightarrow\Gr_{G}$ of ind-schemes.
Let $\cW_{\la}$ be the Wakimoto sheaf on $\Fl_{G}$ for $G$
corresponding to a coweight $\la$. See \cite[\S8]{MR2053952} for the
definition (due to Mirkovi\'c). By \cite[\S8.4]{MR2053952}, we have a
`multiplication'
\begin{equation}\label{eq:28}
        \mathsf E_{\la}\otimes \mathsf E_{\mu}\to \mathsf
        E_{\la+\mu},
        \qquad\qquad \mathsf E_{\la} =
        \Ext^*_{D_{I}(\Gr_{G})}(\mathbf 1_{\Gr_{G}},
        \cW_{\la}\star\scA),
\end{equation}
where $\star$ is the convolution product on $I$-equivariant
complexes on $\Fl_{G}$ and $\Gr_{G}$: Let $x\in\mathsf E_{\la}$,
$y\in\mathsf E_{\mu}$. We consider the composite
\begin{multline*}
    y\cdot x
    \colon \mathsf 1_{\Gr_{G}}\xrightarrow{y} \cW_{\mu}\star\scA
    = \cW_{\mu}\star\mathsf 1_{\Gr_{G}}\star \scA
    \xrightarrow{\cW_{\mu}\star x\star\scA}
    \cW_{\mu}\star\cW_{\la}\star\scA\star\scA
    \\
    = \cW_{\la+\mu}\star\scA\star\scA
    \xrightarrow{\mathsf m}
    \cW_{\la+\mu}\star\scA.
\end{multline*}
Note that $\cW_{\mu}\star x$ is well-defined as $x$ is an
$I$-equivariant homomorphism, and hence $\cW_{\mu}\boxtimes x$
descends for the morphism $q$.

\begin{NB}
    More generally, for $M_1, M_2\in D_{I}(\Gr_{G})$, we consider
    \begin{equation*}
        \mathsf E(M_1, M_2)
        \defeq \bigoplus_{\la\in Y(G)^+}
        \Ext^*_{D_{I}(\Gr_{G})}(M_1,
        \cW_{\lambda}\star M_2\star\scA).
    \end{equation*}
    The above construction gives us $\mathsf E(M_1,M_2)\otimes \mathsf
    E(M_2,M_3) \to \mathsf E(M_1,M_3)$. Thus $\mathsf E(\mathbf
    1_{\Gr_{G}},M)$ is a module of $\mathsf E(\mathbf
    1_{\Gr_{G}},\mathbf 1_{\Gr_{G}})$.
\end{NB}

We have an isomorphism $\mathsf E_\la \cong
H^*_{T_\cO}(i_\la^!\Res_{T_\cO,G_\cO}\scA)$ (see \cite[(8.7.2)]{MR2053952}),
and the above multiplication is the same as \eqref{eq:54}.

\subsection{Gluing construction}\label{subsec:glue}

One of motivations of \cite{Cremonesi:2014kwa} extending the monopole
formula from the Hilbert series of the coordinate ring of the Coulomb
branch $\mathcal M_C$ to the character of the space of sections of a
line bundle (see \remref{rem:HLmonopole}) is to write down the Hilbert
series of a complicated Coulomb branch from simpler ones. We use the
machinery prepared in earlier subsections to introduce the
corresponding construction at the level of commutative ring objects in
$D_G(\Gr_G)$.

The setting in \cite{Cremonesi:2014kwa} is as follows. Suppose that we
have a finite collection $\{ (G_i,\bN_i)\}$ ($i=1,2,\dots$) of gauge
theories sharing the common flavor symmetry group, i.e., $\bN_i$ is a
representation of a larger group $\tilde G_i$ containing $G_i$ as a
normal subgroup with $G_F = \tilde G_i/G_i$, independent of $i$. Then
we define $G$ as the fiber product of $\tilde G_i$ over $G_F$, and
$\bN = \bigoplus \bN_i$. The monopole formula for the Hilbert series
of the Coulomb branch of $(G,\bN)$ is given by extended monopole
formula for $(G_i,\bN_i)$. See also \cite[\S5(i)]{2015arXiv150303676N}
for a review.

An example is a star shaped quiver gauge theory, which is the $3d$
mirror of the Sicilian theory of type $A_{N-1}$, reviewed in
\cite[\S3(iii)]{2015arXiv150303676N}. See \ref{fig:star}. We have
three copies of type $A_{N-1}$ quiver gauge theory with
$\dim V = (N-1,N-2,\dots,1)$, $\dim W = (N,0,\dots,0)$ as in
\subsecref{subsec:ABG}. We divide the group $\GL(V) = \prod \GL(V_i)$
by the diagonal central subgroup $Z$ and take it as the gauge
group. The common flavor symmetry group is $G_F = \PGL(N)$.

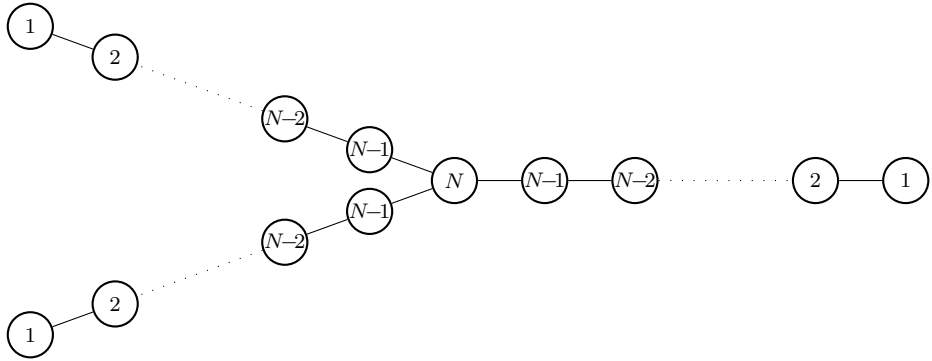
\begin{figure}[htbp]
    \centering
    \begin{tikzpicture}[scale=1.2,
circled/.style={circle,draw
,thick,inner sep=0pt,minimum size=.6cm},
squared/.style={rectangle,draw
,thick,inner sep=0pt,minimum size=6mm}
]
\node[circled] (v0) at ( 0,0)  {$\scriptstyle N$};
\node[circled] (v1) at (0:1)  {$\scriptstyle N\!-\!1$};
\draw [-] (v0) -- (v1);
\node[circled] (v2) at (0:2)  {$\scriptstyle N\!-\!2$};
\draw [-] (v1) -- (v2);
\node[circled] (vN-1) at (0:4) {$\scriptstyle 2$};
\draw [loosely dotted] (v2) -- (vN-1);
\node[circled] (vN) at (0:5) {$\scriptstyle 1$};
\draw [-] (vN-1) -- (vN);
\node[circled] (vs1) at (160:1)  {$\scriptstyle N\!-\!1$};
\draw [-] (v0) -- (vs1);
\node[circled] (vs2) at (160:2)  {$\scriptstyle N\!-\!2$};
\draw [-] (vs1) -- (vs2);
\node[circled] (vsN-1) at (160:4) {$\scriptstyle 2$};
\draw [loosely dotted] (vs2) -- (vsN-1);
\node[circled] (vsN) at (160:5) {$\scriptstyle 1$};
\draw [-] (vsN-1) -- (vsN);
\node[circled] (vn1) at (200:1)  {$\scriptstyle N\!-\!1$};
\draw [-] (v0) -- (vn1);
\node[circled] (vn2) at (200:2)  {$\scriptstyle N\!-\!2$};
\draw [-] (vn1) -- (vn2);
\node[circled] (vnN-1) at (200:4) {$\scriptstyle 2$};
\draw [loosely dotted] (vn2) -- (vnN-1);
\node[circled] (vnN) at (200:5) {$\scriptstyle 1$};
\draw [-] (vnN-1) -- (vnN);
\end{tikzpicture}
    \caption{A star shaped quiver gauge theory}
    \label{fig:star}
\end{figure}

The variety $\cR_{G,\bN}$ is the fiber product of $\cR_{\tilde
  G_i,\bN_i}$ over $\Gr_{G_F}$.
\begin{NB}
    $\cT_{G,\bN} = G_\cK\times_{G_\cO} \bN_\cO = \prod' (\tilde
    G_i)_\cK\times_{\prod' (\tilde G_i)_\cO} \bigoplus (\bN_i)_\cO$,
    where $\prod'$ is the fiber product over $G_F$, i.e., $\tilde
    G_1\times_{G_F} \tilde G_2 \times_{G_F} \tilde
    G_3\times\cdots$. Take its point $[(g_1,g_2,\cdots), (s_1\oplus
    s_2\oplus \cdots)]$. Then the isomorphism is given by sending it
    to $([g_1,s_1], [g_2,s_2],\cdots)$. In fact, as $g_1$, $g_2$, etc
    are sent to the same point in $G_F$, so the latter is a point in
    the fiber product. $\cR_{G,\bN}$ is given by the equation $g_i
    s_i\in\bN_i$ for all $i$. Therefore $\cR_{G,\bN}$ is the fiber
    product of $\cR_{G_i,\bN_i}$ over $\Gr_{G_F}$.
\end{NB}%
Let us denote the natural projections $\cR_{G,\bN} \to \Gr_{G_F}$ and
$\cR_{\tilde G_i,\bN_i}\to \Gr_{G_F}$ by $\pi$ and $\pi_i$
respectively. Then
\begin{equation*}
    \pi_*\DC_{\cR_{G,\bN}}[-2\dim\bN_\cO]
    = i_\Delta^!\left(
    \boxtimes \pi_{i*}\DC_{\cR_{\tilde G_i,\bN_i}}[-2\dim(\bN_i)_\cO]\right),
\end{equation*}
where $i_\Delta\colon \Gr_{G_F}\to \prod_i \Gr_{G_F}$ is the diagonal
embedding.
\begin{NB}
    Let us consider the cartesian square
    \begin{equation*}
        \begin{CD}
            X\times_Y X' @>{\tilde i}>> X\times X' \\ @V{\pi}VV @VV{p}V
            \\
            Y @>>{i_{\Delta Y}}> Y\times Y.
        \end{CD}
    \end{equation*}
    $\pi_* \DC_{X\times_Y X'} = \pi_* \tilde i^! \DC_{X\times X'}
    = i_{\Delta Y}^! p_* \DC_{X\times X'}$.
\end{NB}%
Note that $\pi_{i*}\DC_{\cR_{\tilde G_i,\bN_i}}[-2\dim(\bN_i)_\cO]$ is the commutative ring object in $D_{G_F}(\Gr_{G_F})$, considered in \subsecref{subsec:affG_flavor}.

Motivated by the above example, we consider the following setting. (We
use the convention in \subsecref{subsec:line-bundles-via}, i.e.,
replace $G_F$ by $G$.) Suppose that we have a finite collection $\{
\scA_i\}$ of commutative ring objects in $D_G(\Gr_G)$.
Let $i_\Delta\colon \Gr_G\to \prod_i \Gr_G$ be the diagonal
embedding. Then the following is clear:

\begin{Proposition}
\label{prop:!-product}
    $\scA \defeq i_\Delta^!\left(\boxtimes \scA_i\right)$ is a
    commutative ring object in $D_G(\Gr_G)$. In particular, we can
    consider the affine scheme
    \(
        \Spec H^*_{G_\cO}(\Gr_G, \scA).
    \)
\end{Proposition}

In fact, we have $\boxtimes\mathsf m\colon
(\boxtimes\scA_i)\star(\boxtimes\scA_i) = \boxtimes (\scA_i \star
\scA_i) \to \boxtimes \scA_i$ from $\mathsf m\colon
\scA_i\star\scA_i\to\scA_i$. Then we apply $i_\Delta^!$.
We claim that there is a natural homomorphism
\begin{equation}
  \label{eq:2.40.1}
i_\Delta^!(\boxtimes\scA_i)\star i_\Delta^!(\boxtimes\scA_i)
\to i_\Delta^! \left(\boxtimes (\scA_i\star\scA_i)\right),
\end{equation}
hence its composition with $i_\Delta^!(\boxtimes\mathsf m)$ gives the
desired multiplication homomorphism of $i_\Delta^!(\boxtimes\scA_i)$.
We prove the claim by comparing the convolution diagrams \eqref{eq:1}
for $\Gr_G$ and $\prod_i\Gr_G$. Since $p$, $q$ are smooth, $p^*$, $q^*$
commute with $i_\Delta^!$.
\begin{NB}
We change $p^*$, $q^*$ by $p^!$, $q^!$ up to shifts.

Here is an earlier attempt:

The commutative diagram for the multiplication $m$ factors as
\begin{equation*}
  \begin{CD}
    \Gr_G\tilde\times\Gr_G @>\alpha>> X @>m'>> \Gr_G \\
    @. @V{i'}VV @VV{i_\Delta}V \\
    @. \Gr_{\prod_i G}\tilde\times\Gr_{\prod_i G}
    @>>\prod_i m> \Gr_{\prod_i G} = \prod_i \Gr_G,
  \end{CD}
\end{equation*}
where $X$ is the fiber product. Let
\(
\boxtimes (\scA_i \tilde\boxtimes \scA_i)
\)
denote the complex on
\(
\Gr_{\prod_i G}\tilde\times\Gr_{\prod_i G}
\)
obtained in the course of the convolution product for $\prod_i G$. We
define the homomorphism as
\begin{equation*}
  \begin{split}
    & (m'\alpha)_* (i'\alpha)^! \boxtimes (\scA_i \tilde\boxtimes
    \scA_i)
    = m'_* \alpha_! \alpha^! i^{\prime !}
    \boxtimes (\scA_i \tilde\boxtimes
    \scA_i)
\\
& \longrightarrow m'_* i^{\prime !}  \boxtimes (\scA_i \tilde\boxtimes
    \scA_i)
    \cong
    i_\Delta^! (\prod_i m)_* \boxtimes (\scA_i \tilde\boxtimes \scA_i)
    = i_\Delta^! \boxtimes (\scA_i \star \scA_i)
  \end{split}
\end{equation*}
by the adjunction $\alpha_!\alpha^! \to \operatorname{id}$.
\end{NB}%
The last part of the convolution diagram for $G$ and $\prod_i G$ is
\begin{equation*}
  \begin{CD}
    \Gr_G\tilde\times\Gr_G @>m>>\Gr_G \\
    @V{i'_\Delta}VV @VV{i_\Delta}V \\
    \prod_i \Gr_G\tilde\times\Gr_G
    = \Gr_{\prod_i G}\tilde\times\Gr_{\prod_i G}
    @>>\prod_i m> \Gr_{\prod_i G} = \prod_i \Gr_G,
  \end{CD}
\end{equation*}
where we denote the diagonal embedding of the left column by
$i'_\Delta$ to distinguish it from the right column. Let
\(
\boxtimes (\scA_i \tilde\boxtimes \scA_i)
\)
denote the complex on
\(
\Gr_{\prod_i G}\tilde\times\Gr_{\prod_i G}
\)
obtained in the course of the convolution product for $\prod_i G$.
We define the homomorphism as
\begin{equation*}
    m_* i_\Delta^{\prime !}(\boxtimes (\scA_i \tilde\boxtimes \scA_i))
    = m_* \bigotimes^! (\scA_i \tilde\boxtimes \scA_i)
    \to 
    \bigotimes^! m_* (\scA_i \tilde\boxtimes \scA_i) =
    i_\Delta^{!}(\prod_i m)_* \boxtimes (\scA_i \tilde\boxtimes \scA_i))
\end{equation*}
by the natural homomorphism \cite[(2.6.24) or the dual of
(2.6.22)]{KaSha}.
\begin{NB}
    Note that $A\otimes^! B = \DD(\DD A\otimes \DD
    B)$.
    \cite[(2.6.22)]{KaSha} gives us have
    $f_* \DD A\otimes f_* \DD B\to f_*(\DD A\otimes \DD B)$. Applying
    $\DD$, we get
    $\DD f_* (\DD A\otimes \DD B) \to \DD (f_*\DD A\otimes f_* \DD B)$.
    The LHS is $f_! (A\otimes^! B)$, while the RHS is
    $f_! A\otimes^! f_! B$. In our application, $f = m$ is proper.
\end{NB}%

See \ref{sec:Sicilian} for an application of the gluing construction.

\section{Proof of commutativity}\label{sec:commute}

We denote $\Gr_G$ by $\Gr$ for brevity in this section. In this
section we closely follow \cite[\S5]{MV2}, \cite{MR1826370} and
\cite[\S5.3]{Beilinson-Drinfeld}.

\subsection{Commutativity constraint}\label{subsec:constraint}

Let us give a definition of the commutativity constraint $\Theta$.

Let us choose a smooth curve $X$. We define $\Gr_X$ the moduli space
of triples $(x,\scP,\varphi)$ of a point $x\in X$, a $G$-bundle $\scP$
on $X$ and its trivialization $\varphi$ over $X\setminus\{x\}$.
\begin{NB}
    We have $\Gr_X = \mathscr X\times_{\Aut(\cO)} \Gr$, where
    $\mathscr X$ is the $\Aut(\cO)$-bundle over $X$ parametrizing all
    choices of local coordinates. See \cite{Beilinson-Drinfeld}.
\end{NB}%
We also have a group scheme $G_{X,\cO}$, the global analog of $G_\cO$.
\begin{NB}
    It is $\mathscr X\times_{\Aut(\cO)} G_\cO$.
\end{NB}%

More generally, we introduce an ind-scheme $\Gr_{X^n}$ as the moduli
space of $(x_1,\dots,x_n,\scP,\varphi)$ of $n$ ordered points in $X$,
a $G$-bundle $\scP$ on $X$ and its trivialization $\varphi$ over
$X\setminus\bigcup \{x_i\}$. We also have $G_{X^n,\cO}$, which is the
moduli space of $(x_1,\dots,x_n,\scP,\kappa_{x_1,\dots,x_n})$ where
$(x_1,\dots,x_n)\in X^n$, $\scP$ the trivial $G$-bundle on $X$, and
$\kappa_{x_1,\dots,x_n}$ is a trivialization of $\scP$ on $\hat
X_{x_1,\dots,x_n}$.
\begin{NB}
Then $G_{X,\cO}$ acts on $\Gr_X$.

Also $G_{X,\cO}$ acts on $\widetilde{\Gr_X\times\Gr_X}$ in
\eqref{eq:60} below by changing the trivialization $\kappa$.  Here $X$
for $G_{X,\cO}$ is the second factor in $X\times X$.  This realizes
$p_X$ as a $G_{X,\cO}$-torsor. We can rewrite
$\widetilde{\Gr_X\times\Gr_X}$ as the moduli space of
$(x_1,x_2,\scP_1,\scP_2,\varphi_1,\varphi_2,\kappa)$ where $\varphi_i$
is the trivialization of $\scP_i$ on $\hat X_{x_i}\setminus x_i$. Then
we change both trivialization $\varphi_2$, $\kappa$ (at $\hat
X_{x_2}\setminus x_2$). This is the second action, and $q_X$ is the
associated $G_{X,\cO}$-torsor.
\end{NB}%

Then we define the convolution product of $\scA$, $\scB\in
D_{G_{X,\cO}}(\Gr_X)$ as before, using the global version of the
diagram \eqref{eq:1}:
\begin{equation}\label{eq:60}
    \begin{CD}
        \Gr_X\times\Gr_X @<p_X<<
        \widetilde{\Gr_X\times\Gr_X}
        @>q_X>> \Gr_X\tilde\times\Gr_X @>m_X>> \Gr_{X^2}.
    \end{CD}
\end{equation}
Here $\widetilde{\Gr_X\times\Gr_X}$ is the moduli space of
$(x_1,x_2,\scP_1,\varphi_1,\kappa,\scP_2,\varphi_2)$, a pair of points
$(x_1,x_2)\in X^2$, two $G$-bundles $\scP_1$, $\scP_2$ and their
trivializations $\varphi_i$ over $X\setminus \{x_i\}$ together with a
trivialization $\kappa$ of $\scP_1$ on the formal neighborhood of
$x_2$.
The twisted product $\Gr_X\tilde\times\Gr_X$ is the moduli space of
$(x_1,x_2,\scP_1,\varphi_1,\scP,\eta)$ as above, but
$\eta\colon \scP_1|_{X\setminus x_2}\cong \scP|_{X\setminus x_2}$ instead
of $\varphi_2$ and $\kappa$.
The morphism $q_X$ is given by defining $\scP$ as the gluing of
$\scP_1|_{X\setminus x_2}$ and $\scP_2|_{\hat X_{x_2}}$ by
$\varphi_2^{-1}\circ\kappa$ over $(X\setminus x_2)\cap \hat X_{x_2} =
\hat X_{x_2}\setminus x_2$.
(When $X=D$, the formal disk, $\scP$ and $\scP_2$ are
isomorphic. Hence this construction was omitted before.)
The definitions of morphisms $p_X$, $m_X$ are as before, and are
omitted. (See \cite[\S5]{MV2}.)
\begin{NB}
    \begin{gather*}
        (x_1,x_2,\scP_1,\varphi_1,\scP_2,\varphi_2) \overset{p_X}{\leftmapsto}
        (x_1,x_2,\scP_1,\varphi_1,\kappa,\scP_2,\varphi_2),
\\
        (x_1,x_2,\scP_1,\varphi_1,\kappa,\scP_2,\varphi_2)
        \overset{q_X}{\mapsto}
        (x_1,x_2,\scP_1,\varphi_1,\scP = \scP_1|_{X\setminus x_2}
        \underset{\varphi_2^{-1}\circ\kappa}\bigcup
        \scP_2|_{\hat X_{x_2}},\eta=\id|_{X\setminus x_2}),
\\
        (x_1,x_2,\scP_1,\varphi_1,\scP,\eta) \overset{m_X}{\mapsto}
        (x_1,x_2,\scP,\varphi_1\circ\eta^{-1}).
    \end{gather*}
\end{NB}%
Note that $p_X$ is a $G_{X,\cO}$-torsor by the action changing $\kappa$. The second projection $q_X$ is also a $G_{X,\cO}$-torsor by the action changing
$\kappa$ and $\varphi_2$ simultaneously.
\begin{NB}
    Suppose a point $(x_1,x_2,\scP_1,\varphi_1,\scP,\eta)$ together
    with a trivialization $\kappa\colon\scP_1|_{\hat X_{x_2}}\to \hat
    X_{x_2}\times G$ is given. We define a bundle $\scP_2$ defined
    over $\hat X_{x_2}$ with trivialization $\varphi_2$ over ${\hat
      X_{x_2}\setminus x_2}$ as 
    \begin{equation*}
        \scP_2 \defeq \scP|_{\hat X_{x_2}},\qquad
        \varphi_2 \colon \scP_2|_{\hat X_{x_2}\setminus x_2}
        = \scP_{\hat X_{x_2}\setminus x_2}\xrightarrow{\eta^{-1}}
        \scP_1|_{\hat X_{x_2}\setminus x_2} \xrightarrow{\kappa}
        \hat X_{x_2}\setminus x_2\times G,
    \end{equation*}
    i.e., $\varphi_2 = \kappa\circ\eta^{-1}$. This is an inverse
    construction of $q_X$.
\end{NB}%

The diagram \eqref{eq:60} gives a $G_{X^2,\cO}$-equivariant object
defined on $\Gr_{X^2}$ by $\scA_X\star_X\scB_X \defeq m_{X*}
(q_X^*)^{-1} p_X^*(\scA_X\boxtimes\scB_X)$ for $\scA_X$, $\scB_X\in
D_{G_{X,\cO}}(\Gr_X)$.

We take $X=\mathbb A^1$. We have $\Gr_X \cong X\times \Gr$ thanks to a
choice of a global coordinate on $\mathbb A^1$. In particular, we have
a projection $\tau\colon \Gr_X\to\Gr$.
For an object $\scA\in D_{G}(\Gr)$, we can attach
$\scA_X\in D_{G_{X,\cO}}(\Gr_X)$ by $\tau^*\scA[1]$.
\begin{NB}
    $=\tau^!\scA[-1]$.
\end{NB}%
In fact, we can do more generally if we use the $\Aut(\cO)$-bundle
over $X$ parametrizing all choices of local coordinates and consider
$\Aut(\cO)$-equivariant objects as in \cite{Beilinson-Drinfeld,MR1826370}.

Let $\Delta$ denote the diagonal in $X^2$ and $U$ denote the
complement $X^2\setminus\Delta$. The restrictions of $\Gr_{X^2}$ to
$\Delta$ and $U$ are isomorphic to $\Gr_X$ and $(\Gr_X\times
\Gr_X)|_U$ respectively. In fact, the restriction to $\Delta$ is obvious. For a given $(x_1,x_2,\scP,\varphi)$ with $x_1\neq x_2$, we define
$\scP_i$ by gluing $\scP_i|_{X\setminus x_i} = (X\setminus x_i)\times G$ and
$\scP_i|_{X\setminus x_{3-i}} = \scP|_{X\setminus x_{3-i}}$ by $\varphi$
on $X\setminus \{x_1,x_2\}$.
\begin{NB}
    Conversely from $(x_1,\scP_1,\varphi_1)$, $(x_2,\scP_2,\varphi_2)$
    with $x_1\neq x_2$, we define $\scP$ as the gluing of
    $\scP_1|_{X\setminus x_2}$ and $\scP_2|_{X\setminus x_1}$ by
    $\varphi_2^{-1}\circ\varphi_1$ over $X\setminus\{x_1,x_2\}$.
\end{NB}%
Hence we have the diagram
\begin{equation}\label{eq:65}
    \begin{CD}
        \Gr_X @>{\iota}>> \Gr_{X^2} @<{\jmath}<< (\Gr_X\times\Gr_X)|_U
        \\
        @VVV @VVV @VVV
        \\
        \Delta @>>> X^2 @<<< U.
    \end{CD}
\end{equation}

We consider the nearby cycle functor
\begin{equation*}
    \psi_{\Gr_{X^2}}\colon
    D_{(G_{X,\cO}\times G_{X,\cO})|_U}((\Gr_X\times \Gr_X)|_U)\to
    D_{G_{X,\cO}}(\Gr_X).
\end{equation*}
See \cite[\S8.6]{KaSha}, where we change the source
domain to objects defined on $(\Gr_X\times \Gr_X)|_U$, and shift by $-1$,
following the convention in \cite{MR1826370}.

Then an argument in \cite[Proposition 6]{MR1826370} shows there is a
natural isomorphism
\begin{equation}\label{eq:59}
    \psi_{\Gr_{X^2}}((\scA_X\boxtimes \scB_X)|_U)\cong (\scA\star \scB)_X.
\end{equation}
We have the isomorphism $(\scA_X\boxtimes\scB_X)|_U\cong
(\scB_X\boxtimes\scA_X)|_U$ exchanging the factors. Therefore together
with 
\eqref{eq:59} it gives us an isomorphism
$\scA\star\scB \cong \scB\star\scA$.
This is the definition of the commutativity constraint $\Theta$ used
in \ref{thm:commutative}.

Let us briefly explain how \eqref{eq:59} is constructed. For a later
purpose, we give a slightly different explanation from \cite{MR1826370}.

\newcommand{\ps}{\operatorname{ps}}

By the definition of the nearby cycle functor, we have a natural
homomorphism
\begin{equation}\label{eq:61}
    \ps\colon \psi_{\Gr_{X^2}}((\scA_X\star_X \scB_X)|_U)
    \to \iota^! (\scA_X\star_X \scB_X)[1].
\end{equation}
It is the \emph{dual} of the specialization homomorphism. See
\cite[(8.6.7)]{KaSha}.
We restrict the diagram \eqref{eq:60} to the diagonal to see that
\begin{equation}\label{eq:73}
    (\scA\star\scB)_X \cong \iota^! (\scA_X\star_X\scB_X)[1].
\end{equation}
\begin{NB}
    Let us restrict \eqref{eq:60} to the diagonal. We obtain
\begin{equation*}
    \begin{CD}
        \Gr_X\times\Gr_X @<p_X<<
        \widetilde{\Gr_X\times\Gr_X}
        @>q_X>> \Gr_X\tilde\times\Gr_X @>m_X>> \Gr_{X^2}
        \\ @AAA @AAA @AAA @AAA \\
        X\times \Gr\times\Gr @<\id_X\times p<<
        X\times G_\cK\times \Gr
        @>\id_X \times q>> X \times \Gr\tilde\times\Gr
        @>\id_X\times m>> \Gr_{X} = X\times\Gr
    \end{CD}
\end{equation*}
Starting from
$i^!(\scA_X\star_X\scB_X) = i^! m_{X*}(\scA_X\tilde\boxtimes\scB_X)$,
we apply the base change, etc, to get the assertion.
\end{NB}%
Therefore we need to check

\begin{Claim}
\begin{subequations}
\begin{align}\label{eq:67}
    & \text{We have a natural isomorphism
      $(\scA_X\star_X\scB_X)|_U\cong
      (\scA_X\boxtimes\scB_X)|_U$.} \\
    \label{eq:72}
    & \text{$\ps$ in \eqref{eq:61} is an isomorphism.}
\end{align}
\end{subequations}
\end{Claim}

\begin{proof}
Let us denote the restrictions of $p_X$, $q_X$, $m_X$ to inverse
images of $U$ by $p_U$, $q_U$, $m_U$ respectively.

Over $U$, we have a natural commutative diagram
\begin{equation}\label{eq:69}
    \begin{gathered}[m]
    \xymatrix@C=20pt{
      (\Gr_X\times\Gr_X)|_U 
      & \widetilde{\Gr_X\times\Gr_X}|_U \ar[l]_-{p_U} \ar[r]^-{q_U}
      & \Gr_X\tilde\times\Gr_X|_U \ar[r]^-{m_U}_-\cong
      & \Gr_{X^2}|_U
      \\
      (\Gr_X\times\Gr_X)|_U \ar@{=}[u]
      & G_{X,\cO}\times_{X} (\Gr_X\times \Gr_X)|_U \ar[l] \ar[r] \ar[u]_-\cong
      & (\Gr_X\times\Gr_X)|_U \ar@{=}[r] \ar[u]_-\cong
      & (\Gr_X\times\Gr_X)|_U \ar[u]_-\cong^-\jmath
      }
    \end{gathered}
\end{equation}
\begin{NB}
    \begin{equation*}
        \begin{CD}
        (\Gr_X\times\Gr_X)|_U @<p_U<< \widetilde{\Gr_X\times\Gr_X}|_U
        @>q_U>> \Gr_X\tilde\times\Gr_X|_U @>m_U>\cong> \Gr_{X^2}|_U \\
        @| @A{\cong}AA @AA{\cong}A @A{\cong}A{\jmath}A \\
        (\Gr_X\times\Gr_X)|_U @<<<
        G_{X,\cO}\times_{X} (\Gr_X\times \Gr_X)|_U
        @>>> (\Gr_X\times\Gr_X)|_U @= (\Gr_X\times\Gr_X)|_U,
    \end{CD}
    \end{equation*}
\end{NB}%
where $\Gr_X\times\Gr_X\to X$ in the bottom middle term is through the
projection $X\times X\to X$ to the second factor.
Here the second vertical isomorphism is given by regarding
$\kappa$ as a trivialization of the trivial bundle over $\hat X_{x_2}$
via the trivialization $\varphi_1\colon \scP_1|_{\hat
  X_{x_2}}\xrightarrow{\cong} \hat X_{x_2}\times G$.
The third vertical isomorphism is given by considering $\eta$ as a
trivialization of $\scP$.
The lower left arrow is given by forgetting $G_{X,\cO}$. The lower
right arrow is given by the action of $G_{X,\cO}$ on the second factor
of $\Gr_X\times\Gr_X$.
\begin{NB}
    Let us check the assertions.

    Take $(x_1,x_2,\scP_1,\varphi_1,
    \kappa,\scP_2,\varphi_2)\in\widetilde{\Gr_X\times\Gr_X}$ with
    $x_1\neq x_2$. We view it as $\scP_i$, a $G$-bundle over $\hat
    X_{x_i}$, $\varphi_i$, a trivialization of $\scP_i$ on $\hat
    X_{x_i}\setminus \{x_i\}$, and $\kappa$ as a trivialization of the
    trivial bundle (through $\varphi_1$).

    In the same way, we consider
    $(x_1,x_2,\scP_1,\varphi_1,\scP,\eta)\in \Gr_X\tilde\times\Gr_X|_U$ as
    $\scP_1$, $\varphi_1$ as above, a $G$-bundle $\scP$ over $\hat
    X_{x_2}$, and a trivialization $\eta$ of $\scP$ over $\hat
    X_{x_2}\setminus x_2$.  Namely we trivialize $\scP_1|_{\hat
      X_{x_2}}$ via $\varphi_1$. It is clear that
    $(\Gr_X\tilde\times\Gr_X)|_U \cong (\Gr_X\times\Gr_X)|_U$.

    Now we look at the definition
    \begin{equation*}
        q_X(x_1,x_2,\scP_1,\varphi_1,\kappa,\scP_2,\varphi_2)
        = (x_1,x_2,\scP_1,\varphi_1,\scP = \scP_1|_{X\setminus x_2}
        \underset{\varphi_2^{-1}\circ\kappa}\cup
        \scP_2|_{\hat X_{x_2}},\eta=\id|_{X\setminus x_2})
    \end{equation*}
    again. We can replace $\scP$ by $\scP_2$, as we only care it over
    $\hat X_{x_2}$. Then the trivialization $\eta$ is given by
    $\kappa^{-1}\circ\varphi_2$.
\end{NB}%
Since we are considering equivariant objects, we have a canonical
isomorphism $(q_U^*)^{-1}p_U^*((\scA_X\boxtimes\scB_X)|_U) \cong
(\scA_X\boxtimes\scB_X)|_U$. We now apply $m_{U*}$ and observe that
$m_{U*}(q_U^*)^{-1}p_U^*((\scA_X\boxtimes\scB_X)|_U) =
(\scA_X\star_X\scB_X)|_U$. Thus we have checked (a).

Let us turn to the assertion (b). The idea is to consider nearby
cycle functors for four spaces in \eqref{eq:60}.

Let us start with $m_X$.
\begin{NB}
The restriction $m_U$ becomes an isomorphism.
\begin{NB2}
    We reconstruct $\scP_1$ as $\scP_1|_{X\setminus x_1} = (X\setminus
    x_1)\times G$, $\scP_1|_{X\setminus x_2} = \scP|_{X\setminus x_2}$
    with the gluing data $\varphi$.
\end{NB2}%
Therefore we can identify $(\Gr_X\tilde\times\Gr_X)|_U$ with $(\Gr_X\times
\Gr_X)|_U$ and consider
$\psi_{\Gr_X\tilde\times\Gr_X}((\scA_X\boxtimes\scB_X)|_U)$.
\end{NB}%
Since nearby cycle functors commute with proper morphisms, we have
\begin{equation*}
   \psi_{\Gr_{X^2}}(m_{U*}(\scA_X\tilde\boxtimes \scB_X)|_U) \cong
   m_{\Delta*}\psi_{\Gr_X\tilde\times\Gr_X}((\scA_X\tilde\boxtimes\scB_X)|_U),
\end{equation*}
where $m_\Delta$ is the restriction of $m$ to $\Delta$.
\begin{NB}
    More precisely,
    \begin{equation*}
        \begin{split}
            \psi_{\Gr_{X^2}}(m_{U*}(\scA_X\tilde\boxtimes \scB_X)|_U)
            & \cong
            m_{\Delta*}\psi_{\Gr_X\tilde\times\Gr_X}((\scA_X\tilde\boxtimes\scB_X)|_U)
            \\
            & \cong
            m_{\Delta*} (q_\Delta^*)^{-1}
            \psi_{\widetilde{\Gr_X\times\Gr_X}}(p_U^*(\scA_X\boxtimes\scB_X)|_U)
            \\
            & \cong
            m_{\Delta*} (q_\Delta^*)^{-1} p_\Delta^*
            \psi_{\Gr_X\times\Gr_X}((\scA_X\boxtimes\scB_X)|_U).
        \end{split}
    \end{equation*}
\end{NB}

Next consider $p_X$ and $q_X$. They are both smooth
(\cite[p.114]{MV2}), and hence commute with nearby cycle
functors. Therefore
\begin{equation*}
    \begin{split}
        \psi_{\Gr_X\tilde\times\Gr_X}((\scA_X\tilde\boxtimes\scB_X)|_U)
        & \cong 
        (q_\Delta^*)^{-1}
        \psi_{\widetilde{\Gr_X\times\Gr_X}}(p_U^*(\scA_X\boxtimes\scB_X)|_U)
        \\
        & \cong 
        (q_\Delta^*)^{-1} p_\Delta^*
        \psi_{\Gr_X\times\Gr_X}((\scA_X\boxtimes\scB_X)|_U),
    \end{split}
\end{equation*}
where $p_\Delta$, $q_\Delta$ are restrictions of $p_X$, $q_X$ to $\Delta$.
Hence
\begin{equation*}
    \psi_{\Gr_{X^2}}((\scA_X\star_X\scB_X)|_U)
    \cong m_{\Delta*}(q_\Delta^*)^{-1} p_\Delta^*
    \psi_{\Gr_X\times\Gr_X}((\scA_X\boxtimes\scB_X)|_U).
\end{equation*}

Now $\Gr_X\times\Gr_X = X\times X\times \Gr\times\Gr$, hence
$\psi_{\Gr_X\times\Gr_X}((\scA_X\boxtimes\scB_X)|_U)$ is just
$(\scA\boxtimes \scB)_X$. More precisely, the isomorphism is given by
the dual specialization homomorphism
\[
\ps\colon \psi_{\Gr_X\times\Gr_X}((\scA_X\boxtimes\scB_X)|_U)
\xrightarrow{\cong} \iota^! (\scA_X\boxtimes\scB_X)[1] =
(\scA\boxtimes\scB)_X,
\]
thanks to vanishing of the vanishing cycle functor
$\varphi_{\Gr_X\times\Gr_X}(\scA_X\boxtimes\scB_X)$.
Thus
\[
   \psi_{\Gr_{X^2}}(\scA_X\star_X\scB_X)|_U)
   \xrightarrow[\cong]{m_{\Delta*} (q_\Delta^*)^{-1} p_\Delta^* \ps}
   m_{\Delta*} (q_\Delta^*)^{-1} p_\Delta^*((\scA\boxtimes \scB)_X).
\]
Notice that the restriction of \eqref{eq:60} to the diagonal is just
the product of $X$ and the diagram \eqref{eq:1}. Therefore the right
hand side is $(\scA\star\scB)_X$. Now one can check that dual
specialization homomorphisms commute with proper pushforward and
smooth pull-backs so that they are compatible with the commutation of
nearby cycle functors. (See the argument in the proof of
\lemref{lem:N=0case} below.)
Therefore $m_{\Delta*} (q_\Delta^*)^{-1} p_\Delta^* \ps$ is equal to
$\ps$ over $\Gr_{X^2}$. Thus (b) is checked.
\end{proof}

\subsection{Factorization version of $\mathcal R$}\label{subsec:KV}

We define a global version of the variety of triples $\mathcal R$ in
this subsection.

Let us assume that we are given a smooth connected curve $X$, an algebraic group $G$ and a representation $\mathbf N$ of $G$ and a finite set $I$.
Consider a functor $\mathrm{Schemes}/\mathbb C\to \mathrm{Sets}$
which sends a scheme $S$ to the following data:
1) A map $f\colon S\to X^I$. We shall think about $f$ as a collection of maps $f_i\colon S\to X$ for $i\in I$ and we denote by $\Gamma$ the union of graphs of $f_i$ -- this is a closed subscheme of $S\times X$.

2) A $G$-bundle $\scP$ on $S\times X$.

3) A trivilalization $\varphi$ of $\scP$ over $S\times X\backslash \Gamma$.

4) A section $s$ of the associated bundle $\scP_{\bN}$ over the formal neighbourhood of $\Gamma$ in $S\times X$ and a section $s'$ of the trivial $\mathbf N$ - bundle over the same formal neighbourhood which are equal on the ``formal punctured neighbourhod" (this makes sense because of 3). These notions (formal neighbourhood, formal punctured neighbourhood) are explained in \cite{MR2102701}.

Now we claim that this functor is representable by an ind-scheme.
Moreover, this ind-scheme has a natural closed embedding into $\Gr_{X^I}\underset{X^I}\times \mathcal J_{\mathbf N, X^I}$ where

a) $\Gr_{X^I}$ is the factorization (a.k.a. Beilinson-Drinfeld) Grassmannian over $X^I$

b)  $\mathcal J_{\mathbf N, X^I}$ is the Kapranov-Vasserot factorization version of the $\mathbf N$-jet space over $X^I$.

Indeed it is enough to construct this closed embedding (as a closed subfunctor of an ind-scheme is also an ind-scheme). But an $S$-point of $\Gr_{X^I,G,BD}$ is precisely the data of 1), 2), 3) and an $S$-point of $\mathcal J_{\mathbf N, X^I}$
is the data of 1), 2) and $s'$ from 4). Since $s$ is obviously uniquely determined by all the data and since the existence of $s$ is a closed condition on the other data we get the above closed embedding.

Let us denote the above ind-scheme by $\mathcal R_{X^I}$. Then obviously from 1) we get a morphism $\pi_{X^I}:\mathcal R_{X^I}\to X^I$ and
it is clear that the restriction of
$\mathcal R_{X^I}$ to the complement $U$ of all the diagonals in $X^I$ is naturally isomorphic to the similar restriction
of $(\mathcal R^{(1)})^I$.

On the other hand, assume that we are given a surjective morphism $I\to J$ of finite sets. Such a morphism defines a closed embedding
$X^J\hookrightarrow X^I$ (as a partial diagonal) and it follows that the restriction of $\mathcal R_{X^I}$ to $X^J$ is naturally isomorphic to
$\mathcal R^{(J)}$.

Similarly, we can define a factorization version of the bundle $\mathcal T$ over $\Gr$. By definition an $S$-point of $\mathcal T_{X^I}$
is a quadruple $(f,\scP,\varphi,s)$ as above 
(i.e.\ no $s'$).\footnote{Note that if we instead only choose $s'$ and do not 
choose $s$ then the resulting functor is represented by
$\Gr_{X^I}\underset{X^I}\times \mathcal J_{\mathbf N, X^I}$.}
We claim again that this functor is representable by an ind-scheme. For this it is enough to show that the morphism 
$\mathcal T_{X^I}\to \Gr_{X^I}$ (which corresponds to forgetting $s$) is representable. This can be done by a word-by-word repetition
of the proof of the fact that the factorization version of the jet scheme is representable by a scheme (cf.\ again Section 3 of \cite{MR2102701}).

In what follows we shall only need the above spaces when $I=\{1,2\}$.

\subsection{Definition of another multiplication}

We consider the space $\cR_{X^2}$, its dualizing complex
$\DC_{\cR_{X^2}}$ and the pushforward $\pi_*\DC_{\cR_{X^2}}$.  Its
restriction to $U$ is isomorphic to
$(\pi_*\DC_{\cR_X}\boxtimes\pi_*\DC_{\cR_X})|_U$ under $\jmath$ in
\eqref{eq:65}. We consider two dual specialization homomorphisms
\begin{equation}\label{eq:71}
    \begin{gathered}[m]
    \xymatrix@C=30pt{
        \psi_{\Gr_{X^2}}(\pi_*\DC_{\cR_{X^2}}|_U) \ar[r]^-\ps 
        & \iota^! \pi_* \DC_{\cR_{X^2}}[1]
        \\
        \psi_{\Gr_{X^2}}((\pi_*\DC_{\cR_X}\star_X\pi_*\DC_{\cR_X})|_U)
        \ar[r]^-\ps_-\cong \ar[u]^-\cong
        & \iota^! (\pi_*\DC_{\cR_X}\star_X\pi_*\DC_{\cR_X})[1]
        \ar@{.>}[u],
        }
    \end{gathered}
\end{equation}
where $\iota\colon \Gr_X\to \Gr_{X^2}$ is the inclusion, and the
vertical arrow is given by 
\(
   \pi_*\DC_{\cR_{X^2}}|_U \cong
   (\pi_*\DC_{\cR_X}\boxtimes\pi_*\DC_{\cR_X})|_U
   \cong
   (\pi_*\DC_{\cR_X}\star_X\pi_*\DC_{\cR_X})|_U.
\)
(See \eqref{eq:67}.) 
The lower homomorphism is an isomorphism thanks to \eqref{eq:72}.
Note that $\pi_* \DC_{\cR_{X}}$ is $(\pi_*\DC_{\cR})_X[1]$. Therefore
the right bottom term is $(\pi_*\DC_{\cR}\star\pi_*\DC_\cR)_X[2]$ by
\eqref{eq:73}.
Note also $\iota^!\pi_*\DC_{\cR_{X^2}}[1] = \pi_*\DC_{\cR_X}[1]
= (\pi_*\DC_\cR)_X[2]$.
\begin{NB}
    \begin{equation*}
        \begin{CD}
            \cR_X @>\iota>> \cR_{X^2} \\
            @V{\pi}VV @VV{\pi}V \\
            \Gr_X @>>\iota> \Gr_{X^2}
        \end{CD}
    \end{equation*}
\end{NB}%
Therefore we obtain a homomorphism
\begin{equation}
    \label{eq:62}
    \mathsf m^\psi\colon
    \scA\star\scA \to \scA,
    \qquad \scA = \pi_*\DC_\cR[-2\dim\bN_\cO],
\end{equation}
by specializing the dotted arrow at a point in $X$.

The degree shift should be checked by going back to finite dimensional
approximation of $\cR$. We have shifts by $\dim
\bN_\cO/z^{d_1}\bN_\cO$ and $\dim \bN_\cO/z^{d_2}\bN_\cO$ for two
factors in
$\psi_{\Gr_{X^2}}((\pi_*\DC_{\cR_X}\boxtimes\pi_*\DC_{\cR_X})|_U)$. Then
we have a shift $\dim \bN_\cO/z^{d_1+d_2}\bN_\cO$ for
$\pi_*\DC_{\cR_X}$.

Now our goal is to check two properties:
\begin{enumerate}
  \renewcommand{\theenumi}{\roman{enumi}}%
  \renewcommand{\labelenumi}{\textup{(\theenumi)}}%
    \item\label{item:equal} $\mathsf m^\psi = \mathsf m$,
    \item $\mathsf m^\psi$ is invariant under the exchange of factors of
  $\scA_X\star\scA_X$. (More precisely, exchange after going back to
  $(\scA\boxtimes\scA)|_U$.)
\end{enumerate}

The property (ii) is clear as the diagram \eqref{eq:65} is invariant
under the exchange of two factors of $X^2=X\times X$.

We will check (i) for $\bN=0$ in the next subsection. We have a
difficulty to check (i) directly for general $\bN$, so we will argue
indirectly by reduction to the case $\bN=0$.

\begin{NB}
\begin{enumerate}
  \renewcommand{\theenumi}{\roman{enumi}}%
  \renewcommand{\labelenumi}{\textup{(\theenumi)}'}%
    \item\label{item:prime} the induced product $\mathsf m^\psi\colon
  H^{G_\cO}_*(\scA) \otimes H^{G_\cO}_*(\scA)\to H^{G_\cO}_*(\scA)$ is
  equal to the induced product by $\mathsf m$, and hence to the
  convolution product $\ast$ by \propref{prop:sheav-affine-grassm}(4).
\end{enumerate}
\end{NB}%

\subsection{The case \texorpdfstring{$\bN=0$}{N=0}}\label{subsec:N=0}


We first consider the case $\bN=0$.

We consider the dual specialization homomorphism for
$\DC_{\Gr_X\tilde\times\Gr_X}$:
\[
   \ps_\prime\colon
   \psi_{\Gr_X\tilde\times\Gr_X}((\DC_{\Gr_X}\boxtimes\DC_{\Gr_X}|_U)
   \to \iota_{\prime}^! \DC_{\Gr_X\tilde\times\Gr_X}[1]
   = \DC_{X\times \Gr\tilde\times\Gr}[1]
\]
where $\iota_{\prime}\colon X\times \Gr\tilde\times\Gr =
(\Gr_X\tilde\times\Gr_X)|_\Delta\to \Gr_X\tilde\times\Gr_X$ is the
embedding from the definition of the nearby cycle functor.
Here we have used $\DC_{\Gr_X\tilde\times\Gr_X}|_U\cong
(\DC_{\Gr_X}\boxtimes\DC_{\Gr_X}|_U)$ from \eqref{eq:69}.

The following two assertions identify \eqref{eq:62} with the pull-back
of $\mathsf m$ under \eqref{eq:59}, hence we obtain the property (i)
for $\bN=0$.

\begin{Lemma}\label{lem:N=0case}
    \textup{(1)} $\ps$ is equal to the composition of
    $m_{\Delta*}\ps_{\prime}$ and the natural morphism $m_{\Delta*}
    m_{\Delta}^! = m_{\Delta!} m_{\Delta}^!  \to \id$.

    \textup{(2)} The homomorphism $\ps_{\prime}$ coincides with the
    homomorphism $\DC_{\Gr}\tilde\boxtimes\DC_{\Gr}\to
    \DC_{\Gr\tilde\times\Gr}$ constructed in \eqref{eq:58},
    pull-backed by $\Gr_X\tilde\times\Gr_X\to \Gr\tilde\times\Gr$.
    It is an isomorphism.
\end{Lemma}

\begin{proof}
    (1) Recall the definition of the nearby cycle functor and the dual
    specialization morphism (\cite[\S8.6]{KaSha}). We have $f\colon
    \Gr_{X^2}\to X^2 = \CC^2 \to \CC$, where the second map is
    $(x_1,x_2)\mapsto x_1-x_2$. We then consider $p\colon
    \tilde{\CC}^\times\to\CC$, the composition of the universal
    covering $\tilde\CC^\times\to\CC^\times$ and the inclusion
    $\CC^\times\to\CC$. We then pull back $p$ by $f$ to get
    $\tilde p\colon \smash[b]{\widetilde{\Gr}}_{X^2}^\times\to \Gr_{X^2}$.
    Then
    \begin{equation*}
        \psi_{\Gr_{X^2}}((\DC_{\Gr_X}\boxtimes\DC_{\Gr_X})|_U)
        = \iota^* \tilde p_* \tilde p^* \DC_{\Gr_{X^2}}[-1]
        \cong \iota^* \HHom(f^* p_!  \CC_{\tilde{\CC}^\times},
        \DC_{\Gr_{X^2}}[-1])
    \end{equation*}
    and $\ps$ is defined from $\CC_{\{0\}}\to p_!
    \CC_{\tilde{\CC}^\times}[2]$.

    Let us write the identification $(\Gr_X\tilde\times\Gr_X)|_U
    \xrightarrow{\cong} \Gr_{X^2}|_U\cong (\Gr_X\times\Gr_X)|_U$
    explicitly as $m_U$, the restriction of $m_X$ to $U$. The
    commutativity of
    \begin{equation*}
        \begin{CD}
            \psi_{\Gr_{X^2}}(m_{U*}\DC_{(\Gr_X\tilde\times\Gr_X)|_U})
            @>\ps>>
            \iota^! m_{X*} \DC_{\Gr_X\tilde\times\Gr_X}[1]
            \\ @V{\cong}VV @VV{\cong}V \\
            m_{\Delta*} \psi_{\Gr_X\tilde\times\Gr_X}
            (\DC_{(\Gr_X\tilde\times\Gr_X)|_U})
            @>>m_{\Delta*} \ps_{\prime}>
            m_{\Delta *} \iota_{\prime}^!\DC_{\Gr_X\tilde\times\Gr_X}[1]
        \end{CD}
    \end{equation*}
    is clear as both vertical arrows are given by base change and
    adjunction. This property has been already used in the
    construction of the commutativity constraint above.
    \begin{NB}
        We have
        \begin{equation*}
            \begin{CD}
                \smash[b]{\widetilde{\Gr_X\tilde\times\Gr_X}}^\times
                @>>>
                \smash[b]{\widetilde{\Gr}}_{X^2}^\times
                @>>> \tilde\CC^\times
                \\ @V{\tilde{p'}}VV @VV{\tilde p}V @VV{p}V \\
                \Gr_X\tilde\times\Gr_X @>>m_X> \Gr_{X^2} @>>f> \CC.
            \end{CD}
        \end{equation*}
        Both are cartesian products. Then the commutation of the
        proper pushfoward and the nearby cycle is given as
        \begin{equation*}
            \begin{split}
            & m_{\Delta*}\psi_{\Gr_X\tilde\times\Gr_X}(\DC_{(\Gr_X\tilde\times\Gr_X)|_U})
            \cong
            m_{\Delta*} \iota^*\HHom(m_X^* f^* p_!\CC_{\tilde\CC^\times},
            \DC_{\Gr_X\tilde\times\Gr_X}[-1])
            \\
            \cong \; & \iota^* m_{X*}\HHom(m_X^* f^* p_!\CC_{\tilde\CC^\times},
            \DC_{\Gr_X\tilde\times\Gr_X}[-1])
            \\
            \cong \; & \iota^* \HHom(f^* p_!\CC_{\tilde\CC^\times},
            m_{X*} \DC_{\Gr_X\tilde\times\Gr_X}[-1])
            = \psi_{\Gr_{X^2}}(m_{U*}\DC_{(\Gr_X\tilde\times\Gr_X)|_U}),
            \end{split}
        \end{equation*}
        where the second isomorphism is the base change ($m_{\Delta*}
        = m_{\Delta!}$, $m_{X*} = m_{X!}$), the third isomorphism is
        the adjunction. We replace $p_!\CC_{\tilde\CC^\times}$ by
        $\CC_{\{0\}}$.
    \end{NB}%
    Next the commutativity of
    \begin{equation*}
        \begin{CD}
            \psi_{\Gr_{X^2}}(m_{U*}\DC_{(\Gr_X\tilde\times\Gr_X)|_U})
              @>\ps>> \iota^! m_{X*} \DC_{\Gr_X\tilde\times\Gr_X}[1]
\\
            @V{\psi_{\Gr_{X^2}}(m_{U!}m_U^!\to\id)}V{\cong}V
            @VV{m_{X!}m_X^!\to\id}V \\
            \psi_{\Gr_{X^2}}((\DC_{\Gr_X}\boxtimes\DC_{\Gr_X})|_U)
            @>>\ps>
            \iota^! \DC_{\Gr_{X^2}}[1]
        \end{CD}
    \end{equation*}
    is also clear from the definition. Therefore we get the assertion.

    (2) \begin{NB}
This is probably an isomorphism:
$\DC_{\Gr}\tilde\boxtimes\DC_{\Gr} \cong \DC_{\Gr\tilde\times\Gr}$.
\begin{NB2}
        It seems easier than I expected. So I include the proof of an isomorphicity at the end.
    \end{NB2}

We replace $m_{?!}m_?^!\to \id$ by $\id\to p_{?*} p_?^*$, $\id\to
q_{?*} q_?^*$ ($? = U$, $\Delta$, $X$) in the argument in (1).
\end{NB}%
The following diagram is commutative:
\begin{equation*}
    \begin{CD}
        q_\Delta^* \psi_{\Gr_X\tilde\times\Gr_X}(\DC_{(\Gr_X\tilde\times\Gr_X)|_U})
        @>q_\Delta^*\ps_\prime>>
        q_\Delta^* \iota_\prime^! \DC_{\Gr_X\tilde\times\Gr_X}[1] \\
        @V{\cong}VV @VV{\cong}V \\
        \psi_{\widetilde{\Gr_X\times\Gr_X}}(q_U^*
        \DC_{(\Gr_X\tilde\times\Gr_X)|_U}) @>>\ps_{\prime\prime}>
        \iota_{\prime\prime}^! q_X^* \DC_{\Gr_X\tilde\times\Gr_X}[1],
    \end{CD}
\end{equation*}
where $\iota_{\prime\prime}\colon G_\cK\times \Gr_X\to
\widetilde{\Gr_X\times\Gr_X}$ is the inclusion given by
$\widetilde{\Gr_X\times\Gr_X}|_\Delta \cong G_\cK\times\Gr_X$, and
$\ps_{\prime\prime}$ is the dual specialization for
$\DC_{\widetilde{\Gr_X\times\Gr_X}} = q_X^*\DC_{\Gr_X\tilde\times\Gr_X}$.
In fact, the left vertical arrow is
given by the composition of
\begin{equation*}
    \begin{split}
        & q_\Delta^*\psi_{\Gr_X\tilde\times\Gr_X}(\DC_{\Gr_X\tilde\times\Gr_X}|_U) \cong q_\Delta^*
        \iota_{\prime}^*\HHom(m_X^* f^* p_!\CC_{\tilde\CC^\times},
        \DC_{\Gr_X\tilde\times\Gr_X}[-1])
        \\
        \cong \; & \iota_{\prime\prime}^* q_X^* \HHom(m_X^*
        f^* p_!\CC_{\tilde\CC^\times}, \DC_{\Gr_X\tilde\times\Gr_X}[-1]) \\
        \cong \; & \iota_{\prime\prime}^*\HHom(q_X^* m_X^* f^*
        p_!\CC_{\tilde\CC^\times}, q_X^* \DC_{\Gr_X\tilde\times\Gr_X}[-1]) \cong
        \psi_{\widetilde{\Gr_X\times\Gr_X}}( q_U^* \DC_{\Gr_X\tilde\times\Gr_X}|_U)
    \end{split}
\end{equation*}
where we have used \cite[Prop.~3.1.13]{KaSha} and $q_X^! = q_X^*[2\dim
G_\cO]$ for the third isomorphism. Now we apply $\CC_{\{0\}}\to
p_!\CC_{\tilde\CC^\times}[2]$.

In the same way, we have another commutative diagram
\begin{equation*}
    \begin{CD}
        \psi_{\widetilde{\Gr_X\times\Gr_X}}(p_U^*
        \DC_{(\Gr_X\times\Gr_X)|_U}) @>\ps_{\prime\prime}>>
        \iota_{\prime\prime}^! p_X^* \DC_{\Gr_X\times\Gr_X}[1] \\
        @V{\cong}VV @VV{\cong}V \\
        p_\Delta^*
        \psi_{\Gr_X\times\Gr_X}(\DC_{(\Gr_X\times\Gr_X)|_U})
        @>>p_\Delta^*\ps_{\prime\prime\prime}> p_\Delta^*
        \iota_{\prime\prime\prime}^! \DC_{\Gr_X\times\Gr_X}[1],
    \end{CD}
\end{equation*}
where $\iota_{\prime\prime\prime}\colon X\times\Gr\times\Gr\to
\Gr_X\times \Gr_X$ is the embedding given by
$\Gr_X\times\Gr_X|_\Delta\cong X\times \Gr\times\Gr$. Now the
assertion is proved. Note $p_\Delta^*\ps_{\prime\prime\prime}$ is an
isomorphism, hence so is $\ps_{\prime}$.
%
\begin{NB}
    For $\cR$, we use the pull-back with support, so the argument
    should be more delicate.
\end{NB}%
\end{proof}

\begin{Remark}
    We have a difficulty to generalize the argument in
    \subsecref{subsec:N=0} to $\bN\neq 0$ since we lack an
    $\cR$-version of $\Gr_X\tilde\times\Gr_X$, as a well-defined ind-scheme.
    \begin{NB}
        Misha's Message on Nov.~2, 2015:

        Consider the case $G={\mathbb C}^\times$, ${\mathbf N}$ is
        one-dimensional.  More precisely, $\Gr_X\star \Gr_X$ is the
        moduli space of ${\scP}_{\triv}\leftarrow{\scP}_1\to{\scP}$;
        the left arrow is $\varphi$ regular away from $x_1$, the right
        arrow is $\eta$ regular away from $x_2$. Then what are the
        sections and conditions on them? We are interested in the
        component where $\varphi$ has a first order pole, and $\eta$
        has a first order zero. Naively it seems that in the fiber of
        ${\mathbf N}_\cO$ over $x_1\ne x_2$ we impose 1 equation, but
        in the fiber over $x_1=x_2$ there are no equations. So the
        equations must be divisible by $x_1-x_2$, but there seems to
        be no modular reasons for this.
        
        \begin{NB2}
          Added on Oct.~28, 2017, after receiving Gus' note.
          
          This reasoning seems wrong, as we should impose 1 equation even if
          $x_1 = x_2$.
        \end{NB2}%
      \end{NB}%
      This difficulty will be overcome in \ref{gusgusgus} written by
      Gus Lonergan.
\end{Remark}

\subsection{Completion of the proof}


Let $\bz_{X^2}\colon \Gr_{X^2}\to \cT_{X^2}$ be the factorization
version of the embedding $\bz\colon \Gr\to\cT$ discussed in
\subsecref{sec:z}. It factors as $\bz_{X^2} = i\circ \tilde\bz_{X^2}$,
where $\tilde\bz_{X^2}\colon \Gr_{X^2}\to \cR_{X^2}$, and $i\colon
\cR_{X^2}\to\cT_{X^2}$ is the embedding. Since $\cT_{X^2}\to\Gr_{X^2}$
is a vector bundle, we have $\bz_{X^2}^*\DC_{\cT_{X^2}}\to
\DC_{\Gr_{X^2}}[2\dim\bN_\cO]$, and also
$\DC_{\cR_{X^2}}[-2\dim\bN_\cO]\to \tilde\bz_{X^2*}\DC_{\Gr_{X^2}}$ by
the pull-back with support. We apply $\pi_*$ to obtain
\begin{equation*}
    \bz_{X^2}^*\colon \pi_*\DC_{\cR_{X^2}}[-2\dim\bN_\cO] \to \DC_{\Gr_{X^2}}.
\end{equation*}
\begin{NB}
    as $\pi_* \tilde z_{X^2*} = \id$.
\end{NB}%
We now apply the nearby cycle functor $\psi_{\Gr_{X^2}}$ and the dual
specialization homomorphisms:
\begin{equation*}
    \begin{CD}
        \psi_{\Gr_{X^2}}(\pi_*(\DC_{\cR_{X^2}}[-2\dim\bN_\cO])|_U)
        @>\ps>> \pi_* \DC_{\cR_{X}}[-2\dim\bN_\cO+1]
\\
      @V{\bz_{X^2}^*}VV @VV{\bz^*}V
\\
        \psi_{\Gr_{X^2}}(\DC_{\Gr_{X^2}}|_U)
        @>>\ps> \DC_{\Gr_{X}}[1].
    \end{CD}
\end{equation*}
This is a commutative diagram by the argument in the proof of
\lemref{lem:N=0case}. Removing the unnecessary factor $X$, we get
\begin{equation}\label{eq:74}
    \begin{CD}
        \scA\star\scA @>\mathsf m^\psi_\cR>> \scA
\\
      @V{\bz^*\star \bz^* = \bz_{X^2}^*}VV @VV{\bz^*}V
\\
        \DC_{\Gr}\star\DC_{\Gr} @>>\mathsf m^\psi_\Gr>\DC_{\Gr},
    \end{CD}
\end{equation}
where we put the superscript $\psi$ to emphasize that the
multiplication is defined via nearby cycle functors. 
Since the restriction of $\bz^*_{X^2}$ to $U$ is $\bz^*\boxtimes \bz^*$, the
left vertical arrow is equal to $\bz^*\star \bz^*$.
\begin{NB}
    Added on Nov.\ 24 by H.N.
\end{NB}%
\begin{NB}
At this stage we have proved the following:
  
We have proved that $\mathsf m^\psi_\Gr$ is $\mathsf m_\Gr$ defined in
\propref{prop:sheav-affine-grassm} for $\bN=0$ in the previous
subsection. If we take cohomology groups, we have the induced
homomorphism of algebras $(H^*_{G_\cO}(\scA),\mathsf m^\psi_\cR) \to
(H^*_{G_\cO}(\DC_{\Gr}),\ast) = (H_*^{G_\cO}(\Gr), \ast)$. This is
injective by \lemref{lem:injective}. It is also an algebra
homomorphism with respect to the convolution product. Therefore
$\mathsf m^\psi_\cR
$ coincides with the convolution
product.
\begin{NB2}
    Let $\varTheta$ be the embedding $\varTheta\mathsf m^\psi_\cR(x,y)
    = \varTheta(x)\ast\varTheta(y)$. On the other hand we also know
    $\varTheta(x\ast y) = \varTheta(x)\ast\varTheta(y)$. Since
    $\varTheta$ is injective, we have $\mathsf m^\psi_\cR(x,y) = x\ast
    y$.
\end{NB2}%
\end{NB}


Let us view $\mathsf m$, $\mathsf
m^\psi$ as elements of
$\Ext^*_{D_{G}(\Gr)}(\scA\star\scA,\scA)$. It is a module over
$H^*_G(\mathrm{pt})$. We consider the restriction functor
$\Res_{T_\cO,G_\cO}$ from the $G_\cO$-equivariant derived category to
the $T_\cO$-equivariant one. Then we have
\begin{equation*}
    \Ext^*_{D_{G}(\Gr)}(\scA\star\scA,\scA)
    \to 
    \Ext^*_{D_{T}(\Gr)}(\Res_{T_\cO,G_\cO}(\scA\star\scA),
    \Res_{T_\cO,G_\cO}(\scA)),
\end{equation*}
and the latter is a module over $H^*_T(\mathrm{pt}) = \CC[\ft]$.

We have
\begin{enumerate}
  \renewcommand{\theenumi}{\roman{enumi}}%
  \renewcommand{\labelenumi}{\textup{(\theenumi)}'}%
    \item $\mathsf m^\psi$ and $\mathsf m$ are equal in
  $\Ext^*_{D_{T}(\Gr)}(\Res_{T_\cO,G_\cO}(\scA\star\scA),
  \Res_{T_\cO,G_\cO}(\scA))\otimes_{\CC[\ft]}\CC(\ft)$. Hence $\mathsf
  m$ is commutative up to an element which vanishes in
  $\Ext^*_{D_{T}(\Gr)}(\Res_{T_\cO,G_\cO}(\scA\star\scA),
  \Res_{T_\cO,G_\cO}(\scA))\otimes_{\CC[\ft]}\CC(\ft)$.
\end{enumerate}

\begin{NB}
This statement recovers (\ref{item:prime})', as $H^*_{G_\cO}(\scA)\to
H^*_{T_\cO}(\scA)\otimes_{\CC[\ft]}\CC(\ft)$ is injective by
\lemref{lem:flat}.
\begin{NB2}
    The restriction $\Res_{T_\cO,G_\cO}\mathsf m$ defines
    $H^*_{T_\cO}(\scA)\otimes H^*_{G_\cO}(\scA)\to H^*_{T_\cO}(\scA)$,
    which is nothing but the right $H^{G_\cO}_*(\cR)$-module structure
    on $H^{T_\cO}_*(\cR)$ discussed in \subsecref{sec:bimodule}.
\end{NB2}%
\end{NB}

This statement recovers (\ref{item:equal}), as
$\mathrm{Ext}^*_{D_G(\Gr_G)}(\scA\star\scA,\scA)$ is the Weyl group
invariant part of
$\mathrm{Ext}^*_{D_T(\Gr_G)}(\Res_{T_\cO,G_\cO}(\scA\star\scA),
\Res_{T_\cO,G_\cO}(\scA))$, and
$\mathrm{Ext}^*_{D_T(\Gr_G)}(\Res_{T_\cO,G_\cO}(\scA\star\scA),
\Res_{T_\cO,G_\cO}(\scA))$ is a free $\CC[\ft]$-module: More
generally, for $\CF,\CG\in D_G(\Gr)$,
$\Ext^*_{D_T(\Gr)}(\Res_{T_\cO,G_\cO}(\CF),\Res_{T_\cO,G_\cO}(\CG))$
is a free $\BC[\ft]$-module. Indeed, by devissage it reduces to the
case of irreducible perverse $\CF,\CG$ where it is well known, see
e.g.~\cite{ginsburg}.

Let us suppress $\Res_{T_\cO,G_\cO}$ hereafter. 

Let us consider the commutative diagram \eqref{eq:74}. We have the
corresponding diagram for $\mathsf m_\cR$, the multiplication
constructed in \propref{prop:sheav-affine-grassm}, where the lower arrow
is $\mathsf m_{\Gr}$, cf.\ \ref{lem:convolution}.
We compose $\bz^*$ to get
\begin{equation*}
    \Ext^*_{D_{G}(\Gr)}(\scA\star\scA,\scA)
    \xrightarrow{\bz^*}
    \Ext^*_{D_{G}(\Gr)}(\scA\star\scA,\DC_{\Gr}).
\end{equation*}
The commutativity of the diagram and $\mathsf m_{\Gr} = \mathsf
m^\psi_{\Gr}$ (\subsecref{subsec:N=0}) imply that $\bz^*\mathsf m^\psi_\cR$
and $\bz^*\mathsf m_\cR$ are equal in 
$\Ext^*_{D_{G}(\Gr)}(\scA\star\scA,\DC_{\Gr})$.
Therefore it is enough to check that
\begin{equation*}
    \Ext^*_{D_{T}(\Gr)}(
    \scA\star\scA
    ,
    \scA
    )\otimes_{\CC[\ft]}\CC(\ft) 
    \xrightarrow{\bz^*}
    \Ext^*_{D_{T}(\Gr)}(
    \scA\star\scA
    ,
    \DC_\Gr
    )\otimes_{\CC[\ft]}\CC(\ft)
\end{equation*}
is an isomorphism. The argument is almost same as one in the proof of
\lemref{lem:injective}.

By the definition of $\bz^*$, it factors through
\[
   \Ext^*_{D_{T}(\Gr)}(\scA\star\scA, \pi_*\DC_{\cT}[-2\dim\bN_\cO])
   \otimes_{\CC[\ft]}\CC(\ft).
\]
Since $\cT\to\Gr$ is a vector bundle of rank $2\dim\bN_\cO$, we have
$\pi_*\DC_{\cT}[-2\dim\bN_\cO]\cong \DC_{\Gr}$.
\begin{NB}
    We have
    $\pi_*\DC_{\cT}[-2\dim\bN_\cO] \xrightarrow{z^*} \pi_* z_* z^* 
    \DC_{\cT}[-2\dim\bN_\cO] \cong \DC_{\Gr}$, which induces
    $H^{T_\cO}_*(\cT)\to H^{T_\cO}_*(\Gr)$.
    On the other hand, we have $\DC_{\Gr}\xrightarrow{\pi^*} \pi_*
    \pi^*\DC_{\Gr} \cong \pi_* \DC_{\cT}[-2\dim\bN_\cO]$, which
    induces $H^{T_\cO}_*(\Gr)\to H^{T_\cO}_*(\cT)$. They are inverse
    to each other.
\end{NB}%
Therefore it is enough to check that
\begin{equation*}
    \Ext^*_{D_{T}(\Gr)}(\scA\star\scA,
    \scA )\otimes_{\CC[\ft]}\CC(\ft) 
    \xrightarrow{i_*}
    \Ext^*_{D_{T}(\Gr)}(\scA\star\scA,
    \pi_*\DC_\cT)\otimes_{\CC[\ft]}\CC(\ft)
\end{equation*}
given by the closed embedding $i\colon\cR\to\cT$ is an
isomorphism. Let $j\colon\cT\setminus\cR\to \cT$ be the inclusion of
the complement. We have the distinguished triangle $i_!  i^!
\DC_\cT\to \DC_\cT \to j_* j^*\DC_\cT$. From the associated long exact
sequence, it is enough to show that 
\(
   \Ext^*_{D_{T}(\Gr)}(\scA\star\scA, \pi_* j_* j^*
   \DC_\cT)\otimes_{\CC[\ft]}\CC(\ft)
\) 
vanishes. But 
\(
\Ext^*_{D_{T}(\Gr)}(\scA\star\scA, \pi_* j_* j^*
   \DC_\cT)
   = \Ext^*_{D_{T}(\cT\setminus\cR)}(j^*\pi^* (\scA\star\scA),
   \DC_{\cT\setminus\cR})
\)
is an equivariant cohomology group over $\cT\setminus\cR$ which does
not contain $T$-fixed points by \lemref{lem:fixedpts}.
Therefore it is torsion and vanishes once we take a tensor product
with $\CC(\ft)$.

\begin{NB}
\renewenvironment{NB}{
\color{blue}{\bf NB2}. \footnotesize
}{}
\renewenvironment{NB2}{
\color{purple}{\bf NB3}. \footnotesize
}{}
\input{BD_temp}
\end{NB}

\section{Proof of \ref{thm:ABG}}
\label{monopole}
In this section we prove \ref{thm:ABG}.\footnote{The second
named author thanks Roman Bezrukavnikov for his numerous explanations about the
Andersen-Jantzen sheaves on Kleinian surfaces and nilpotent cones.}

During the proof, the $\CC^\times$-action on the Coulomb branch will
play an important role. The $\CC^\times$-action is given by the
homological grading, shifted according to the convention in
\ref{discrepancy}(2). Then the monopole formula in
\ref{prop:monopole_formula} is modified to
\begin{equation}\label{eq:modified}
    P_t^{\mathrm{mod}}(\cR)=
    \sum_{\lambda}t^{2\Delta(\lambda)}P_{G}(t;\lambda).
\end{equation}
As mentioned in \ref{discrepancy}(2), this modification is harmless as
the difference
$d_\lambda - 2\langle\rho,\lambda\rangle - \Delta(\lambda)$ depends
only on connected components of $\cR$. Nevertheless we will see that
this convention is a correct choice.

\subsection{Characters of global sections of line bundles on Kleinian surfaces}
\label{line Klein}
Recall that $\cS_N$ is the hypersurface in $\BA^3$ given by the equation
$zy=w^N$. It is the categorical quotient $\BA^2/\!\!/(\BZ/N\BZ)$ where
$\zeta\in\BZ/N\BZ$ takes $(u,v)\in\BA^2$ to $(\zeta u,\zeta^{-1}v)$.
We consider the following action of $\BC^\times\times\BC^\times$ on
$\BA^2\colon (x,t)\cdot(u,v)=(t^{-1}x^{-1}u,t^{-1}xv)$.   This action descends
to $\cS_N$. 
\begin{NB}
$\cS_N\colon (x,t)\cdot(w,y,z) = (t^{-2}w,t^{-N}x y,t^{-N}x^{-1}z)$.
\end{NB}%
The action of the second $\BC^\times$ 
$t\cdot(u,v)=(t^{-1}u,t^{-1}v)$ 
is a restriction of an $\SU(2)$-action on $\BA^2 = \RR^4$ rotating
hyper-K\"ahler structures. Hence it is natural in view of
\ref{discrepancy}.

We are interested in characters of certain
$\BC^\times\times\BC^\times$-equivariant sheaves on $\cS_N$. The tautological
characters of $\BC^\times\times\BC^\times$ will be denoted by $x$ and $t$.
We denote by $\pi\colon \wit\cS_N\to\cS_N$ the minimal resolution of $\cS_N$. 
The action of $\BC^\times\times\BC^\times$ lifts to $\wit\cS_N$. We recall the well
known facts about the $\BC^\times\times\BC^\times$-fixed points in $\wit\cS_N$.

We will denote these points by $p_0,\ldots,p_{N-1}$, so that the exceptional
divisor $E\subset\wit\cS_N$ consists of projective lines $E_1,\ldots,E_{N-1}$,
and $E_r$ contains $p_{r-1},p_r$. The character of the tangent space
$T_{p_r}\wit\cS_N$ is
\begin{NB}
$t^{N-2r-2}x + t^{2r-N}x^{-1}$.
\end{NB}%
$t^{N-2r-2}x^{-N}+t^{2r-N}x^N$.     
\begin{NB}
    Note $x$ is replaced by $x^{-1/N}$.
\end{NB}
The Picard group 
$\on{Pic}(\wit\cS_N)$ is canonically identified with the weight lattice of 
$\SL(N)$. Namely, $\wit\cS_N\subset T^*\CB$ is the preimage of a subregular 
(Slodowy)
slice in the Springer resolution $T^*\CB\to\CN$ of the nilpotent cone for 
$\SL(N)$. For an $\SL(N)$-weight $\lambda$ the corresponding line bundle
$\CL_\lambda$ on $\wit\cS_N$ is the restriction to $\wit\cS_N\subset T^*\CB$
of the pullback to $T^*\CB$ of the line bundle $\CO(\lambda)$ on the flag
variety $\CB$.
The line bundle $\CL_{\omega_i}$ corresponding to the fundamental
weight $\omega_i,\ 1\leq i\leq N-1$,
\begin{NB}
  I have swapped $i$ and $r$. HN.
\end{NB}%
admits a natural 
$\BC^\times\times\BC^\times$-equivariant structure such that the character of
its fiber at $p_r$ is $t^{N-i}x^{i-N}$
\begin{NB}
$t^{N-r}x^{r-N}$   
\end{NB}%
provided $0\leq r\leq i-1$, 
and $t^ix^i$
\begin{NB}
$t^rx^r$     
\end{NB}%
provided $i\leq r\leq N-1$. 
This is defined so that $\Gamma(\wit\cS_N,\CL_{\omega_i})$ is the
space of semi-invariants $\BC[\BA^2]^{\chi_i}$ where
$\chi_i(\zeta)=\zeta^i$.
Under the above identification,
$\CO_{\wit\cS_N}(-E_r)$ is nothing but $\CL_{\alpha_r}$ where $\alpha_r$ is the 
$r$-th simple root of $\SL(N)$.

We will write a dominant $\SL(N)$-weight $\lambda$ as a partition 
$\lambda_1\geq\lambda_2\geq\ldots\geq\lambda_N$ defined up to simultaneous
shift of all $\lambda_i$. In other words, 
$\lambda=\sum_{i=1}^{N-1}(\lambda_i-\lambda_{i+1})\omega_i$.
Then $\CL_\lambda$ admits a natural $\BC^\times\times\BC^\times$-equivariant
structure (as a tensor product of fundamental line bundles) such that
the character of its fiber at $p_r,\ 0\leq r\leq N-1$, is
$t^{\sum_{i=1}^N|\lambda_i-\lambda_{r+1}|}x^{\sum_{i=1}^N(\lambda_i-\lambda_{r+1})}$.
If $\lambda$ is not necessarily dominant, we get the character
$t^{\sum_{i=1}^r(\lambda_i-\lambda_{r+1})+\sum_{i=r+2}^N(\lambda_{r+1}-\lambda_i)}
x^{\sum_{i=1}^N(\lambda_i-\lambda_{r+1})}$.
\begin{NB}
  We have
  \(
  \prod_{i\le r} (t^ix^i)^{\lambda_i-\lambda_{i+1}}
  \prod_{i > r} (t^{N-1-i}x^{i-N})^{\lambda_i - \lambda_{i+1}}
  =
  t^{\sum_{i\le r}(\lambda_i - \lambda_{r+1})+\sum_{i\ge r+2}(\lambda_{r+1}-\lambda_i)}
  x^{\sum_i (\lambda_i - \lambda_{r+1})}
  \).
\end{NB}%

\begin{Lemma}
\label{monopole Klein}
For dominant $\lambda$, the character of $\Gamma(\wit\cS_N,\CL_\lambda)$
equals 
\begin{NB}
\begin{equation*}
\sum_{m\in\BZ}
x^{m}t^{\sum_{i=1}^N|\lambda_i-m|}(1+t^2+t^4+\ldots).  
\end{equation*}
\end{NB}%
$$\sum_{m\in\BZ}x^{\sum_{i=1}^N(\lambda_i-m)}t^{\sum_{i=1}^N|\lambda_i-m|}(1+t^2+t^4+\ldots).$$
\end{Lemma}

\begin{proof}
We compute the above expression as
\begin{equation*}
  \begin{split}
    & \frac{x^{\sum_{i=1}^N\lambda_i}}{1-t^2}\Biggl(
      \sum_{m=\lambda_1}^\infty x^{-Nm} t^{-\sum_{i=1}^N (\lambda_i - m)}
      \begin{aligned}[t]
        & + \sum_{m=\lambda_2}^{\lambda_1-1} x^{-Nm}
        t^{\lambda_1-m  -\sum_{i=2}^N (\lambda_i - m)}
      + \cdots
      \\ & \qquad \cdots
      + \sum_{m=-\infty}^{\lambda_N-1} x^{-Nm} t^{\sum_{i=1}^N (\lambda_i - m)}
    \Biggr)
  \end{aligned}
\\
    =\; & \frac{x^{\sum_{i=1}^N\lambda_i}}{1-t^2}
    \begin{aligned}[t]
          \Biggl( 
          \frac{x^{-N\lambda_1} t^{-\sum_{i>1} (\lambda_i-\lambda_1)}}{1 - x^{-N}t^N}
      - \frac{x^{-N\lambda_1} t^{-\sum_{i>1} (\lambda_i-\lambda_1)}}{1 - x^{-N}t^{N-2}}
      + & \frac{x^{-N\lambda_2} t^{\lambda_1-\lambda_2-\sum_{i>2}(\lambda_i-\lambda_2)}}
      {1 - x^{-N}t^{N-2}} + \cdots \\
      & \quad
      \cdots - \frac{x^{-N\lambda_N} t^{\sum_{i<N} \lambda_i - \lambda_N}}
      {1 - x^{-N}t^{-N}}
    \Bigr).
    \end{aligned}
  \end{split}
\end{equation*}
We combine $(2r-1)$th and $(2r)$th terms ($1\le r\le N$) to get
\begin{gather*}
  \frac{x^{\sum_i (\lambda_i - \lambda_{r})}
    t^{\sum_{i<r} (\lambda_i - \lambda_{r})
      - \sum_{i>r}(\lambda_i - \lambda_r)}}
  {1-t^2}\left(\frac1{1-x^{-N}t^{N-2r+2}} - \frac1{1-x^{-N}t^{N-2r}}\right) \\
  =
    \frac{x^{\sum_i (\lambda_i-\lambda_{r})} t^{\sum_{i<r} (\lambda_i - \lambda_{r})
      - \sum_{i>r}(\lambda_i - \lambda_r)}}
  {(1-x^{-N}t^{N-2r+2})(1-x^{N}t^{2r-N})}.
\end{gather*}
This is the contribution of $p_{r-1}$ to the Lefschetz fixed point
formula to the Euler characteristic of $\CL_\lambda$.  (The
denominator is $\Lambda_{-1} T^*_{p_{r-1}}\wit\cS_N$, and the
numerator is $(\CL_\lambda)_{p_{r-1}}$.)
Since $\lambda$ is dominant, higher cohomology vanishes. Hence this is
the character of $\Gamma(\wit\cS_N,\CL_\lambda)$.
\end{proof}

\begin{NB}
\begin{proof}
The statement is clear for $\lambda=0$ since we are computing the character
of $\BZ/N\BZ$-invariants in $\BC[\BA^2]$. More generally, for 
$\lambda=\omega_r=(1,\ldots,1,0,\ldots,0)\ (r$ 1's), the global sections
$\Gamma(\wit\cS_N,\CL_{\omega_r})$ is the space of semi-invariants 
$\BC[\BA^2]^{\chi_r}$ where $\chi_r(\zeta)=\zeta^r$, up to $x^{-r/N}$. The desired assertion 
follows since $|r-Nm|=r|1-m|+(N-r)|\!\!-m|$ for any $m\in\BZ$.
Now any dominant weight $\lambda$ is equal to 
$\bar{\lambda}+\sum_{i=1}^{N-1}a_i\alpha_i,\ a_i\in\BN,\ 
\bar{\lambda}\in\{0,\omega_1,\ldots,\omega_{N-1}\}$. More precisely, if we put 
$\omega_0:=0$, and $|\lambda|:=\lambda_1+\ldots+\lambda_N$, then
$\bar\lambda=\omega_{|\lambda|\pmod{N}}$.
So the proof proceeds by induction, with already established formula
for the character of $\Gamma(\wit\cS_N,\CL_{\omega_r})$ as the base.
Given a dominant $\lambda\not\in\{0,\omega_1,\ldots,\omega_{N-1}\}$, 
we can find either $1<r\leq N$ such that $\lambda_{r-1}>\lambda_r+1$, and
$\lambda'=(\lambda_1,\ldots,\lambda_{r-2},\lambda_{r-1}-1,\lambda_r+1,
\lambda_{r+1},\ldots,\lambda_N)$ is still dominant,  
or $1<r<s\leq N$ such that 
$\lambda_{r-1}>\lambda_r=\ldots=\lambda_{s-1}>\lambda_s$, and 
$$\lambda''=(\lambda_1,\ldots,\lambda_{r-2},\lambda_{r-1}-1,\lambda_r=\ldots=
\lambda_{s-1},\lambda_s+1,\lambda_{s+1},\ldots,\lambda_N)$$ is still dominant, 
and the monopole formula is known for $\Gamma(\wit\cS_N,\CL_{\lambda'})$ and
$\Gamma(\wit\cS_N,\CL_{\lambda''})$.
The difference of monopole formulas for $\Gamma(\wit\cS_N,\CL_{\lambda'})$ and
$\Gamma(\wit\cS_N,\CL_\lambda)$ is
\begin{equation*}
    \begin{split}
        & -\sum_{\lambda_r<m<\lambda_{r-1}} x^{m}t^{\sum_{i=1}^N|\lambda_i-m|}
        (1-t^{-2})(1+t^2+t^4+\ldots)
\\
   =\; & \sum_{\lambda_r<m<\lambda_{r-1}} x^{m}t^{-2+\sum_{i=1}^N|\lambda_i-m|}.
    \end{split}
\end{equation*}
\begin{NB2}
$\Gamma(\wit\cS_N,\CL_\lambda)$ is
$$-\sum_{\lambda_r<m<\lambda_{r-1}}x^{\sum_{i=1}^N(\lambda_i-m)}t^{\sum_{i=1}^N|\lambda_i-m|}
(1-t^{-2})(1+t^2+t^4+\ldots)=$$
$$=\sum_{\lambda_r<m<\lambda_{r-1}}x^{\sum_{i=1}^N(\lambda_i-m)}t^{-2+\sum_{i=1}^N|\lambda_i-m|}.$$
\end{NB2}
But this coincides with the character of 
$R\Gamma(\wit\cS_N,\CL_{\lambda'}/\CL_\lambda)$. More precisely, the above expression
\begin{NB2}
$$\sum_{\lambda_r<m<\lambda_{r-1}}x^{\sum_{i=1}^N(\lambda_i-m)}t^{-2+\sum_{i=1}^N|\lambda_i-m|}$$
\end{NB2}%
is the character of the global sections of the the line bundle
$\CL_{\lambda'}/\CL_\lambda$ on $E_r$.

Furthermore, the difference of monopole formulas for 
$\Gamma(\wit\cS_N,\CL_{\lambda''})$ and $\Gamma(\wit\cS_N,\CL_\lambda)$ is
\begin{equation*}
    \begin{split}
       & -\sum_{\lambda_s<m<\lambda_{r-1}} x^{m}t^{\sum_{i=1}^N|\lambda_i-m|}
        (1-t^{-2})(1+t^2+t^4+\ldots)
       \\ =\; & \sum_{\lambda_s<m<\lambda_{r-1}}x^{m}
       t^{-2+\sum_{i=1}^N|\lambda_i-m|}.
    \end{split}
\end{equation*}
\begin{NB2}
$$-\sum_{\lambda_s<m<\lambda_{r-1}}x^{\sum_{i=1}^N(\lambda_i-m)}t^{\sum_{i=1}^N|\lambda_i-m|}
(1-t^{-2})(1+t^2+t^4+\ldots)=$$
$$=\sum_{\lambda_s<m<\lambda_{r-1}}x^{\sum_{i=1}^N(\lambda_i-m)}t^{-2+\sum_{i=1}^N|\lambda_i-m|}.$$
\end{NB2}
But this coincides with the character of 
$R\Gamma(\wit\cS_N,\CL_{\lambda''}/\CL_\lambda)$. More precisely, 
this expression
\begin{NB2}
$$\sum_{\lambda_s<m<\lambda_{r-1}}x^{\sum_{i=1}^N(\lambda_i-m)}t^{-2+\sum_{i=1}^N|\lambda_i-m|}$$
\end{NB2}%
is the character of the global sections of the line bundle
$\CL_{\lambda''}/\CL_\lambda$ on $\bigcup_{r\leq i\leq s}E_i$.
\end{proof}
\end{NB}

\subsection{Pushforwards of line bundles on Kleinian surfaces}
\label{push Klein}
For dominant $\lambda$ we denote by $\CF_\lambda$ the torsion free sheaf
$R\pi_*\CL_\lambda=\pi_*\CL_\lambda$ on $\cS_N$. We also set
$\bar\lambda=\omega_{|\lambda|\pmod{N}}$ where $\omega_0:=0$.

\begin{Lemma}
\label{charac}
For dominant weight $\lambda$ let $\CF$ be a 
$\BC^\times\times\BC^\times$-equivariant torsion-free sheaf on $\cS_N$
such that the character of $\Gamma(\cS_N,\CF)$ coincides with the
character of $\Gamma(\cS_N,\CF_\lambda)$. Then 

\textup{(a)} The restriction $\CF|_{\cS_N^\circ}$ is a line bundle, isomorphic
to $\CF_{\bar\lambda}|_{\cS_N^\circ}$. Here $\cS_N^\circ :=\cS_N\setminus\{0\}$.

\textup{(b)} An isomorphism in \textup{(a)} is defined uniquely up to
multiplication by a scalar, even if one forgets the 
$\BC^\times\times\BC^\times$-equivariance.

\textup{(c)} The composition of isomorphisms
$\CF|_{\cS_N^\circ}\simeq\CF_{\bar\lambda}|_{\cS_N^\circ}\simeq\CF_\lambda|_{\cS_N^\circ}$
gives an isomorphism $\CF|_{\cS_N^\circ}\iso\CF_\lambda|_{\cS_N^\circ}$ which extends to
an isomorphism $\CF\iso\CF_\lambda$.
\end{Lemma}

\begin{proof}
An automorphism of a line bundle on $\cS_N^\circ$ is given by multiplication
by an invertible function on $\cS_N^\circ$. Any invertible function on
$\cS_N^\circ$ is constant. Indeed, it lifts to $\BA^2\setminus\{0\}$ where
all the invertible functions are constant. Hence uniqueness in \textup{(b)}.

A torsion free sheaf $\CF$ is locally free on the complement of $\cS_N$
to finitely many points. Due to the $\BC^\times\times\BC^\times$-equivariance,
$\CF$ is locally free on $\cS_N^\circ$. Let us
denote by $j\colon \cS_N^\circ\hookrightarrow\cS_N$ the open embedding.
Then $\CF\hookrightarrow j_*(\CF|_{\cS_N^\circ})$. Since we know
the character of $\Gamma(\cS_N,\CF)$, we conclude that $\CF$ is
generically of rank one, i.e.\ $\CF|_{\cS_N^\circ}$ is a line bundle.
Now $\on{Pic}(\cS_N^\circ)=\BZ/N\BZ$, and any line bundle on $\cS_N^\circ$
is isomorphic to $\CF_{\bar\mu}|_{\cS_N^\circ}$ for 
$\bar\mu\in\{0,\omega_1,\ldots,\omega_{N-1}\}$.
Thus $\CF\hookrightarrow j_*(\CF_{\bar\mu}|_{\cS_N^\circ})$ (if we disregard the
$\BC^\times\times\BC^\times$-equivariant structure). But any two
$\BC^\times\times\BC^\times$-equivariant structures on the line bundle
$\CF_{\bar\mu}|_{\cS_N^\circ}$ are isomorphic up to twist by a character $\chi$
of $\BC^\times\times\BC^\times$. So we have a 
$\BC^\times\times\BC^\times$-equivariant embedding
$\CF\otimes\chi\hookrightarrow j_*(\CF_{\bar\mu}|_{\cS_N^\circ})$.
We claim that $\bar\mu$ is congruent to $\lambda$ modulo the root lattice,
that is $\bar\mu=\omega_{|\lambda|\pmod{N}}=\bar\lambda$. 
Indeed, we take a sufficiently negative $m$ in the formula 
of~\lemref{monopole Klein} for the character of $\Gamma(\cS_N,\CF)$, so that
$\lambda_i-m>0$ for any $i$. Then $\sum_i(\lambda_i-m)=-Nm+\sum_i\lambda_i$,
and so $|\lambda|\pmod{N}$ is determined from the character of 
$\Gamma(\cS_N,\CF)$.

However, 
$\CF_{\bar{\lambda}}\iso j_*(\CF_{\bar{\lambda}}|_{\cS_N^\circ})$ (see~\lemref{CM} below 
and restrict to a subregular slice $\cS_N\subset\CN$). Thus we have 
$\CF\otimes\chi\hookrightarrow\CF_{\bar{\lambda}}\hookleftarrow\CF_\lambda$,
and we have to check that the images of $\CF\otimes\chi$ and $\CF_\lambda$ inside
$\CF_{\bar{\lambda}}$ coincide, and $\chi=1$. But the character of 
(global sections of) $\CF_{\bar{\lambda}}$ is multiplicity free, 
and the characters of $\CF_{\lambda_1}\otimes\chi_1,\ \CF_{\lambda_2}\otimes\chi_2$
coincide if and only if $\lambda_1=\lambda_2,\ \chi_1=\chi_2$,
so the equality of characters of
$\CF$ and $\CF_\lambda$ guarantees $\chi=1$ and the coincidence of the images
of $\CF$ and $\CF_\lambda$ in $\CF_{\bar\lambda}$.
\end{proof}

\subsection{Line bundles on Kleinian surfaces via homology groups of fibers}
\label{Klein via}
Recall the setup of~\ref{subsec:affG_flavor}
and~\ref{subsec:line-bundles-via}.  We consider the quiver gauge
theory of type $A_1$ with $\dim V=1,\ \dim W=N$ with
$G=\GL(V)=\BC^\times$, $\tilde G = \GL(V)\times \GL(W)/Z$,
$G_F=\PGL(W)=\PGL(N)$, and varieties of triples $\cR$, $\tilde\cR$ for
$(G,\bN)$, $(\tilde G,\bN)$ and the corresponding complex $\tilscA$
\begin{NB} or $\scA^{\on{for}}$?
\end{NB}%
on $\Gr_{\PGL(N)}$. See \subsecref{subsec:affG_flavor}.  We are
interested in its costalks at the points $\lambda\in\Gr_{\PGL(N)}$
where $\lambda=(\lambda_1\geq\lambda_2\geq\ldots\geq\lambda_N)$ is a
dominant coweight of $\PGL(N)$. According to~\eqref{eq:47}, the
costalk $i^!_\lambda\tilscA^{\on{for}}$ forms a module over the
algebra $i^!_0\tilscA^{\on{for}}$.  The algebra
$i^!_0\tilscA^{\on{for}}$ is nothing but the Coulomb branch
$H_*^{G_\CO}(\cR)\simeq\BC[\cS_N]$ where $\bN=\on{Hom}(W,V)$
by \ref{subsec:abel_examples}. 
The costalk $i^!_\lambda\tilscA^{\on{for}}$ is nothing but
$H_*^{G_\CO}(\tilde\cR^{\lambda})$ where $\tilde\cR^{\lambda}$
is the fiber of
$\tilde\pi\colon \tilde\cR \to\Gr_{G_F}=\Gr_{\PGL(N)}$ over $\lambda\in\Gr_{G_F}$,
see~\eqref{eq:52}.

\begin{Lemma}
\label{torsion free}
The $i^!_0\tilscA^{\on{for}}$-module $i^!_\lambda\tilscA^{\on{for}}$ is torsion free.
\end{Lemma}

\begin{proof}
Both $i^!_0\tilscA^{\on{for}}$ and $i^!_\lambda\tilscA^{\on{for}}$ are free
$H^*_G(\on{pt})$-modules. So if $i^!_\lambda\tilscA^{\on{for}}$ had torsion, then
it would still have torsion after the base change to $H^*_T(\on{pt})$ and
localization to the generic point of $H^*_T(\on{pt})$. However, this is 
impossible since after this localization, $i^!_\lambda\tilscA^{\on{for}}$ becomes
a free (rank 1) $i^!_0\tilscA^{\on{for}}$-module by the Localization Theorem.
\end{proof}

Recall that $H^{G_\cO}_*(\cR)$ has an additional grading
induced from $\pi_0(\Gr_G) = \pi_1(G) = \pi_1(\CC^\times) \cong \ZZ$
compatible with the convolution product (see \ref{sec:grading}). We
also have an additional grading on $H^{G_\cO}_*(\tilde\cR^{\lambda})$ compatible with the $H^{G_\cO}_*(\cR)$-module structure from $\pi_0(\Gr_{\tilde G}) = \pi_1(\tilde G)$ in the same way.
We choose $\pi_0(\Gr_{\tilde G})\to\ZZ$ so that the connected
component of $\tilde\cR^{\lambda}$ corresponding to the $m$-th
component of $\Gr_{G}$
goes to $\sum_{i=1}^N (\lambda_i - m)$. This is well-defined as it is
invariant under simultaneous shift of all $\lambda_i$ and $m$.

\begin{Proposition}
\label{costalk Klein}
Under the identification 
$i^!_0\tilscA^{\on{for}}=H_*^{G_\CO}(\cR)\simeq\BC[\cS_N]$, the 
$i^!_0\tilscA^{\on{for}}$-module 
$i^!_\lambda\tilscA^{\on{for}}=H_*^{G_\CO}(\tilde\cR^{\lambda})$ is isomorphic
to the $\BC[\cS_N]$-module $\Gamma(\cS_N,\CF_\lambda)$. More precisely,

\textup{(a)} The localization of $i^!_\lambda\tilscA^{\on{for}}$ to $\cS_N^\circ$ 
is a line bundle isomorphic to $\CF_\lambda|_{\cS_N^\circ}$.

\textup{(b)} An isomorphism in \textup{(a)} is
defined uniquely up to multiplication by a scalar.

\textup{(c)} An isomorphism in \textup{(a)} extends to an isomorphism
$i^!_\lambda\tilscA^{\on{for}}\iso\Gamma(\cS_N,\CF_\lambda)$.
\end{Proposition}

\begin{proof}
\begin{NB}
The connected components of the affine Grassmannian $\Gr_G=\Gr_{\BC^\times}$ are
canonically numbered by $\BZ$; this gives rise to a grading of
$H_*^{G_\CO}(\tilde\cR^{\lambda})$. More precisely, to a homology class
supported on the connected component of $\tilde\cR^{\lambda}$ living over
the $m$-th compoment of $\Gr_G$, we assign its $x$-degree 
$\deg_x=\sum_{i=1}^N(\lambda_i-m)$. When $\lambda=0$ the multiplication in
$i^!_0\tilscA^{\on{for}}=H_*^{G_\CO}(\cR)$ is compatible with the grading
by~\ref{sec:grading}. For arbitrary $\lambda$ the action of
$i^!_0\tilscA^{\on{for}}$ on $i^!_\lambda\tilscA^{\on{for}}$ is compatible with the
grading for evident reasons.

The homological degree of
$H_*^{G_\CO}(\tilde\cR^{\lambda})$ gives rise to another grading.
More precisely, to a homology class of degree $d$ having $x$-degree $\deg_x$,
we assign its $t$-degree $\deg_t=d-\deg_x$.

By the monopole formula of~\ref{prop:monopole_formula} 
and~\ref{discrepancy}, the Hilbert series of the bigraded module
$H_*^{G_\CO}(\tilde\cR^{\lambda})$ is
$\sum_{m\in\BZ}x^{\sum_{i=1}^N(\lambda_i-m)}t^{\sum_{i=1}^N|\lambda_i-m|}(1+t^2+t^4+\ldots).$
By~\lemref{torsion free}, $H_*^{G_\CO}(\tilde\cR^{\lambda})$ is a torsion-free
$H_*^{G_\CO}(\cR)$-module.
Comparing its Hilbert series with the formula of~\lemref{monopole Klein} and 
applying the criterion of~\lemref{charac} we obtain the desired result.
\end{NB}%
By the monopole formula of~\ref{prop:monopole_formula} 
with the convention~\ref{discrepancy}(2), the Hilbert series of the bigraded module
$H_*^{G_\CO}(\tilde\cR^{\lambda})$ is
$\sum_{m\in\BZ}x^{\sum_{i=1}^N (\lambda_i - m)}
t^{\sum_{i=1}^N|\lambda_i-m|}(1+t^2+t^4+\ldots).$
By~\lemref{torsion free}, $H_*^{G_\CO}(\tilde\cR^{\lambda})$ is a torsion-free
$H_*^{G_\CO}(\cR)$-module.
Comparing its Hilbert series with the formula of~\lemref{monopole Klein} and 
applying the criterion of~\lemref{charac} we obtain the desired result.
\end{proof}

Let us write down the isomorphism more concretely when $\lambda$ is
the $\iialp$-th fundamental coweight $\omega_\iialp$.

Recall $w$, $y$, $z$ are identified with elements in
$H^{G_\cO}_*(\cR)$ as follows (see~\ref{subsec:abel_examples}):
\begin{itemize}
\item $w$ is the generator of $H^*_G(\mathrm{pt})$.
\item $y$ is the fundamental class of the fiber $\pi^{-1}(1)$, where
  $\pi\colon \cR\to\Gr_G \simeq\ZZ$.
\item $z$ is the fundamental class of the fiber $\pi^{-1}(-1)$.
\end{itemize}

The space $\Gamma(\widetilde{\cS}_N,\CL_{\omega_\iialp})$ of sections of
the line bundle corresponding to $\omega_\iialp$ is identified with the
space of semi-invariants $\BC[\BA^2]^{\chi_\iialp}$ where
$\chi_\iialp(\zeta)=\zeta^\iialp$. It has a linear basis
\begin{equation*}
  u^\iialp z^m w^k, \quad v^{N-\iialp} y^m w^k \quad(m,k\in\ZZ_{\ge 0}),
\end{equation*}
where $w = uv$, $z=u^N$, $y=v^N$.
  
Let us consider a coweight $(m,\underbrace{1,\dots,1}_{\text{$\iialp$
    times}},\underbrace{0,\dots,0}_{\text{$N-\iialp$ times}})$ ($m\in\ZZ$)
of $\tilde G$, where the first $m$ is a coweight of $G$. Let $r^m$
denote the fundamental class of the corresponding fiber for the
projection $\tilde\cR\to\Gr_{\tilde G}$. Note that the pairing between
the coweight above and weights of $\Hom(W,V)$ are $m-1$,\dots, $m-1$
($\iialp$ times) and $m$, \dots, $m$ ($N-\iialp$ times). Thus we have $\iialp$
negative terms if $m=0$, $N-\iialp$ positive terms if $m=1$, all negative
or all positive otherwise. Therefore
\begin{equation*}
  y r^m =
  \begin{cases}
    r^{m+1} & \text{if $m > 0$},
    \\
    w^\iialp r^{m+1} & \text{if $m=0$},
    \begin{NB}
      y u^\iialp = v^N u^\iialp = v^{N-\iialp} w^\iialp,
    \end{NB}
    \\
    w^N r^{m+1} & \text{if $m < 0$},
    \begin{NB}
      y u^\iialp z^{-m} = v^N u^\iialp z^{-m} = w^N u^\iialp z^{-m-1}
    \end{NB}
  \end{cases}
  \quad
  z r^m =
  \begin{cases}
    r^{m-1} & \text{if $m \le 0$},
    \\
    w^{N-\iialp} r^{m-1} & \text{if $m=1$},
    \begin{NB}
      z v^{N-\iialp} = w^{N-\iialp} u^\iialp,
    \end{NB}
    \\
    w^N r^{m-1} & \text{if $m > 1$},
  \end{cases}
\end{equation*}
by \ref{sec:abelian}. (Note that we can replace $\tilde G$ by
$\GL(V)\times T(W)/Z$ where $T(W)\subset\GL(W)$ is a maximal torus of
$\GL(W)$ as in \ref{subsec:flav-symm-group2}. Hence we can use
computation in \ref{sec:abelian}.)
Now we get an isomorphism
$i_{\omega_\iialp}^!\scAfor\xrightarrow{\cong}
\CC[\BA^2]^{\chi_\iialp}$ of
$\CC[\cM_C] = \CC[y,z,w]/(yz=w^N)$-modules by setting
\begin{equation*}
  w^{k} r^m =
  \begin{cases}
    v^{N-\iialp} y^{m-1} w^{k} & \text{if $m > 0$},\\
    u^\iialp z^{-m} w^{k} & \text{if $m \le 0$}.
  \end{cases}
\end{equation*}
\begin{NB}
  In particular,
  \begin{itemize}
  \item $u^\iialp$ is the fundamental class of the fiber of the above
    coweight with $m=0$.
  \item $v^{N-\iialp}$ is the fundamental class of the fiber of the above
    coweight with $m=1$.
  \end{itemize}
\end{NB}%

\begin{NB}
  Consider $l\omega_\iialp = (\underbrace{l,\dots,l}_{\text{$\iialp$
      times}}, \underbrace{0,\dots,0}_{\text{$N-\iialp$ times}})$. We
  denote the fundamental class of the fiber over
  $(m,l\omega_\iialp) = (m,l\omega_\iialp =
  (\underbrace{l,\dots,l}_{\text{$\iialp$ times}},
  \underbrace{0,\dots,0}_{\text{$N-\iialp$ times}})$ by ${}^lr^m$. We
  would like to express it in products of ${}^1 r^{m'} =$ (previous
  $r^{m'}$) with various $m'$. Suppose $0\le m\le l$.
  \begin{equation*}
    \underbrace{{}^1 r^0 \cdots {}^1 r^0}_{\text{$l-m$ times}}
    \underbrace{{}^1 r^1 \cdots {}^1 r^1}_{\text{$m$ times}}
    = {}^l r^m,
  \end{equation*}
  as pairings of weights of $\Hom(W,V)$ and coweights
  $(0,\omega_\iialp)$, $(1,\omega_\iialp)$ is $-1$, $0$ ($\iialp$
  times) and $0$, $1$ ($N-\iialp$ times).
  \begin{NB2}
    Hence ${}^l r^m$ is identified with $u^{(l-m)\iialp} v^{(N-\iialp)m}$.
  \end{NB2}%
  For $m < 0$ or $m > l$, we multiply ${}^l r^0$ or ${}^l r^l$ with
  $y$ or $z$ to get ${}^l r^m$ respectively. In particular,
  $\bigoplus_{l\ge 0} i^!_{l\omega_\iialp}\scAfor$ is generated by
  $i^!_{\omega_\iialp}\scAfor$ as a
  $i^!_0\scAfor = \CC[\cS_N]$-algebra. Thus we have a line bundle
  $\shfO(1)$ over
  $\Proj(\bigoplus_{l\ge 0}i^!_{l\omega_\iialp}\scAfor)$ and a
  morphism
  $\widetilde{\cS}_N\to \Proj(\bigoplus_{l\ge 0}
  i^!_{l\omega_\iialp}\scAfor)$ such that $\cL_{\omega_\iialp}$ is the
  pull-back of $\shfO(1)$.

  This defines a natural homomorphism
  \begin{equation*}
      i^!_{l\omega_\iialp}\scAfor \to 
      \Gamma(\Proj(\bigoplus_{l\ge 0}i^!_{l\omega_\iialp}\scAfor), \shfO(l))
      \cong \Gamma(\widetilde{\cS}_N,\cL_{\omega_\iialp}^{\otimes l}).
  \end{equation*}

  \begin{NB2}
      Why this is an isomorphism ? It seems that this is injective by
      a trivial reason, probably as
      $\bigoplus_{l\ge 0} i^!_{l\omega_\iialp}\scAfor$ is an integral
      domain. Then we compare characters.
  \end{NB2}
\end{NB}%

\subsection{Andersen-Jantzen sheaves on a nilpotent cone}
\label{AJ}
We denote by $\CN$ the nilpotent cone of $\algsl_N$.
\begin{NB}
  It was originally ${\mathfrak{gl}}_N$. It is hard to distinguish
  $\gl_N$ from $\algsl_N$, but we only have line bundles
  $\shfO(\lambda)$ for a dominant weight $\lambda$ of $\SL(N)$ = a
  dominant coweight of $\PGL(N)$, not of $\GL(N)$. Also
  $i_\lambda^!\tilscA^{\mathrm{for}}$ does not have a natural
  $\GL(N)$-equivariant structure. We have an equivariant
  structure, where the maximal torus is identified with the Pontryagin
  dual of $\pi_1(\tilde G)$. Here
  \begin{equation*}
    \pi_1(\tilde G) \cong 
    \underbrace{\ZZ\oplus\cdots\oplus\ZZ}_{\text{$N$ times}}/
    (N,N-1,\cdots,2,1)\ZZ \cong \ZZ^{N-1},
  \end{equation*}
  where the isomorphism is given by
  $[\lambda_1,\dots,\lambda_N] \mapsto
  (\lambda_1-N\lambda_N,\dots,\lambda_{N-1}-2\lambda_N)$. We have an
  exact sequence
  \begin{equation*}
    0 \to \pi_1(G)\cong\ZZ^{N-1} \to \pi_1(\tilde G) \cong \ZZ^{N-1}
    \to \pi_1(\PGL(N)) \cong \ZZ/N\ZZ \to 0,
  \end{equation*}
  where the first inclusion is given by $\lambda_1 = 0$, and the last
  projection is $[\lambda_1,\dots,\lambda_N]\mapsto \lambda_1\bmod N$.

  Now the Pontryagin dual of $\pi_1(G)$ is identified with a maximal
  torus of $\PGL(N)$ acting on $\CN$, while that of $\pi_1(\tilde G)$
  is a maximal torus of $\SL(N)$, which acts on line
  bundles. ($\PGL(N)$-action cannot be lifted to line bundles in
  general.)
\end{NB}%
We denote by
$\CB$ the flag variety of $\algsl_N$, and by $T^*\CB$ its cotangent
bundle. We denote by $\pi\colon T^*\CB\to\CN$ the Springer resolution.
We denote by $j\colon\BO_\reg\hookrightarrow\CN$ the embedding of the 
regular nilpotent orbit.
For a dominant weight $\lambda=(\lambda_1\geq\ldots\geq\lambda_N)$ we denote
by $\CO(\lambda)$ the line bundle on $T^*\CB$ obtained by the pullback of
the corresponding line bundle on $\CB$. It is known that 
$\CJ_\lambda:=\pi_*\CO(\lambda)=R\pi_*\CO(\lambda)$ is a torsion-free sheaf
on $\CN$ (an {\em Andersen-Jantzen sheaf}, 
see e.g.~\cite[Theorem~5.2.1]{brion}).

\begin{Lemma}
\label{CM}
For $\bar{\lambda}\in\{0,\omega_1,\ldots,\omega_{N-1}\}$ we have
$\CJ_{\bar\lambda}=j_*(\CJ_{\bar\lambda}|_{\BO_\reg})$.
\end{Lemma}

\begin{proof}
We have to check that $\CJ_{\bar\lambda}$ is Cohen-Macaulay. It follows from the
fact that its Grothendieck-Serre dual $R\pi_*(\CO(-\bar{\lambda}))$ has no 
higher cohomology by~\cite[Theorem~5.2.1]{brion}.
\end{proof}

Recall that according to~\cite{Lu-Green}, $\CN$ is isomorphic to the 
transversal slice $\ol\CW{}^{N\omega_1}_0$ in the affine Grassmannian 
$\Gr_{\GL(N)}$. Recall the factorization morphism $\varPi:=\pi_{N\omega_1^*}\circ 
s^{N\omega_1}_\mu\colon \CN=\ol\CW{}^{N\omega_1}_0\to\BA^{N\omega_1^*}$ 
of~\ref{semismall} (it is also called the 
{\em Gelfand-Tsetlin integrable system}).

\begin{Lemma}
\label{ssmall}
The morphism $\varPi\circ\pi\colon T^*\CB\to\BA^{N\omega_1^*}$ is flat.
\end{Lemma}

\begin{proof}
It suffices to prove that all the fibers of $\varPi\circ\pi$
have the same dimension $N(N-1)/2$. We recall the proof of~\ref{semismall}.
There the dimension estimate on the fibers of $\varPi$ followed from the
semismallness of the convolution morphism $\bq$. Under the identification
$\CN=\ol\CW{}^{N\omega_1}_0$, the Springer resolution $\pi\colon T^*\CB\to\CN$
corresponds to the iterated convolution morphism ${\mathbf m}\colon 
\Gr^{\omega_1}_{\GL(N)}\wit\times\ldots\wit\times\Gr^{\omega_1}_{\GL(N)}\to\Gr_{\GL(N)}$
restricted to the slice $\ol\CW{}^{N\omega_1}_0\subset\Gr_{\GL(N)}$. 
Now the convolution morphism ${\mathbf m}$ is semismall, and moreover, its 
composition with $\bq$ is semismall as well, so the proof of~\ref{semismall}
goes through in the present situation as well.
\end{proof}  

\subsection{Andersen-Jantzen sheaves via homology groups of fibers}
\label{AJ via}
We change the setup of~\ref{Klein via} to that of~\ref{subsec:ABG}. 
According to~\eqref{eq:47}, the costalk
$i^!_\lambda\tilscA^{\on{for}}$ forms a module over the algebra $i^!_0\tilscA^{\on{for}}$.
The algebra $i^!_0\tilscA^{\on{for}}$ is nothing but the Coulomb branch
$H_*^{G_\CO}(\cR)\simeq\BC[\CN]$.
\begin{NB}
where $\bN=\on{Hom}(W,V_1)\oplus
\on{Hom}(V_1,V_2)\oplus\ldots\oplus\on{Hom}(V_{N-2},V_{N-1})$. 
\end{NB}%
The costalk $i^!_\lambda\tilscA^{\on{for}}$ is nothing but 
$H_*^{G_\CO}(\tilde\cR^{\lambda})$ where 
$\tilde\pi\colon \tilde\cR\to\Gr_{G_F}=\Gr_{\PGL(N)}$
and $\tilde\cR^{\lambda} = \tilde\pi^{-1}(\lambda)$, 
see~\eqref{eq:52}.

We have the $L^{\on{bal}}=\PGL(N)$-action on the Coulomb branch
$H^{G_\cO}_*(\cR)$ by~\ref{prop:Integrable} and~\ref{ex:psln}. 
By \ref{ex:stab} it coincides with the standard action on $\CN$.

\begin{Theorem}
\label{costalk nilpotent}
Under the identification 
$i^!_0\tilscA^{\on{for}}=H_*^{G_\CO}(\cR)\simeq\BC[\CN]$, the 
$i^!_0\tilscA^{\on{for}}$-module 
$i^!_\lambda\tilscA^{\on{for}}=H_*^{G_\CO}(\tilde\cR^{\lambda})$ is isomorphic
to the $\BC[\CN]$-module $\Gamma(\CN,\CJ_\lambda)$.
\end{Theorem}

\begin{proof}
The $\BC[\CN]$-module $i^!_\lambda\tilscA^{\on{for}}$ is torsion free generically 
of rank 1, see~\lemref{torsion free}.
By \ref{prop:Integrable} and~\ref{ex:psln}, we have an action of
$\tilde{L}^{\on{bal}}=\SL(N)$ on $i^!_\lambda\tilscA^{\on{for}}$.
The $i^!_0\tilscA^{\on{for}}$-module $i^!_\lambda\tilscA^{\on{for}}$
is $\SL(N)$-equivariant (under the natural projection $\SL(N)\to\PGL(N)$).
Hence, the restriction of the associated coherent sheaf 
$(i^!_\lambda\tilscA^{\on{for}})_{\on{loc}}$ to
$\BO_\reg\subset\CN$ is a line bundle. Now $\on{Pic}(\BO_\reg)=\BZ/N\BZ$, 
and any line bundle on $\BO_\reg$ is isomorphic to $\CJ_{\bar\mu}|_{\BO_\reg}$
for $\bar\mu\in\{0,\omega_1,\ldots,\omega_{N-1}\}$. Thus we obtain an
embedding $i^!_\lambda\tilscA^{\on{for}}\hookrightarrow
\Gamma(\CN,j_*(\CJ_{\bar\mu}|_{\BO_\reg}))=\Gamma(\CN,\CJ_{\bar\mu})$.
We claim that $\bar\mu=\omega_{|\lambda|\pmod{N}}=\bar\lambda$.
Indeed, $\SL(N)$-module $i^!_\lambda\tilscA^{\on{for}}$ has the same central 
character as $V^{\bar\lambda}$. 

Thus we obtain an embedding  
$i^!_\lambda\tilscA^{\on{for}}\hookrightarrow\Gamma(\CN,\CJ_{\bar\lambda})$.
Similarly, we have an embedding 
$\Gamma(\CN,\CJ_\lambda)\hookrightarrow\Gamma(\CN,\CJ_{\bar\lambda})$.
In other words, denoting $i^!_\lambda\tilscA^{\on{for}}|_{\BO_\reg}$ the
restriction of $(i^!_\lambda\tilscA^{\on{for}})_{\on{loc}}$ to $\BO_\reg$,
we obtain an isomorphism of line bundles
$i^!_\lambda\tilscA^{\on{for}}|_{\BO_\reg}\simeq\CJ_\lambda|_{\BO_\reg}$. Note that
this isomorphism is defined uniquely up to a scalar multiplication since the 
automorphism group of any
line bundle on $\BO_\reg$ is $\Gamma(\BO_\reg,\CO^\times)=\BC^\times$.
Indeed, an invertible function on $\BO_\reg$ extends to a regular function on
$\CN$ due to normality of $\CN$. This extended function is still invertible
since otherwise its zero divisor would intersect $\BO_\reg$. Its lift to
$T^*\CB$ is invertible and hence constant on each fiber of $T^*\CB\to\CB$.
So it is lifted from $\CB$ and hence constant.

We will show that the above isomorphism extends to $\CN$.
To this end we use the factorization morphism 
$\varPi\colon \CN\to\BA^{N\omega_1^*}=\ft(V)/\Weyl$
as in~\ref{prop:flat} and~\ref{rem:conditions},
 where $\ft(V)$ is a Cartan subalgebra of $\fg=\gl(V)$,
and $\Weyl$ is the Weyl group of $(\gl(V),\ft(V))$. 
The condition $\varPi_*\CJ_\lambda=\varPi_*\pi_*\CO(\lambda)\iso 
j_*\varPi_*\pi_*\CO(\lambda)|_{T^*\CB^\bullet}=j_*\varPi_*\CJ_\lambda|_{\CN^\bullet}$
of~\ref{rem:conditions} is satisfied since the complement of $T^*\CB^\bullet$
in $T^*\CB$ is of codimension 2 by~\ref{ssmall}.
So it suffices to check the regularity of our
rational isomorphism after the base
change $\ft(V)\to\ft(V)/\Weyl$ and localizations
at general points of the root hyperplanes. Moreover, since we already 
know that our isomorphism is regular at $\BO_\reg$,
it remains to check the regularity at the localizations at general points
of the coordinate hyperplanes $w_{1,r}=0,\ r=1,\ldots,N-1$, cf.\ the proof
of \ref{Coulomb_quivar}. By an application of the Localization Theorem, just
as in {\em loc.~cit.}, the comparison reduces to~\propref{costalk Klein}.
Namely, let $t$ be a general point of the hyperplane $w_{1,r}=0$, and let
$x$ be a point of the subregular nilpotent orbit above $t$. Then there is 
a slice $\cS_N\subset\CN$ through $x$ such that the isomorphism of
$\CJ_\lambda|_{\BO_\reg}$ and $i^!_\lambda\tilscA^{\on{for}}|_{\BO_\reg}$ restricted
to $\cS_N$ extends to the localization $(\cS_N)_t$ (by~\propref{costalk Klein}).
Due to the $\SL(N)$-equivariance, the pullback of the above isomorphism to
$\SL(N)\times\cS_N^\circ\stackrel{\on{act}}{\longrightarrow}\BO_\reg$ extends to
$(\SL(N)\times\cS_N)_t$. By the faithfully flat descent, the above isomorphism
extends to $\CN_t$, and hence to the whole of $\CN$.
\end{proof}

\subsection{Modified homological grading}
\label{Delta}

Let us write down the modified monopole formula \eqref{eq:modified} in
our case explicitly. (This appeared first in
\cite[(3.9)]{Cremonesi:2014kwa}.) It is
\begin{NB}
Following~\ref{discrepancy}, we modify the homological grading of
$i^!_\lambda\tilscA^{\on{for}}=H_*^{G_\CO}(\tilde\cR^{\lambda})=
\bigoplus_{\vec\Lambda}H_*^{G_\CO}(\tilde\cR^{\lambda}_{\vec\Lambda})$
(the direct sum over the connected components of $\Gr_{\GL(V)}$
numbered by the $(N-1)$-tuples of integers
$\vec\Lambda=(\Lambda^1,\ldots,\Lambda^{N-1})$).  Namely, by
subtracting from the homological degree a discrepancy
$${\mathfrak d}_{\vec\Lambda}:=(N-1)|\lambda|-2\Lambda^1-\ldots-2\Lambda^{N-1}$$ 
on each connected component, we obtain the 
following Hilbert series $P_t^{\on{mod}}(i^!_\lambda\tilscA^{\on{for}})$ of 
$i^!_\lambda\tilscA^{\on{for}}$ with modified homological 
grading~\cite[(3.9)]{Cremonesi:2014kwa}:
\end{NB}%
\begin{equation}
\label{eq:mod}
P_t^{\on{mod}}(i^!_\lambda\tilscA^{\on{for}})=
\sum_{\vec\lambda}t^{2\Delta(\vec\lambda)}P_{\GL(V)}(t,\vec\lambda)
\end{equation}
(the sum over the dominant coweights 
$\vec\lambda=(\lambda^1,\ldots,\lambda^{N-1})$ of 
$\GL(N-1)\times\ldots\times\GL(1)$), where
$$2\Delta(\lambda,\lambda^1,\ldots,\lambda^{N-1}):=
\sum_{j=1}^{N-1}\sum_{i,i'}|\lambda_i^{j-1}-\lambda_{i'}^j|-
2\sum_{j=1}^{N-1}\sum_{i<i'}|\lambda_i^j-\lambda_{i'}^j|,$$
and we set for convenience $\lambda^0:=\lambda$.
We also set $n(\lambda)=\sum_{i=1}^N(i-1)\lambda_i$. Then $\dim\Gr_{\PGL(N)}^\lambda
=\langle2\rho^\vee_{\PGL(N)},\lambda\rangle=(N-1)|\lambda|-2n(\lambda)$.

\begin{Lemma}
\label{perverse degree}
$i^!_\lambda\tilscA^{\on{for}}$ lives in (modified) degrees $\geq\dim\Gr_{\PGL(N)}^\lambda$,
and its component of this degree has the same dimension as the irreducible 
$\SL(N)$-module $V^\lambda$.
\end{Lemma}

\begin{proof}
We have to compute $$t^{2n(\lambda)-(N-1)|\lambda|}\sum_{\vec\lambda}
t^{2\Delta(\vec\lambda)}P_{\GL(V)}(t,\vec\lambda)|_{t=0}=t^{2n(\lambda)-(N-1)|\lambda|}
\sum_{\vec\lambda}t^{2\Delta(\vec\lambda)}|_{t=0}.$$ One checks that this is the sum of 1's over
the set of $(N-1)$-tuples $\vec\lambda$ which interlace, i.e.\
$\lambda_i^j\geq\lambda_i^{j+1}\geq\lambda^j_{i+1},\ 0\leq j\leq N-2,\ 
1\leq i\leq N-j-1$ (recall that $\lambda^0=\lambda$).\footnote{We learned this observation in \cite{268987}.}
In other words, this is the cardinality of the set of Gelfand-Tsetlin patterns
of shape $\lambda$, that is $\dim V^\lambda$.
\end{proof}


\begin{Remark}
  Characters of $\Gamma(\CN,\CJ_\lambda)$ are given by Hall-Littlewood
  polynomials by computation of Euler characteristic
  \cite{MR593631,Brylinski} and the vanishing theorem
  \cite{MR1223221}. Therefore \eqref{eq:mod} gives a combinatorial
  expression of Hall-Littlewood polynomials. We asked several people
  (including mathoverflow \cite{268865}) whether it is known or
  not. But we could not find earlier appearance. In view of the
  argument in the special case $t=0$ in \ref{perverse degree}, there
  should be a purely combinatorial proof.
\end{Remark}

\subsection{Modified grading of Andersen-Jantzen modules}
\label{delta}
We have the dilatation action of $\BC^\times$ on $T^*\CB$ and the natural 
$\BC^\times$-equivariant structure on $\CO(\lambda)$; hence a grading on
$\Gamma(T^*\CB,\CO(\lambda))=\Gamma(\CN,\CJ_\lambda)$ starting in degree 0
with $\Gamma(\CB,\CO(\lambda))=(V^\lambda)^\vee$.
We modify the grading by doubling all the degrees and shifting it by
$(N-1)|\lambda|-2n(\lambda)$. From now on we consider $\Gamma(\CN,\CJ_\lambda)$
with this modified grading only.

\begin{Theorem}
\label{costalk graded}
The isomorphism of $\BC[\CN]$-modules 
$i^!_\lambda\tilscA^{\on{for}}\simeq\Gamma(\CN,\CJ_\lambda)$ of~\thmref{costalk nilpotent}
is a graded isomorphism.
\end{Theorem}

\begin{proof}
For $\lambda=0$ the claim is nothing but~\ref{compare_degrees_slice}.
Clearly, $\Gamma(\CN,\CJ_\lambda)$ is a graded $\SL(N)\ltimes\BC[\CN]$-module; 
$i^!_\lambda\tilscA^{\on{for}}$ is also a graded $\SL(N)\ltimes\BC[\CN]$-module by
construction of~\ref{sec:group_action} (see~\ref{ex:psln}). Both embeddings 
$i^!_\lambda\tilscA^{\on{for}}\hookrightarrow\Gamma(\CN,\CJ_{\bar\lambda})$ and
$\Gamma(\CN,\CJ_\lambda)\hookrightarrow\Gamma(\CN,\CJ_{\bar\lambda})$ are compatible
with the gradings up to a shift since the structure of a $\SL(N)$-equivariant
line bundle on $\CJ_{\bar\lambda}|_{\BO_\reg}$ extends to a 
$\SL(N)\times\BC^\times$-equivariant structure uniquely up to tensoring with
a character of $\BC^\times$. The above shifts match because both the grading
of $i^!_\lambda\tilscA^{\on{for}}$ and the grading of $\Gamma(\CN,\CJ_\lambda)$ start
in the same degree $(N-1)|\lambda|-2n(\lambda)$.
\end{proof}

\subsection{The regular sheaf}
\label{ABG}
By~\lemref{perverse degree}, $\tilscA\in{} ^p\!D^{\geq0}_{\PGL(N)}(\Gr_{\PGL(N)})$,
\begin{NB}
We need some abstract nonsense about Ind.
\end{NB}
and $^p\!H^0(\tilscA)\simeq
\bigoplus_\lambda(V^\lambda)^\vee\otimes\on{IC}(\ol\Gr{}^\lambda_{\PGL(N)})=:\Areg$.

\begin{Theorem}
\label{Areg}
The natural morphism $\sigma\colon \Areg={} ^p\!H^0(\tilscA)\to\tilscA$ 
is an isomorphism of ring objects.
\end{Theorem}

\begin{proof}
First we prove that $\sigma$ is an isomorphism disregarding the ring structure.
We have to check $\tau_{>0}\tilscA=\on{Cone}(\sigma)=0$. 
Note that all the costalks of $\on{IC}(\ol\Gr{}^\lambda_{\PGL(N)})$ live in the 
degrees of the same parity as $|\lambda|$, see~\cite{Lus-ast}. 
We will call this phenomenon
{\em parity vanishing}. The parity vanishing for $\tilscA$ also holds true
(on a given connected component of $\Gr_{\PGL(N)}$, all the costalks of
$\tilscA^{\on{for}}$ live in the same parity as all the costalks of any IC sheaf on 
this component, see~\eqref{eq:mod}). This implies that 
$^p\!H^{\on{odd}}(\tilscA)=0$, and hence
$^p\!H^{\on{odd}}(\tau_{>0}\tilscA)=0$. Now the Hilbert series of
$i^!_\lambda\tilscA^{\on{for}}$ and $i^!_\lambda\Areg$ coincide by~\thmref{costalk graded}
and the comparison of~\cite{Brylinski} and~\cite{Lus-ast}.
Hence if $\sigma$ were not an isomorphism, its costalk $\sigma_\lambda$ would
have both kernel and cokernel for some $\lambda$. Thus, $\on{Cone}(\sigma)$
would have a costalk of wrong parity at $\lambda$. This would contradict
the parity vanishing for $\tau_{>0}\tilscA=\on{Cone}(\sigma)$.
We conclude that $\sigma$ is an isomorphism.

Now we compare the ring structures. Since both $\tilscA$ and $\Areg$
are perverse, it suffices to check that the fiber functor $H^\bullet(\sigma)$
induces an isomorphism of the rings $H^\bullet(\Gr_{\PGL(N)},\Areg)$
and $H^\bullet(\Gr_{\PGL(N)},\tilscA^{\on{for}})$. 
It is enough to check the assertion for $\GL(N)$ instead of $\PGL(N)$,
as $\Gr_{\GL(N)}$ is the union of copies of $\Gr_{\PGL(N)}$.
We have 
$H^\bullet(\Gr_{\GL(N)},\Areg)\simeq\BC[\GL(N)]$ by geometric Satake equivalence. On the other hand, the cohomology
$H^\bullet(\Gr_{\GL(N)},\tilscA^{\on{for}})$ is the quotient of the equivariant cohomology
$H^\bullet_{\GL(N)}(\Gr_{\GL(N)},\tilscA)$ modulo the augmentation ideal of
$H^\bullet_{\GL(N)}(\on{pt})$. And $H^\bullet_{\GL(N)}(\Gr_{\GL(N)},\tilscA)=
H_*^{(\GL(V)\times\GL(W))_\CO}(\cR_{\GL(V)\times\GL(W),\bN})$ where $(\GL(V)\times\GL(W),\bN)$ is the quiver gauge theory obtained from $(\GL(V),\bN)$ by turning $\GL(W)$ to a gauge group. By~\ref{pestun}, 
$H_*^{(\GL(V)\times\GL(W))_\CO}(\cR_{\GL(V)\times\GL(W),\bN})\simeq\BC[\oZ^\alpha_{\PGL(N+1)}]$
where $\alpha=N\alpha_1+(N-1)\alpha_2+\ldots+\alpha_N$.  
By~\cite[Theorem~1]{MR2443330}, $\oZ^\alpha_{\PGL(N+1)}\simeq\GL(W)\times W$,
and its projection to $\on{Spec}H^\bullet_{\GL(N)}(\on{pt})$ is nothing but
the projection of $\GL(W)\times W$ to $W$. Hence the zero fiber of this 
projection is isomorphic to $\GL(W)=\GL(N)$.
\end{proof}


\section{Mirrors of Sicilian theories}\label{sec:Sicilian}

In the first half of this section, we study examples of Coulomb
branches $\cM_C$ of star shaped quiver gauge theories as
in~\ref{fig:star}. As explained at the end of Introduction, they are
conjectural Higgs branches of Sicilian theories.

Let us briefly review \cite{MR2985331} on expected properties of Higgs
branches of Sicilian theories. It is conjectured that there exists a
functor from the category of $2$-bordisms to a category $\mathrm{HS}$
of holomorphic symplectic varieties with Hamiltonian group
actions. For the latter, objects are complex algebraic semisimple
groups. A homomorphism from $G$ to $G'$ is a holomorphic symplectic
variety $X$ with a $\CC^\times$-action scaling the symplectic form
with weight $2$ together with hamiltonian $G\times G'$ action
commuting with the $\CC^\times$-action. For $X\in \Hom(G',G)$,
$Y\in\Hom(G,G'')$, their composition $Y\circ X\in\Hom(G',G'')$ is
given by the symplectic reduction of $Y\times X$ by the diagonal
$G$-action. The identity $\in\Hom(G,G)$ is the cotangent bundle $T^*G$
with the left and right multiplication of $G$.

Let us fix a complex semisimple group $G$. Physicists associate a $3d$
Sicilian theory to $G$ and a Riemann surface with boundary, and
consider its Higgs branch. It depends only on the topology of the
Riemann surface, and gives a functor as above. We associate $S^1$ with
$G$, and a cylinder with $T^*G$. Since $T^*G$ is the identity in
$\mathrm{HS}$, it is one of requirements.

Physical argument shows that the variety associated with a disk is
$G\times S$, where $S$ is the Kostant slice to the regular nilpotent
orbit.

Let $W \equiv W_G$ be the variety associated with $S^2$ three disks
removed. This is a fundamental piece as other varieties are obtained
by reductions of products of its copies. It has an action of
$\mathfrak S_3\rtimes G^3$. It is expected that
\begin{itemize}
\item $W = \CC^2\otimes\CC^2\otimes\CC^2$ if $G = \SL(2)$.
  \item $W$ is the minimal nilpotent orbit of $E_6$ if $G = \SL(3)$.
\end{itemize}
For other groups, $W$ is unknown.

Recently Ginzburg and Kazhdan \cite{GK} construct a functor, and check
most of properties, in particular show that the gluing of Riemann
surfaces corresponds to the hamiltonian reduction with respect to the
diagonal action. Via a result of \cite{MR3366026} their symplectic
variety associated with $S^2$ minus $b$ disks is defined as
\begin{equation*}
    W^b \defeq
    \Spec H^*_{G_\cO}(\Gr_{G}, i_\Delta^!(\boxtimes_{k=1}^b(\Areg)_k)),
\end{equation*}
where $(\Areg)_k$ is a copy of the regular sheaf on $\Gr_G$. Here the
complex symplectic group taken as the object of the target category is
$G^\vee$, the Langlands dual group. (E.g., $b=2$ gives $T^*G^\vee$.)

By \ref{thm:ABG} together with \ref{subsec:glue} we immediately get the
following:
\begin{Theorem}\label{thm:GK}
  The symplectic variety $W^b$ of Ginzburg-Kazhdan for $G^\vee=\SL(N)$
  is isomorphic to the Coulomb branch of the star shaped quiver gauge
  theory in \ref{fig:star} with $b$ legs instead of $3$.
\end{Theorem}
More precisely, as we divide $\GL(V)\times\GL(W)$ by $\CC^\times$ in
\ref{subsec:ABG}, we also divide $\GL(V)$ for the star shaped quiver
gauge theory also by the diagonal central subgroup $Z\cong \CC^\times$.
\begin{NB}
  Consider the case of three legs.
  Take $\tilde G = \left(\prod_{i=1}^N \GL(i)\right)/\CC^\times$,
  $G = \prod_{i=1}^{N-1} \GL(i)$ so that $G_F = \tilde G/G =
  \PGL(N)$. Then
  $\PGL(V) = \tilde G\times_{\PGL(N)}\tilde G\times_{\PGL(N)}\tilde G$.
\end{NB}%
If we replace the central $\GL(N)$ by $\SL(N)$ instead of taking the
quotient by $Z$, we get $W$ for $G^\vee = \PGL(N)$.\footnote{We thank
  Yuji Tachikawa for an explanation of this procedure.}
\begin{NB}
    $\dim \mathcal M_C = 3N(N-1) + 2(N-1) = 3N^2 - N - 2$. On the
    other hand, the expected dimension \cite[(3.14)]{MR2985331} is
    $3\dim G - \rank G = 3 N^2 - 3 - (N-1) = 3N^2 - N - 2$.
\end{NB}%

We could consider the Coulomb branch for more general quiver gauge
theory associated with a Riemann surface with boundary as in
\cite[3(iii) Figure~5]{2015arXiv150303676N} (namely we have $b$ legs,
as well as $g$ loops at the central vertex), which is the Higgs branch
of a $3d$ Sicilian theory, obtained by compactifying $6d$
$\mathcal N=(2,0)$ theory of type $A$ by
$S^1\times(\text{punctured Riemann surface})$. Ginzburg-Kazdhan
construction is also generalized. See \ref{cylinder G}. We conjecture
that \thmref{thm:GK} is generalized.

\begin{Conjecture}\label{conj:genGK}
    Let
    $W^{g,b} \defeq \Spec H^*_{G_\CO}(\Gr_G,\scA^b\otimes^!\scB^g)$ as
    in \ref{cylinder G} for $G^\vee = \SL(N)$. It is isomorphic to the
    Coulomb branch of the gauge theory associated with the quiver
    \cite[3(iii) Figure~5]{2015arXiv150303676N}.
\end{Conjecture}

By \ref{subsec:glue} it is enough to show that the complex
$\scB = \scB^{g=1}$ introduced in \ref{cylinder G} is isomorphic to
the object $\scA = \pi_*\DC_\cR[-2\dim\bN_\cO]$ associated with
$(G,\bN) = (\PGL(N),\mathfrak{pgl}(N))$. We conjecture that this is
true for general $G$ and its adjoint representation $\bN=\mathfrak{g}$.
Note that we prove that
$\Spec H^*_{G_\cO}(\Gr_G,\scB) = (T^\vee\times\ft)/W$ (using Losev's
result in~\ref{subsec:Loseu}), which coincides with the Coulomb branch for the
adjoint representation.  (See
\ref{subsec:prev}\oldref{Coulomb2-subsub:adjoint}.)

The remainder of this section is as follows.
In the first five subsections, we study examples of $\cM_C$, in
particular check two cases $\SL(2)$, $\SL(3)$ above. These are
basically applications of \cite{blowup} and \ref{sec:group_action},
and we will not use the sheaf $\scA$.
In the subsequent five subsections, we show the gluing property and
also $W^{b=2} = T^*G^\vee$. They were shown in \cite{GK}, but we give
proofs for completeness. They are direct consequences of
\cite{MR2422266}.
In \ref{subsec:glu_ham,subsec:gluHiggs} we explain similarities
between the gluing property and hamiltonian reduction.

Let us use the following notation as in \cite{blowup}. Let $Q$ be a
quiver with sets $I$, $\Omega$ of vertices and arrows respectively. We
take an $I$-graded vector space $V = \bigoplus V_i$ with dimension
vector $\alpha = (\dim V_i)_{i\in I}$. We set
$\GL(V) = \prod \GL(V_i)$,
$\bN \equiv \bN^\alpha = \bigoplus \Hom(V_i,V_j)$, where the sum is
over the arrows $i\xrightarrow{h} j\in\Omega$. We also take the
diagonal central subgroup $Z=\CC^\times\subset\GL(V)$ and set
$\PGL(V) = \GL(V)/Z$. We consider $\cR\equiv \cR_{\PGL(V),\bN}$ and
$\mathcal M_C = \Spec(H^{\PGL(V)_\cO}_*(\cR_{\PGL(V),\bN}))$.

\begin{Remark}
  Consider the regular sheaf $\Areg$ on $\Gr_G$. In type $A$, it
  arises as the ring object associated with a quiver gauge theory by
  \thmref{thm:ABG}. Using $\grpSp/\SO$ quiver as in
  \cite[App.A.2]{2015arXiv150303676N}, we can conjecture that $\Areg$
  for classical groups is constructed in a similar way, once we can
  generalize our definition to the case when $\bM$ is not necessarily
  of cotangent type. For exceptional groups, we do not expect that
  $\Areg$ appears in this way, as argued in
  \cite[\S3(i)]{2015arXiv150303676N}.
  Nevertheless it is expected that $\Areg$ arises from the $3d$
  $\mathcal N=4$ quantum field theory $T(G)$, which was introduced in
  Gaiotto-Witten \cite{MR2610576}. This theory is not a usual gauge
  theory nor a lagrangian theory for an exceptional group, hence is
  difficult to understand from a mathematical point of view. But it
  has a $G\times G^\vee$-symmetry, and its Higgs/Coulomb branches are
  nilpotent cones $\mathcal N$ and $\mathcal N^\vee$ of $G$ and
  $G^\vee$ respectively. The Sicilian theory $S_{G^\vee}(g,b)$
  associated with $b$ punctured genus $g$ Riemann surface $C$
  considered above is constructed from $T(G)$ by `gauging' 
  quantum field theories up to $3d$ mirror:
  \begin{equation}\label{eq:fund}
    S_{G^\vee}({g,b}) \overset{\text{3d mirror}}{\longleftrightarrow}
    T[G]^{b}\times\Hyp(\fg\oplus\fg^*)^g 
    \tslabar G_{\mathrm{diag}},
  \end{equation}
  where we use the notation $\tslabar$ for the gauging in
  \cite{Tach-review}. (See also \cite{2017arXiv171209456T}.) This
  observation was given in \cite{Benini:2010uu}. Note that we ignore
  the parameter $\tau$ in \cite[\S2.6]{Tach-review}. The deformation
  parameter, which corresponds to the complex structure of $C$, is not
  relevant to Higgs branches as complex symplectic varieties. Hence we
  can safely write $S_{G^\vee}(g,b)$ instead of $S_{G^\vee}(C)$, and
  understand that the Higgs branch of $S_{G^\vee}(g,b)$ is the Coulomb
  branch of the right hand side.
  A similarity between \eqref{eq:fund} and the definition
  $W^{g,b} = \Spec H^*_{G_\CO}(\Gr_G,\scA^b\otimes^!\scB^g)$ is
  clear. We identify $T[G]$ with $\scA$, $\Hyp(\fg\oplus\fg^*)$ with
  $\scB$, and $\tslabar G_{\mathrm{diag}}$ with taking
  $H^*_{G_\cO}(\Gr_G,\bullet)$ after the $!$-restriction to the
  diagonal subgroup. See \ref{subsec:glu_ham} for a further discussion.
  We thank Davide Gaiotto and Yuji Tachikawa for this remark.
\end{Remark}

\subsection{Cylinder}\label{subsec:cylinder}

Consider the two legs star shaped quiver gauge theory instead of three
legs in~\ref{fig:star}. It is a quiver gauge theory of type $A_{2N-1}$ with
$\dim V = (1,2,\dots,N-1,N,N-1,\dots,2,1)$. We first consider the
Coulomb branch for $\GL(V)$. By \subsecref{Co_bra}, $\mathcal
M_C(\GL(V),\bN)$ is the moduli space $\oZ^\alpha_{\PGL(2N)}$ of based maps from
$\CP^1$ to a flag variety of type $A_{2N-1}$ with degree $\alpha =
\dim V$. By \cite[Th.~7.2]{MR2443330}, it is isomorphic to
$T^*\GL(N)$.
\begin{NB}
  At this stage, there is no difference between $\SL(2N)$ and
  $\PGL(2N)$, but the stabilizer is for $\PGL(2N)$.

  Note also the maximal torus is
  $\pi_1(\GL(V))^\wedge = (\CC^\times)^{Q_0}$, which is naturally
  identified with the weight lattice of $\PGL(2N)$, when we consider
  the action on $\oW^\alpha_0$.
\end{NB}%

Note that $\alpha=2\omega_N$.  By~\ref{Cartan_grading}, we have an
isomorphism $\oZ^\alpha_{\PGL(2N)}\iso S_\alpha\cap\oW^\alpha_0$ and
the natural action of
$\on{Stab}_{\PGL(2N)}(\alpha)=\PGL(N,N):=
\GL(N)\times\GL(N)/\CC^\times$ on $\oZ^\alpha_{\SL(2N)}$, where
$\CC^\times$ is the diagonal central subgroup of $\GL(N)\times\GL(N)$.
It coincides with the natural action of $\GL(N)\times\GL(N)$ on
$T^*\GL(N)$ through the quotient homomorphism
$\GL(N)\times\GL(N)\twoheadrightarrow\PGL(N,N)$. By~\ref{ex:stab} this
action coincides with the one given in~\ref{sec:group_action}.
More precisely the $\PGL(N,N)$ action on $\oW^\alpha_0$ coincides with
the one given in \ref{sec:group_action}, and the embedding
$S_\alpha\cap\oW^\alpha_0\to \oW^\alpha_0$ is equivariant for both
actions, as it is given by \ref{open_piece} as Coulomb branches.

By \ref{Ober} $H^{\GL(V)_\cO}_*(\cR_{\GL(V),\bN})\twoheadrightarrow
H^{\GL(V)_\cO}_*(\cR_{\PGL(V),\bN})$ is nothing but the restriction to
the level set $F_\alpha^{-1}(1)$, where $F_\alpha$ is the boundary
function (see \subsecref{zastava}). In this particular case, we have
$F_\alpha = c\det$ for $c\in\CC^\times$: all the invertible regular
function on $\oZ^\alpha$ are of the form $c F_\alpha^k$, $k\in\ZZ$,
$c\in\CC^\times$ \cite[Lemma~5.4]{bdf}. Now by degree reasons, $\det =
c F_\alpha$.
\begin{NB}
    Misha, please check. Is it possible to remove the ambiguity of $c$
    ?
\end{NB}%
Therefore $H^{\GL(V)_\cO}_*(\cR_{\PGL(V),\bN})) \cong
\gl(N)\times\SL(N)$.  Moreover,
$H^{\PGL(V)_\cO}_*(\cR_{\PGL(V),\bN})) \to
H^{\GL(V)_\cO}_*(\cR_{\PGL(V),\bN}))$ is nothing but the projection
$\gl(N)\times\SL(N)\to \operatorname{Lie}\PGL(N)\times\SL(N)$.
\begin{NB}
    $H^*_{\PGL(V)}(\mathrm{pt}) = \CC[\operatorname{Lie}\PGL(V)]$,
    $H^*_{\GL(V)}(\mathrm{pt}) = \CC[\operatorname{Lie}\GL(V)]$. The
    projection is giving by the pull-back with respect to the
    projection $\operatorname{Lie}\GL(V)\twoheadrightarrow
    \operatorname{Lie}\PGL(V)$.
  \end{NB}%
  Identifying $\operatorname{Lie}\PGL(N)$ with $\algsl(N)^*$ via the
  Killing form, we get $\mathcal M_C(\PGL(V),\bN) \cong
  T^*\SL(N)$.
  This is the symplectic variety associated with a cylinder as expected.

  Let us check how the action in \ref{sec:group_action} is affected by
  the replacement $\GL(V)\to\PGL(V)$. The semisimple Lie algebra
  $\fl^{\mathrm{bal}}_{\mathrm{ss}}$ remains the same: the variety
  $\cR_{\PGL(V),\bN}$ is obtained from $\cR_{\GL(V),\bN}$ by
  identifying isomorphic connected components. Therefore the
  construction of \ref{sec:quantization} applies. On the other hand
  $\pi_1(\GL(V))\cong \bigoplus_{i\in Q_0} \ZZ \alpha_i$ is replaced
  by
  $\pi_1(\PGL(V))\cong \pi_1(\GL(V))/ \ZZ(\sum_{i\in Q_0} \dim
  V_i\, \alpha_i)$. The root datum is
  $R^{\mathrm{bal}}\subset \pi_1(\PGL(V))$,
  $R^{\mathrm{bal} \vee}\subset\pi_1(\PGL(V))^\vee$.
  \begin{NB}
    We identify $\pi_1(\PGL(V))^\vee$ with
    $\{ \lambda \in \Hom(\bigoplus_{i\in Q_0} \ZZ\alpha_i, \ZZ) \mid
    \langle\lambda,\sum_{i\in Q_0} \dim V_i\, \alpha_i\rangle = 0\}$.
    We have
    $R^{\mathrm{bal}\vee} = \{ \alpha_1^\vee, \dots,
    \alpha_{N-1}^\vee, \alpha_{N+1}^\vee,\dots,
    \alpha_{2N-1}^\vee\}$. Since the middle one is omitted, we still
    have $R^{\mathrm{bal}\vee} \subset \pi_1(\PGL(V))^\vee$.
  \end{NB}%
  Thus $\PGL(N,N)$ is replaced by its subgroup
  $\PGL(N,N)' \defeq \{[g_1,g_2]\mid \det g_1 = \det g_2\}$.
  \begin{NB}
    Really ? At least the action on $T^*\SL(N)$ by
    $(g_1,g_2)\cdot (g,\beta) = (g_1 g g_2^{-1},
    \operatorname{Ad}(g_1)\beta)$ is well-defined.
  \end{NB}%
  We have $\SL(N)\times\SL(N)\twoheadrightarrow \PGL(N,N)'$
  \begin{NB}
    as $[g_1,g_2] = [\det g_1^{-1/N} g_1, \det g_1^{-1/N} g_2]
    = [\det g_1^{-1/N}g_1, \det g_2^{-1/N}g_2]$
  \end{NB}%
  with kernel $\ZZ/N\ZZ$, the diagonal central subgroup. The standard
  action on $T^*\SL(N)$ coincides with the one given in
  \ref{sec:group_action}.

  On the other hand, if we replace the central $\GL(N)$ by $\SL(N)$,
  the corresponding Coulomb branch is the hamiltonian reduction of
  $T^*\GL(N)$ with respect to the $\CC^\times$-action corresponding to
  $\pi_1(\GL(N))\cong\ZZ$. (See \ref{prop:reduction}.) In this case
  $\CC^\times$-action is the scalar multiplication on $T^*\GL(N)$
  (\ref{cartan_grading}), hence the reduction is $T^*\PGL(N)$ as
  expected.

\subsection{Disk}

The variety for the disk is calculated as for the cylinder. We
consider a quiver gauge theory of type $A$ with
$\dim V = (1,2,\dots,N)$. As we remarked in the proof of \ref{Areg},
the Coulomb branch is $\oZ^\alpha_{\PGL(N+1)}\simeq\GL(N)\times\CC^N$, where
$\CC^N$ is identified with the Kostant slice for $\GL(N)$, and
$\alpha=(N+1)\omega_N$.
\begin{NB}
    The action of $\GL(N)$ should be identified with the standard
    one. We should deduce it from \ref{ex:stab}, and Misha will write
    down the explanation.

    Again I was careless. We have $\SL(N)$ action from
    \ref{sec:group_action}. However we have
    $\pi_1(\GL(V))^\wedge = (\CC^\times)^N$, which is strictly larger
    than the maximal torus of $\SL(N)$.
\end{NB}%
By~\ref{Cartan_grading}, we have an isomorphism
$\oZ^\alpha_{\PGL(N+1)}\iso S_\alpha\cap\oW^\alpha_0$ and the natural action
of $\on{Stab}_{\PGL(N+1)}(\alpha)=\GL(N)$ on $\oZ^\alpha_{\PGL(N+1)}$ coinciding with
the natural action of $\GL(N)$ on $\GL(N)\times\CC^N$ (trivial on $\CC^N$ and by
left shifts on $\GL(N)$). By~\ref{ex:stab} this action coincides with the one
given in~\ref{sec:group_action}.

The modification to cases $\SL(N)$, $\PGL(N)$ are similar to the above.

\subsection{$S^2$ with three punctures for $\SL(2)$}
\label{D4}

We next consider the Higgs branch of the Sicilian theory of type
$\SL(2)$ associated with $S^2$ with three punctures. The mirror quiver
gauge theory is of type $D_4$.

We consider the $D_4$ quiver with the central vertex $1$ and other
vertices $2,3,4$. We orient the edges from the central vertex.
We take $V_1=\BC^2,\ V_2=V_3=V_4=\BC$. The diagonal
central subgroup $Z=\BC^\times\subset\GL(V)$ acts trivially on
$\bN=\bigoplus_{i=2}^4\Hom(V_1,V_i)$, so the action of
$\GL(V)$ factors through $\PGL(V):=\GL(V)/Z$. We will prove
$\CM_C(\PGL(V),\bN)\simeq\BA^8$.

According to~\ref{pestun}, $\cM_C(\GL(V),\bN)\simeq\oZ^\alpha$, the moduli
space of degree $\alpha$ based maps from $\BP^1$ to the flag variety $\CB$
of the simply connected group $G=\operatorname{Spin}(8)$ of type $D_4$.
Here $\alpha=2\alpha_1+\alpha_2+\alpha_3+\alpha_4$ is the highest coroot.
Note that $\alpha=\omega_1$ is a fundamental coweight.
We also consider the transversal slice
$s^\alpha_0\colon \oW^\alpha_{G,0}\to Z^\alpha$ (see~\ref{nondom}; note that
$-w_0=\on{Id}$ for our $G$). It is the moduli space of the data
$(\scP_{\on{triv}}\stackrel{\sigma}{\longrightarrow}\scP)$ where $\scP_{\on{triv}}$
is the trivialized $G$-bundle on $\BP^1$, and $\sigma$ is an isomorphism
on $\BP^1\setminus\{0\}$ with a trivial $G$-bundle $\scP$ possessing a
degree $\alpha$ pole at $0\in\BP^1$. We consider an open subset
$U\subset\oW^\alpha_{G,0}$ formed by the data
$(\scP_{\on{triv}}\stackrel{\sigma}{\longrightarrow}\scP)$ such that
the transformation of the (unique) degree 0 complete flag in $\scP$ with
value $B_-$ at $\infty\in\BP^1$ viewed as a generalized $B$-structure in
$\scP_{\on{triv}}$ acquires no defect at $0\in\BP^1$. We have
$s^\alpha_0\colon U\iso\oZ^\alpha$. We also have another open subset
$U'\subset\oW^\alpha_{G,0}$ formed by the data
$(\scP_{\on{triv}}\stackrel{\sigma}{\longrightarrow}\scP)$ such that
the transformation of the (unique) degree 0 complete flag in $\scP_{\on{triv}}$
with value $B_-$ at $\infty\in\BP^1$ viewed as a generalized $B$-structure in
$\scP$ acquires no defect at $0\in\BP^1$. This open subset $U'$ is nothing
but the intersection of $\oW^\alpha_{G,0}$ with the semiinfinite orbit
$T_{-\alpha}\subset\Gr_G$. Since the trivialization of $\scP$ at $\infty\in\BP^1$
(arising from $\sigma$) uniquely extends to the trivialization of $\scP$
over the whole of $\BP^1$, we obtain an involution
$\iota\colon\oW^\alpha_{G,0}\iso\oW^\alpha_{G,0}$ reversing the roles of
$\scP$ and $\scP_{\on{triv}}$ and replacing $\sigma$ by $\sigma^{-1}$.
We have $\iota\colon U\iso U'$. Thus we obtain an isomorphism
$s^\alpha_0\circ\iota\colon U'\iso\oZ^\alpha$.

Since $\alpha$ is the highest coroot,
$\oW^\alpha_{G,0}$ is isomorphic to the minimal nilpotent orbit closure
$\CN_{\min}=0\sqcup\BO_{\on{min}}\subset\fg$,
see~\cite[4.5.12, page 182]{Beilinson-Drinfeld}
or~\cite[Lemma~2.10]{mov} for a published account.
The projectivization of $\CN_{\min}$ is
the partial flag variety $\CB_\alpha=G/P_{234}$: the quotient with respect
to a submaximal parabolic subgroup. Thus we have a $\BC^\times$-bundle
$p\colon\BO_{\on{min}}\to\CB_\alpha$.
The big Bruhat cell (the open $B_-$-orbit) $C\subset\CB_\alpha$ is the
free orbit of the unipotent radical $U^-_{234}$ of $P^-_{234}$. Via the
exponential map, $U^-_{234}\simeq\fu^-_{234}$, the nilpotent radical of the
Lie algebra of $P^-_{234}$. For the Levi subgroup $L_{234}\subset P^-_{234}$
we have $[L_{234},L_{234}]\simeq\SL(2)_2\times\SL(2)_3\times\SL(2)_4$, and
$\fu^-_{234}$ as a $[L_{234},L_{234}]$-module is isomorphic to
$\BC^2_2\otimes\BC^2_3\otimes\BC^2_4\oplus
\BC^1_2\otimes\BC^1_3\otimes\BC^1_4$.
Note that the center $Z(\SL(2)_2\times\SL(2)_3\times\SL(2)_4)=(\BZ/2\BZ)^3$
has a natural projection onto $\BZ/2\BZ$ (the sum of coordinates). The kernel
$K$ of this projection, as a subgroup of $\on{Spin}(8)$, coincides with the
center $Z(\on{Spin}(8))$. The action of $\SL(2)_2\times\SL(2)_3\times\SL(2)_4$
on $\fu^-_{234}$ factors through the action of
$\ol{L}_{234}:=(\SL(2)_2\times\SL(2)_3\times\SL(2)_4)/K$.

Finally, we have $U'=p^{-1}(C)$. Thus we obtain a projection
(a $\BC^\times$-bundle)
$p\circ(s^\alpha_0\circ\iota)^{-1}\colon\oZ^\alpha\to C$. This action of
$\BC^\times$ is nothing but the composition of the natural $T$-action
(Cartan torus $T=B\cap B_-$) with the cocharacter
$\alpha_1\colon \BC^\times\to T$.
The boundary equation $F_\alpha\colon \oZ^\alpha\to\BC^\times$ has weight 1
with respect to
$\BC^\times\stackrel{\alpha_1}{\longrightarrow}T$~\cite[Proposition~4.4]{bf14}.
It follows that
$(p\circ(s^\alpha_0\circ\iota)^{-1},F_\alpha)\colon\oZ^\alpha\iso C\times\BC^\times$.
The action of $\ol{L}_{234}$ on $\oZ^\alpha=C\times\BC^\times$ is via the above
action on $C$. Note that $\ol{L}_{234}$ is nothing but
$[L^{\on{bal}},L^{\on{bal}}]\subset\on{PSO}(8)$ of~\ref{sec:group_action}, and the
action of $\ol{L}_{234}$ on $\CM_C(\GL(V),\bN)$ coincides with the action
constructed in~\ref{sec:group_action} by~\ref{ex:stab}.

Now the surjection 
\begin{equation*}\BC[\oZ^\alpha]=\BC[\CM_C(\GL(V),\bN)]=
H_*^{\GL(V)_\CO}(\CR_{\GL(V),\bN})
\twoheadrightarrow H_*^{\GL(V)_\CO}(\CR_{\PGL(V),\bN})
\end{equation*} 
is nothing but the
restriction to the level set $F_\alpha^{-1}(1)$ (see~\ref{Ober}), hence
\linebreak[4]$\Spec H_*^{\GL(V)_\CO}(\CR_{\PGL(V),\bN})\simeq C$.
Furthermore, the embedding 
\begin{equation*}\BC[\CM_C(\PGL(V),\bN)]=
H_*^{\PGL(V)_\CO}(\CR_{\PGL(V),\bN})\hookrightarrow H_*^{\GL(V)_\CO}(\CR_{\PGL(V),\bN})=
\BC[F_\alpha^{-1}(1)]
\end{equation*} 
is nothing but the embedding of the ring of functions
invariant with respect to the translations action of $\BG_a$ on $\oZ^\alpha$.
Here we view $\BG_a$ as a subgroup of automorphisms of $\BP^1$ preserving
$\infty\in\BP^1$; its action on $\oZ^\alpha$ preserves the boundary equation
$F^\alpha$ and its level set $F_\alpha^{-1}(1)$. In terms of the
identification $F_\alpha^{-1}(1)\simeq C\simeq\fu^-_{234}\simeq
\BC^2_2\otimes\BC^2_3\otimes\BC^2_4\oplus\BC$, the action of $\BG_a$ is
nothing but the action of the last summand $\BC$, and hence $\CM_C(\PGL(V),\bN)=
F_\alpha^{-1}(1)/\BG_a\simeq\BC^2_2\otimes\BC^2_3\otimes\BC^2_4$.



The above action of $\ol{L}_{234}=[L^{\on{bal}},L^{\on{bal}}]$ on $\CM_C(\GL(V),\bN)$
induces its action on $\CM_C(\PGL(V),\bN)$.
One can also see that $\ol{L}_{234}$ is the reductive group
corresponding to the root datum
$R^{\mathrm{bal}}\subset\pi_1(\PGL(V))$,
$R^{\mathrm{bal}\vee}\subset\pi_1(\PGL(V))^\vee$ via
$\pi_1(\GL(V))\twoheadrightarrow\pi_1(\PGL(V))$ as in
\ref{subsec:cylinder}.

\begin{Remark}
    Let us give another argument, which the third named author was taught by
    Amihay Hanany.

    Let us consider functions $E_1^{(1)}$, $F_1^{(1)}$ for the middle
    vertex $1$ by \ref{rem:deghalf}.
    \begin{NB}
        Recall coweights are $(1,0)$, $(0,-1)$ for the vertex $1$, and
        $0$ otherwise.
    \end{NB}%
    Since $\langle \alpha,\alpha_1\rangle = 1$, we have the action of
    $\BG_a^2$ by integrating hamiltonian vector fields
    $H_{E_1^{(1)}}$, $H_{F_1^{(1)}}$.
    We combine it with the action of
    $\SL(2)_2\times\SL(2)_3\times\SL(2)_4$. Let us consider the Lie
    subalgebra of $\CC[\cM_C(\PGL(V),\bN)]$ generated by $E_i^{(1)}$,
    $F_i^{(1)}$ ($i=1,2,3,4$). Viewing $(\PGL(V),\bN)$ as a framed
    quiver gauge theory of type $A_3$ with $\dim V =
    \begin{smallmatrix}
        1 & 2 & 1
    \end{smallmatrix}$,
    $\dim W =
    \begin{smallmatrix}
        0 & 1 & 0
    \end{smallmatrix}$,
    we see that $\mu = \dim W - C\dim V$ satisfies the condition
    $\langle \mu,\alpha\rangle \ge -1$ for any positive root
    $\alpha$. Hence elements in the Lie subalgebra have either degree
    $1$, $1/2$, or $0$, and the degree $0$ part consists of constant
    functions by \ref{rem:Delta-grading}. Since the Poisson bracket is
    of degree $-1$, $\{f, g\}$ is a constant if $f$, $g$ are of degree
    $1/2$.

    Commutator relations in \secref{sec:quantization} imply that
    $E_1^{(1)}$ (resp.\ $F_1^{(1)}$) is a lowest (resp.\ highest)
    weight vector in the tensor product
    $\CC^2_2\otimes\CC^2_3\otimes\CC^2_4$ of vector representations.
    \begin{NB}
        $\{ F_2^{(1)}, E_1^{(1)}\} = 0$,
        $\{ H_2^{(1)}, E_1^{(1)}\} = -\frac12 (H_2^{(0)} E_1^{(1)} +
        E_1^{(1)} H_2^{(0)}) = - E_1^{(1)}$ and
        $\{ E_2^{(1)}, F_1^{(1)}\} = 0$,
        $\{H_2^{(1)}, F_1^{(1)}\} = \frac12 (H_2^{(0)} F_1^{(1)} +
        F_1^{(1)} H_2^{(0)}) = F_1^{(1)}$.
    \end{NB}%
    \begin{NB}
        Note also that $F_1^{(1)}$ is a highest weight, as $
\begin{smallmatrix}
       & 1 & \\
    1 & (1,0) & 1
\end{smallmatrix}
= \begin{smallmatrix}
       & 0 & \\
    0 & (0,-1) & 0
\end{smallmatrix}
$ as coweight of $\PGL(V)$.
\end{NB}%
Hence we have a factorization
$\CM_C(\PGL(V),\bN)\cong \BA^8\times \cM_C'$ by \ref{rem:deghalf}.
\begin{NB}
    Note that
    $\{ \varphi^*(E_1^{(1)}), \varphi^*(F_1^{(1)})\} = \varphi^* 1 =
    1$
    for $\varphi\in \SL(2)_2\times\SL(2)_3\times\SL(2)_4$. Therefore
    we have a non-degenerate pairing on the degree $1/2$ space.
\end{NB}%
But $\cM'_C$ must be a point as $\cM_C(\PGL(V),\bN)$ is
$8$-dimensional.  (Both $E_1^{(1)}$ and $F_1^{(1)}$ live in the
\emph{same} representation, as $\cM_C(\PGL(V),\bN)$ would have a
factor $\BA^{16}$ otherwise.)

The same argument shows that the Coulomb branch $\cM_C(\PGL(V),\bN)$ for
$\dim V =
\begin{smallmatrix}
    &&&& 1 &&&&
    \\
    1 & 2 & \dots & N-1 & N & N-1 & \dots & 2 & 1
\end{smallmatrix}
$ is $\Hom(\CC^N_l, \CC_r^N)\oplus \Hom(\CC^N_r,\CC^N_l)$, where we
have an $\SL(N)_l\times\SL(N)_r$-action from the balanced vertices in
the left and right legs, and $\CC^N_l$, $\CC^N_r$ are its vector
representations. See \cite[(4.6)]{MR2985331}.
\end{Remark}

\subsection{$S^2$ with three punctures for $\SL(3)$}
\label{subsec:E6}

We next consider type $\SL(3)$. The mirror quiver gauge theory is of
type affine $E_6$. We start with a simple observation. Let us denote
by $0$ (resp.\ $6$) a special vertex (resp.\ a vertex adjacent to $0$)
of the affine quiver of type $E_6$. This choice breaks $\mathfrak S_3$
symmetry of the quiver. We have an isomorphism
$\prod_{i\neq 0}\GL(V_i) \cong \PGL(V) = \GL(V)/Z$.  Therefore we can
view $(\PGL(V),\bN)$ as a framed quiver gauge theory of finite type
$E_6$ with $\dim V =
\begin{smallmatrix}
    &   & 2 & & \\
    1 & 2 & 3 & 2 & 1
\end{smallmatrix}
$, $\dim W =
\begin{smallmatrix}
    && 1 && \\
    0 & 0 & 0 & 0 & 0
\end{smallmatrix}$.
Therefore its Coulomb branch is $\oW^{\varpi_6}_{G,0}$ by
\ref{Coulomb_quivar}, where $G$ is the group $E_6$ of adjoint type.
By~\cite[4.5.12, page 182]{Beilinson-Drinfeld}
(see~\cite[Lemma~2.10]{mov} for a published account),
this is isomorphic to the closure $\CN_{\min}$ of the minimal nilpotent orbit
$\BO_{\min}$. We have the action of $G$ by \ref{prop:Integrable},
which is identified with the standard one by \ref{ex:stab}.

The action of $\SL(3)^2$ corresponding to two legs not containing the
chosen special vertex $0$ is coming from the standard inclusion
$\SL(3)^2\subset E_6$. The remaining $\SL(3)$ action for the leg
containing $0$ is given as follows.

First let us note that the Lie algebra $\fl$ of degree $1$ elements in
$\CC[\cM_C]$ is $\mathfrak{e}_6$, as we already know
$\cM_C = \CN_{\min}$.
\begin{NB}
  We have $\cM_C\subset \mathfrak{e}_6^*$ by the moment map. There are
  clearly degree $1$ functions on $\mathfrak{e}_6^*$ are just linear
  ones.

  Note also, all vertices are balanced, hence
  $\fl^{\mathrm{bal}} = \mathfrak{e}_6 = \fl$.
\end{NB}%

Returning back to the original gauge group $\PGL(V) = \GL(V)/Z$, we
have degree $1$ elements $E_0^{(1)}$, $F_0^{(1)}$, $H_0^{(1)}\in\fl$
corresponding to the special vertex $0$ by \ref{lem:balanc-vert-quiv}.
The variety of triples $\cR_{\PGL(V),\bN}$ is obtained from
$\cR_{\GL(V),\bN}$ by identifying isomorphic connected
components. Therefore the construction of \ref{sec:quantization}
applies.
The computation of Poisson brackets $\{ H_i^{(1)}, E_0^{(1)}\}$,
$\{ H_i^{(1)}, F_0^{(1)}\}$ remains unchanged by the replacement
$\GL(V)\to\PGL(V)$, hence we conclude that $E_0^{(1)}$, $F_0^{(1)}$
are root vectors corresponding to the highest weight of
$\fl=\mathfrak{e}_6$. It also follows that $E_0^{(1)}$, $F_0^{(1)}$
together with $E_6^{(1)}$, $F_6^{(1)}$ generate an additional
$\algsl(3)$, and $\SL(3)$.

We have the $\mathfrak S_3$-action on $\cM_C$ induced by permutation
of three legs. From the above consideration, it is clear that it
corresponds to $\mathfrak S_3$ of automorphisms of $\mathfrak e_6$
exchanging root subspaces corresponging to the highest weight and two
remaining special vertices. (See \cite[Th.~8.6]{Kac} for the detail of
the construction of automorphisms.)

\subsection{Torus with one puncture for $\SL(3)$}

We consider the Higgs branch of the Sicilian theory of type $\SL(3)$
associated with a torus with one puncture.
According to \ref{conj:genGK} the mirror quiver gauge theory is
$1\to 2\to 3\righttoleftarrow$, where numbers are dimensions (and we
use them also for indices of vertices). Note that we have an edge loop
at the vertex $3$. Let us denote the Coulomb branch of this quiver
gauge theory by $\cM$. The following result is informed to the
third-named author by Amihay Hanany. (It is based on an earlier
observation in \cite[\S2.1]{2012JHEP...05..145G},
\cite[(3.3.2)]{Cremonesi:2014vla}.)

\begin{Proposition}
\label{prop:g2}
$\cM$ is isomorphic to the subregular orbit closure
$\ol{\mathbb O}_{\on{subreg}}\subset\fg_2$.
\end{Proposition}

\begin{proof}
Let us first construct the action of $G_2$.

We consider operators $E_i^{(1)}$, $F_i^{(1)}$, $H_i^{(1)}$
($i=1,2,3$) as in \ref{lem:balanc-vert-quiv}. The vertices $1$, $2$
are balanced, while $3$ is \emph{not}. But we still have
$\deg E_3^{(1)}$, $F_3^{(1)} = 1$ by \eqref{eq:Deltadeg}. Let us
consider the Lie subalgebra $\mathfrak{g}$ of $\CC[\cM]$ generated by
these operators.
By \ref{subsec:balanced} $E_i^{(1)}$, $F_i^{(1)}$, $H_i^{(1)}$
($i=1,2$) define the Lie algebra $\algsl(3)$, and the corresponding
hamiltonian vector fields are integrated to an $\SL(3)$-action on
$\cM$.

The proof of commutation relations
$\{ E_i^{(1)}, F_3^{(1)}\} = 0 = \{ F_i^{(1)}, E_3^{(1)}\}$,
$\{ H_1^{(1)}, E_3^{(1)}\} = 0 = \{ H_1^{(1)}, F_3^{(1)}\}$,
$\{ H_2^{(1)}, E_3^{(1)}\} = -E_3^{(1)}$,
$\{ H_2^{(1)}, F_3^{(1)}\} = F_3^{(1)}$ in \ref{sec:quantization} remains
to work even though $3$ has an edge loop. Similarly
to~\cite[Lemma~6.8]{2016arXiv160800875K} we calculate
$\{E_3^{(1)},F_3^{(1)}\}=2(3w_{2,1}+3w_{2,2}-2w_{3,1}-2w_{3,2}-2w_{3,3})$.
Since $H_1^{(1)}=2w_{1,1}-w_{2,1}-w_{2,2}$ and
$H_2^{(1)}=2w_{2,1}+2w_{2,2}-w_{1,1}-w_{3,1}-w_{3,2}-w_{3,3}$, we conclude that
$\{E_3^{(1)},F_3^{(1)}\}=4H_2^{(1)}+2H_1^{(1)}$. Hence $\fg$ is simple of type
$G_2$ with Cartan subalgebra $\fh$ spanned by $H_1^{(1)}$ and $H_2^{(1)}$, and
$\algsl(3)$ is spanned by $\fh$ and the long roots. Note that $E_3^{(1)}$ and
$F_3^{(1)}$ generate the fundamental representations $V$ and $V^\vee$ of
$\algsl(3)$, and $\fg=\algsl(3)\oplus V\oplus V^\vee$.

More concretely,
$F^{(1)}_3=u_{3,1}+u_{3,2}+u_{3,3}$, and
$E^{(1)}_3=(w_{3,1}-w_{2,1})(w_{3,1}-w_{2,2})u_{3,1}^{-1}+(w_{3,2}-w_{2,1})(w_{3,2}-
w_{2,2})u_{3,2}^{-1}+(w_{3,3}-w_{2,1})(w_{3,3}-w_{2,2})u_{3,3}^{-1}$, and
$E^{(1)}_2=(w_{2,1}-w_{1,1})(w_{2,2}-w_{2,1})^{-1}u_{2,1}^{-1}+(w_{2,2}-w_{1,1})
(w_{2,1}-w_{2,2})^{-1}u_{2,2}^{-1}$, and
$F^{(1)}_2=(w_{2,1}-w_{3,1})(w_{2,1}-w_{3,2})(w_{2,1}-w_{3,3})(w_{2,2}-w_{2,1})^{-1}u_{2,1}+
(w_{2,2}-w_{3,1})(w_{2,2}-w_{3,2})(w_{2,2}-w_{3,3})(w_{2,1}-w_{2,2})^{-1}u_{2,2}$, while
$F^{(1)}_1=(w_{1,1}-w_{2,1})(w_{1,1}-w_{2,2})u_{1,1}$, and $E^{(1)}_1=u_{1,1}^{-1}$.

Now consider an auxiliary quiver gauge theory of type $D_4$ with 1-dimensional
framing at the middle vertex numbered by 2, and the outer vertices numbered by
$(3,1),(3,2),(3,3)$. We take $\dim V''_{3,1}=\dim V''_{3,2}=\dim V''_{3,3}=1,\
\dim V''_2=2$ (and $\dim W''_2=1$).
We know that $\cM_C(\GL(V''),\bN'')=\CN_{\on{min}}$
is the closure of the minimal orbit in $\mathfrak{so}_8$. On the other hand,
we consider a quiver gauge theory of the affine type $D_4$ with the extra
vertex numbered by 1, $\dim V_1=1$, and all the other dimensions
as before, but no framing. We denote the corresponding graded vector space by
$V'=V''\oplus V_1$, and the corresponding representation of $\GL(V')$ by $\bN'$.
Then $\cM_C(\GL(V''),\bN'')=\cM_C(\PGL(V'),\bN')$.
We have an embedding $\BC[\cM_C(\PGL(V'),\bN')]\hookrightarrow
\BC(w_{1,1},w_{2,1},w_{2,2},w_{3,1},w_{3,2},w_{3,3},u_{1,1},u_{2,1},u_{2,2},u_{3,1},
u_{3,2},u_{3,3})^{S_2}$ where the symmetric group $S_2$ acts by permuting
$(w_{2,1},u_{2,1})$ and $(w_{2,2},u_{2,2})$. Also we have an embedding
$\BC[\cM]\hookrightarrow
\BC(w_{1,1},w_{2,1},w_{2,2},w_{3,1},w_{3,2},w_{3,3},u_{1,1},u_{2,1},u_{2,2},u_{3,1},
u_{3,2},u_{3,3})^{S_2\times S_3}$ where the symmetric group $S_3$ acts by permuting
$(w_{3,1},u_{3,1}),\ (w_{3,2},u_{3,2})$, and $(w_{3,3},u_{3,3})$.
By inspection of~\eqref{eq:82},~\eqref{eq:83},~\ref{th:PhiBar} we check
$F^{(1)}_2={}'\!F^{(1)}_2,\ E^{(1)}_2={}'\!E^{(1)}_2,\
F^{(1)}_3={}'\!F^{(1)}_{3,1}+{}'\!F^{(1)}_{3,2}+{}'\!F^{(1)}_{3,3},\
E^{(1)}_3={}'\!E^{(1)}_{3,1}+{}'\!E^{(1)}_{3,2}+{}'\!E^{(1)}_{3,3}$ where $'\!E,{}'\!F$
refer to the generators of $\mathfrak{so}_8$ in $\BC[\cM_C(\PGL(V'),\bN')]$,
while $E,F$ refer to the generators of $\fg_2$ in the previous paragraph.
Since the projection $\cM_C(\PGL(V'),\bN')=\CN_{\on{min}}\to\mathfrak{so}_8$
is an embedding, we conclude that the projection
$\cM\to\fg_2$ is generically an embedding. Hence the differential of the
$G_2$-action on $\cM$ is generically surjective, so $\cM$ has an open
$G_2$-orbit $\widetilde{\mathbb O}\subset\cM$ which is a nonramified cover of
its image adjoint orbit ${\mathbb O}\subset\fg_2$.

Now the monopole formula for $\cM$ gives degrees in $\BN$. Indeed,
the contribution of a dominant coweight $\lambda=(\lambda_{1,1},
\lambda_{2,1}\geq\lambda_{2,2},\lambda_{3,1}\geq\lambda_{3,2}\geq\lambda_{3,3})$ of
$(\GL(1)\times\GL(2)\times\GL(3))/Z$ equals
$\Delta(\lambda)=\lambda_{2,2}-\lambda_{2,1}+
\frac{1}{2}|\lambda_{2,1}-\lambda_{1,1}|+\frac{1}{2}|\lambda_{2,2}-\lambda_{1,1}|+
\frac{1}{2}\sum_{r,s}|\lambda_{2,r}-\lambda_{3,s}|$ which is easily seen to be
nonnegative and integral. We conclude that $\cM$ is conical, and hence its
image $\overline{\mathbb O}\subset\fg_2$ is conical as well. It follows that
${\mathbb O}$ is a nilpotent orbit. But $\fg_2$ has a unique 10-dimensional
nilpotent orbit: the subregular one. Hence
${\mathbb O}={\mathbb O}_{\on{subreg}}$.

Now we have to identify the cover $\widetilde{\mathbb O}\to{\mathbb O}$.
It is known that the universal cover of ${\mathbb O}_{\on{subreg}}$ is an open
piece of the minimal nilpotent orbit
${\mathbb O}_{\on{min}}\subset\mathfrak{so}_8$, see e.g.~\cite{brko}, and the
Galois group of this cover is $S_3$. Moreover, the degree 1 functions on the
universal cover constitute the Lie algebra $\mathfrak{so}_8$. It follows that
if the cover $\widetilde{\mathbb O}\to{\mathbb O}$ corresponds to a subgroup
$\pi_1\subset S_3$, then the degree 1 functions on $\widetilde{\mathbb O}$
constitute $\mathfrak{so}_8^{\pi_1}$. Since we know that the degree 1
functions on $\cM$ constitute $\fg_2=\mathfrak{so}_8^{S_3}$, we conclude that
$\pi_1=S_3$, so that $\widetilde{\mathbb O}\iso{\mathbb O}_{\on{subreg}}$.
Finally, the normality property of the orbit closure
$\ol{\mathbb O}_{\on{subreg}}$ guarantees that $\cM\iso\ol{\mathbb O}_{\on{subreg}}$.
\end{proof}

 Note that a torus with one puncture is obtained from $S^2$ with
three punctures by gluing two punctures. We have computed the Higgs
branch associated with the latter in \ref{subsec:E6}. The Higgs branch
is the closure $\CN_{\min}(\mathfrak{e}_6)$ of the minimal nilpotent orbit of
$\mathfrak{e}_6$.
Therefore the Higgs branch $\cM$ for a torus with one puncture is the
Hamiltonian reduction $\CN_{\min}(\mathfrak{e}_6)\tslash \Delta_{\SL(3)}$ with
respect to the diagonal $\SL(3)$ in $\SL(3)\times\SL(3)$ corresponding
to two legs which are glued. Therefore we have an action of the
centralizer of $\Delta_{\SL(3)}$ in $E_6$, which is $G_2$. (See e.g.,
\cite[\S3.2]{MR2342006} and the references therein.)
Combining with \ref{prop:g2}, we should have
$\CN_{\min}(\mathfrak{e}_6)\tslash \Delta_{\SL(3)}\cong \ol{\mathbb
  O}_{\on{subreg}}(\mathfrak{g}_2)$, the closure of the subregular
nilpotent orbit of $\mathfrak{g}_2$. We do not have a proof of this
statement, though it might be known to an expert.
\begin{NB}
  e.g., Gordan Savin.
\end{NB}%

\subsection{Recollections on derived Satake equivalence}
\label{derived Satake}
We consider a reductive group $G$ with Langlands dual group
$G^\vee$ and its Lie algebra $\gvee$. We have a commutative ring object
$\Areg=\bigoplus_\lambda\on{IC}(\ol\Gr{}^\lambda)\otimes(V^\lambda)^*\in
D_G(\Gr_G)$.

Let $e,h,f\in\gvee$ be a principal $\algsl_2$-triple such that $f$ is lower
triangular, and $e$ is upper triangular. We consider the Kostant slice
$e+\fz(f)$ to the regular nilpotent orbit. Let $\Sigma$ be the image of
$e+\fz(f)$ under a $G^\vee$-invariant isomorphism $\gvee\simeq(\gvee)^*$.
let $\Upsilon$ be the image of $e+\fb^{\!\scriptscriptstyle\vee}_-$ (Borel subalgebra)
under a $G^\vee$-invariant isomorphism $\gvee\simeq(\gvee)^*$.
We have canonical isomorphisms $\Sigma=\ft/\Weyl=\Upsilon/U^\vee_-$ (unipotent
subgroup) by the compositions $\Sigma\hookrightarrow(\gvee)^*\twoheadrightarrow
(\gvee)^*\dslash G^\vee\cong\ft/\Weyl$ and
$\Sigma\hookrightarrow\Upsilon\twoheadrightarrow\Upsilon/U^\vee_-$.

According to~\cite[Theorem~5]{MR2422266}, there is an equivalence of
monoidal triangulated categories $\Psi\colon
D^{G^\vee}(\on{Sym}^{[]}(\gvee))\to D_{G}(\Gr_{G})$. Recall that
$D_G(\Gr_G)$ stands for the Ind-completion of the bounded derived equivariant
constructible category on $\Gr_G$. Accordingly,
$D^{G^\vee}(\on{Sym}^{[]}(\gvee))$ stands for the Ind-completion of the
triangulated category $D^{G^\vee}_{\on{perf}}(\on{Sym}^{[]}(\gvee))$ formed by the
$G^\vee$-equivariant perfect dg-modules over
$\on{Sym}^{[]}(\gvee)$: the graded symmetric algebra of $\gvee$ where any
element of $\gvee$ is assigned degree 2 (with trivial differential).
The monoidal structure on $D_{G}(\Gr_{G})$ is given by the convolution $\star$,
and the monoidal structure on $D^{G^\vee}(\on{Sym}^{[]}(\gvee))$ is
$M_1,M_2\mapsto M_1\otimes_{\on{Sym}^{[]}(\gvee)}M_2$. The algebra
$\BC[\Sigma]=\on{Sym}(\gvee)^{G^\vee}$ acts on
$\Ext^*_{D^{G^\vee}(\on{Sym}^{[]}(\gvee))}(M_1,M_2)$, and this action is compatible with
the action of $\BC[\Sigma]=H^*_G(\on{pt})$ on
$\Ext^*_{D_G(\Gr_G)}(\Psi(M_1),\Psi(M_2))$. 
Since $D^{G^\vee}(\on{Sym}^{[]}(\gvee))$ is the homotopy category of a
dg-category, we have
$\on{RHom}_{D^{G^\vee}(\on{Sym}^{[]}(\gvee))}\colon
D^{G^\vee}(\on{Sym}^{[]}(\gvee))\times
D^{G^\vee}(\on{Sym}^{[]}(\gvee))\to D(\on{Vect})$.
\begin{NB}\linelabel{line:dg}
  HN is not sure how this should be explained. But it will be used
  later, hence should be explained at some point. See also
  l.\lineref{line:dg2} for his complain.
\end{NB}%

The functor $\Psi^{-1}$ is uniquely characterized by the property that
$\Psi^{-1}(\on{IC}(\ol\Gr{}^\lambda))=\on{Sym}^{[]}(\gvee)\otimes V^\lambda$.
For $\CF\in D_G(\Gr_G)$ we have~\cite[Theorem~2]{MR2422266}:
\begin{equation}
\label{eq:coho}
H^*_{G_\CO}(\Gr_G,\CF)=\kappa^l(\Psi^{-1}(\CF)):=
H^*(\Psi^{-1}(\CF)\otimes_{\on{Sym}^{[]}_{\on{new}}(\gvee)}\BC[\Sigma]^{[]})
\end{equation}
in the following sense. The eigenvalues of $-h$ from the above
$\algsl_2$-triple define a grading of $\gvee$, hence a grading of $(\gvee)^*$
and a grading of $\on{Sym}^{[]}(\gvee)$. Thus $\gvee,\ (\gvee)^*$ and
$\on{Sym}^{[]}(\gvee)$ acquire a bigrading such that the total degree of
$e$ is zero. We consider $\on{Sym}^{[]}(\gvee)$ with the new grading given by
the total degree and denote it $\on{Sym}^{[]}_{\on{new}}(\gvee)$. Now the
projection $\on{Sym}(\gvee)\to\BC[\Sigma]$ is compatible with the new grading,
and induces a grading on $\BC[\Sigma]$ denoted $\BC[\Sigma]^{[]}$. Finally,
we consider both $\on{Sym}^{[]}_{\on{new}}(\gvee)$ and $\BC[\Sigma]^{[]}$ as
dg-algebras with trivial differential (and zero components of odd degrees).
Note that $\Psi^{-1}(\CF)$ is still a dg-module over
$\on{Sym}^{[]}_{\on{new}}(\gvee)$ due to its $G^\vee$-equivariance.
\begin{NB}
  The $H^*_G(\on{pt})=\BC[\Sigma]^{[]}$-module structure in the LHS
  of~\eqref{eq:coho}
  is compatible with the $\BC[\Sigma]^{[]}$-module structure in the RHS.
  The RHS of~\eqref{eq:coho} is the cohomology of a dg-module over
  $\BC[\Sigma]^{[]}$ (note that the LHS is a graded module over
  $H_G(\on{pt})=\BC[\Sigma]^{[]}$).  According
  to~\cite[Section~2.3]{MR2422266}, this cohomology also has a
  structure of a graded module over the ring of functions on the
  tangent bundle of $\Sigma$. The ring of functions on the tangent
  bundle of $\Sigma$ is identified with $H^*_{G_\cO}(\mathrm{Gr}_G)$.
\end{NB}%

\linelabel{line:duality}
Let ${\mathbf D}$ stand for the duality
$M\mapsto\on{RHom}_{\on{Sym}^{[]}(\gvee)}(M,\on{Sym}^{[]}(\gvee))$ in
$D^{G^\vee}_{}(\on{Sym}^{[]}(\gvee))$. Let ${\mathbb D}$ stand for the Verdier
duality in $D_G(\Gr_G)$. We denote by ${\mathfrak C}_{G^\vee}$
the autoequivalence of $D^{G^\vee}_{}(\on{Sym}^{[]}(\gvee))$ induced by
the canonical outer automorphism of $G^\vee$ interchanging conjugacy classes of
$g$ and $g^{-1}$ (the Chevalley involution). We denote by ${\mathcal C}_G$
the autoequivalence of $D_G(\Gr_G)$ induced by
$g\mapsto g^{-1},\ G((z))\to G((z))$. Then
${\mathcal C}_G\circ\Psi=\Psi\circ{\mathfrak C}_{G^\vee}$.
According to~\cite[Lemma~14]{MR2422266},
we have $\Psi\circ{\mathfrak C}_{G^\vee}\circ{\mathbf D}={\mathbb D}\circ\Psi$
and $\Psi\circ{\mathbf D}={\mathcal C}_G\circ{\mathbb D}\circ\Psi$.
\begin{NB}
Clearly, ${\mathbf D}\BC[T^*G]^{[]}=\BC[T^*G]^{[]}$, while
${\mathbb D}\Areg={\mathcal C}_G\Areg$.

We define $\Phi:={\mathfrak C}_{G^\vee}\circ\Psi^{-1}\colon D_G(\Gr_G)\to
D^{G^\vee}_{}(\on{Sym}^{[]}(\gvee))$. We have $\Phi(\Areg)=\BC[T^*G]^{[]}$.
\end{NB}%

The following lemma is well known. (See
\cite[\S2.4]{1995alg.geom.11007G}. Also the proof of
\cite[Lemma~14]{MR2422266} depends on it.) Let us give its proof
for completeness. Recall ${\mathbf 1}_{\Gr_G}$ denotes the skyscraper
sheaf at the base point in $\Gr_G$.

\begin{NB}
  Does this also follows from the corresponding statement in
  $D^{G^\vee}_{}(\on{Sym}^{[]}(\gvee))$ ? Namely $\mathbf D$ is the
  rigidity in $D^{G^\vee}_{}(\on{Sym}^{[]}(\gvee))$, i.e.\
  \[
    \on{RHom}_{D^{G^\vee}_{}(\on{Sym}^{[]}(\gvee))}(\on{Sym}^{[]}(\gvee),
    M_1\otimes_{\on{Sym}^{[]}(\gvee)} M_2) \cong
    \on{RHom}_{D^{G^\vee}_{}(\on{Sym}^{[]}(\gvee))}(\mathbf DM_1,M_2).
  \]
  MF: yes it does. However, the proof of 
$\Psi\circ{\mathfrak C}_{G^\vee}\circ{\mathbf D}={\mathbb D}\circ\Psi$
in~\cite[Lemma~14]{MR2422266} uses the rigidity in $D_G(\Gr_G)$.
\end{NB}%

\begin{Lemma}
\label{lem:rigidity}
${\mathcal C}_G\circ\BD$ is the rigidity for $(D_G(\Gr_G),\star)$. That is,
for any $\CF_1,\CF_2\in D_G(\Gr_G)$ we have a canonical isomorphism
$\on{RHom}_{D_G(\Gr_G)}({\mathbf 1}_{\Gr_G},\CF_1\star\CF_2)\iso
\on{RHom}_{D_G(\Gr_G)}(\linebreak[3]{\mathcal C}_G\circ\BD\CF_1,\CF_2)$.
\end{Lemma}

\begin{proof}
For any group ${\mathsf H}$, the convolution operation 
$\CF_1\star\CF_2={\mathsf m}_*(\CF_1\boxtimes\CF_2)$ on
$D({\mathsf H})$ has rigidity $\CF\mapsto{\mathcal C}_{\mathsf H}\circ\BD\CF$ 
where ${\mathcal C}_{\mathsf H}$
is induced by the automorphism $h\mapsto h^{-1},\ {\mathsf H}\to {\mathsf H}$. 
Namely,
$\on{RHom}({\mathbf 1}_{\mathsf H},\CF_1\star\CF_2)=i_e^!(\CF_1\star\CF_2)=
\on{RHom}(\CC_{\mathsf H},\nabla^!(\CF_1\boxtimes\CF_2))=
\on{RHom}(\CC_{\mathsf H},{\mathcal C}_{\mathsf H}\CF_1\otimes^!\CF_2)=
\on{RHom}({\mathcal C}_{\mathsf H}\circ\BD\CF_1,\CF_2)$, where
$\nabla\colon {\mathsf H}\hookrightarrow{\mathsf H}\times{\mathsf H}$
is the antidiagonal embedding $h\mapsto (h^{-1},h)$.

We apply this to the category of $G_\CO$-left-right equivariant sheaves
on ${\mathsf H}=G_{\mathcal K}$.

More formally, let us use the six operations for constructible derived 
categories on Artin stacks. There is a reference~\cite{lasols} for 
$\ol{\mathbb Q}_l$-coefficients. We choose an isomorphism 
$\ol{\mathbb Q}_l\cong\BC$ and use it for complex coefficients. Our stack
is ${\mathcal X}:=G_\CO\backslash\Gr_G$. It is the moduli stack of pairs 
$\scP_1,\scP_2$ of $G$-bundles on the formal disc $D$ equipped with an 
isomorphism $\eta\colon \scP_1|_{D^*}\iso\scP_2|_{D^*}$. There is an involution
$\fri\colon \cX\to\cX$ induced by the inversion $g\mapsto g^{-1}$ of 
$G_{\mathcal K}$. In modular terms, 
$\fri(\scP_1,\scP_2,\eta)=(\scP_2,\scP_1,\eta^{-1})$. 
Recall that $\Gr_G$ is the moduli space of $G$-bundles $\scP$ on $D$ equipped
with an isomorphism $\sigma\colon \scP_{\on{triv}}|_{D^*}\iso\scP|_{D^*}$.
We have a projection $\on{pr}_2\colon \Gr_G\to\cX$ sending $(\scP,\sigma)$ to
$(\scP_{\on{triv}},\scP,\sigma)$. Similarly, we define $_G\!\on{rG}$ as the moduli
space of $G$-bundles $\scP$ on $D$ equipped with an isomorphism
$\tau\colon \scP|_{D^*}\iso\scP_{\on{triv}}|_{D^*}$. We have a projection
$\on{pr}_1\colon _G\!\on{rG}\to\cX$ sending $(\scP,\tau)$ to
$(\scP,\scP_{\on{triv}},\tau)$. We have isomorphisms 
$\fri\colon \Gr_G\iso{}_G\!\on{rG}$ and $\fri\colon _G\!\on{rG}\iso\Gr_G$
sending $\sigma$ to $\tau=\sigma^{-1}$ and $\tau$ to $\sigma=\tau^{-1}$.
Obviously, $\fri\on{pr}_1=\on{pr}_2\fri$ and $\fri\on{pr}_2=\on{pr}_1\fri$.
There is a morphism
$\sfm\colon _G\!\on{rG}\times\Gr_G\to\cX$ induced by the multiplication in 
$G_{\mathcal K}$. In modular terms, 
$\sfm(\scP_1,\tau;\scP_2,\sigma)=(\scP_1,\scP_2,\sigma\circ\tau)$.
The convolution on $D(\cX)$ is defined as 
$\CF_1\star\CF_2:=\sfm_*(\on{pr}_1^*\CF_1\boxtimes\on{pr}_2^*\CF_2)$.
The unit object ${\mathbf 1}$ is $\BC_{G_\CO\backslash\Gr^0_G}$. We claim that the
rigidity is $\fri^*\circ\BD$. Indeed, 
$\on{RHom}_\cX({\mathbf 1},\CF_1\star\CF_2)=i_0^!(\CF_1\star\CF_2)=
\on{RHom}_{G_\CO\backslash\Gr_G}(\BC_{\Gr_G},\nabla^!(\on{pr_1}^*\CF_1\boxtimes\on{pr}_2^*\CF_2))=
\on{RHom}_{G_\CO\backslash\Gr_G}(\BC_{\Gr_G},\fri^*\on{pr}_1^*\CF_1\otimes^!\on{pr}_2^*\CF_2)=
\on{RHom}_{G_\CO\backslash\Gr_G}(\fri^*\circ\BD\on{pr}_1^*\CF_1,\on{pr}_2^*\CF_2)=
\on{RHom}_\cX(\fri^*\circ\BD\CF_1,\CF_2)$, where 
$\nabla\colon \Gr_G\to{}_G\!\on{rG}\times\Gr_G$ is $(\fri,\on{id})$.
\end{proof}

\subsection{Regular sheaf and derived Satake equivalence}
Under the equivalence $\Psi^{-1}$, the ring object $\Areg\in D_{G}(\Gr_{G})$
corresponds to the $G^\vee$-equivariant free
$\on{Sym}^{[]}(\gvee)$-module $\BC[G^\vee]\otimes\on{Sym}^{[]}(\gvee)$
which will be denoted $\BC[T^*G^\vee]^{[]}$ for short. The $G^\vee$-action comes
from the {\em left} action of $G^\vee$ on
$T^*G^\vee=G^\vee\times(\gvee)^*,\ g_1(g_2,\xi)=(g_1g_2,\xi)$. And the action
of $\on{Sym}^{[]}(\gvee)$ on $\BC[T^*G^\vee]^{[]}$ comes from the morphism
$\mu_l\colon T^*G^\vee\to(\gvee)^*,\ (g,\xi)\mapsto\on{Ad}_g\xi$
(the moment map of the left action).
Recall that $\Areg$ is equipped with an action
of $G^\vee$. Under the equivalence $\Psi^{-1}$, this action goes to the
action on $\BC[T^*G^\vee]^{[]}$ coming from the {\em right} action of $G^\vee$ on
$T^*G^\vee=G^\vee\times(\gvee)^*,\ g_1\cdot(g_2,\xi)=
(g_2g_1^{-1},\on{Ad}_{g_1}\xi)$. For this reason the action of $G^\vee$ on $\Areg$
will be called the {\em right} action. \linelabel{lne:1488} 
The moment map of the right $G^\vee$-action on $T^*G^\vee$ is
$\mu_r\colon T^*G^\vee\to(\gvee)^*,\ (g,\xi)\mapsto\xi$.

Also note that $\on{RHom}_{D_G(\Gr_G)}(\Areg,\Areg)$ is a formal dg-algebra (since e.g.\
$\BC[T^*G^\vee]^{[]}$ is a free $\on{Sym}^{[]}(\gvee)$-module), so $\mu_r$
gives rise to a $G^\vee$-equivariant morphism of dg-algebras
\begin{equation}
\label{eq:right-g-action}
  \on{Sym}^{[]}(\gvee)\to\on{RHom}_{D_G(\Gr_G)}(\Areg,\Areg).
\end{equation}
  Altogether we have the action of $G^\vee\ltimes\on{Sym}^{[]}(\gvee)$ on $\Areg$
that will be called the {\em right} action.

\begin{Remark}
\label{rem:ext-rhom}
For $\CF_1,\CF_2\in D_G(\Gr_G)$ we distinguish
$\Ext^*_{D_G(\Gr_G)}(\CF_1,\CF_2)$ and
$\on{RHom}_{D_G(\Gr_G)}\linebreak[2](\CF_1,\CF_2)$. They are
isomorphic in $D(\on{Vect})$, the derived category of vector spaces,
which is equivalent to $\on{Vect}^{\on{gr}}$, the category of graded
vector spaces. But when we consider additional structures, such as a
dg-algebra structure or a structure of a dg-module over
$G^\vee\ltimes\on{Sym}^{[]}(\gvee)$, they are not isomorphic. We thus
understand
$\Ext^*_{D_G(\Gr_G)}(\CF_1,\CF_2)=H^*(\on{RHom}_{D_G(\Gr_G)}(\CF_1,\CF_2))$.
\begin{NB}
  Let $\cC$ be an abelian category. Then
  $\on{RHom}_\cC\colon D^-(\cC)^\circ\times D^+(\cC)\to
  D^+(\on{Mod}\ZZ)$ is defined as a derived functor of
  $\Hom_\cC(\bullet,\bullet)$.  See \cite[1.10.11]{KaSha}.
\end{NB}%
\end{Remark}

\begin{Definition}
\label{def:dg-action}
The morphism~\eqref{eq:right-g-action}
$\on{Sym}^{[]}(\gvee)\to\on{RHom}_{D_G(\Gr_G)}(\Areg,\Areg)$ induces,
for any $\CF\in D_G(\Gr_G)$, the composed morphism 
\begin{equation*}\on{Sym}^{[]}(\gvee)\to
\on{RHom}_{D_G(\Gr_G)}(\Areg,\Areg)\to\on{RHom}_{D_G(\Gr_G)}(\Areg\otimes^!\CF,
\Areg\otimes^!\CF)
\end{equation*} 
of dg-algebras. Also, the morphism~\eqref{eq:right-g-action} induces,
for any $\CF\in D_G(\Gr_G)$, the composed morphism
\begin{multline*}\on{Sym}^{[]}(\gvee)\otimes\on{RHom}_{D_G(\Gr_G)}(\BC_{\Gr_G},\Areg\otimes^!\CF)\\
\to\on{RHom}_{D_G(\Gr_G)}(\Areg,\Areg)\otimes
\on{RHom}_{D_G(\Gr_G)}(\BC_{\Gr_G},\Areg\otimes^!\CF)
\to\on{RHom}_{D_G(\Gr_G)}(\BC_{\Gr_G},\Areg\otimes^!\CF)
\end{multline*} 
of complexes of vector 
spaces. This morphism is $G^\vee$-equivariant for the $G^\vee$-action on 
$\on{RHom}_{D_G(\Gr_G)}(\BC_{\Gr_G},\Areg\otimes^!\CF)$ induced by the right 
$G^\vee$-action on $\Areg$. Thus, the complex 
$\on{RHom}_{D_G(\Gr_G)}(\BC_{\Gr_G},\Areg\otimes^!\CF)$ acquires
the structure of an object of $D^{G^\vee}(\on{Sym}^{[]}(\gvee))$, and 
$\on{Ext}^*_{D_G(\Gr_G)}(\BC_{\Gr_G},\Areg\otimes^!\bullet)$ gets upgraded to the  
functor 
\begin{equation*}\on{RHom}_{D_G(\Gr_G)}(\BC_{\Gr_G},\Areg\otimes^!\bullet)\colon 
D_G(\Gr_G)\to D^{G^\vee}(\on{Sym}^{[]}(\gvee)).
\end{equation*}

Similarly, 
$\on{RHom}_{D_G(\Gr_G)}(\BD\Areg,\CF), \on{RHom}_{D_G(\Gr_G)}({\mathbf 1}_{\Gr_G},\BD\Areg\star\CF), \on{RHom}_{D_G(\Gr_G)}(\Areg,\CF)$, etc.\ all acquire the structures of objects
of $D^{G^\vee}(\on{Sym}^{[]}(\gvee))$, and $\on{Ext}^*_{D_G(\Gr_G)}(\Areg,\bullet)$ 
gets upgraded to the functor 
$\on{RHom}_{D_G(\Gr_G)}(\Areg,\bullet)\colon D_G(\Gr_G)\to D^{G^\vee}(\on{Sym}^{[]}(\gvee))$.
\end{Definition}

\begin{Lemma}
\label{rem:formality}
From~\ref{def:dg-action} we obtain an action of 
$(G^\vee)^b\ltimes\on{Sym}^{[]}(\gvee)^{\otimes b}$ on \linebreak[4]
$\on{RHom}_{D_G(\Gr_G)}(\BC_{\Gr_G},\Areg^{\otimes^!b})$. The resulting dg-module
over $(G^\vee)^b\ltimes\on{Sym}^{[]}(\gvee)^{\otimes b}$ is formal.
\end{Lemma}

\begin{proof}
Using~\cite[Proposition~5]{MR2422266} we reformulate the claim for
$D_G(\Gr_G)$ replaced by $D_{G_{\ol{\mathbb F}_q}}(\Gr_{G,\ol{\mathbb F}_q})$. 
Then $\Areg$ is pointwise pure (meaning that all its costalks are pure with
respect to the Frobenius action). Hence $\Areg^{\otimes^!b}$ is also pointwise
pure. Then the Cousin spectral sequence for the Schubert stratification of
$\Gr_G$ shows that $H^*_{G_{\CO,\ol{\mathbb F}_q}}(\Gr_{G,\ol{\mathbb F}_q},\Areg^{\otimes^!b})$
is pure. Also, $\Ext^*_{D_{G_{\ol{\mathbb F}_q}}(\Gr_{G,\ol{\mathbb F}_q})}(\Areg,\Areg)$ is 
pure. Now the argument 
of~\cite[Section~6.5]{MR2422266} proves that the dg-algebra
$\on{RHom}_{D_G(\Gr_G)}(\Areg,\Areg)$ is formal as well as its $b$-th tensor 
power, and the dg-module $\on{RHom}_{D_G(\Gr_G)}(\BC_{\Gr_G},\Areg^{\otimes^!b})$ over
$\on{RHom}_{D_G(\Gr_G)}(\Areg,\Areg)^{\otimes b}$ is formal.  
\end{proof}

The Kostant-Whittaker (hamiltonian) reduction of $T^*G^\vee$ with respect
to the right action is
$T^*G^\vee\stslash_{U^\vee_-,\psi}:=\mu_r^{-1}(\Upsilon)/U^\vee_-$.
(We use $\stslash$ for a hamiltonian reduction in order to avoid a conflict with $\dslash$ for a GIT quotient.)
At the level of dg-modules, $\kappa^r(\BC[T^*G^\vee]^{[]}):=
(\BC[\mu_r^{-1}(\Upsilon)]^{[]})^{U^\vee_-}=\BC[\mu_r^{-1}(\Sigma)]^{[]}:=
\BC[T^*G^\vee]^{[]}\otimes_{\on{Sym}^{[]}_{\on{new}}(\gvee)}\BC[\Sigma]^{[]}$,
tensor product with respect to the action of the right copy of
$\on{Sym}^{[]}_{\on{new}}(\gvee)$.\footnote{Passing to cohomology, we obtain the
usual hamiltonian reduction: $H^*(\kappa^r(\BC[T^*G^\vee]^{[]}))=
\BC[T^*G^\vee]\otimes_{\on{Sym}(\gvee)}\BC[\Sigma]=:\kappa^r(\BC[T^*G^\vee])$. We use
the same notation $\kappa^r$ for the hamiltonian reduction of the usual
modules and of dg-modules. It is clear from the context which one is used in
what follows.}
\begin{NB}
This is not a free $\on{Sym}^{[]}(\gvee)$-module anymore. Accordingly, the
dualizing sheaf $\DC_{\Gr_{G}}$ is not a direct sum of IC sheaves.
\end{NB}
(We have an isomorphism $U^\vee_-\times \Sigma\iso\Upsilon$ given
by the action of $U^\vee_-$ on $\Upsilon$. Hence
$\CC[\Upsilon]^{U^\vee_-} \cong \CC[\Sigma]$. Moreover, for any
$U^\vee_-$-equivariant sheaf $\CF$ on $\Upsilon$, we have
$\CF^{U^\vee_-}=\CF|_\Sigma$. Similarly, for a $U^\vee_-$-equivariant dg-module $M$
over $\BC[\Upsilon]^{[]}$ we have
$M^{U^\vee_-}=M\otimes_{\BC[\Upsilon]^{[]}}\BC[\Sigma]^{[]}$.)
\begin{NB}
  HN asked why $\kappa^r(\CC[T^*G^\vee]^{[]})$ is an object of
  $D^{G^\vee}_{\mathrm{perf}}(\mathrm{Sym}^{[]}(\mathfrak g^\vee))$,
  before it is explained that we actually need to take the
  ind-completion.

  I guess I do not understand the question: $\BC[T^*G^\vee]^{[]}\in
  D^{G^\vee\times G^\vee}_{}(\on{Sym}^{[]}(\gvee\oplus\gvee))$. We perform
  the right reduction. The result is formal (the differential is zero), so
  it is a graded $G^\vee\rtimes\on{Sym}^{[]}(\gvee)$-module (with respect to
  the residual left structure). However, it is not free anymore. So we
  replace it by a free resolution in $D^{G^\vee}_{}(\on{Sym}^{[]}(\gvee))$.

  We had some discussion on Oct.20, 21, 2017.
\end{NB}%
This is an object of $D^{G^\vee}_{}(\on{Sym}^{[]}(\gvee))$ (with
respect to the residual {\em left} action of $G^\vee$) corresponding
under the equivalence $\Psi^{-1}$ to the dualizing complex
$\DC_{\Gr_{G}}$~\cite[Proposition~4]{MR2422266}.  (In
fact,~\cite[Proposition~4]{MR2422266} is proved for the extra
equivariance under the loop rotations.) Instead of
$\DC_{\Gr_{G}}=\Psi(\kappa^r(\Psi^{-1}\Areg))$, we will write
$\DC_{\Gr_{G}}=\kappa^r(\Areg)$ for short.

\begin{NB}
\begin{Lemma}
\label{0.2}
The functor $\on{Ext}(\Areg,\CF):=
\bigoplus_\lambda\on{Ext}_G(\on{IC}(\ol\Gr{}^\lambda)\otimes(V^\lambda)^*,\CF),\
D_G(\Gr_G)\to D(\on{Vect})$ is equal to the composition of $\Psi^{-1}$ and the
forgetful functor
$\on{Forg}\colon D^{G^\vee}_{}(\on{Sym}^{[]}(\gvee))\to D(\on{Vect})$.
In particular, the algebra $\on{Sym}^{[]}(\gvee)$ acts on $\on{Ext}(\Areg,\CF)$.
\end{Lemma}

\begin{proof}
For a $G\ltimes\on{Sym}^{[]}(\gvee)$-module $M$ we have
\begin{equation*}
\on{Forg}M=\bigoplus_\lambda\on{Hom}_{G\ltimes\on{Sym}^{[]}(\gvee)}
(\on{Sym}^{[]}(\gvee)\otimes V^\lambda\otimes(V^\lambda)^*,M),
\end{equation*}
where the third factor $(V^\lambda)^*$ is
understood as the underlying vector space (disregarding its $G$-module
structure). The lemma follows.
\end{proof}

Following~\cite[Lemma~14]{MR2422266}, we will denote by ${\mathfrak C}_{G^\vee}$
the autoequivalence of $D^{G^\vee}_{}(\on{Sym}^{[]}(\gvee))$ induced by
the canonical outer automorphism of $G^\vee$ interchanging conjugacy classes of
$g$ and $g^{-1}$ (the Chevalley involution). Then we have
\begin{equation*}
\on{Forg}\circ{\mathfrak C}_{G^\vee}\circ\Psi^{-1}(\CF)=
H^*_{G_\CO}(\Gr_G,\Areg\otimes^!\CF).
\end{equation*}
Indeed, $H^*_{G_\CO}(\Gr_G,\Areg\otimes^!\CF)=
\begin{NB2}
H^*(\Gr_G,\underline{\on{RHom}}({\mathbb D}\Areg,\CF))=
\end{NB2}%
\on{RHom}_{D_G(\Gr_G)}({\mathbb D}\Areg,\CF)=\on{RHom}_{D_G(\Gr_G)}({\mathbb D}\CF,\Areg)$ where
${\mathbb D}$ stands for the Verdier duality in $D_G(\Gr_G)$. On the other hand,
we know that 
\begin{equation*}
M=\on{RHom}_{D^{G^\vee}(\on{Sym}^{[]}(\gvee))}(\BC[T^*G^\vee],M)=
\on{RHom}_{D^{G^\vee}(\on{Sym}^{[]}(\gvee))}({\mathbf D}M,\BC[T^*G^\vee])
\end{equation*} 
where
${\mathbf D}$ stands for the duality
$M\mapsto\on{RHom}_{\on{Sym}^{[]}(\gvee)}(M,\on{Sym}^{[]}(\gvee))$ in
$D^{G^\vee}_{}(\on{Sym}^{[]}(\gvee))$. But ${\mathbb D}$ and ${\mathbf D}$
differ by ${\mathfrak C}_{G^\vee}$ according to~\cite[Lemma~14]{MR2422266}.

In particular, it follows that the RHS acquires a natural action of
$\on{Sym}^{[]}(\gvee)$ for any $\CF\in D_G(\Gr_G)$.
\end{NB}%

Under the dualities $\mathbf D$, $\DD$ we have
${\mathbf D}\BC[T^*G^\vee]^{[]}=\BC[T^*G^\vee]^{[]}$, while
${\mathbb D}\Areg={\mathcal C}_G\Areg$.

We define $\Phi:={\mathfrak C}_{G^\vee}\circ\Psi^{-1}\colon D_G(\Gr_G)\to
D^{G^\vee}_{}(\on{Sym}^{[]}(\gvee))$. We have $\Phi(\Areg)=\BC[T^*G^\vee]^{[]}$.

\begin{Lemma}
  \label{0.1}
  \textup{(a)} Let us define a functor 
$\on{Ext}^*_{D_G(\Gr_G)}(\Areg,\bullet)\colon D_G(\Gr_G)\to D(\on{Vect})=
\on{Vect}^{\on{gr}}$ by
  \begin{equation*}\on{Ext}^*_{D_G(\Gr_G)}(\Areg,\CF):=
  \bigoplus_\lambda\on{Ext}^*_{D_G(\Gr_G)}(\on{IC}(\ol\Gr{}^\lambda)\otimes(V^\lambda)^*,\CF).
\end{equation*}
It is canonically isomorphic to $\on{Ext}^*_{D_G(\Gr_G)}(\mathbf 1_{\Gr_G},
      \Areg\star\bullet)$ by the rigidity together with
      $\cC_G\DD\Areg = \Areg$ (see \ref{lem:rigidity}).
  Then both $\on{Ext}^*_{D_G(\Gr_G)}(\Areg,\bullet)$ and
  $\on{Ext}^*_{D_G(\Gr_G)}(\mathbf 1_{\Gr_G},
  \Areg\star\bullet)$ 
  are canonically isomorphic to the composition $\on{Forg}\circ\Psi^{-1}$,
  where $\on{Forg}$ is the forgetful functor 
  $D^{G^\vee}_{}(\on{Sym}^{[]}(\gvee))\to D(\on{Vect})=\on{Vect}^{\on{gr}}$.
  
Their upgrades
$\on{RHom}_{D_G(\Gr_G)}(\Areg,\bullet)$ and
$\on{RHom}_{D_G(\Gr_G)}(\mathbf 1_{\Gr_G},
\Areg\star\bullet)\colon 
  D_G(\Gr_G)\to D^{G^\vee}(\on{Sym}^{[]}(\gvee))$
(see~\ref{def:dg-action}) are canonically isomorphic to $\Psi^{-1}$.



\textup{(b)} For $\CF_1,\CF_2\in D_G(\Gr_G)$, there are canonical isomorphisms
$H^*_{G_\cO}(\Gr_G,\cF_1\otimes^! \cF_2) = \on{Ext}^*_{D_G(\Gr_G)}(\BC_{\Gr_G},\cF_1\otimes^!\cF_2)
\cong\on{Ext}^*_{D_G(\Gr_G)}(\BD\CF_1,\CF_2)\cong
    \on{Ext}^*_{D_G(\Gr_G)}(\mathbf 1_{\Gr_G},{\mathcal C}_G\cF_1\star\cF_2)$
\begin{NB}
$=i_0^!({\mathcal C}_G\cF_1\star\cF_2)$\end{NB}%
in $D(\on{Vect})=\on{Vect}^{\on{gr}}$.

\textup{(c)} Let us define a functor 
$\on{Ext}^*_{D_G(\Gr_G)}(\BD\Areg,\bullet)\colon D_G(\Gr_G)\to D(\on{Vect})=
\on{Vect}^{\on{gr}}$ by
  \begin{equation*}\on{Ext}^*_{D_G(\Gr_G)}(\BD\Areg,\CF):=
  \bigoplus_\lambda\on{Ext}^*_{D_G(\Gr_G)}(\on{IC}(\ol\Gr{}^\lambda)
\otimes V^\lambda,\CF).
\end{equation*}
It is canonically isomorphic to $\on{Ext}^*_{D_G(\Gr_G)}(\mathbf 1_{\Gr_G},
      \cC_G\Areg\star\bullet)$ by the rigidity together with
      $\cC_G\Areg = \BD\Areg$ (see \ref{lem:rigidity}).
We have canonical isomorphisms
\begin{multline*}
\on{Forg}\circ\Phi(\bullet)\iso
\on{Ext}^*_{D_G(\Gr_G)}(\mathbf 1_{\Gr_G},\cC_G\Areg\star\bullet)
\underset{\mathrm{(b)}}{\overset{\sim}{\longrightarrow}}
\Ext^*_{D_G(\Gr_G)}(\BD\Areg,\bullet)\\
\underset{\mathrm{(b)}}{\overset{\sim}{\longrightarrow}}
\on{Ext}^*_{D_G(\Gr_G)}(\BC_{\Gr_G},\Areg\otimes^!\bullet)
\begin{NB}
  = H^*_{G_\CO}(\Gr_G,\Areg\otimes^!\bullet) 
\end{NB}%
\end{multline*}
of functors $D_G(\Gr_G)\to D(\on{Vect})=\on{Vect}^{\on{gr}}$.

The upgraded functors (see~\ref{def:dg-action})
\begin{multline*}
\Phi\iso\on{RHom}_{D_G(\Gr_G)}({\mathbf 1}_{\Gr_G},\cC_G\Areg\star\bullet)\\\iso
\on{RHom}_{D_G(\Gr_G)}(\BD\Areg,\bullet)\iso
\on{RHom}_{D_G(\Gr_G)}(\BC_{\Gr_G},\Areg\otimes^!\bullet)
\end{multline*}
are isomorphic as functors
from $D_G(\Gr_G)$ to $D^{G^\vee}(\on{Sym}^{[]}(\gvee))$.
\end{Lemma}

\begin{proof}
  \textup{(a)}
  \begin{NB}
For a $G\ltimes\on{Sym}^{[]}(\gvee)$-module $M$ we have
\begin{equation*}
\on{Forg}M=\bigoplus_\lambda\on{Hom}_{G\ltimes\on{Sym}^{[]}(\gvee)}(\on{Sym}^{[]}(\gvee)\otimes
V^\lambda\otimes(V^\lambda)^*,M),
\end{equation*} 
where the third factor $(V^\lambda)^*$ is
understood as the underlying vector space (disregarding its $G$-module
structure).

\end{NB}%
We consider $\CC[T^*G^\vee]^{[]}$ as a
$G^\vee\times G^\vee$-equivariant
$(\on{Sym}^{[]}(\gvee), \on{Sym}^{[]}(\gvee))$ bimodule by the left
and right action. We have the canonical matrix coefficient morphisms
\begin{equation}\label{eq:phiM}
  M \xrightarrow[\varphi_M]{\cong}
  \on{RHom}_{D^{G^\vee}(\on{Sym}^{[]}(\gvee))}(\on{Sym}^{[]}(\gvee),\CC[T^*G^\vee]^{[]}\otimes M),
\end{equation}
where $G^\vee\ltimes(\on{Sym}^{[]}(\gvee))$ acts on $\CC[T^*G^\vee]^{[]}$ by
the left action, and the right hand side is regarded as an object in
$D^{G^\vee}(\on{Sym}^{[]}(\gvee))$ by the residual right action on
$\CC[T^*G^\vee]^{[]}$. Here $\varphi_M$ is defined as follows: Given a
$G^\vee$-module $M$, we have ${\boldsymbol\varphi}_M\colon M\otimes M^*\to\CC[G^\vee]$ by
${\boldsymbol\varphi}_M(m\otimes m^*)(g):=\langle gm,m^*\rangle$. It is a morphism of
$G^\vee\times G^\vee$-modules. By swapping $M^*$ to the target, ${\boldsymbol\varphi}_M$ can be
viewed as a morphism $\varphi_M\colon M\to\Hom_{G^\vee}(\BC,\CC[G^\vee]\otimes M)$.
The morphism $\varphi_M$ is an isomorphism. One can think about it as the 
usual fiber functor
$\on{Forg}$ on $\on{Rep}(G^\vee)$ being represented by $\CC[G^\vee]$.
\begin{NB}
  Let $\pi(g)\in\End(M)$ denote the action of $g\in G^\vee$, i.e.\
  $g\cdot m = \pi(g) m$. Then $\pi(\bullet)$ is an $\End(M)$-valued
  function on $G^\vee$. Therefore $\pi(\bullet)m$ for $m\in M$ is an
  $M$-valued function on $G^\vee$. Since
  $g(\pi(g^{-1}\bullet)m) = \pi(\bullet)m$, it is contained in
  $\Hom_{G^\vee}(\CC,\CC[G^\vee]\otimes M)$.
\end{NB}%
Its inverse $\varphi_M^{-1}$ is given by the evaluation
$\CC[G^\vee]\to\CC$ at $g=1$.
We consider this over $\on{Sym}^{[]}(\gvee)$ to get
\eqref{eq:phiM}. We now apply the derived Satake equivalence. We get
\begin{multline*}
  \Psi^{-1}(\cF)
  \underset{\varphi_{\Psi^{-1}(\cF)}}{\overset{\sim}{\longrightarrow}}
  \on{RHom}_{D^{G^\vee}(\on{Sym}^{[]}(\gvee))}(\on{Sym}^{[]}(\gvee),
  \BC[T^*G^\vee]^{[]}\otimes\Psi^{-1}(\CF))\\
  \iso \on{RHom}_{D_G(\Gr_G)}(\mathbf 1_{\Gr_G}, \Areg\otimes\CF).
\end{multline*}
The second isomorphism holds since the construction of $\Psi^{-1}$ actually
passes through dg-categories (as opposed to being defined at the level of
derived categories),
\begin{NB}\linelabel{line:dg2}
  HN understand this explanation as follows: A dg-category has $\Hom$
  as differential graded modules, i.e., $\ZZ$-graded modules with
  differentials. Since we are inverting quasi-isomorphisms,
  $D^{G^\vee}(\on{Sym}^{[]}(\gvee))$ is not literally a dg-category,
  but it is obtained from the dg-category of dg-modules of
  $\on{Sym}^{[]}(\gvee)$ by taking $H^0(\Hom)$. Then \cite{MR2422266}
  constructed two dg-algebras $A^\bullet$, $B^\bullet$ which are
  quasi-isomorphic and yield $D_G(\Gr_G)$ and
  $D^{G^\vee}(\on{Sym}^{[]}(\gvee))$ by $H^0(\Hom)$ respectively. Note
  that dg-categories $\cC_{\text{dg}}(A^\bullet)$,
  $\cC_{\text{dg}}(B^\bullet)$ of dg $A^\bullet$-modules, dg
  $B^\bullet$-modules are not equivalent as dg-categories (as
  $A^\bullet$, $B^\bullet$ are only quasi-isomorphic), but we regard
  $\on{RHom}$ as objects in $D(\on{Vect})$, hence the claim is true.

  Maybe this and also l.\lineref{line:dg} should be explained a little
  more carefully.
\end{NB}%
and is compatible with the action of the formal dg-algebra
$\on{Sym}^{[]}(\gvee)
\begin{NB}
\to\on{RHom}_{D_G(\Gr_G)}(\Areg,\Areg)\cong
\on{RHom}_{D^{G^\vee}(\on{Sym}^{[]}(\gvee))}(\BC[T^*G^\vee],\BC[T^*G^\vee])
\end{NB}%
$. We have an isomorphism
$\on{RHom}_{D_G(\Gr_G)}(\Areg,\bullet)\cong
\on{RHom}_{D_G(\Gr_G)}(\mathbf 1_{\Gr_G},\Areg\star\bullet)$ by the
rigidity plus $\Areg=\cC_G\circ\BD\Areg$.

\begin{NB}
  An older version:
  
We have the canonical matrix coefficient morphisms
\begin{multline}\label{eq:phiMNB}
  \on{RHom}_{D^{G^\vee}(\on{Sym}^{[]}(\gvee))}(\CC[T^*G]^{[]}, M)
  \xrightarrow[\phi_M]{\cong} M \\
  \xrightarrow[\varphi_M]{\cong}
  \on{RHom}_{D^{G^\vee}(\on{Sym}^{[]}(\gvee))}(\on{Sym}^{[]}(\gvee),\CC[T^*G]^{[]}\otimes M),
\end{multline}
where $G^\vee\ltimes(\on{Sym}^{[]}(\gvee))$ acts on $\CC[T^*G]^{[]}$ by
the left action, and both LHS and RHS are regarded as objects in
$D^{G^\vee}(\on{Sym}^{[]}(\gvee))$ by the residual right action on
$\CC[T^*G]^{[]}$. Here $\varphi_M,\phi_M$ are defined as follows: Given a
$G^\vee$-module $M$, we have ${\boldsymbol\varphi}_M\colon M\otimes M^*\to\CC[G^\vee]$ by
${\boldsymbol\varphi}_M(m\otimes m^*)(g):=\langle gm,m^*\rangle$. It is a morphism of
$G^\vee\times G^\vee$-modules. By swapping $M^*$ to the RHS, ${\boldsymbol\varphi}_M$ can be
viewed as a morphism $\varphi_M\colon M\to\Hom_{G^\vee}(\BC,\CC[G^\vee]\otimes M)$.
Alternatively, swapping $\CC[G^\vee]$ to the LHS and $M^*$ to the RHS, ${\boldsymbol\varphi}_M$ 
can be viewed as a morphism $\phi_M\colon\Hom_{G^\vee}(\CC[G^\vee],M)\to M$. The
morphisms $\varphi_M,\phi_M$ are isomorphisms. One can think about it as the 
usual fiber functor
$\on{Forg}$ on $\on{Rep}(G^\vee)$ being represented by $\CC[G^\vee]$.
\begin{NB2}
  Let $\pi(g)\in\End(M)$ denote the action of $g\in G^\vee$, i.e.\
  $g\cdot m = \pi(g) m$. Then $\pi(\bullet)$ is an $\End(M)$-valued
  function on $G^\vee$. If we have
  $\varphi\in\Hom_{G^\vee}(\CC[G^\vee],M)$, we consider
  $\varphi(\pi(\bullet))\in\End(M)\otimes M$. Then $\phi_M(\varphi)$
  is given by $\varphi(\pi(\bullet))$ contracting $\End(M)\otimes M$
  to $M$. On the other hand, $\varphi_M$ is given by
  $m\mapsto (1\to (g\to g\cdot m))$. In particular, $\varphi_M^{-1}$
  is given by the evaluation $\CC[G^\vee]\to\CC$ at $g=1$.

  The isomorphism of rigidity $\bD\BC[G^\vee]\cong\BC[G^\vee]$ is given by 
  ${\mathbf 1}_{G^\vee}=\BC\stackrel{\iota_{G^\vee}}{\longrightarrow}
  \BC[G^\vee]\otimes\BC[G^\vee]
  \stackrel{\tau_{G^\vee}}{\longrightarrow}\BC={\mathbf 1}_{G^\vee}$, 
 where $\iota_{G^\vee}(c)=c$, and $\tau_{G^\vee}(f_1\otimes f_2)=f_1(e)\cdot f_2(e)$.
  Thus the rigidity provides an isomorphism 
  $\Hom_{G^\vee}(\BC[G^\vee],M)\iso\Hom_{G^\vee}(\BC,\BC[G^\vee]\otimes M)$ equal to
  $\varphi_M\circ\phi_M$.
\end{NB2}%
We consider this over $\on{Sym}^{[]}(\gvee)$ to get
\eqref{eq:phiMNB}. We now apply the derived Satake equivalence.
\begin{NB2}
  Please explain why the upgraded functor is isomorphic. Indeed,
the first isomorphism below holds since the construction of $\Psi^{-1}$ actually
passes through dg-categories (as opposed to being defined at the level of
derived categories), and is compatible with the action of formal dg-algebra
$\on{Sym}^{[]}(\gvee)\to\on{RHom}_{D_G(\Gr_G)}(\Areg,\Areg)\cong
\on{RHom}_{D^{G^\vee}(\on{Sym}^{[]}(\gvee))}(\BC[T^*G^\vee],\BC[T^*G^\vee])$.
\end{NB2}%
We get
\begin{multline*}
 \on{RHom}_{D_G(\Gr_G)}(\mathbf 1_{\Gr_G},\Areg\star\CF)\iso 
\on{RHom}_{D_G(\Gr_G)}(\Areg,\CF)\iso
\on{RHom}_{D^{G^\vee}(\on{Sym}^{[]}(\gvee))}(\Psi^{-1}(\Areg),\Psi^{-1}(\CF))\\
 =\on{RHom}_{D^{G^\vee}(\on{Sym}^{[]}(\gvee))}(\BC[T^*G^\vee]^{[]},\Psi^{-1}(\CF))\iso
\Psi^{-1}(\CF).
\end{multline*}
The first isomorphism is the rigidity plus $\Areg=\cC_G\circ\BD\Areg$.
The second isomorphism holds since the construction of $\Psi^{-1}$ actually
passes through dg-categories (as opposed to being defined at the level of
derived categories), and is compatible with the action of the formal dg-algebra
$\on{Sym}^{[]}(\gvee)
\begin{NB2}
\to\on{RHom}_{D_G(\Gr_G)}(\Areg,\Areg)\cong
\on{RHom}_{D^{G^\vee}(\on{Sym}^{[]}(\gvee))}(\BC[T^*G^\vee],\BC[T^*G^\vee])
\end{NB2}%
$.
\end{NB}

\textup{(b)}
\begin{NB}
We have $H^*_{G_\CO}(\Gr_G,\cF_1\otimes^! \cF_2)=
H^*_{G_\CO}(\Gr_G,i_\Delta^!(\cF_1\boxtimes\cF_2))=
    \on{Ext}^*_{D_G(\Gr_G)}({\mathbb D}\cF_1,\cF_2)=\on{Ext}^*_{D_G(\Gr_G)}({\mathbf
      1}_{\Gr_G},{\mathcal C}_G\cF_1\star\cF_2)=
    \on{Ext}^*_{D_G(\Gr_G)}(\on{IC}(\Gr^0_G),{\mathcal C}_G\cF_1\star\cF_2)$
\begin{NB2}$=i_0^!({\mathcal C}_G\cF_1\star\cF_2)$
\end{NB2}%
. Here the third equality holds since ${\mathcal C}_G\circ{\mathbb D}$ is
the rigidity for $(D_G(\Gr_G),\star)$, see~\ref{lem:rigidity}.
\end{NB}%
The first isomorphism is the rigidity for the monoidal category
$(D_G(\Gr_G),\otimes^!)$, while the second is for
$(D_G(\Gr_G),\star)$. See \ref{lem:rigidity}.

\textup{(c)}
\begin{NB}
We have \begin{NB2}$i_0^!({\mathbb D}\Areg\star\CF)\iso$\end{NB2}%
\begin{multline*}\on{RHom}_{D_G(\Gr_G)}(\on{IC}(\Gr^0_G),{\mathbb D}\Areg\star\CF)\iso
\on{RHom}_{D_G(\Gr_G)}({\mathbf 1}_{\Gr_G},{\mathbb D}\Areg\star\CF)\\
\iso\on{RHom}_{D_G(\Gr_G)}({\mathcal C}_G\Areg,\CF)\iso
\on{RHom}_{D_G(\Gr_G)}(\Areg,{\mathcal C}_G\CF),
\end{multline*}
and all the quasiisomorphisms are compatible with the action of
$G^\vee\ltimes\on{Sym}^{[]}(\gvee)$.
Now the last expression is
$\Psi^{-1}\circ\mathcal C_G(\cF)$ by (a), which is nothing but $\Phi(\cF)$.
\end{NB}%
The first isomorphism is a consequence of (a) together with
${\mathbb D}\Areg={\mathcal C}_G\Areg$. The second and third are nothing but
(b). In order to see that the second and third isomorphisms are upgraded to
$\on{RHom}$, we observe that quasi-isomorphisms
\begin{equation*}
\on{RHom}_{D_G(\Gr_G)}({\mathbf 1}_{\Gr_G},\cC_G\Areg\star\bullet)\iso
\on{RHom}_{D_G(\Gr_G)}(\BD\Areg,\bullet)\iso
\on{RHom}_{D_G(\Gr_G)}(\BC_{\Gr_G},\Areg\otimes^!\bullet)
\end{equation*}
are compatible with the action of $G^\vee\ltimes\on{Sym}^{[]}(\gvee)$.
\end{proof}

Let us suppose $\cF\in D_G(\Gr_G)$ is a ring object, i.e.\ it is
equipped with a commutative multiplication homomorphism
$\mathsf m_\cF\colon \cF\star\cF \to \cF$.
Then $\Psi^{-1}(\cF)\in D^{G^\vee}(\on{Sym}^{[]}(\gvee))$ is also a
ring object, i.e.\ it is equipped with
$\Psi^{-1}(\mathsf m_\cF)\colon
\Psi^{-1}(\cF)\otimes_{\on{Sym}^{[]}(\gvee)}\Psi^{-1}(\cF)\to\Psi^{-1}(\cF)$.
The same is true for $\Phi$.
On the other hand,
$\on{RHom}_{D_G(\Gr_G)}(\mathbf 1_{\Gr_G},\Areg\star\cF) =
\on{RHom}_{D_G(\Gr_G)}(\Areg,\cF)$ in \ref{0.1}(b) is equipped with a
multiplication by $\mathsf m_\cF$ and $\sfm_{\Areg}$ (equivalently, a
coproduct $\cC_G\BD\sfm_{\Areg}\colon \Areg\to\Areg\star\Areg$).
%
Similarly,
$\on{RHom}_{D_G(\Gr_G)}(\mathbf 1_{\Gr_G}, \cC_G\Areg\star\cF) =
\on{RHom}_{D_G(\Gr_G)}(\BD\Areg, \cF)$ in \ref{0.1}(c) is equipped
with a multiplication by $\mathsf m_\cF$ and
$\cC_G\sfm_{\Areg}\colon \cC_G\Areg\star\cC_G\Areg\to\cC_G\Areg$
(equivalently, a coproduct
$\BD\sfm_{\Areg}\colon\BD\Areg\to\BD\Areg\star\BD\Areg$).
Finally, a multiplication on 
$\on{RHom}_{D_G(\Gr_G)}(\CC_{\Gr_G},\Areg\otimes^!\cF)$ is defined as 
in~\ref{prop:!-product}.
\begin{NB}
  Really ? What is $\Areg\star\Areg\to\Areg$ (or 
  $\mathcal C_G\Areg\to \mathcal C_G\Areg\star\mathcal C_G\Areg$,
  $\DD\Areg\star\DD\Areg\to\Areg$) ? Isn't it just $\cC_G\sfm$ ?

  \begin{NB2}
    Misha's message on Dec.\ 2, 2017:
  Dear Hiraku,

  I suggest to use the isomorphism $\DD\Areg=\cC_G\Areg$ and the
  covariant autoequivalence $\cC_G$ to transfer the multiplication
  ${\mathsf m}$ and comultiplication $\varDelta$ from $\Areg$ to
  $\DD\Areg=\cC_G\Areg$. Note that ${\mathsf m}$ is commutative, while
  $\varDelta$ is not cocommutative, so if we swap their duals for
  $\DD\Areg$, we obtain something different (more like $U(\fg^\vee)$).
  \end{NB2}
\end{NB}%

\begin{Proposition}
\label{prop:comproducts}
\textup{(a)} Multiplications on $\Psi^{-1}(\CF)$,
$\on{RHom}_{D_G(\Gr_G)}(\Areg,\cF)$ and
$\on{RHom}_{D_G(\Gr_G)}\linebreak[3]({\mathbf 1}_{\Gr_G},\Areg\star\cF)$ are equal
under the isomorphism in \ref{0.1}(a).

\textup{(b)} The same is true for $\Phi(\CF)$,
$\on{RHom}_{D_G(\Gr_G)}(\mathbf 1_{\Gr_G},\cC_G\Areg\star\cF)$,
$\on{RHom}_{D_G(\Gr_G)}(\BD\Areg, \cF)$, and
$\on{RHom}_{D_G(\Gr_G)}(\CC_{\Gr_G},\Areg\otimes^!\cF)$, under the
isomorphisms of~\ref{0.1}(c).
\end{Proposition}

\begin{proof}
\textup{(a)} 
\begin{NB}
Old version: 
For an algebra $A$ in the category
$D^{G^\vee}(\on{Sym}^{[]}(\gvee))$, the following diagram is commutative:
\begin{equation*}
\tiny{\begin{CD}
\on{RHom}_{D^{G^\vee}(\on{Sym}^{[]}(\gvee))}(\on{Sym}^{[]}(\gvee),\BC[T^*G^\vee]^{[]}\otimes A)\otimes
\on{RHom}_{D^{G^\vee}(\on{Sym}^{[]}(\gvee))}(\on{Sym}^{[]}(\gvee),\BC[T^*G^\vee]^{[]}\otimes A)
@<\sim<\varphi_A\otimes\varphi_A< A\otimes A \\
@VVV @| \\
\on{RHom}_{D^{G^\vee}(\on{Sym}^{[]}(\gvee))}(\on{Sym}^{[]}(\gvee),\BC[T^*G^\vee]^{[]}\otimes
\BC[T^*G^\vee]^{[]}\otimes A\otimes A) @. A\otimes A \\
@V{\sfm_{\BC[T^*G^\vee]^{[]}}}VV @| \\
\on{RHom}_{D^{G^\vee}(\on{Sym}^{[]}(\gvee))}(\on{Sym}^{[]}(\gvee),\BC[T^*G^\vee]^{[]}\otimes A\otimes A) 
@<\sim<\varphi_{A\otimes A}< A\otimes A \\
@V{\sfm_A}VV @V{\sfm_A}VV \\
\on{RHom}_{D^{G^\vee}(\on{Sym}^{[]}(\gvee))}(\on{Sym}^{[]}(\gvee),\BC[T^*G^\vee]^{[]}\otimes A) @<\sim<\varphi_A< A.
\end{CD}
}
\end{equation*}
Equivalently, the following diagram is commutative:
\begin{equation*}
\begin{CD}
\on{RHom}_{D^{G^\vee}(\on{Sym}^{[]}(\gvee))}(\BC[T^*G^\vee]^{[]},A)\otimes
\on{RHom}_{D^{G^\vee}(\on{Sym}^{[]}(\gvee))}(\BC[T^*G^\vee]^{[]},A)
@>\sim>\phi_A\otimes\phi_A> A\otimes A \\
@VVV @| \\
\on{RHom}_{D^{G^\vee}(\on{Sym}^{[]}(\gvee))}(\BC[T^*G^\vee]^{[]}\otimes
\BC[T^*G^\vee]^{[]},A\otimes A) @. A\otimes A \\
@V{\bD\sfm_{\BC[T^*G^\vee]^{[]}}}VV @| \\
\on{RHom}_{D^{G^\vee}(\on{Sym}^{[]}(\gvee))}(\BC[T^*G^\vee]^{[]},A\otimes A) 
@>\sim>\phi_{A\otimes A}> A\otimes A \\
@V{\sfm_A}VV @V{\sfm_A}VV \\
\on{RHom}_{D^{G^\vee}(\on{Sym}^{[]}(\gvee))}(\BC[T^*G^\vee]^{[]},A) @>\sim>\phi_A> A.
\end{CD}
\end{equation*}
Here in the middle left vertical morphism we identify
$\BC[T^*G^\vee]^{[]}=\bD\BC[T^*G^\vee]^{[]}$, so that $\bD\sfm_{\BC[T^*G^\vee]^{[]}}$
goes from $\BC[T^*G^\vee]^{[]}=\bD\BC[T^*G^\vee]^{[]}$ to
$\bD\BC[T^*G^\vee]^{[]}\otimes\bD\BC[T^*G^\vee]^{[]}=
\BC[T^*G^\vee]^{[]}\otimes\BC[T^*G^\vee]^{[]}$.
We now apply the derived Satake equivalence to 
$A=\Psi^{-1}(\CF)$.
\end{NB}%
    The isomorphism
    $\on{RHom}_{D_G(\Gr_G)}(\Areg,\cF)\cong
    \on{RHom}_{D_G(\Gr_G)}(\mathbf 1_{\Gr_G},\Areg\star\cF)$
    is given by the rigidity, and respects the multiplication by
    definition. Therefore it is enough to check the compatibility
    under the isomorphism
    $\Psi^{-1}(\cF) \cong \on{RHom}_{D_G(\Gr_G)}(\mathbf
    1_{\Gr_G},\Areg\star\cF)$.
    This isomorphism is nothing but $\varphi_{M}$
    (where $M = \Psi^{-1}(\cF)$) in \eqref{eq:phiM} under the derived Satake
    equivalence. Therefore it is enough to check that $\varphi_{M}$
    respects the multiplication when $M$ is an algebra in the category
    $D^{G^\vee}(\on{Sym}^{[]}(\gvee))$. This is trivial as
    $\varphi_{M}^{-1}$ is given by the evaluation
    $\on{ev}_1\colon \CC[G^\vee]\to \CC$ at $1\in G^\vee$.
    \begin{NB}
  \begin{equation*}
    \begin{CD}
      \CC[G^\vee]\otimes\CC[G^\vee]
      @>\on{ev}_1\otimes\on{ev}_1>> \CC
      \\
      @V{m_{\CC[G^\vee]}}VV @|
      \\
      \CC[G^\vee]@>>\on{ev}_1>\CC
    \end{CD}
  \end{equation*}
  is commutative as $f\otimes g$ is sent to $f(1)g(1) = (fg)(1)$.
\end{NB}%

\textup{(b)} Applying (a) to ${\mathcal C}_G\CF$, we see that the
multiplications on $\Phi(\CF)$ and
$\on{RHom}_{D_G(\Gr_G)}(\mathbf 1_{\Gr_G},
\linebreak[2]\cC_G\Areg\star\cF)=\on{RHom}_{D_G(\Gr_G)}(\BD\Areg, \cF)$
are equal under the isomorphisms of~\ref{0.1}(c). It remains to
compare them with the multiplication on
$\on{RHom}_{D_G(\Gr_G)}(\CC_{\Gr_G},\Areg\otimes^!\cF)$ defined
in~\ref{prop:!-product} as the composition
\begin{multline*}\on{RHom}_{D_G(\Gr_G)}(\CC_{\Gr_G},\Areg\otimes^!\CF)\otimes 
\on{RHom}_{D_G(\Gr_G)}(\CC_{\Gr_G},\Areg\otimes^!\CF)\\
\to\on{RHom}_{D_G(\Gr_G)}\left(\CC_{\Gr_G},(\Areg\otimes^!\CF)\star
(\Areg\otimes^!\CF)\right)\\
\xrightarrow{\eqref{eq:2.40.1}}
\on{RHom}_{D_G(\Gr_G)}\left(\CC_{\Gr_G},(\Areg\star\Areg)\otimes^!
(\CF\star\CF)\right)
\stackrel{\mathsf m}{\to}\on{RHom}_{D_G(\Gr_G)}(\CC_{\Gr_G},\Areg\otimes^!\CF).
\end{multline*}

Note that
$\on{RHom}_{D_G(\Gr_G)}(\DD\Areg, \cF) \cong
\on{RHom}_{D_G(\Gr_G)}(\CC_{\Gr_G}, \RHHom(\DD\Areg,\cF))$.
Recall that the convolution product $\star$ is defined as
$m_* (q^*)^{-1} p^*$, see \eqref{eq:40}, where we omit $\bar{\ \ }$ for
brevity.
Since $p$ and $q$ are smooth with fibers both $G_\cO$, we have
$(q^*)^{-1} p^* = (q^!)^{-1} p^!$.
By \cite[(2.6.24)]{KaSha} for $m_*$ and \cite[Prop.~3.1.13]{KaSha} for
$p^!$, $q^!$, we have a natural homomorphism
\begin{equation}\label{eq:starHom}
  \RHHom(\DD\Areg,\cF) \star \RHHom(\DD\Areg,\cF)
  \to \RHHom(\DD\Areg\star\DD\Areg, \cF\star\cF).
\end{equation}
\begin{NB}
\begin{multline*}
  m_* (q^!)^{-1} p^! \left(\RHHom(\DD\Areg,\cF) \boxtimes
    \RHHom(\DD\Areg,\cF)\right)
  \\
  \to
  \RHHom(m_* (q^!)^{-1} p^! (\DD\Areg\boxtimes\DD\Areg),
  m_* (q^!)^{-1} p^! (\cF\boxtimes\cF))
\end{multline*}
\end{NB}%
Hence we have the multiplication on
$\on{RHom}_{D_G(\Gr_G)}(\CC_{\Gr_G}, \RHHom(\DD\Areg,\cF))$ by
$\sfm_\cF$ and $\DD\sfm_{\Areg}$. Then the isomorphism
$\on{RHom}_{D_G(\Gr_G)}(\DD\Areg, \cF) \cong
\on{RHom}_{D_G(\Gr_G)}(\CC_{\Gr_G}, \RHHom(\DD\Areg,\cF))$ is
compatible with the multiplication.
\begin{NB}
  Let us first consider \cite[(2.6.24)]{KaSha}. It is given, e.g.,
  by the combination of \cite[(2.6.15) and (2.6.17)]{KaSha} as
  $f_*\RHHom(G_1, G_2)\to f_*\RHHom(f^* f_* G_1, G_2)
  \cong \RHHom(f_* G_1, f_* G_2)$.
  By the definition, it is compatible with the `global' homomorphisms
  $\on{RHom}_{D(Y)}(G_1,G_2) \to \on{RHom}_{D(X)}(f_* G_1, f_* G_2)$
  under
  \begin{equation*}
    \begin{split}
      & \on{RHom}_{D(Y)}(G_1,G_2) = \on{RHom}_{D(Y)}(\CC_Y,
      \RHHom(G_1,G_2)) = \on{RHom}_{D(X)}(\CC_X, f_*\RHHom(G_1,G_2)),\\
      & \on{RHom}_{D(X)}(\CC_X, \RHHom(f_* G_1,f_* G_2)) =
  \on{RHom}_{D(X)}(f_*G_1,f_*G_2).
    \end{split}
  \end{equation*}
  \begin{NB2}
  Next consider \cite[(2.6.27)]{KaSha}. It is given, e.g., by the
  combination of \cite[(2.6.15), (2.6.16), (2.6.17)]{KaSha} as
  $f^* \RHHom(F_1, F_2) \to f^*\RHHom(F_1, f_* f^* F_2)
  \cong f^* f_* \RHHom(f^* F_1, f^* F_2)
  \to \RHHom(f^*F_1, f^* F_2)$.
  We have the global homomorphism
  $\on{RHom}_{D(X)}(F_1,F_2) \to \on{RHom}_{D(Y)}(f^* F_1, f^* F_2)$,
  which is compatible with the interior one under
  \begin{equation*}
    \begin{split}
      & \on{RHom}_{D(X)}(F_1,F_2) = \on{RHom}_{D(X)}(\CC_X,
      \RHHom(F_1,F_2)) \xrightarrow{\id\to f_* f^*}
      \on{RHom}_{D(Y)}(\CC_Y, f^*\RHHom(F_1,F_2)),
      \\
      & \on{RHom}_{D(Y)}(f^* F_1, f^* F_2) = \Hom_{D(Y)}(\CC_Y,
      \RHHom(f^*F_1,f^*F_2)).
    \end{split}
  \end{equation*}
  \end{NB2}%

  Next consider \cite[Prop.~3.1.13]{KaSha}. Note $f^! = f^*[2d]$ ($d$
  is the dimension of fibers). We have a homomorphism
  \begin{multline*}
    \on{RHom}_{D(X)}(F_1,F_2) = \on{RHom}_{D(X)}(\CC_X,
    \RHHom(F_1,F_2))\\
    \xrightarrow{f_! f^!\to \id}
      \on{RHom}_{D(Y)}(f_! f^!\CC_X, \RHHom(F_1,F_2)) \\
      \cong \on{RHom}_{D(Y)}(\CC_Y[2d], f^!\RHHom(F_1,F_2))\\
      \xrightarrow[\cong]{\text{\cite[Prop.~3.1.13]{KaSha}}}
      \on{RHom}_{D(Y)}(\CC_Y[2d], \RHHom(f^* F_1, f^! F_2)) \\
      \cong \on{RHom}_{D(Y)}(\CC_Y, \RHHom(f^* F_1, f^* F_2))
      \cong \on{RHom}_{D(Y)}(f^* F_1, f^* F_2).
    \end{multline*}
    In view of the proof of \cite[Prop.~3.1.13]{KaSha} with
    $G = f^!\CC_X = \CC_Y[2d]$, this is nothing but
    \begin{equation*}
      \on{RHom}_{D(X)}(\CC_X\otimes F_1, F_2) \to
      \on{RHom}_{D(Y)}(f^!\CC_X\otimes f^* F_1, f^! F_2)
    \end{equation*}
    in \cite[Prop.~3.1.11]{KaSha}. It is the `global' homomorphism.
\end{NB}%

Now our remaining task is to check that the isomorphism
$\Areg\otimes^!\cF\cong \RHHom(\DD\Areg,\cF)$ (see
\cite[Prop.~3.4.6]{KaSha}) is compatible with the multiplication.
Since the following diagram commutes:
\begin{equation*}
\begin{CD}
  (\Areg\star\Areg)\otimes^!(\CF\star\CF)
  @>{\sfm_{\Areg}\otimes^!\sfm_\CF}>> \Areg\otimes^!\CF\\
  @V{\wr}VV @V{\wr}VV\\
  \RHHom(\BD\Areg\star\BD\Areg,\CF\star\CF)
  @>{\sfm_\CF\circ\BD\sfm_{\Areg}}>>
  \RHHom(\BD\Areg,\CF),
\end{CD}
\end{equation*}
\begin{NB}
  Note we implicitly use
  $\DD\Areg\star\DD\Areg \cong \DD(\Areg\star\Areg)$ here, as we apply
  $\DD\sfm_{\Areg}$. To show this isomorphism, we again use
  $(q^*)^{-1} p^* = (q^!)^{-1} p^!$.
\end{NB}%
the proof is reduced to the commutativity of
\begin{equation}
  \label{eq:241vs618}
  \begin{CD}
    (\Areg\otimes^!\cF)\star (\Areg\otimes^!\cF) @>\eqref{eq:2.40.1}>>
    (\Areg\star\Areg)\otimes^! (\cF\star\cF) \\
    @V{\wr}VV @V{\wr}VV\\
    \RHHom(\DD\Areg,\cF)\star \RHHom(\DD\Areg,\cF)
    @>>\eqref{eq:starHom}>
    \RHHom(\DD\Areg\star\DD\Areg,\cF\star\cF).
  \end{CD}
\end{equation}
Recall that horizontal arrows in \eqref{eq:241vs618} are defined as
composite of homomorphisms for $p^!$, $q^!$, $m_*$ under
$\star = m_* (q^!)^{-1} p^!$. Thus the commutativity follows from
compatibilities of homomorphisms for $p^!$, $q^!$, $m_*$ under the
isomorphism \cite[Prop.~3.4.6]{KaSha}. We leave the reader to check
the detail.
\begin{NB}
  A reference for the rigidity $(D_G(\Gr_G),\otimes^!)$ is
  \cite[Prop.~3.4.6]{KaSha}. It is given as the composite
  \begin{equation}\label{eq:3.4.6}
    \begin{split}
      & \cF\otimes^! \DD\cG \cong \DD(\DD\cF\otimes \cG) \cong
    \RHHom(\DD\cF\otimes\cG,\DC_X)
    \xrightarrow[\sim]{\text{\cite[(2.6.7)]{KaSha}}}
    \RHHom(\cG,\RHHom(\DD\cF,\DC_X))\\
    \cong\; & \RHHom(\cG,\DD\DD\cF)
    \xrightarrow[\sim]{\text{\cite[Prop.~3.4.3]{KaSha}}} \RHHom(\cG,\cF).
    \end{split}
  \end{equation}
  We need to use the compatibility of this rigidity with various
  functors. For example, consider $!$-pullback. We have
  \cite[Prop.~2.6.5, Prop.~3.1.13]{KaSha}:
  \begin{gather*}
    f^*\cF_1\otimes f^*\cF_2 \cong f^*(\cF_1\otimes \cF_2),\\
    f^!\RHHom(\cF_1,\cF_2) \cong \RHHom(f^*\cF_1, f^!\cF_2).
  \end{gather*}
  Thus
  \begin{gather*}
    f^! (\cF\otimes^!\DD\cG) \cong f^!\cF \otimes^! f^!\DD\cG
    \cong f^!\cF\otimes^! \DD f^*\cG,\\
    f^!\RHHom(\cG,\cF) \cong \RHHom(f^*\cG, f^!\cF).
  \end{gather*}
  Thus it is enough to check that \eqref{eq:3.4.6} is compatible with
  respect to $f^!$.

  \cite[Prop.~3.1.13]{KaSha} is proved as follows:
  \begin{equation*}
    \begin{split}
    & \on{Hom}(\mathcal H, f^!\RHHom(\cG,\cF)) \cong
    \on{Hom}(f_!\mathcal H, \RHHom(\cG, \cF)) \\
    & \xrightarrow[\cong]{\text{\cite[Prop.~2.6.3]{KaSha}}}
    \on{Hom}(f_!\mathcal H\otimes \cG, \cF)
    \xrightarrow[\cong]{\text{\cite[Prop.~2.6.6]{KaSha}}}
    \on{Hom}(f_!(\mathcal H\otimes f^*\cG), \cF)\\
    & \cong
    \on{Hom}(\mathcal H\otimes f^*\cG, f^! \cF)
    \cong
    \on{Hom}(\mathcal H, \RHHom(f^*\cG, f^! \cF)).
    \end{split}
  \end{equation*}
  \begin{NB2}
  For pushforward:
  \begin{equation*}
    \begin{split}
    & \on{RHom}(G_1\otimes^! G_2, G_1\otimes^! G_2)
    \xrightarrow{\text{\cite[(2.6.17)]{KaSha}}}
    \on{RHom}(G_1\otimes^! G_2, f^!f_!G_1 \otimes^! f^! f_! G_2) \\
    & \xrightarrow[\cong]{\text{\cite[Prop.~2.6.5]{KaSha}}}
    \on{RHom}(G_1\otimes^! G_2, f^!(f_!G_1 \otimes^! f_! G_2))
    \cong 
    \on{RHom}(f_!(G_1\otimes^! G_2), f_! G_1\otimes^! f_! G_2).
    \end{split}
  \end{equation*}
  \begin{equation*}
    \begin{split}
    & \on{RHom}(\RHHom(F,G), \RHHom(F,G))
    \xrightarrow{\text{\cite[(2.6.17)]{KaSha}}}
    \on{RHom}(\RHHom(F,G), \RHHom(f^* f_* F, f^! f_! G)) \\
    & \xrightarrow[\cong]{\text{\cite[Prop.~3.1.13]{KaSha}}}
    \on{RHom}(\RHHom(F,G), f^!\RHHom(f_* F, f_! G))
    \cong 
    \on{RHom}(f_!\RHHom(F, G), \RHHom(f_* F, f_! G)).
    \end{split}
  \end{equation*}
  HN cannot check the assertion in this way.
  \end{NB2}

  In order to check the commutativity of~\eqref{eq:241vs618},
  we have to check the commutativity of the diagram
  \begin{equation}
    \label{pushforward}
    \begin{CD}
      m_*(\cX\otimes^!\cY) @>{\BD{\text{\cite[(2.6.22)]{KaSha}}}}>>
      m_*\cX\otimes^!m_*\cY\\
      @V{\wr}VV @V{\wr}VV\\
      m_*\RHHom(\BD\cX,\cY) @>{\text{\cite[(2.6.24)]{KaSha}}}>>
      \RHHom(m_*\BD\cX,m_*\cY)
    \end{CD}
  \end{equation}
  for a proper morphism $m$, as well as the commutativity of the
  diagram
  \begin{equation}
    \label{pullback}
    \begin{CD}
      p^*(\cX\otimes^!\cY) @= \BD p^!(\BD\cX\otimes\BD\cY) @>\eta>>
      p^!\cX\otimes^!p^*\cY\\
      @V{\wr}VV @. @V{\wr}VV\\
      p^*\RHHom(\BD\cX,\cY) @>{\text{\cite[(2.6.27)]{KaSha}}}>>
      \RHHom(p^*\BD\cX,p^*\cY) @= \RHHom(\BD p^!\cX,p^*\cY)
    \end{CD}
  \end{equation}
  for a smooth morphism $p$, where $\eta$ is the composition
  \begin{equation*}
    \BD p^!(\BD\cX\otimes\BD\cY)
    \xrightarrow{\BD\text{\cite[3.1.11]{KaSha}}}
    \BD(p^*\BD\cX\otimes p^!\BD\cY)\cong p^!\cX\otimes^!p^*\cY.
  \end{equation*}
  Here I am confused: should we also use the isomorphism of $p^*$ and
  $p^!$ up to shift for a smooth $p$?

  It suffices to check the commutativity at the level of costalks
  at the points $w\in W\xrightarrow{p}Z$. We denote the embedding of
  $z=p(w)$ into $Z$ by $u_z$. We denote the relative dimension of $p$
  by $d$. Then the costalk of $p^*(\cX\otimes^!\cY)$ at $w$ is
  $u_z^!\cX\otimes u_z^!\cY[-2d]$ as well as the costalk of
  $p^!\cX\otimes^!p^*\cY$ at $w$, and $\eta$ is identity.
  On the other hand, the costalk of $p^*\RHHom(\BD\cX,\cY)$ at $w$
  is $\Hom(u_z^*\BD\cX,u_z^!\cY)[-2d]=
  \Hom((u_z^!\cX)^\vee,u_z^!\cY)[-2d]$
  by~\cite[Proposition~3.1.13]{KaSha} as well as the costalk of
  $\RHHom(\BD p^!\cX,p^*\cY)$ at $w$, and the morphism
  of~\cite[(2.6.27)]{KaSha} is identity.

  As for~\eqref{pushforward}, it suffices to check the commutativity
  at the level of costalks at the points $z\in Z\xleftarrow{m}W$. Let
  $W_z$ be the fiber of $W$ over $z\in Z$, and let
  $i_z\colon W_z\hookrightarrow W$ be its closed embedding. Then we can
  replace $W$ by $W_z$, and $Z$ by $z$, and $\cX,\cY$ by
  $i_z^!\cX,i_z^!\cY$. So the claim follows from the commutativity of
  \begin{equation*}
     \begin{CD}
      R\Gamma(W,\cX\otimes^!\cY) @>{\BD{\text{\cite[(2.6.22)]{KaSha}}}}>>
      R\Gamma(W,\cX)\otimes R\Gamma(W,\cY)\\
      @V{\wr}VV @V{\wr}VV\\
      \on{RHom}_{D(W)}(\BD\cX,\cY) @>{\text{\cite[(2.6.24)]{KaSha}}}>>
      \Hom(R\Gamma(W,\BD\cX),R\Gamma(W,\cY))
    \end{CD}
  \end{equation*}
  The lower arrow is just the functoriality of $R\Gamma(W,\bullet)$.
  Its dual is $R\Gamma(W,\BD\cX\otimes\BD\cY)\leftarrow
  R\Gamma(W,\BD\cX)\otimes R\Gamma(W,\BD\cY)$ given by the definition
  of tensor product of sheaves, i.e.~\cite[(2.6.22)]{KaSha}.
\end{NB}%

\begin{NB} 
The proof is reduced to the commutativity of the following diagram
applied in the case $\CG_1=\CG_2=\Areg,\ \CF_1=\CF_2=\CF$:
\begin{equation*}
{\scriptsize
      \begin{CD}
      \on{Ext}^*_{D_G(\Gr_G)}(\CC_{\Gr_G}, \CG_1\otimes^! \CF_1)\otimes
      \on{Ext}^*_{D_G(\Gr_G)}(\CC_{\Gr_G}, \CG_2\otimes^! \CF_2) @>\cong>>
      \on{Ext}^*_{D_G(\Gr_G)}(\mathbb D\CG_1,\CF_1)\otimes
      \on{Ext}^*_{D_G(\Gr_G)}(\mathbb D\CG_2,\CF_2) \\
      @| @| \\
      H^*_{G_\cO}({\Gr_G}, \CG_1\otimes^! \CF_1)\otimes
      H^*_{G_\cO}({\Gr_G}, \CG_2\otimes^! \CF_2) @>\cong>>
      H^*_{G_\CO}(\Gr_G,\RHHom(\BD\CG_1,\CF_1))\otimes
      H^*_{G_\CO}(\Gr_G,\RHHom(\BD\CG_2,\CF_2))\\
      @V{(q^*)^{-1}p^*}VV @VV{(q^*)^{-1}p^*}V \\
      H^*_{G_\cO}({\Gr_G}\tilde\times{\Gr_G}, (\CG_1\otimes^! \CF_1)
      \tilde\boxtimes (\CG_2\otimes^! \CF_2)) @>\cong>>
      H^*_{G_\CO}(\Gr_G\tilde{\times}\Gr_G,\RHHom(\BD\CG_1,\CF_1)
      \tilde{\boxtimes}\RHHom(\BD\CG_2,\CF_2))\\
    @V\text{$\otimes^!$ commutes with $p^*$, $q^*$}VV @VV\text{same as left}V \\
      H^*_{G_\cO}({\Gr_G}\tilde\times{\Gr_G}, (\CG_1\tilde\boxtimes\CG_2)
      \otimes^! (\CF_1\tilde\boxtimes\CF_2)) @>\cong>>
      H^*_{G_\CO}(\Gr_G\tilde{\times}\Gr_G,\RHHom(\BD\CG_1
      \tilde{\boxtimes}\BD\CG_2,\CF_1\tilde{\boxtimes}\CF_2)) \\
      @| @| \\
      H^*_{G_\cO}({\Gr_G}, m_*\left( (\CG_1\tilde\boxtimes\CG_2)
      \otimes^! (\CF_1\tilde\boxtimes\CF_2)\right)) @>\cong>>
      H^*_{G_\CO}(\Gr_G, m_* \left(\RHHom(\BD\CG_1
      \tilde{\boxtimes}\BD\CG_2,\CF_1\tilde{\boxtimes}\CF_2)\right)) \\
    @V{\text{\cite[(2.6.22)]{KaSha}}}VV @VV{\text{\cite[(2.6.24)]{KaSha}}}V \\      H^*_{G_\cO}({\Gr_G},  m_*(\CG_1\tilde\boxtimes\CG_2)
      \otimes^! m_*(\CF_1\tilde\boxtimes\CF_2)) @>\cong>>
      H^*_{G_\CO}(\Gr_G, \RHHom(m_*(\BD\CG_1
      \tilde{\boxtimes}\BD\CG_2),m_* (\CF_1\tilde{\boxtimes}\CF_2)))\\
      @| @| \\
      H^*_{G_\cO}({\Gr_G}, (\CG_1\star \CG_2) \otimes^! (\CF_1\star\CF_2)) @>\cong>>
      H^*_{G_\CO}(\Gr_G,\RHHom(\BD\CG_1\star\BD\CG_2,\CF_1\star\CF_2))\\
      @| @| \\
      \on{Ext}^*_{D_G(\Gr_G)}(\CC_{\Gr_G}, (\CG_1\star \CG_2) \otimes^! (\CF_1\star\CF_2))
      @>>\cong>  \on{Ext}^*_{D_G(\Gr_G)}(\mathbb D\CG_1\star \mathbb
      D\CG_2,
      \CF_1\star\CF_2).
    \end{CD}
}
\end{equation*}
Note that in our case $\CG_1=\CG_2=\Areg$, the graded vector spaces
$\on{Ext}^*_{D_G(\Gr_G)}(\CC_{\Gr_G}, \CG_1\otimes^!\CF_1),\\
      \on{Ext}^*_{D_G(\Gr_G)}(\CC_{\Gr_G}, \CG_2\otimes^!\CF_2)$, etc.\ are upgraded
to $\on{RHom}^*_{D_G(\Gr_G)}(\CC_{\Gr_G}, \CG_1\otimes^! \CF_1),\\
      \on{RHom}^*_{D_G(\Gr_G)}(\CC_{\Gr_G}, \CG_2\otimes^! \CF_2)$, etc., and 
the upgraded diagram remains commutative.
\end{NB}%
\begin{NB}
  HN thinks that a key point of the proof is not explained. We do not
  explain (a) the reason why each small square is commutative, (b) the
  reason why the vertical composite is equal to the original
  multiplication (where the final part given by
  $\sfm_{\Areg}\otimes^!\sfm_\cF$,
  $\sfm_\cF\circ\DD\sfm_{\Areg}$ is omitted.) For (a), HN already
  writes two squares where reasoning should be given. For (b), HN is
  wondering the arrow involving \cite[(2.6.24)]{KaSha}, namely why
  \begin{equation*}
    \begin{CD}
      \on{RHom}_{D_G(\Gr_G)}(\CC_{\Gr_G},
      m_*\left(\RHHom(\BD\Areg\tilde{\boxtimes}\BD\Areg,
        \CF\tilde{\boxtimes}\CF)\right)) @>\cong>>
      \on{RHom}_{D_G(\Gr_G\tilde\times\Gr_G)}
      (\BD\Areg\tilde{\boxtimes}\BD\Areg,\CF\tilde{\boxtimes}\CF)
      \\
      @V{\text{\cite[(2.6.24)]{KaSha}}}VV @VV{m_*}V \\
      \on{RHom}_{D_G(\Gr_G)}(\CC_{\Gr_G},
      \RHHom(m_*(\BD\Areg\tilde{\boxtimes}\BD\Areg),
      m_*(\CF\tilde{\boxtimes}\CF)))
      @>>\cong>
      \on{RHom}_{D_G(\Gr_G)}(m_*(\BD\Areg\tilde{\boxtimes}\BD\Areg),m_*(\CF\tilde{\boxtimes}\CF))
\end{CD}
\end{equation*}
is commutative.
\end{NB}%
\end{proof}

\subsection{Hamiltonian reduction}
\label{ham}
The right
Kostant-Whittaker reduction of $\Areg=\Psi(\BC[T^*G^\vee]^{[]})$ equipped with
$G^\vee$-action is a particular case of the following construction.

Let $\CG$ be a commutative ring object of $D^{G^\vee}_{}(\on{Sym}^{[]}(\gvee))$
equipped with an action of an algebraic group $H$. Let $\fh$ be the Lie algebra
of $H$. Let $\on{Sym}^{[]}(\fh)$ be the symmetric algebra of $\fh$ equipped
with a nonnegative grading (not necessarily the standard one, nor the doubled
standard one) and viewed as a dg-algebra with a trivial differential.
Let $\mu^*\colon \on{Sym}^{[]}(\fh)\to\on{RHom}_{D^{G^\vee}(\on{Sym}^{[]}(\gvee))}(\CG,\CG)$ be an $H$-equivariant
homomorphism of dg-algebras such that the multiplication morphism
${\mathsf m}\colon \CG\otimes_{\on{Sym}^{[]}(\gvee)}\CG\to\CG$ is $H$-equivariant
and $\on{Sym}^{[]}(\fh)$-linear.

In all the examples below the following property holds:
after applying the forgetful functor and taking
cohomology and their spectrum, $\on{Spec}H^*(\on{Forg}\CG)$ is equipped with
an $H$-invariant symplectic form, and $\mu^*$ is compatible with a moment map
$\mu\colon \on{Spec}H^*(\on{Forg}\CG)\to\fh^*$.

Given an $H$-invariant subvariety $X\subset\fh^*$ such that the projection
$\on{Sym}(\fh)\twoheadrightarrow\BC[X]$ is compatible with the grading
$\on{Sym}^{[]}(\fh)$ and induces the grading $\BC[X]^{[]}$, we define
$\CG\tslash(H,X):=(\CG\otimes_{\on{Sym}^{[]}(\fh)}\BC[X]^{[]})^H$. This is a
commutative ring
object of $D^{G^\vee}_{}(\on{Sym}^{[]}(\gvee))$. If $X=\{0\}\subset\fh^*$,
we simply write $\CG\tslash H$ for $\CG\tslash(H,\{0\})$.

\subsection{Leg amputation}
\label{amputation}
Following~\ref{prop:!-product},
\begin{NB} (or the beginning of current~\ref{sec:Sicilian})
\end{NB}
we consider a commutative ring object
$\scA^b:=i_\Delta^!(\boxtimes_{k=1}^b(\Areg)_k)$ (in particular, the ring object
associated with $S^2$ with three punctures
is $\scA^3$ in our present notation).
According to~\ref{sec:group_action}, $\scA^b$ is equipped with an action of
$\SL(N)^b=(G^\vee)^b$. More generally, we consider a commutative ring object
$\scA^b:=i_\Delta^!(\boxtimes_{k=1}^b(\Areg)_k)$ on $\Gr_{G}$ equipped with
an action of $(G^\vee)^b$ for a reductive flavor group $G$.
We set $W_G^b:=\on{Spec}H^*_{G_\CO}(\Gr_G,\scA^b)$. We conjecture that $H^*_{G_\cO}(\Gr_G,\scA^b)$ is finitely generated, which we checked so far only in type $A$.
We assume it hereafter. Then $W_G^b$ is an affine variety
with Poisson structure equipped with a hamiltonian action of $(G^\vee)^b$.
In particular, $W_{\PGL(N)}^3$ is $W$ of the beginning of
current~\ref{sec:Sicilian}.

Also, $W^2_G=T^*G^\vee$ since
$\BC[W^2_G]=\on{Forg}\circ\Phi(\Areg)=\BC[T^*G^\vee]$, see~\ref{0.1}(c).

According to~\ref{def:dg-action}, we have the action of $b$ copies of
$\on{Sym}^{[]}(\gvee)$ on $\scA^b$. We can consider its Kostant-Whittaker
reduction $\kappa^r_b(\scA^b)=\scA^b\otimes_{\on{Sym}^{[]}_{\on{new}}(\gvee)}\BC[\Sigma]^{[]}$
with respect to the last copy of $G^\vee$ in $(G^\vee)^b$ (cf.~\ref{ham}).
More precisely we apply $\Psi$ to $\kappa^r_b(\Psi^{-1}\scA^b)=
(\Psi^{-1}\scA^b)\otimes_{\on{Sym}^{[]}_{\on{new}}(\gvee)}\BC[\Sigma]^{[]}$.
\begin{NB}
    Misha, this seems still imprecise. Give an explanation of the
    definition of $\kappa^r_b(\Psi^{-1}\scA^b)$. Is it the
    $U^\vee_-$-invariant part ? If so, is it really in
    $D^{G^\vee}_{}(\on{Sym}^{[]}(\gvee))$, i.e., is the
    perfectness preserved under the $U^\vee_-$-invariant part ?
    \begin{NB2}
      Recall that
      $\scA_R=\bigoplus_{\lambda\in X^+}\on{IC}(\ol\Gr{}^\lambda_G)\otimes V^\lambda$,
      and $\scA^b=\bigoplus_{\lambda_1,\ldots,\lambda_b\in X^+}i^!_\Delta
      \left(\boxtimes_{k=1}^b\on{IC}(\ol\Gr{}^{\lambda_k}_G)\right)
      \otimes\bigotimes_{k=1}^b
      V^{\lambda_k}$. Now $\kappa^r_b$ does something to the vector factors
      $\bigotimes_{k=1}^bV^{\lambda_k}$ leaving the sheaf factors
      $i^!_\Delta\left(\boxtimes_{k=1}^b\on{IC}(\ol\Gr{}^{\lambda_k}_G)\right)$
      intact. This is wrong. However, as Sasha argues, the right reduction
is defined directly at the level of Satake category.
    \end{NB2}%
\end{NB}%

\begin{Lemma}
\label{last copy}
$\kappa^r_b(\scA^b)=\scA^{b-1}$.
\end{Lemma}

\begin{proof}
We have $\kappa^r_b(\scA^b)=\kappa^r_b(i_\Delta^!(\boxtimes_{k=1}^b(\Areg)_k))=
\kappa^r_b(i_\Delta^!(\scA^{b-1}\boxtimes\Areg))=
i_\Delta^!(\scA^{b-1}\boxtimes\kappa^r(\Areg))=
i_\Delta^!(\scA^{b-1}\boxtimes\DC_{\Gr_{G}})=\scA^{b-1}$.
\begin{NB}
  Misha, please explain why $i_\Delta^!$ commutes with $\kappa^r_b$.
  \begin{NB2}
      Recall that
      $\scA_R=\bigoplus_{\lambda\in X^+}\on{IC}(\ol\Gr{}^\lambda_G)\otimes V^\lambda$,
      and $\scA^b=\bigoplus_{\lambda_1,\ldots,\lambda_b\in X^+}i^!_\Delta
      \left(\boxtimes_{k=1}^b\on{IC}(\ol\Gr{}^{\lambda_k}_G)\right)
      \otimes\bigotimes_{k=1}^b
      V^{\lambda_k}$. Now $\kappa^r_b$ does something to the vector factors
      $\bigotimes_{k=1}^bV^{\lambda_k}$ leaving the sheaf factors
      $i^!_\Delta\left(\boxtimes_{k=1}^b\on{IC}(\ol\Gr{}^{\lambda_k}_G)\right)$
      intact. This is wrong.
    \end{NB2}%
\end{NB}%
Indeed,
$\kappa^r_b(\bullet)=\bullet\otimes_{\on{Sym}^{[]}_{\on{new}}(\gvee)}\BC[\Sigma]^{[]}$ (with
respect to the action of the $b$-th copy of $\on{Sym}^{[]}_{\on{new}}(\gvee)$).
In the third equality we use that for
$\CF=\scA^{b-1}\in D_G(\Gr_G)$ and $\CF'=\Areg\in D_G(\Gr_G)$ with a dg-algebra
$A=\on{Sym}^{[]}(\gvee)$ equipped with a homomorphism to
$\on{RHom}_{D_G(\Gr_G)}(\CF',\CF')$, and a dg-module $M=\BC[\Sigma]^{[]}$ over $A$, we
have $(\CF\otimes^!\CF')\otimes_AM=\CF\otimes^!(\CF'\otimes_AM)$ by
associativity of tensor product. This equality is compatible with the
commutative ring structures by the construction in~\ref{prop:!-product}
(the reduction $\kappa^r_b(\scA^b)$ carries the induced ring structure
by the explanation in~\ref{ham} since the multiplication
${\mathsf m}\colon \Areg\star\Areg\to\Areg$ is $\on{Sym}^{[]}(\gvee)$-linear
for the right action of $\on{Sym}^{[]}(\gvee)$ on $\Areg=\Psi(\BC[T^*G]^{[]})$.)
\end{proof}

\begin{Corollary}
\label{last copy W}
$\kappa_b(W_G^b)=W_G^{b-1}$.
\end{Corollary}

\begin{proof}
We have to check that $\kappa_b$ commutes with $H^*_{G_\CO}(\Gr_G,\bullet)$.
After applying the derived Satake equivalence we have to check that $\kappa_b^r$
commutes with $\kappa^l$. Recall that
$\kappa^l(\bullet)=H^*(\bullet\otimes_{\on{Sym}^{[]}_{\on{new}}(\gvee)}\BC[\Sigma]^{[]})$
(with respect to the action of the left copy of $\on{Sym}^{[]}_{\on{new}}(\gvee)$),
while $\kappa^b_r(\bullet)=\bullet\otimes_{\on{Sym}^{[]}_{\on{new}}(\gvee)}\BC[\Sigma]^{[]}$
(with respect to the action of the $b$-th right copy of
$\on{Sym}^{[]}_{\on{new}}(\gvee)$). We have
$\Psi^{-1}(\scA^b)\in D^{(G^\vee)^{b+1}}_{}(\on{Sym}^{[]}((\gvee)^{\oplus b+1})$
(one left structure and $b$ right structures). We assign number 0 to the
left structure. Then
\begin{multline*}
\BC[W_G^{b-1}]=H^*\left((\Psi^{-1}(\scA^b)\otimes_{_b\on{Sym}^{[]}_{\on{new}}(\gvee)}
\BC[\Sigma]^{[]})\otimes_{_0\on{Sym}^{[]}_{\on{new}}(\gvee)}\BC[\Sigma]^{[]}\right)\\
=H^*\left((\Psi^{-1}(\scA^b)\otimes_{_0\on{Sym}^{[]}_{\on{new}}(\gvee)}
\BC[\Sigma]^{[]})\otimes_{_b\on{Sym}^{[]}_{\on{new}}(\gvee)}\BC[\Sigma]^{[]}\right)\\
=H^*\left((\Psi^{-1}(\scA^b)\otimes_{_0\on{Sym}^{[]}_{\on{new}}(\gvee)}
\BC[\Sigma]^{[]})\right)\otimes_{_b\on{Sym}(\gvee)}\BC[\Sigma]=
\kappa_b(\BC[W^b_G]).
\end{multline*}
The third equality (commutation of taking cohomology and tensor product with a
$_b\on{Sym}^{[]}_{\on{new}}(\gvee)$-module) is clear for free modules, and then
for perfect complexes by devissage, and then for Ind-perfect complexes since
cohomology commutes with direct images.
\end{proof}

\subsection{General surfaces for arbitrary reductive groups and fusion}
\label{cylinder G}
First we study the case of cylinder and give another explanation of the
identification $W_G^2=T^*G^\vee$.

We consider the equivalence 
\begin{equation*}
\Psi\boxtimes\Psi\colon
D^{G^\vee\times G^\vee}_{}(\on{Sym}^{[]}(\gvee\oplus\gvee))\to
D_{G\times G}(\Gr_{G}\times\Gr_{G}).
\end{equation*}
Under this equivalence, the ring object $\Areg\boxtimes\Areg\in
D_{G\times G}(\Gr_{G}\times\Gr_{G})$ corresponds to the
$G^\vee\times G^\vee$-equivariant free
$\on{Sym}^{[]}(\gvee\oplus\gvee)$-module
$\BC[G^\vee\times G^\vee]\otimes\on{Sym}^{[]}(\gvee\oplus\gvee)$
which will be denoted $\BC[T^*G^\vee\times T^*G^\vee]^{[]}$ for short.
It is equipped with the {\em right} action of $G^\vee\times G^\vee$ with
the right moment map $(\mu_r,\mu_r)$. The hamiltonian reduction with
respect to the diagonal right action
\begin{equation*}
\left(T^*G^\vee\times T^*G^\vee\right)\tslash\Delta_{G^\vee}:=
\on{Spec}\left(\BC[(\mu_r,\mu_r)^{-1}(\Delta_{(\fg^\vee)^*})]^{\Delta_{G^\vee}}\right)=
(\mu_r,\mu_r)^{-1}(\Delta_{(\fg^\vee)^*})\dslash \Delta_{G^\vee}
\end{equation*}
(the categorical quotient is the set-theoretical one, as it is with respect
to the free action of $G^\vee$) is
nothing but $T^*G^\vee$ equipped with the residual left action of
$G^\vee\times G^\vee\colon
(h_1,h_2)(g,\xi)=(h_2gh_1^{-1},\on{Ad}_{h_1}\xi)$,
and the equivariant morphism to $(\gvee)^*\oplus(\gvee)^*\colon
(g,\xi)\mapsto(\xi,\on{Ad}_g\xi)$. \linelabel{lne:1518}
Note that the natural projection $\BC[T^*G^\vee\times T^*G^\vee]\twoheadrightarrow
\BC[(\mu_r,\mu_r)^{-1}(\Delta_{(\fg^\vee)^*})]$ is compatible with the grading of
$\BC[T^*G^\vee\times T^*G^\vee]$, and so it induces a grading on the target,
to be denoted $\BC[(\mu_r,\mu_r)^{-1}(\Delta_{(\fg^\vee)^*})]^{[]}$. This in turn
induces a grading on the $\Delta_{G^\vee}$-invariant subalgebra, to be denoted
$\BC[T^*G^\vee\times T^*G^\vee\tslash\Delta_{G^\vee}]^{[]}$. Viewing it
as a $G^\vee\times G^\vee$-equivariant graded module over
$\on{Sym}^{[]}(\gvee\oplus\gvee)$ (with zero differential) and taking its free
resolution, we obtain the same named object of
$D^{G^\vee\times G^\vee}(\on{Sym}^{[]}(\gvee\oplus\gvee))$.
We will denote $\Psi\boxtimes\Psi\left(\BC[T^*G^\vee\times
T^*G^\vee\tslash\Delta_{G^\vee}]^{[]}\right)$ by
$\Areg\boxtimes\Areg\tslash\Delta_{G^\vee}\in
D_{G\times G}(\Gr_{G}\times\Gr_{G})$ for short, cf.~\ref{ham}.
\begin{NB}
By construction, $\BC[T^*G^\vee\times T^*G^\vee\tslash\Delta_{G^\vee}]^{[]}$ is a
graded module over $\on{Sym}^{[]}(\gvee\oplus\gvee)$ (with zero differential),
but not a free module. So it has to be replaced by a free resolution.
\end{NB}%

Now
$\Psi^{-1}(\Areg\star\Areg)=
\BC[T^*G^\vee]^{[]}\otimes_{\on{Sym}^{[]}(\gvee)}\BC[T^*G^\vee]^{[]}$,
and $\Psi^{-1}(\on{IC}(\Gr^0_G))=\on{Sym}^{[]}(\gvee)$. Hence, $W_G^2=
\on{Spec}H^*_{G_\CO}(\Gr_G,\scA^2)=(T^*G^\vee\times T^*G^\vee)\tslash\Delta_{G^\vee}=
T^*G^\vee$. The action $G^\vee\times G^\vee$ on $W_G^2$ is the natural action
of $G^\vee\times G^\vee$ on
$T^*G^\vee\colon (h_1,h_2)\cdot(g,\xi)=(h_2gh_1^{-1},\on{Ad}_{h_1}\xi)$;
in particular, the diagonal action of $\Delta_{G^\vee}$ is the adjoint action
$h(g,\xi)=(hgh^{-1},\on{Ad}_h\xi)$.

We denote $\scA^2\tslash\Delta_{G^\vee}$ by $\scB\in D_G(\Gr_G)$.  We have
$H^*_{G_\CO}(\Gr_G,\scB)=H^*_{G_\CO}(\Gr_G,\scA^2)\tslash\Delta_{G^\vee}=
\BC[(T^*G^\vee)\tslash\Delta_{G^\vee}]=\BC[T^\vee\times\ft]^W$.
Here the last equality is a multiplicative analog of the isomorphism
$(\fg^\vee\times\fg^\vee)\tslash \Delta_{G^\vee}=
(\ft^\vee\times\ft^\vee)/W$~\cite{MR1446576} due to I.~Losev. Its proof is given
in~\ref{subsec:Loseu} below.
More generally, we have
$H^*_{G_\CO}(\Gr_G,i^!_\Delta(\scA^b\boxtimes\scB))=
\BC[W^{b+2}_G\tslash\Delta^{b+1,b+2}_{G^\vee}]=
\BC[(\mu_r^{b+1},\mu_r^{b+2})^{-1}(\Delta^{b+1,b+2}_{(\fg^\vee)^*})]^{\Delta_{G^\vee}^{b+1,b+2}}$
where $\Delta^{b+1,b+2}_{G^\vee}$ stands for the diagonal in the
product of the last two copies in $(G^\vee)^{b+2}$.

We denote $\scB^g:=i^!_\Delta(\boxtimes_{k=1}^g\scB_k)\in D_G(\Gr_G)$. Then
$\on{Spec}H^*_{G_\CO}(\Gr_G,\scA^b\otimes^!\scB^g)$ is an object of HS associated
with a surface of genus $g$ with $b$ punctures. Now we turn to the study of
fusion of surfaces.


\begin{NB}
  \begin{Lemma}
\label{2:0}
$\on{pr}_{2*}(\Areg\boxtimes\Areg\tslash\Delta_{G^\vee})=\DC_{\Gr_{G}}$.
\end{Lemma}

\begin{proof}
According to~\cite[Theorem~4]{MR2422266}, under the equivalence
$\Psi\boxtimes\Psi$, the direct image $\on{pr}_{2*}$ corresponds to
the Kostant-Whittaker reduction $\kappa^l_1$ with respect to the {\em left}
action of $U^\vee_-$ in the first copy of $G^\vee$. The resulting object of
$D^{G^\vee}_{}(\on{Sym}^{[]}(\gvee))$ is
$\kappa^l_1\left(\BC[T^*G^\vee\times
T^*G^\vee\tslash\Delta_{G^\vee}]^{[]}\right)=\kappa^r(\BC[T^*G^\vee]^{[]})$
\begin{NB2} (compare lines \lineref{lne:1488} and \lineref{lne:1518}) \end{NB2}%
of~\ref{derived Satake} corresponding under $\Psi$ to $\DC_{\Gr_{G}}$.
\end{proof}
\end{NB}%

\begin{Proposition}
\label{2:0 W}
Let $\Delta_{G^\vee}^{b_1,b_2}$ denote the diagonal action of the $b_1$-st and
$b_2$-nd copy of $G^\vee$ on $W_G^{b_1}\times W_G^{b_2}$. Then $W_G^{b_1+b_2-2}=
(W_G^{b_1}\times W_G^{b_2})\tslash\Delta_{G^\vee}^{b_1,b_2}$.
\end{Proposition}

\begin{proof}
  We have 
\begin{multline*}
\BC[W_G^{b_1+b_2-2}]=H^*_{G_\CO}(\Gr_G,\scA^{b_1-1}\otimes^!\scA^{b_2-1})=
  \on{Ext}^*_{D_G(\Gr_G)}({\mathbf 1}_{\Gr_G},
{\mathcal C}_G\scA^{b_1-1}\star\scA^{b_2-1})\\
  =\on{Ext}^*_{D^{G^\vee}(\on{Sym}^{[]}(\gvee))}\left(\on{Sym}^{[]}(\gvee),
{\mathfrak C}_{G^\vee}\Psi^{-1}(\scA^{b_1-1})
  \otimes_{\on{Sym}^{[]}(\gvee)}\Psi^{-1}(\scA^{b_2-1})\right)\\
  =\on{Ext}^*_{D^{G^\vee}(\on{Sym}^{[]}(\gvee))}\left(\on{Sym}^{[]}(\gvee),
  \Phi(\scA^{b_1-1})\otimes_{\on{Sym}^{[]}(\gvee)}{\mathfrak C}_{G^\vee}
  \Phi(\scA^{b_2-1})\right),
\end{multline*}
(the second equality is~\ref{0.1}(b)).

Now
  $\on{Ext}^*_{D^{G^\vee}(\on{Sym}^{[]}(\gvee))}\left(\on{Sym}^{[]}(\gvee),
\Phi(\scA^{b_1-1})\otimes_{\on{Sym}^{[]}(\gvee)}{\mathfrak C}_{G^\vee}\Phi(\scA^{b_2-1})\right)$
is the hamiltonian reduction
  $(\Phi(\scA^{b_1-1})\boxtimes{\mathfrak C}_{G^\vee}
\Phi(\scA^{b_2-1}))\tslash\Delta_{G^\vee}$ of
    $\Phi(\scA^{b_1-1})\boxtimes{\mathfrak C}_{G^\vee}\Phi(\scA^{b_2-1})$ with
    respect to the diagonal (left) action of $G^\vee$. According to~\ref{0.1}(c)
and~\ref{rem:formality},
    \begin{equation*}
    \Phi(\scA^{b-1})=H^*_{G_\CO}(\Gr_G,\Areg\otimes^!\scA^{b-1})=
    H^*_{G_\CO}(\Gr_G,\scA^b)=\BC[W_G^b],
    \end{equation*} and the left
    $G^\vee\ltimes\on{Sym}^{[]}(\gvee)$-module structure in the LHS coincides
    with the {\em right} $G^\vee\ltimes\on{Sym}^{[]}(\gvee)$-module structure
    in the RHS (with respect to the last copy of
    $G^\vee\ltimes\on{Sym}^{[]}(\gvee)$). This completes the proof.
\end{proof}

\begin{Remark}\label{rem:generalizedGaiotto}
    The same argument shows that
    \begin{equation*}
        H^*_{G_\cO}(\Gr_G,\scA_1\otimes^!\scA_2) \cong
        H^*_{G_\cO}(\Gr_G,\Areg\otimes^!\scA_1)\otimes
        \mathfrak C_{G^\vee} H^*_{G_\cO}(\Gr_G,\Areg\otimes^!\scA_2)
        \tslash \Delta_{G^\vee}
    \end{equation*}
    for ring objects $\scA_1$, $\scA_2$ in $D_G(\Gr_G)$.

The natural action of $H^*_{G_\CO}(\on{pt})\otimes H^*_{G_\CO}(\on{pt})=
\BC[\Sigma]^{[]}\otimes \BC[\Sigma]^{[]}$ on the RHS factors through the
multiplication homomorphism
$\BC[\Sigma]^{[]}\otimes \BC[\Sigma]^{[]}\to\BC[\Sigma]^{[]}$, and the resulting
action of $H^*_{G_\CO}(\on{pt})=\BC[\Sigma]^{[]}$ in the RHS coincides with its
natural action in the LHS.
\begin{NB}
  HN is not sure where $\CC[\Sigma]^{[]}$-action is explained for
  $\Ext^*_{D^{G^\vee}(\on{Sym}(\gvee)^{[]}}$ in the displayed equation
  in the proof of \ref{2:0 W}. Therefore it is difficult for him to
  check this claim.

  Well, $\Ext^*_{D^{G^\vee}(\on{Sym}^{[]}(\gvee))}(\on{Sym}^{[]}(\gvee),M)=M^{G^\vee}$
  is equipped with the natural action of
  $\on{Sym}^{[]}(\gvee)^{G^\vee}=\BC[\Sigma]^{[]}$.
  
\end{NB}%
\end{Remark}

\begin{Remark}
  \label{rem:quantum Satake}
  The results of~\ref{derived Satake}--\ref{cylinder G} have their quantum
  counterparts if we consider an extra equivariance with respect to the
  loop rotations. They are based on the equivalence of monoidal triangulated
  categories $D^{G^\vee}(U_\hbar^{[]}(\gvee))\iso
  D_{G_\CO\rtimes\BC^\times}(\Gr_G)$~\cite[Theorem~5]{MR2422266}
  (recall that $D_G(\Gr_G)$ is a shorthand for the $G_\CO$-equivariant derived
  category $D_{G_\CO}(\Gr_G)$). In particular, the
  regular sheaf $\Areg\in D_{G_\CO\rtimes\BC^\times}(\Gr_G)$ corresponds to
  $U_\hbar^{[]}(\gvee)\ltimes\BC[G^\vee]$.
  For a loop rotation equivariant ring object $\scA\in D_{G_\CO\rtimes\BC^\times}(\Gr_G)$
  one can consider the loop rotation equivariant cohomology ring
  $H^*_{G_\CO\rtimes\BC^\times}(\Gr_G,\scA)$. Similarly to~\ref{rem:generalizedGaiotto},
  we have
  \begin{equation*}
        H^*_{G_\cO\rtimes\BC^\times}(\Gr_G,\scA_1\otimes^!\scA_2) \cong
        H^*_{G_\cO\rtimes\BC^\times}(\Gr_G,\Areg\otimes^!\scA_1)\otimes
        \mathfrak C_{G^\vee} H^*_{G_\cO\rtimes\BC^\times}(\Gr_G,\Areg\otimes^!\scA_2)
        \tslash \Delta_{G^\vee}
    \end{equation*}
  ({\em quantum Hamiltonian reduction}) for ring objects $\scA_1$, $\scA_2$
  in $D_{G_\CO\rtimes\BC^\times}(\Gr_G)$.

  In particular, we set
  $\BC_\hbar[W_G^b]:=H^*_{G_\CO\rtimes\BC^\times}(\Gr_G,\scA^b)$,
  a quantization of $\BC[W_G^b]$. Then $\BC_\hbar[W_G^{b_1+b_2-2}]=
  (\BC_\hbar[W_G^{b_1}]\otimes\BC_\hbar[W_G^{b_2}])\tslash\Delta_{G^\vee}$.
\end{Remark}


\subsection{Gluing construction vs hamiltonian reduction}
\label{subsec:glu_ham}

Let us slightly change the point of view to our gluing construction
\ref{subsec:glue} so that it formally looks similar to a hamiltonian
reduction.

Let $\scA$ be a commutative ring object on $\Gr_G$. Let $G'$ be a
subgroup of $G$, which is also reductive. We have an inclusion
$i\colon \Gr_{G'}\to\Gr_G$. Then
\begin{quote}
  The $!$-pull back $i^!\scA$ is a ring object on $\Gr_{G'}$.
\end{quote}
When $\scA$ arises as $\pi_*(\DC_{\cR}[-2\dim\bN_\cO])$ from a
representation $\bN$, $i^!\scA$ is the ring object associated with
$\bN$ viewed as a representation of $G'$.

Next suppose we have a homomorphism $G\to G''$ to another reductive
group $G''$. We consider the induced morphism
$p\colon\Gr_G\to \Gr_{G''}$, which is equivariant under the induced
group homomorphism $G_\cO\to G''_\cO$. Then
\begin{quote}
  The pushfoward $Q_{p*}\scA$ is a ring object on $\Gr_{G''}$.
\end{quote}
Here $Q_{p*}$ is the general pushforward as in \ref{subsec:affG_flavor}.
The construction of \ref{subsec:affG_flavor} is an example of the
pushforward, where $G$, $G''$ here are $\tilde G$, $G_F$ there, and
$\scA\in D_G(\Gr_G)$ here is the ring object on
$D_{\tilde G}(\Gr_{\tilde G})$ associated with a representation $\bN$
of $\tilde G$ there.
When $G''$ is the trivial group, the pushforward is nothing but taking
the cohomology $H^*_{G_\cO}(\Gr_G,\scA)$.
In physics terminology this operation corresponds to the
\emph{gauging} with respect to the kernel of the homomorphism
$G\to G''$.

Note that this construction is \emph{formally} similar to a
hamiltonian reduction: suppose that we have a hamiltonian $G$ space
$X$. We take a hamiltonian reduction $X\tslash G'$ with respect to a
normal subgroup $G'\triangleleft G$. Then $X\tslash G'$ is a
hamiltonian $G'' = G/G'$ space.
This is not just an analogy if we consider gauging in quantum field
theories:
The Higgs branch of a gauge theory associated with $(G,\bN)$ is the
hamiltonian reduction $\bN\oplus\bN^*\tslash G$. (See
\cite{Tach-review} for a review for mathematicians.)

As an example of the similarity, let us consider \eqref{eq:fund} which
we regard as a quantum field theory \emph{upgrade} of the definition
$W^{g,b} = \Spec H^*_{G_\CO}(\Gr_G,\scA^b\otimes^!\scB^g)$. Let us
consider the Coulomb branch of the left hand side, which should be
equal to the Higgs branch of the right hand side. Under the gauging
$\tslabar$, the Higgs branch is replaced by the symplectic
reduction as we have just mentioned. Hence we get
\begin{equation*}
    \cM_C(S_{G^\vee}(C)) = 
      \mathcal N_{G}^b \times (\fg\oplus\fg^*)^g
    \tslash G_{\mathrm{diag}},
\end{equation*}
where $\mathcal N_G$ is the nilpotent cone of $G$, and
$\fg\oplus \fg^*$ is symplectic by the natural pairing.
Thus the Coulomb branch $\cM_C(S_{G^\vee}(g,b))$ is the `additive
version' of the $G$-character variety on the punctured Riemann surface
$C$, where the monodromy around punctures sit in the regular
unipotent orbit.
\begin{NB}
  For a general nilpotent element $e_i$, one replaces the regular
  nilpotent orbit by the Lusztig-Spaltenstein dual of $e_i$.
\end{NB}%
\begin{NB}
    Let us consider $\cT = S_{SL(2)}(S^2,\text{three punctures})$. The
    corresponding Hitchin moduli space is a point. Therefore we
    conclude that
    $\cM_C(\Hyp(\CC^2\otimes\CC^2\otimes\CC^2)) = \mathrm{pt}$. Thus
    my naive consideration might be correct.
\end{NB}%
When $G$ is of type $A$, this is the Higgs branch of the quiver gauge
theory associated with the quiver \cite[3(iii)
Figure~5]{2015arXiv150303676N}. See the references therein to see why
it is an additive version of a $G$-character variety.


\subsection{Gluing in the Higgs branch side}\label{subsec:gluHiggs}

Let us pursue the analogy between the gluing construction and
hamiltonian reduction further. Let us consider a ring object
associated with $S_G(g,b)$ in the Coulomb branch side instead of the
Higgs branch side. It is the Higgs branch ring object associated with
the right hand side of \eqref{eq:fund} after exchanging $G$ and
$G^\vee$. Hence it is
\begin{equation*}
    \cA_{S_G(g,b)} = \boxtimes_{k=1}^b (\cA_R)_k \boxtimes
    \boxtimes_{l=1}^g \on{Sym}(\fg^\vee \oplus(\fg^\vee)^*)_l
    \tslash \Delta_{G^\vee},
\end{equation*}
where $\on{Sym}(\fg^\vee \oplus(\fg^\vee)^*)$ is considered as a ring object
on the affine Grassmannian $\Gr_{\{e\}}$ for the trivial group $\{e\}$
with the diagonal $G^\vee$-action. Therefore $\cA_{S_G(g,b)}$ is a
ring object in $D_{G^b}(\Gr_{G^b})$.
Since \ref{2:0 W} is a consequence of an upgraded equality in quantum
field theories (due to Gaiotto \cite{MR3006961}), we have the
corresponding property also for $\cA_{S_G(g,b)}$. It is nothing but
the following:

\begin{Proposition}\label{prop:d1}
\begin{equation*}
    p_* i_{\Delta^{b_1,b_2}}^! (\cA_{S_G(g_1,b_1)} \boxtimes \cA_{S_G(g_2,b_2)})
    = \cA_{S_G(g_1+g_2,b_1+b_2-2)},
\end{equation*}
where \textup{(a)}
$i_{\Delta^{b_1,b_2}}\colon \Gr_G^{b_1+b_2-1}\to \Gr_{G}^{b_1}\times
\Gr_G^{b_2}$
is the product of the evident map
$\Gr_{G}^{b_1+b_2-2}\xrightarrow{\cong}
\Gr_G^{b_1-1}\times\Gr_G^{b_2-1}$
and the diagonal embedding $\Gr_{G}\to (\Gr_{G})^2$ of the last factor
to the product of the $b_1$st and the $b_2$nd factors, and
\textup{(b)} $p\colon (\Gr_{G})^{b_1+b_2-1}\to (\Gr_{G})^{b_1+b_2-2}$
is the projection given by forgetting the last factor.
\end{Proposition}

\begin{proof}
Let us identify $\fg^\vee$ and $(\fg^\vee)^*$ by a non-degenerate
invariant form.
Let us consider
\[
  (T^*G^\vee)^b \times (\fg^\vee\times\fg^\vee)^g \tslash
\Delta^r_{G^\vee} = \underbrace{T^*G^\vee\times\cdots\times
  T^*G^\vee}_{\text{$b$ times}} \times
\underbrace{(\fg^\vee\times\fg^\vee) \times\cdots\times
  (\fg^\vee\times\fg^\vee)}_{\text{$g$ times}} \tslash
\Delta^r_{G^\vee},
\]
where $\Delta^r_{G^\vee}$ is the diagonal subgroup acting on
$T^*G^\vee\times\cdots\times T^*G^\vee$ by the right action, and on
$(\fg^\vee\times\fg^\vee)\times\cdots \times (\fg^\vee\times\fg^\vee)$
by the adjoint action. The dg-version of its coordinate ring is
$\Psi^{-1}\cA_{S_G(g,b)}$.
It is isomorphic to
$(T^*G^\vee)^{b-1}\times (\fg^\vee\times\fg^\vee)^g$ by
\begin{multline*}
    \left[ g_1,\xi_1, \dots, g_b,\xi_b, \eta_1,\zeta_1,\dots,
    \eta_g,\zeta_g \right] \mapsto 
    (g_1',\xi'_1,\dots,g_{b-1}',\xi'_{b-1},\eta'_1,\zeta'_1,\dots,
    \eta'_g,\zeta'_g) 
\\
    g'_k = g_k g_b^{-1}, \quad
    \xi'_k = \on{Ad}_{g_b}\xi_k \quad (k=1,\dots,{b-1}), \quad
    \eta'_l = \on{Ad}_{g_b} \eta_l, \quad
    \zeta'_l = \on{Ad}_{g_b} \zeta_l \quad (l=1,\dots,g).
\end{multline*}
\begin{NB}
    Here $\Delta^r_{G^\vee}$ acts by
    $\left[ g_1 h^{-1} ,\on{Ad}_h \xi_1, \dots, g_b h^{-1},\on{Ad}_h \xi_b, \on{Ad}_h \eta_l, \on{Ad}_h \zeta_l \right]$, hence this is well-defined.
\end{NB}%
The left $(G^\vee)^b$-action on $(T^*G^\vee)^b\times (\fg\times\fg)^g \tslash\Delta^r_{G^\vee}$ is
identified with the left $G^\vee$-action (and the trivial action
on $(\fg^\vee\times\fg^\vee)^g$) for the first $(b-1)$ factors, but
the last factor acts by
\begin{equation*}
    (T^*G^\vee)^{b-1}\ni
    (g_1',\xi_1',\dots,g'_{b-1},\xi_{b-1}') 
    \mapsto (g_1' h_b^{-1}, \on{Ad}_{h_b}\xi'_1,
    \dots, g'_{b-1} h_b^{-1}, \on{Ad}_{h_b}\xi'_{b-1},
    \on{Ad}_{h_b} \eta'_l, \on{Ad}_{h_b} \zeta'_l)
\end{equation*}
for $h_b\in G$. The corresponding moment map is also the standard one
for the first $(b-1)$-factors, and the last one is
\begin{equation*}
    \begin{NB}
        \on{Ad}_{g_b}\xi_b
        = - \on{Ad}_{g_b}\xi_1 - \cdots - \on{Ad}_{g_b}\xi_{b-1} 
        - \sum_{l=1}^g \on{Ad}_{g_b}([\eta_l,\zeta_l]) =
    \end{NB}%
    - \xi_1' - \dots - \xi_{b-1}' - \sum_{l=1}^g [\eta'_l, \zeta'_l].
\end{equation*}
This is nothing but the restriction to the diagonal subgroup of the
product of the right action and the adjoint action.

Now by~\ref{0.1}(b) and~\ref{rem:formality},
$p_*i_\Delta^!(\cA_{S_G(g_1,b_1)}\boxtimes\cA_{S_g(g_2,b_2)})$ goes to
\begin{equation*}
    \on{Ext}^*_{D^{G^\vee}(\on{Sym}(\gvee))}(\on{Sym}^{[]}(\fg^\vee),
    {\mathfrak C}_{G^\vee}\Psi^{-1}\cA_{S_G(g_1,b_1)}
    \otimes_{\on{Sym}^{[]}(\fg^\vee)}
    \Psi^{-1}\cA_{S_G(g_2,b_2)}),
\end{equation*}
under the derived Satake equivalence. Here $\fg^\vee$ is the Lie
algebra of the diagonal subgroup in the product of last factors of
$(G^\vee)^{b_1}$ and $(G^\vee)^{b_2}$.
It is equal to
\begin{equation*}
    \CC[(T^*G^\vee)^{b_1+b_2-2}\times (\fg^\vee\times\fg^\vee)^{g_1+g_2}
    \tslash \Delta^r_{G^\vee}]^{[]}
\end{equation*}
by the above computation. This is nothing but $\Psi^{-1}\cA_{S_G(g_1+g_2,b_1+b_2-2)}$.
\end{proof}

\begin{NB}
    A key of this result and \ref{2:0} can be summarized by the identity
    \begin{equation*}
        G_1\times G_2 \curvearrowright 
        T^*G^\vee\times T^*G^\vee \tslash \Delta^r_{G^\vee}
        = G_1 \curvearrowright T^*G^\vee \curvearrowleft G_2,
    \end{equation*}
    where $G_1$, $G_2$ are copies of $G^\vee$. Here $\curvearrowright$
    is the left action, and $\curvearrowleft$ is the right action.
\end{NB}

\begin{NB}
We have $\kappa^r(T^*G)=\{(g_1,\xi_1) : \on{Ad}_{g_1}\xi_1\in\Sigma\}$. 
The residual $G$-action is $g(g_1,\xi_1)=(g_1g^{-1},\on{Ad}_g\xi_1)$,
and the moment map is $\mu_l(g_1,\xi_1)=\xi_1$.
We have $\kappa^l(T^*G)=\{(g_2,\xi_2) : \xi_2\in\Sigma\}$.
The residual $G$-action is $g(g_2,\xi_2)=(gg_2,\xi_2)$, and the moment
map is $\mu_r(g_2,\xi_2)=\on{Ad}_{g_2}\xi_2$.
Hence $(\kappa^r(T^*G)\times\kappa^l(T^*G))\tslash\Delta_G=
\{(g_1,\xi_1,g_2,\xi_2) : \on{Ad}_{g_1}\xi_1\in\Sigma\ni\xi_2,\ 
\xi_1=\on{Ad}_{g_2}\xi_2\}/\Delta_G$.
Choosing $g\in\Delta_G$ to be $g_1$ we see the latter quotient coincides with
$\{(\xi'_1:=\on{Ad}_{g_1}\xi_1,\ g':=g_1g_2,\ \xi'_2:=\xi_2) : 
\xi'_1=\on{Ad}_{g'}\xi'_2\}$. Recall that $\Sigma$ is a slice, hence it 
follows $\xi'_1=\xi'_2=\on{Ad}_{g'}\xi'_2$, i.e.\ we obtain the universal
centralizer, i.e.\ $\kappa^r\kappa^l(T^*G)$ as desired.
\end{NB}


\begin{NB}

  \begin{quote}
    HN comments out a short review of 3d version of
    \cite{Tach-review}. It is kept as a file \verb+cat_tmp.tex+, hence
    will appear when we delete \verb+%+ of \verb+\input{cat_tmp}+
    several lines below.
  \end{quote}
  
\renewenvironment{NB}{
\color{blue}{\bf NB2}. \footnotesize
}{}
\renewenvironment{NB2}{
\color{purple}{\bf NB3}. \footnotesize
}{}
\end{NB}

\subsection{Hamiltonian reduction of \texorpdfstring{$T^*G$}{T^*G} with respect to the adjoint action}
\label{subsec:Loseu}

Let $G$ be a connected reductive group over $\BC$ and let $\fg$ be its Lie
algebra.
Consider the adjoint action of $G$ on itself and the induced Hamiltonian
action of the cotangent
bundle $T^*G$. Using a non-degenerate invariant form we identify $\fg$ with
$\fg^*$, this gives rise to
an identification $T^*G\cong G\times \fg$ (with the diagonal action of $G$).
The moment map
$\mu\colon T^*G\rightarrow \fg$ becomes $(g,x)\mapsto \on{Ad}_gx-x$. It follows that
$\mu^{-1}(0)=\{(g,x)|
\on{Ad}_gx=x\}$. Consider the Hamiltonian reduction $\mu^{-1}(0)\dslash G$ with
the reduced scheme structure.

Now consider $T^*T=T\times \ft$. We have a natural morphism of varieties
$\psi\colon (T\times \ft)/W\rightarrow
\mu^{-1}(0)\dslash G$ induced from $T\times \ft\hookrightarrow G\times \fg$.

\begin{Proposition}[I.~Losev]
\label{Prop:iso}
The morphism $\psi\colon  (T\times \ft)/W\rightarrow \mu^{-1}(0)\dslash G$ is an isomorphism of varieties.
\end{Proposition}

We can consider the analogous situation for the Lie algebras: we have
the moment map $\underline{\mu}\colon \fg^2\rightarrow \fg, (x,y)\mapsto [x,y]$.
In this situation,
a direct analog of~\ref{Prop:iso} is known thanks to~\cite{MR1446576}:
we have $\ft^2/W\iso\underline{\mu}^{-1}(0)\dslash G$.
In particular,
the variety $\underline{\mu}^{-1}(0)\dslash G$ is normal.

\begin{proof} The proof is in several steps.

{\it Step 1}. Let us show that $\psi$ is a bijection. The variety $\mu^{-1}(0)\dslash G$
parameterizes the closed $G$-orbits in $\mu^{-1}(0)=\{(g,x)| \on{Ad}_gx=x\}$. It follows easily
from the Hilbert-Mumford theorem that the orbit $G(g,x)$ is closed if and only if both
$g,x$ are semisimple. Also any $G$-orbit of semisimple commuting elements $(g,x)$
intersects $T\times \ft$ in a single $W$-orbit. The claim in the beginning of the step follows.

{\it Step 2}. We claim that it is enough to show that $\mu^{-1}(0)\dslash G$ is a normal
algebraic variety. Indeed, any bijective morphism to a normal variety is an isomorphism.
The normality of $\mu^{-1}(0)\dslash G$ will follow if we check that the formal neighborhood
of every point in $\mu^{-1}(0)\dslash G$ is normal. In order to do that we will describe
the formal neighborhood using a version of a slice theorem for Hamiltonian actions
on affine symplectic varieties, see, e.g., \cite{slice} (in that paper complex analytic
neighborhoods were considered, but the result carries over to the formal neighborhood
in a straightforward way).

{\it Step 3}. Let us recall the slice theorem. Let $Y$ be a smooth affine symplectic
variety equipped with a Hamiltonian action (with moment map $\mu$) of a reductive group
$G$ and let $y\in Y$ be a point with closed $G$-orbit. Let us write $H$ for the stabilizer
of $y$ in $G$. The normal space $T_yY/T_y Gy$ can be decomposed as $\fh^\perp\oplus V$,
where $V$ is a symplectic vector space with $H$ acting on $V$ by linear symplectomorphisms.
Then the formal neighborhood of $Gy$ in $Y$ is $G$-equivariantly isomorphic to the formal neighborhood of
the zero section in $G\times^H (\fh^\perp\oplus V)$. An isomorphism can be chosen
to be compatible with symplectic forms and moment maps. In particular, the moment
map $\mu'\colon  G\times^H (\fh^\perp\oplus V)\rightarrow \fg$ is the unique $G$-equivariant
map that on the fiber $\fh^{\perp}\oplus V$ over $1H$ is given by $\mu(z,v)=z+\mu_H(v)$, where $\mu_H\colon V\rightarrow \fh$
is the standard moment map for a linear symplectic action. In particular, we see that
the formal neighborhood of $Gy$ in $\mu^{-1}(0)\dslash G$ is isomorphic to the formal
neighborhood of $0$ in $\mu_H^{-1}(0)\dslash H$.

{\it Step 4}. Consider $Y=G\times \fg$ and $y=(g,0)$ for a semisimple element $g\in G$.
We can identify $T_y Gy$ with $\{\on{Ad}_gx-x| x\in \fg\}$ so $T_y Y/T_y Gy=\fh\oplus \fg$
and $H=Z_G(g)$. We conclude that $V\cong \fh\oplus \fh$ with diagonal action of $H$.
By~\cite{MR1446576}, we see that $\mu^{-1}_H(0)\dslash H$ is normal.
So the formal neighborhood of $G(g,0)$ in $\mu^{-1}(0)\dslash G$ is normal, equivalently,
$G(g,0)$ is a normal  point.

{\it Step 5}. To finish the proof note that $\BC^\times$ acts on $\mu^{-1}(0)\dslash G$,
the action is induced from the dilation action on $\fg$. This action contracts
$\mu^{-1}(0)\dslash G$ to $G\dslash G$. Since the points in the latter are normal,
$\mu^{-1}(0)\dslash G$ is a normal algebraic variety.
\end{proof}


\appendix

\section{Group action on the Coulomb branch}\label{sec:group_action}

In this section we give a proof of the expected property
\cite[\S4(iii)(d)]{2015arXiv150303676N}, using an idea of Namikawa
\cite{2016arXiv160306105N}. See also
\cite{2017arXiv170103825C}.\footnote{The third named author thanks
  Amihay Hanany for his explanation of the idea to use the Lie algebra
  of degree $1$ subspace.}

\subsection{The degree $1$ subspace}

Let us consider the $\CC^\times$-action on the Coulomb branch $\cM_C$
given by $\Delta(\lambda)$ as in \remref{discrepancy}(2). Recall that
the $\CC^\times$-action is shifted from one given by the homological
degree by a hamiltonian action. In particular, the Poisson bracket
$\{\ ,\ \}$ is of degree $-1$ as in \subsecref{subsec:Cartan}.

Consider the subspace $\fl$ of degree $1$ elements in $\CC[\cM_C]$. It
forms a Lie subalgebra under the Poisson bracket $\{\ ,\ \}$. Then
$\CC[\cM_C]$ can be considered as a representation of this Lie algebra
$\fl$ by the Poisson bracket: $\{ f,\bullet\}$ ($f\in\fl$). If we
restrict it to the regular locus of $\cM_C$, it is nothing but the
hamiltonian vector field $H_f$ associated with $f\in\fl$ by the
symplectic form. The action preserves the Poisson bracket
\begin{NB}
  by Jacobi identity
\end{NB}%
and the degree.
\begin{NB}
  as $1 + (-1) = 0$
\end{NB}%
In more geometric term, $H_f$ preserves the symplectic form and
commutes with the $\CC^\times$-action.

\begin{Remark}
  Namikawa \cite{2016arXiv160306105N} shows that $\cM_C$ is the
  closure of a nilpotent orbit if $\CC[\cM_C]$ is generated by $\fl$,
  under the assumption that $\cM_C$ has symplectic singularities. In
  this case $\fl$ is the Lie algebra of
  $\operatorname{Aut}^{\CC^\times}(\cM_C,\omega)$, the group of
  $\CC^\times$-equivariant symplectic automorphisms of $\cM_C$. We
  conjecture that this statement is true for general $\cM_C$.
  Namikawa's argument works in much more general cases without the
  assumption that the coordinate ring is generated by
  $\fl$.\footnote{The third named author thanks Yoshinori Namikawa for
    explanation.} But we are not sure as we do not know $\cM_C$ has
  symplectic singularities, and the $\CC^\times$-action is not conical
  in general. These seem essential in Namikawa's argument.
\end{Remark}

\subsection{Balanced vertices in quiver gauge
  theories}\label{subsec:balanced}

Let us take a quiver $Q = (Q_0,Q_1)$ and two $Q_0$-graded vector
spaces $V = \bigoplus V_i$, $W = \bigoplus W_i$. We consider the
associated quiver gauge theory $(\GL(V),\bN)$ as in
\cite[\S2(iv)]{2015arXiv150303676N} and \secref{QGT}, i.e.,
\begin{equation*}
  \GL(V) = \prod_{i\in Q_0} \GL(V_i), \qquad
  \bN = \bigoplus_{h\in Q_1} \Hom(V_{\vout{h}}, V_{\vin{h}})\oplus
  \bigoplus_{i\in Q_0} \Hom(W_i, V_i).
\end{equation*}

In order to treat a group action on a line bundle in
\ref{subsec:flav-symm-group2} and \ref{subsec:line-bundles-via}, we
also consider a larger symmetry group
$\tilde G = \GL(V)\times \GL(W)/\CC^\times$ with
$\tilde G/\GL(V) = \PGL(W) = \prod_{i\in Q_0} \GL(W_i)/\CC^\times$,
where both $\CC^\times$ are diagonal scalar subgroups.

Recall $\CC[\cM_C]$ has a grading parametrized by $\pi_1(\GL(V))$
(\secref{sec:grading}). In our situation, we have
$\pi_1(\GL(V)) = \bigoplus \pi_1(\GL(V_i)) \cong \ZZ^{Q_0}$.
For the larger symmetry group, we have
$\pi_1(\tilde G) \cong \ZZ^{Q_0}\oplus \ZZ^{\{i \in Q_0| W_i\neq
  0\}}/\ZZ$, where $\ZZ$ is embedded into
$\ZZ^{Q_0}\oplus \ZZ^{\{i \in Q_0| W_i\neq 0\}}$ by
$1\mapsto (\dim V_i, \dim W_i)$.
We have the corresponding action of
$\pi_1(\GL(V))^\wedge \cong (\CC^\times)^{Q_0}$ on $\cM_C$ and
$\pi_1(\tilde G)^\wedge \cong (\CC^\times)^{\# Q_0 + \#\{i \mid W_i\neq
    0\} - 1}$
(modulo finite groups) on a line bundle in \ref{subsec:flav-symm-group2},
\ref{subsec:line-bundles-via}. Here $(\ )^\wedge$ is the Pontryagin
dual.
We will not be interested in the action of finite groups, hence we replace $\pi_1(\tilde G)$ by its free part $\pi_1(\tilde G)_\fr$ hereafter.
\begin{NB}
    We have a natural embedding $\pi_1(G)\subset\pi_1(\tilde G)_\fr$.
\end{NB}%
We have the corresponding space $H^2_{\GL(V)}(\mathrm{pt})$ (or
$H^2_{\tilde G}(\mathrm{pt})$ for $\pi_1(\tilde G)^\wedge$), which
consists of degree $1$ elements. 

\begin{NB}
  Since a coweight $\lambda$ of the flavor symmetry group $G_F$ is
  fixed when we consider a line bundle, $\pi_1(\tilde G)_\fr^\vee$ is
  too big unless $W_i$ is nonzero except a single $i$. I do not see
  how to fix this problem.
\end{NB}%

A vertex $i$ is \emph{balanced} if there is no edge loop at $i$, and
the corresponding coweight $\mu$ satisfies
$\langle\mu,\alpha_i\rangle = 0$, i.e.,
$2\dim V_i = \dim W_i + \sum_j a_{ij} \dim V_j$, where $a_{ij}$ is the
number of edges (either in $Q_1$ or its opposite) between $i$ and $j$.
We consider the subquiver $Q^{\mathrm{bal}}$ of $Q$ consisting of
balanced vertices and edges among them. By a well-known result (e.g.\
\cite[Th.~4.3]{Kac}), $Q^{\mathrm{bal}}$ is a union of finite $ADE$
quivers, unless $Q^{\mathrm{bal}}$ is a union of connected components
of $Q$ of affine type with $W=0$ on them. We suppose it is not the
latter case.

We consider elements $E_i^{(1)}$, $F_i^{(1)}$, $H_i^{(1)}$ from the
shifted Yangian considered in \secref{sec:quantization}. Looking at
relations therein, we see that their Poisson brackets satisfy the
relations of $\algsl_2$, as $H_i^{(p)} = 0$ ($p < 0$), $H_i^{(0)} = 1$ as
$\langle\mu,\alpha_i\rangle = 0$.
Moreover if both $i$ and $j$ are balanced, $E_i^{(1)}$, $F_i^{(1)}$,
$H_i^{(1)}$, $E_j^{(1)}$, $F_j^{(1)}$, $H_j^{(1)}$ satisfy the
relations of $\algsl_3$ or $\algsl_2\oplus\algsl_2$ according to
whether $i$ and $j$ are connected in the quiver or not. 
\begin{NB}
    In \ref{sec:quantization} we assumed that the number of edges
    between $i$ and $j$ are at most $1$. This is automatically
    satisfied under our assumption.
\end{NB}%
We then have the corresponding semisimple Lie algebra
$\fl^{\mathrm{bal}}_{\on{ss}}$ generated by $E_i^{(1)}$, $F_i^{(1)}$,
$H_i^{(1)}$ ($i\in Q_0^{\mathrm{bal}}$). 
\begin{NB}
Let $L^{\mathrm{bal}}$ be the
simply-connected semisimple Lie group whose Lie algebra is
$\fl^{\mathrm{bal}}$.
\end{NB}%

From the definition of $H_i(z)$ in \ref{sec:quantization}, $H_i^{(1)}$
is the coefficient of $z^{-1}$ in
$Z_i(z) \prod_{j} W_j(z)^{a_{ij}} / W_i(z)^2$, where
$Z_i(z) = \prod_{k:i_k=i} (z - z_k)$,
$W_i(z) = \prod_r (z - w_{i,r})$. (Note that we set $\hbar = 0$.) Here
$z_k$, $w_{i,r}$ are equivariant variables for $\prod \GL(W_i)$ and
$\GL(V_i)$ respectively. Therefore $H_i^{(1)}$ is
$-\sum_{k:i_k=i} z_k - \sum_{j,s} a_{ij} w_{j,s} + 2\sum_r w_{i,r}$.
This is nothing but $-c_1(W_i) - \sum_{j} a_{ij} c_1(V_j) + 2c_1(V_i)$ if
we regard $V_i$, $W_i$ as representations of $\GL(V_i)$, $\GL(W_i)$
respectively.
Now we apply \lemref{lem:former_claim}. The Poisson bracket
$\{ H_i^{(1)}, \bullet\}$ is given by
$\gamma^W_i + \sum_{j}a_{ij} \gamma_j - 2\gamma_i$ on the component
with grading
$\gamma = (\gamma_j,\gamma^W_j)\in\ZZ^{Q_0}\oplus\ZZ^{Q_0}/\ZZ$.
\begin{NB}
    Note that this is well defined as $i$ is balanced:
    $\dim W_i + \sum_j a_{ij} \dim V_j - 2\dim V_i = 0$.
\end{NB}%
\begin{NB}
This was originally written:

It means that $\gamma$, regarded as a weight for $\fl^{\mathrm{bal}}$,
is equal to $-\sum_{j\in Q_0} \gamma_j \alpha_j$, where
$\langle\alpha_j, H_i^{(1)}\rangle = 2\delta_{ij} - a_{ij}$ as
usual.
\end{NB}%
In particular, the action of $H_i^{(1)}$ is lifted to
$\pi_1(\GL(V))^\wedge \cong (\CC^\times)^{Q_0}$ for $\cM_C$, 
and to $\pi_1(\tilde G)_\fr^\wedge$ for line bundles.

\begin{Lemma}\label{lem:balanc-vert-quiv}
  $E_i^{(1)}$, $F_i^{(1)}$, $H_i^{(1)}$
  \textup($i\in Q_0^{\mathrm{bal}}$\textup) are of degree $1$.
\end{Lemma}

Therefore the Lie algebra $\fl$ in the previous subsection at least
contains the semisimple Lie algebra $\fl^{\mathrm{bal}}_{\on{ss}}$
above. 

Note that $H^{(1)}_i$ is in $H^2_{\GL(V)}(\mathrm{pt})$ or
$H^2_{\tilde G}(\mathrm{pt})$ when we consider the larger group
$\tilde G$. Let $\fl^{\mathrm{bal}}$ (resp.\
$\tilde{\fl}^{\mathrm{bal}}$) be the Lie subalgebra of $\fl$ generated
by $\fl^{\mathrm{bal}}_{\mathrm{ss}}$ and $H^2_{\GL(V)}(\mathrm{pt})$
(resp.\ $H^2_{\tilde G}(\mathrm{pt})$).

\begin{proof}
We have already checked the assertion for $H_i^{(1)}$.

Looking at the definition of the homomorphism in
\thmref{th:PhiBar}, we see that $E_i^{(1)}$, $F_i^{(1)}$ are
fundamental classes $[\cR_{\pm\varpi_{i,1}}]$ up to sign. By the
formula for $\Delta(\pm\varpi_{i,1})$ in \eqref{eq:Deltadeg}, their degree is
$1$ as $i$ is a balanced vertex.
\begin{NB}
\begin{equation*}
  - (\dim V_i - 1) + \frac12 \left(\dim W_i + \sum_j a_{ij}\dim V_j\right)
  = 1.
\end{equation*}
\end{NB}%
\begin{NB}
Note also that $H_i^{(1)}$ is a cohomological degree $2$ element in
$H^*_{\GL(V)}(\mathrm{pt})$, hence it is of degree $1$ with respect to
$\Delta$. It can be also checked from the commutator relation
$\{ E_i^{(1)}, F_i^{(1)}\} = H_i^{(1)}$.
\end{NB}%
\end{proof}

Recall $\{ H_i^{(1)}, \bullet\}$ defines an element of $\pi_1(\tilde G)_\fr^\vee$ 
by $\gamma\mapsto \gamma_i^W + \sum_j a_{ij} \gamma_j - 2\gamma_i$. More
generally $H_\alpha$ corresponding to a root $\alpha$ of
$\fl^{\mathrm{bal}}_{\on{ss}}$ defines an element of $\pi_1(\tilde
G)_\fr^\vee$.
Thus we regard naturally coroots $\in R^{\mathrm{bal} \vee}$ of
$\fl^{\mathrm{bal}}_{\on{ss}}$ as elements in $\pi_1(\tilde G)_\fr^\vee$.
If we disregard the flavor symmetry, we consider the restriction to
$\gamma_j^W = 0$, hence we still have
$R^{\mathrm{bal} \vee}\subset\pi_1(G)^\vee$.

On the other hand,
\(
  \pi_1(\prod_{i\in Q^{\mathrm{bal}}_0}\GL(V_i))\cong \ZZ^{Q^{\mathrm{bal}}_0}
\)
is naturally identified with the root lattice of the Lie algebra
$\fl^{\mathrm{bal}}_{\on{ss}}$ by sending the $i$-th coordinate vector to the
$i$-th simple root $\alpha_i$. Thus we consider roots
$\in R^{\mathrm{bal}}$ as elements of
$\pi_1(\prod_{i\in Q^{\mathrm{bal}}_0}\GL(V_i))\subset
\pi_1(\GL(V))\subset\pi_1(\tilde G)_\fr$.

We regard $R^{\mathrm{bal}}\subset\pi_1(\GL(V))$,
$R^{\mathrm{bal}\vee}\subset \pi_1(\GL(V))^\vee$ as a root datum, and
consider the corresponding reductive group $L^{\mathrm{bal}}$. For
line bundles, we consider
$R^{\mathrm{bal}}\subset\pi_1(\tilde G)_\fr$,
$R^{\mathrm{bal}\vee}\subset \pi_1(\tilde G)_\fr^\vee$. We denote the
corresponding reductive group by ${\tilde L}^{\mathrm{bal}}$.

\begin{NB}
\begin{Remark}\label{rem:adjointgroup}
  Since $j$ may not be in $Q_0^{\mathrm{bal}}$, the restriction of
  $\sum_{j\in Q_0} \gamma_j \alpha_j$ to the Cartan subalgebra
  $\bigoplus \CC H^{(1)}_i$ of $\fl^{\mathrm{bal}}$ may not be in a
  root lattice in general.
  However in the example in \ref{subsec:ABG}, all vertices are
  balanced, hence it is in a root lattice. Therefore the action of
  $L^{\mathrm{bal}}$ on $\CC[\cM_C]$ given in the next proposition
  factors through the adjoint group of $L^{\mathrm{bal}}$. In fact
  $\cM_C$ is the nilpotent cone and the action is the standard one in
  this example, this is obviously true.
\end{Remark}
\end{NB}%

\begin{Proposition}
\label{prop:Integrable}
The actions of $\{ E_i^{(1)}, \bullet\} $, $\{ F_i^{(1)}, \bullet\}$
are locally nilpotent. Hence the action of $\fl^{\mathrm{bal}}$
\textup(resp.\ ${\tilde\fl}^{\mathrm{bal}}$\textup) is lifted to
$L^{\mathrm{bal}}$ \textup(resp.\ $\tilde L^{\mathrm{bal}}$\textup).
\end{Proposition}

\begin{proof}
  When the $\CC^\times$-action on $\cM_C$ is conical, this is clear as
  subspaces of $\CC[\cM_C]$ with given degree are finite dimensional,
  and $E_i^{(1)}$, $F_i^{(1)}$ preserve them. In order to deal with
  general cases, we modify the argument.

  Consider a closed subvariety $\cR_{\le\lambda}$ as in
  \subsecref{triples}. Since $\GL(V) = \prod_{j\in Q_0} \GL(V_j)$, we can
  modify it by imposing the constraint at $j\neq i$, but not on
  $i$. Let us denote the resulted closed subvariety by
  $\cR_{\preceq\la}$. It is still true that $H^{\GL(V)_{\cO}}_*(\cR)$ is the
  limit of $H^{\GL(V)_{\cO}}_*(\cR_{\preceq\la})$. Operators
  $\{ E_i^{(1)}, \bullet\} $, $\{ F_i^{(1)}, \bullet\}$ are
  well-defined on $H^{\GL(V)_{\cO}}_*(\cR_{\preceq\la})$, as we do not impose
  the constraint at $i$.

  Let $H^{\GL(V)_{\cO}}_*(\cR_{\preceq\la})[d]$ denote the subspace of
  $H^{\GL(V)_{\cO}}_*(\cR_{\preceq\la})$ of degree $d$ elements. It is enough
  to show that it is finite dimensional, as $\{ E_i^{(1)}, \bullet\}$,
  $\{F_i^{(1)},\bullet\}$ preserve this subspace.

  Suppose that an element in $H^{\GL(V)_{\cO}}_*(\cR_{\preceq\la})[d]$ is
  contained in $H^{\GL(V)_{\cO}}_*(\cR_{\le\mu})[d]$. If we decompose $\mu$
  as $(\mu^j)$ according to $j\in Q_0$, the component $\mu^j$ with
  $j\neq i$ is less than or equal to the component $\la^j$ of $\la$ by
  the definition of $\preceq$.

  In order to bound the remaining component $\mu^i$, Let us look at
  the formula of $\Delta(\mu)$:
  \begin{equation*}
    \Delta(\mu) = - \sum_{\alpha\in\Delta^+} |\langle\alpha,\mu\rangle|
    + \frac12 \sum_{\chi} |\langle\chi,\mu\rangle|\dim \bN(\chi).
  \end{equation*}
  We have $\Delta(\mu)\le d$ by our assumption. Let us look at terms
  involving $\mu^i$:
  \begin{equation}\label{eq:monobound}
    - \sum_{a<b} |\mu^i_a - \mu^i_b| +
    \frac12\left( \sum_j a_{ij} \sum_{a,b} | \mu^i_a - \mu^j_b|
      + \dim W_i \sum_a |\mu^i_a|\right),
  \end{equation}
  where we write $\mu^j = (\mu^j_1,\mu^j_2,\dots)$. This is bounded by
  a constant from above, as we have bounds on $\mu^j_a$ ($j\neq i$).
  Since $\mu^j_b$ is bounded, the middle term can be replaced by
  $\sum_j a_{ij} \dim V_j \sum_a |\mu^i_a|$. Now by the assumption
  $2\dim V_i = \dim W_i + \sum a_{ij}\dim V_j$, the first term can be
  absorbed in the middle and last term, so that we still have a bound
  on $\sum_a |\mu^i_a|$.
  \begin{NB}
    For example, suppose $\dim V_i = 3$. We have
    $|\mu^i_1 - \mu^i_2| + |\mu^i_2 - \mu^i_3| + |\mu^i_3 - \mu^i_1|
    \le 2 (|\mu^i_1|+|\mu^i_2|+|\mu^i_3|)$. On the other hand, we
    get $3(|\mu^i_1|+|\mu^i_2|+|\mu^i_3|)$ from the second term.
  \end{NB}%
  Thus $\mu$ is bounded by a constant depending on $\lambda$ and $d$. Hence
  $H^{\GL(V)_{\cO}}_*(\cR_{\preceq\la})[d]$ is finite dimensional.

  This argument works also for the case of $\tilde{L}^{\mathrm{bal}}$.
\end{proof}

Note that the comoment map of the $L^{\mathrm{bal}}$-action on $\cM_C$
is the natural homomorphism
\begin{equation*}
  \CC[\fl^{\mathrm{bal}*}] = \operatorname{Sym}(\fl^{\mathrm{bal}})
  \to \CC[\cM_C]
\end{equation*}
by the definition of the action.

\begin{Example}[cf.\ \ref{Cartan_grading}]\label{ex:stab}
    Consider a framed quiver gauge theory of type $ADE$. Let us define
    two coweights $\lambda = \sum_{i\in Q_0} \dim W_i \omega_i$,
    $\mu = \lambda - \sum_{i\in Q_0} \dim V_i\alpha_i$. Then the
    Coulomb branch is the generalized slice $\oW^{\lambda^*}_{\mu^*}$
    for the adjoint group $G$ of type $Q$, where
    $\lambda^* = - w_0(\lambda)$, $\mu^* = -w_0(\mu)$
    (\ref{Coulomb_quivar}). Then the group $L^{\mathrm{bal}}$, acting
    on $\cM_C$ is $\operatorname{Stab}_G(\mu^*)$, as $\pi_1(\GL(V))$
    is the weight lattice of $G$. The action is the standard one, at
    least when $\mu$ is dominant. The following argument is explained
    to the authors by Joel Kamnitzer:

    First consider the case $\mu=0$. Then $\oW^{\lambda^*}_{0}$ is the
    intersection of $\overline{\Gr}_{G}^{\la}$ and
    $\Gr_{G,0} = G_1[[t^{-1}]]$, where $G_1[[t^{-1}]]$ is the first
    congruence subgroup of $G[[t^{-1}]]$ as in \cite{kwy}. Then the
    assertion follows from a computation of Poisson brackets on
    $\CC[G_1[[t^{-1}]]]$ in \cite[Prop.~2.13]{kwy}. For a general
    dominant $\mu$, we replace $\Gr_{G,0}$ by $\Gr_{G,\mu^*}$ the
    orbit of $G_1[[t^{-1}]]$ through $\mu^*$. But $\CC[\Gr_{G,\mu^*}]$
    is a Poisson subalgebra of $\CC[\Gr_{G,0}]$ preserved by the
    action of $\operatorname{Stab}_G(\mu^*)$ (see \cite[a paragraph
    before Lemma~2.19]{kwy}). Hence the assertion follows from the
    $\mu=0$ case.
\end{Example}

\begin{Example}\label{ex:psln}
    Let us consider the quiver gauge theory in \ref{subsec:ABG}. All
    vertices are balanced in this case. We have
    \begin{equation*}
        \pi_1(\tilde G) \cong 
        \underbrace{\ZZ\oplus\cdots\oplus\ZZ}_{\text{$N$ times}}/
        (N,N-1,\cdots,2,1)\ZZ \cong \ZZ^{N-1},
    \end{equation*}
    where the isomorphism is given by
    $[\lambda_1,\dots,\lambda_N] \mapsto
    (\lambda_1-N\lambda_N,\dots,\lambda_{N-1}-2\lambda_N)$.
    We have an exact sequence
    \begin{equation*}
    0 \to \pi_1(G)\cong\ZZ^{N-1} \to \pi_1(\tilde G) \cong \ZZ^{N-1}
    \to \pi_1(\PGL(N)) \cong \ZZ/N\ZZ \to 0,
  \end{equation*}
  where the first inclusion is given by $\lambda_1 = 0$, and the last
  projection is $[\lambda_1,\dots,\lambda_N]\mapsto \lambda_1\bmod N$.
  It is clear that $\pi_1(G)$ is the weight lattice of $\PGL(N)$,
  while $\pi_1(\tilde G)$ is that of $\SL(N)$. Therefore
  $L^{\mathrm{bal}} = \PGL(N)$, but $\tilde L^{\mathrm{bal}} = \SL(N)$.
\end{Example}

\begin{Remark}\label{rem:deghalf}
    Let us consider the case $\langle \mu,\alpha_i\rangle = -1$
    instead of $0$. Then $H_i^{(p)} = 0$ ($p\le 0$), $H_i^{(1)} =
    1$. Thus we have $\{ E_i^{(1)}, F_i^{(1)}\} = 1$. (Note also
    $\Delta(\pm\varpi_{i,1}) = 1/2$ by \eqref{eq:Deltadeg}.)
    Looking at the argument in the proof of \ref{prop:Integrable}, we
    see that we only need a bound $\dim V_i - 1 \le \dim W_i + \sum
    a_{ij} \dim V_j$
    \begin{NB}
        $= 2\dim V_i - 1$
    \end{NB}%
    to derive a bound on $\sum_a |\mu^i_a|$ from
    \eqref{eq:monobound}. In particular, the proof of
    \ref{prop:Integrable} works in the case
    $\langle\mu,\alpha_i\rangle = -1$, hence $\{ E_i^{(1)},\bullet\}$,
    $\{F_i^{(1)},\bullet\}$ are locally nilpotent, and the
    corresponding hamiltonian vector fields $H_{E_i^{(1)}}$,
    $H_{F_i^{(1)}}$ are integrable. Moreover $[H_{E_i^{(1)}},
    H_{F_i^{(1)}}] = 0$ as $\{ E_i^{(1)}, F_i^{(1)}\} = 1$. Therefore
    we have an action of $\BG_a^2$.
    Let $\varPhi = (F_i^{(1)}, -E_i^{(1)})\colon \cM_C\to\BA^2$. Then
    $\varPhi$ is $\BG_a^2$-equivariant, and the action map
    $\BG_a^2\times\varPhi^{-1}(0)\cong \cM_C$ is an isomorphism.
    \begin{NB}
        The inverse is $\cM_C\ni x\mapsto
        (\varPhi(x),\varPhi(x)^{-1}\cdot x)$.
    \end{NB}%
    See also \cite[Th.(i)]{2016arXiv160306105N}.

    Suppose further that $j$ is balanced, i.e., $\langle
    \mu,\alpha_j\rangle = 0$. Commutation relations in
    \secref{sec:quantization} imply that $\{ F_j^{(1)}, E_i^{(1)}\} =
    0$, $\{ H_j^{(1)}, E_i^{(1)}\} =
    \begin{NB}
        -\frac{\alpha_i\cdot\alpha_j}2 (H_j^{(0)} E_i^{(1)} +
    E_i^{(1)} H_j^{(0)}) = 
    \end{NB}%
    - (\alpha_i\cdot\alpha_j)E_i^{(1)}$ and $\{ E_j^{(1)}, F_i^{(1)}\}
    = 0$, $\{H_j^{(1)}, F_i^{(1)}\}
    \begin{NB}
        = \frac{\alpha_i\cdot\alpha_j}2 (H_j^{(0)} F_i^{(1)} +
    F_i^{(1)} H_j^{(0)}) 
    \end{NB}%
    = (\alpha_i\cdot\alpha_j) F_i^{(1)}$.
    Thus $E_i^{(1)}$ (resp.\ $F_i^{(1)}$) is a lowest (resp.\ highest)
    weight vector of an $\algsl(2)_j = \langle E_j^{(1)}, F_j^{(1)},
    H_j^{(1)}\rangle$ module with the highest (resp.\ lowest) weight
    $\mp(\alpha_i\cdot\alpha_j)$.
\end{Remark}



\section{A global convolution diagram for the variety of triples}
\label{gusgusgus}

\centerline{By Gus Lonergan}

\bigskip

The aim of this appendix is to give another proof of the commutativity of the Coulomb branch by constructing a global convolution diagram for $\cR$. This is a direct generalization of the traditional proof of the case $\bN=0$, which uses the Beilinson-Drinfeld global convolution diagram for $\Gr_G$.

\subsection{Preliminaries on arc-spaces and loop-spaces}\label{one}

\subsubsection{}In this section, we recall certain standard constructions and facts of \cite[Chapters~4-5]{Beilinson-Drinfeld}.

\subsubsection{}\label{abbrv}Let $X$ be a smooth complex curve and let $S$ be a finite set. Given a commutative ring $R$ and an $R$-point $x$ of $X^S$, we denote the coordinates of $x$ by $x_s$ ($s\in S$), and write $\Delta_{S}(x)$ for the formal neighborhood of the union of the graphs of $x_s$ ($s\in S$). For notational simplicity, we frequently remove commas and braces from $S$, and also drop the part $(x)$, when it is clear which point we refer to. So for example the expression:
$$
\Delta_{\{1,2\}}(x)
$$
becomes:
$$
\Delta_{12}.
$$
\subsubsection{}Now fix an affine algebraic group $A$ over $\BC$. Consider the following functor from commutative rings to groups over $X^S$:
$$
A_{S}(R):=\{(x,f)|x\in X^S(R),f\colon \Delta_S\to A\}.
$$
Then $A_{S}$ is represented by the limit of a projective system of smooth affine group schemes over $X^S$:
$$
A_{S}=\lim\limits_{\longleftarrow} (\ldots\to (A_{S})_2\to (A_{S})_1)
$$
such that each transition morphism is a smooth homomorphism. In particular, $A_{S}$ is a formally smooth affine group scheme (of countably infinite type) over $X^S$, but this is not so important for us. Recall that in the definition of the Coulomb branch as a convolution algebra formal homological shifts such as
$$
[2\dim A(\cO)]
$$
appear (for $A=G,N$). Similarly, in the global situation formal homological shifts such as
$$
[2\dim A_{S}]
$$
will appear\footnote{Only for $S$ of cardinality $1$ or $2$; but it clarifies the picture and simplifies the exposition to work more generally at this point.}. For example, in the case where the underlying space is $A_{S}$, for each $d$ we have $\DC_{(A_{S})_d}\cong \BC_{(A_{S})_d}[2\dim (A_{S})_d]$. These complexes are compatible in the natural way under $!$-pullbacks along the transition morphisms. We thus consider $\DC_{A_{S}}$ as the formal homological shift
$$
\DC_{A_{S}}\cong \BC_{A_{S}}[2\dim A_{S}],
$$
where both sides are to be understood by evaluating on smooth quotients of $A_{S}$ and `piecing together' using $!$-pullbacks. Likewise we have a formal expression
$$
\DC_{A_{S}}[-2\dim A_{S}]\cong \BC_{A_{S}}
$$
where both sides are to be understood by evaluating on the smooth quotients $(A_{S})_d$ of $A_{S}$ and `piecing together' using $*$-pullbacks.

\subsubsection{}\label{pq}Let $\theta\colon S'\to S$ be a morphism of finite sets. It induces a map $X^{S}\to X^{S'}$. Given an $R$-point $x$ of $X^S$, this map determines an $R$-point $x'$ of $X^{S'}$, and an embedding $\Delta_{S'}(x')\to \Delta_{S}(x)$. Hence by restriction along this embedding we obtain a map
$$
p^\theta\colon A_{S}\to A_{S'}.
$$
This induces a homomorphism 
$$
q^{\theta}\colon A_{S}\to A_{S'}\times_{X^{S'}}X^S
$$
over the base $X^S$. If $\theta$ is surjective, then $q^\theta$ is an isomorphism. If $\theta$ is injective, then $q^\theta$ seems strange at first sight. For instance if $\theta'$ is a section of $\theta$ then $q^\theta$ is an isomorphism over the resulting copy of $X^{S'}\subset X^S$, whereas over a typical point of $X^S$, $q^\theta$ takes the form of a projection map
$$
A(\cO)^S\to A(\cO)^{S'}.
$$
However, this is misleading: $q^\theta$ is \emph{pro-smooth} when $\theta$ is injective. What we mean by this is that the projective systems of smooth affine group schemes over $X^S$ with smooth transition morphisms
$$
((A_{S})_d)_{d\in\bN}
$$
underlying $A_S$ may be taken, simultaneously for all $S$, to be compatible with all $q^\theta$, i.e. so that $q^\theta$ is the limit of a morphism
$$
(q^\theta_d\colon (A_{S})_d\to (A_{S'})_d\times_{X^{S'}}X^S)_{d\in\bN}
$$
of projective systems, where each map $q^\theta_d$ is a smooth homomorphism over $X^S$. Thus, it makes sense to write (and is true that):
$$
(q^\theta)^*\DC_{A_{S'}\times_{X^{S'}}(X^S)}[-2\dim A_{S'}-2(|S|-|S'|)] = \DC_{A_{S}}[-2\dim A_{S}],
$$
\emph{et cetera}, where the formula should be understood as a statement about complexes on smooth quotients over $X^S$, compatible under $*$-pullbacks.

\subsubsection{Example}Consider the case $S=\{1\}$. Then $A_1$ is a Zariski-locally trivial $A(\cO)$-bundle over $X$. Then, the formal homological shift $[2\dim A(\cO)]$ also makes sense in this context, and we have $[2\dim A(\cO)]=[2\dim A_1-2]$.

\subsubsection{Example}Consider for instance the case $A=\BC$ and $S=\{1,2\}$. Then $A_{12}$ should be thought of as a deformation of the first following projective system into the second:
$$
\begin{matrix}
(\BC[[t]]/t^{2d})_d & \rightsquigarrow & (\BC[[t]]/t^d\times \BC[[t]]/t^d)_d
\end{matrix}
$$
while $A_1$ should be thought of as a trivial deformation: 
$$
\begin{matrix}
(\BC[[t]]/t^{d})_d & \rightsquigarrow & (\BC[[t]]/t^d)_d
\end{matrix}
$$
and we have the morphism of deformations of projective systems:
$$
\begin{matrix}
(\BC[[t]]/t^{2d})_d & \rightsquigarrow & (\BC[[t]]/t^d\times \BC[[t]]/t^d)_d\\
\downarrow & & \downarrow\\
(\BC[[t]]/t^{d})_d & \rightsquigarrow & (\BC[[t]]/t^d)_d
\end{matrix}
$$
where the first downward arrow is the quotient map, and the second downward arrow is the projection map (to the first factor), both of which halve dimension in the $d^{th}$ approximation. It just happens that the limit of the first downward arrow is an isomorphism, while the limit of the second downward arrow is a non-trivial projection.

\subsubsection{}From now on, we assume $\theta$ is an injection, and identify $S'$ with its image under $\theta$. Now, in addition to the formal neighborhood $\Delta_S$ we have the punctured formal neighborhood
$$
\Delta_{S}^{S'}(x):=\Delta_{S}(x)-\cup_{s\in S'}x_s
$$
where in this formula we conflate the point $x_s$ with its graph. The general notational paradigm\footnote{Warning: this doesn't apply to $X$!} here is that subscripts determine discs and superscripts determine punctures. Consider the functor
$$
A_{S}^{S'}(R):=\{(x,f)|x\in X^S(R), f\colon \Delta_S^{S'}(x)\to A\}.
$$
Then $A_{S}^{S'}$ is represented by an ind-scheme, formally smooth over $X^S$. It is a group in ind-schemes (over $X^S$), but not an inductive limit of groups. Nonetheless, it is an \emph{ind-locally nice, reasonable} ind-scheme in the sense of \cite{drin}, meaning that it is a direct limit of closed embeddings with finitely generated ideals:
$$
(A_S^{S'})^1\to (A_S^{S'})^2\to\ldots
$$
of schemes over $X^S$, each of which is locally nice, meaning that Zariski-locally\footnote{In \cite{drin} this is relaxed to `Nisnevich-locally'.} it is the product of a finite-type scheme with an affine space (of countable dimension). We shall call such an ind-scheme \emph{reasonably nice}. The subgroup $A_S$ may be taken as the first subscheme $(A_S^{S'})^1$ in this inductive structure. The left- and right-regular actions of the subgroup $A_S$ preserve the inductive structure, meaning that each $(A_S^{S'})^c$ has an action on both sides by $A_S$ over $X^S$, even though it is not itself a group. Moreover the quotient $(A_S^{S'})^c/A_S$ is of finite-type over $X^S$, and flat, although generally quite singular. The result is that the quotient
$$
A_S^{S'}/A_S
$$
has the structure of ind-finite-type flat ind-scheme over $X^S$.

\begin{Lemma}\label{proper}\begin{enumerate}\item $A_S^{S'}/A_S$ is ind-projective if and only if $A$ is reductive.\
\item $A_S^{S'}/A_S$ is reduced if and only if $A$ has non no-trivial characters.\end{enumerate}\end{Lemma}

\begin{Remark}$A_S^S$ is the Beilinson-Drinfeld grassmannian (on $|S|$ points).\end{Remark}

\subsubsection{}For any chain of inclusions $S''\xrightarrow{\theta'} S'\xrightarrow{\theta} S$ we have natural maps
$$
p^{\theta}\colon A_S^{S''}\to A_{S'}^{S''},
$$
$$
q^{\theta}\colon A_S^{S''}\to A_{S'}^{S''}\times_{X^{S'}}X^S,
$$
defined as in subsection \ref{pq}. Then $q^{\theta}\colon A_S^{S''}\to A_{S'}^{S''}\times_{X^{S'}}X^S$ has as a subgroup $q^\theta\colon A_S\to A_{S'}\times_{X^{S'}}X^S$, and the resulting map
$$
A_S^{S''}/A_S\to (A_{S'}^{S''}/A_{S'})\times_{X^{S'}}X^S
$$
is an isomorphism. 

\subsubsection{Warning}Observe that $A_S^{S''}$ is an ind-$A_S$-torsor over the ind-scheme $(A_{S'}^{S''}/A_{S'})\times_{X^{S'}}X^S$, and the homomorphism $q^\theta\colon A_S^{S''}\to A_{S'}^{S''}\times_{X^{S'}}X^S$ is surjective. It is tempting therefore to try to view $A_S^{S''}$ as being in some sense a torsor over $A_{S'}^{S''}\times_{X^{S'}}X^S$ for some group $\ker q^\theta$. However, the kernel of the projective system
$$
((A_S)_d\to (A_{S'})_d\times_{X^{S'}}X^S)_{d\in\bN}
$$
of subsection \ref{pq} is not Mittag-Leffler. We are not sure how to overcome this issue, so do not attempt to take this point of view.


\subsection{Global convolution diagram for $\cR$}


\subsubsection{}For a finite set $S$, we put
$$
\cT_S^{S'}(R) = \{(x,\cE,f,\tilde{v})\}/\sim
$$
where $x\in X^S(R)$, $\cE$ is a principal $G$-bundle on $\Delta_S$, $f$ is a trivialization of $\cE$ on $\Delta_S^{S'}$, and $\tilde{v}$ is an $\bN$-section of $\cE$, taken up to equivalence. This is the same as the balanced product
$$
\cT_S^{S'}=G_S^{S'}\frac{\times_{X^S}}{^{G_S}}\bN_S.
$$
Thus, $\cT_S^{S'}$ is represented by a reasonably nice ind-scheme with an ind-pro-smooth map to the Beilinson-Drinfeld grassmannian $G_S^{S'}/G_S$. In particular it is formally smooth. Multiplication gives us a map
$$
\cT_S^{S'}\to \bN_S^{S'}
$$
and we define $\cR_S^{S'}$ to be the fiber product
$$
\cR_S^{S'}:= \cT_S^{S'}\times _{\bN_S^{S'}}\bN_S^{}.
$$
Over any closed $X^S$-subscheme of $G_S^{S'}/G_S$, the embedding $\cR_S^{S'}\to \cT_S^{S'}$ has finite codimension. Therefore $\cR_S^{S'}$ is also a reasonably nice ind-scheme, mapping to $G_S^{S'}/G_S$, and of ind-finite codimension in $\cT_S^{S'}$. Note that $\cR_S^{S'}$ is \emph{not} formally smooth, and in particular the map $\cR_S^{S'}\to G_S^{S'}/G_S$ is no longer ind-pro-smooth. As a functor we have
$$
\cR_S^{S'}(R)=\{(x,\cE,f,v)\}/\sim
$$
where $x,\cE,f$ are as in $\cT_S^{S'}$, and $v$ is an $\bN$-section of $\cE$ such that $f(v)$ extends\footnote{It is a priori defined only on $\Delta_S^{S'}$. The extension is necessarily unique.} to $\Delta_S$. We define the shifted dualizing complex on $\cT_S^{S'}$, $\cR_S^{S'}$ as for $\cT$, $\cR$. Namely:
\begin{enumerate}\item On each closed subscheme $(\cT_S^{S'})^c$ of $(\cT_S^{S'})^c$, pro-smooth over $(G_S^{S'}/G_S)^c$ we set
$$
\DC_{(\cT_S^{S'})^c}[-2\dim \bN_S+2|S|]
$$
to be the pullback of the dualizing complex of $(G_S^{S'}/G_S)^c$, i.e. the collection of its pullbacks to each formally smooth quotient $(\cT_S^{S'})^c_d$ of $(\cT_S^{S'})^c_d$ smooth over $(G_S^{S'}/G_S)^c$, compatible under $*$-pullback;
\item Since $\cT_S^{S'}$ is a reasonably nice ind-scheme, we can apply the $!$-pullback to such a collection of complexes on $(\cT_S^{S'})^c$, and obtain one on $(\cT_S^{S'})^{c-1}$. In this way, the collections $\DC_{(\cT_S^{S'})^c}[-2\dim \bN_S+2|S|]$ are compatible under $!$-pullbacks. The resulting compatible collection is called $\DC_{\cT_S^{S'}}[-2\dim \bN_S+2|S|]$.\
\item Using the ind-finite codimensionality of the embedding $i\colon \cR_S^{S'}\to\cT_S^{S'}$, we form a $!$-compatible collection of $*$-compatible collections of complexes
$$
\DC_{\cR_S^{S'}}[-2\dim \bN_S+2|S|]:=i^!\DC_{\cT_S^{S'}}[-2\dim \bN_S+2|S|].
$$\end{enumerate}

\subsubsection{}We will apply the abbreviations of subsection \ref{one}\oldref{abbrv} to our spaces $\cR,\cT$ etc. so that for instance
$$
\cR_{\{1,2\}}^{\{2\}}
$$
becomes
$$
\cR_{12}^2.
$$
We will also write $X^S$ as $\Pi_{s\in S}X_s$, e.g. $X^{\{1,2\}}=X_1\times X_2$. The obvious starting point for the global convolution diagram is $\cR_1^1\times\cR_2^2$, a Zariski-locally trivial $\cR$-bundle over $X_1\times X_2$. Consider the following space:
$$
\cR_{1+2}(R)=\{((x_1,x_2),\cE_1,\cE_2,f_1,f_2,v_1,v_2)\}/\sim
$$
where $x_1,x_2$ are $R$-points of $X$, each $\cE_i$ a principal $G$-bundle on $\Delta_{12}$, $f_i$ is a trivialization of $\cE_i$ on $\Delta_{12}^i$, and $v_i$ is an $\bN$-section of $\cE_i$ such that $f_i(v_i)$ extends to $\Delta_{12}$. It is constructed as
$$
\cR_{1+2}=\cR_{12}^1\times_{X_1\times X_2}\cR_{12}^2,
$$
a reasonably nice ind-scheme over $X_1\times X_2$. It is of ind-finite codimension in the formally smooth reasonably nice ind-scheme
$$
\cT_{1+2}=\cT_{12}^1\times_{X_1\times X_2}\cT_{12}^2=\{((x_1,x_2),\cE_1,\cE_2,f_1,f_2,\tilde{v}_1,\tilde{v}_2)\}/\sim.
$$
There is a map
$$
\alpha\colon \cR_{1+2}\to \cR_1^1\times \cR_2^2
$$
given by restricting $\cE_i,f_i,v_i$ to $\Delta_i\subset\Delta_{12}$. Over the diagonal $X_0\subset X_1\times X_2$, this map $\alpha$ is an isomorphism. But on the complement $U$ of the diagonal, we have a canonical isomorphism
$$
\cR_{1+2}|_{U}=(\cR_1^1\times\cR_2^2)|_U\times_U (\bN_1\times \bN_2)|_U
$$
and $\alpha$ is just the projection. Nonetheless, $\alpha$ \emph{is ind-pro-smooth}. Indeed, it is the product over $X_1\times X_2$ of maps
$$
\cR_{12}^1\to \cR_1^1\times X_2,
$$
$$
\cR_{12}^2\to \cR_2^2\times X_1;
$$
so it suffices to see that the former is ind-pro-smooth. But note that we can write
$$
\cT_1^1\times X_2 = G_1^1\frac{\times_{X_1}}{^{G_1}}\bN_1\times X_2 = G_{12}^1\frac{\times_{X_1\times X_2}}{^{G_{12}}}\bN_1
$$
where $G_{12}$ acts on $\bN_1$ via the homomorphism $G_{12}\to G_1$. Then, the natural map
$$
\cT_{12}^1\to \cT_1^1\times X_2
$$
is that associated to the pro-smooth map $\bN_{12}\to\bN_1$, so is ind-pro-smooth. The fact that the diagram
$$
\begin{matrix}\cR_{12}^1 &\to& \cR_1^1\times X_2\\
\downarrow&&\downarrow\\
\cT_{12}^1&\to&\cT_1^1\times X_2\end{matrix}
$$
is Cartesian gives the result. We have:
\begin{align}
\label{alpha}\alpha^*\DC_{\cR_1^1\times\cR_2^2}[-2\dim \bN_1\times\bN_2] \cong \DC_{\cR_{1+2}}[-2\dim\bN_{12}\times_{X_1\times X_2}\bN_{12}].
\end{align}
Note that $\cR_1^1\times \cR_2^2$, $\cT_1^1\times\cT_2^2$ are acted on factor-wise by $G_1\times G_2$, which receives the factor-wise map from $G_{12}\times_{X_1\times X_2} G_{12}$. This latter group also acts in the natural way on $\cR_{1+2}$, $\cT_{1+2}$, and the diagram
$$
\begin{matrix}\cR_{1+2} &\to& \cR_1^1\times \cR_2^2\\
\downarrow&&\downarrow\\
\cT_{1+2}&\to&\cT_1^1\times \cT_2^2\end{matrix}
$$
is $G_{12}\times_{X_1\times X_2} G_{12}$-equivariant. This action preserves the inductive structure of the diagram, and also the locally nice structure of each closed piece, which allows us to view the appropriately shifted dualizing complex on each space as $G_{12}\times_{X_1\times X_2} G_{12}$-equivariant. We may thus define the shifted equivariant Borel-Moore homologies:
$$
H^{G_1\times G_2}_{*-2\dim \bN_1\times\bN_2}(\cR_1^1\times\cR_2^2),
$$
$$
H^{G_{12}\times_{X_1\times X_2} G_{12}}_{*-2\dim \bN_1\times\bN_2}(\cR_1^1\times\cR_2^2),
$$
$$
H^{G_{12}\times_{X_1\times X_2} G_{12}}_{*-2\dim \bN_{12}\times_{X_1\times X_2}\bN_{12}}(\cR_{1+2}),
$$
as the colimits of the equivariant cohomologies of the appropriately shifted dualizing complexes on the various finite-dimensional approximations. We have maps
$$
H^{G_1\times G_2}_{*-2\dim \bN_1\times\bN_2}(\cR_1^1\times\cR_2^2)\to H^{G_{12}\times_{X_1\times X_2} G_{12}}_{*-2\dim \bN_1\times\bN_2}(\cR_1^1\times\cR_2^2)\to H^{G_{12}\times_{X_1\times X_2} G_{12}}_{*-2\dim \bN_{12}\times_{X_1\times X_2}\bN_{12}}(\cR_{1+2}).
$$
The first map is the restriction of the equivariant structure, while the second is induced by $\alpha^*$, using equation \ref{alpha}. This is the first step of our global convolution story. 

\subsubsection{}Let's define the remaining parts of the global convolution diagram. We set
$$
\widetilde{\cR}_{1+2}=\{((x_1,x_2),\cE_1,\cE_2,f_1,f_2,v_1,v_2,g_1)\}/\sim
$$
where $x_1,x_2,\cE_1,\cE_2,f_1,f_2,v_1,v_2$ are as in $\cR_{1+2}$, and $g_1$ is a trivialization of $\cE_1$ (on $\Delta_{12}$) required to satisfy:
$$
g_1v_1=f_2v_2.
$$
Note that $v_1$ is determined by the rest of the data as $v_1=g_1^{-1}f_2v_2$. That is, $\widetilde{\cR}_{1+2}$ is related to
$$
\widetilde{\cT}_{1+2} := \{((x_1,x_2),\cE_1,\cE_2,f_1,f_2,v_2,g_1)\}/\sim = G_{12}^1\times_{X_1\times X_2}\cR_{12}^2
$$
by the Cartesian square
$$
\begin{matrix}
\widetilde{\cR}_{1+2}&\to&\widetilde{\cT}_{1+2}\\
\downarrow&&\downarrow\\
\cR_{12}^1&\to&\cT_{12}^1
\end{matrix}
$$
where the rightmost downward arrow is the composition
$$
\widetilde{\cT}_{1+2} = G_{12}^1\times_{X_1\times X_2}\cR_{12}^2 \xrightarrow{} G_{12}^1\times_{X_1\times X_2}\bN_{12}\to G_{12}^1\frac{\times_{X_1\times X_2}}{^{G_{12}}}\bN_{12} = \cT_{12}^1.
$$
We have factor-wise actions of $G_{12}\times_{X_1\times X_2}G_{12}$ on $\widetilde{\cR}_{1+2}$, $\widetilde{\cT}_{1+2}$, such that the Cartesian diagram
\begin{align}\label{beta}
\begin{matrix}\widetilde{\cR}_{1+2}&\xrightarrow{\beta}&\cR_{1+2}\\
\downarrow&&\downarrow\\
\widetilde{\cT}_{1+2}&\xrightarrow{b}&\cT_{12}^1\times_{X_1\times X_2}\cR_{12}^2\end{matrix}
\end{align}
is equivariant. In terms of points, the left-hand $G_{12}$ acts by changing the trivialization $f_1$, while the right-hand factor acts by changing simultaneously the trivializations $g_1$, $f_2$; $\beta$ is the map which simply forgets $g_1$. The right-hand $G_{12}$ acts freely, and the quotient space is
$$
\overline{\cR}_{1+2}=\{((x_1,x_2),\cE_1,\cE_2,f_1,g_1^{-1}f_2,v_1,v_2)\}/\sim
$$
where $x_1,x_2,\cE_1,\cE_2,f_1,v_1$ are as in $\cR_{1+2}$, while $g_1^{-1}f_2$ is an isomorphism from $\cE_2$ to $\cE_1$ over $\Delta_{12}^2$, and $v_2$ is an $\bN$-section of $\cE_2$ such that $g_1^{-1}f_2v_2$ extends to $\Delta_{12}$ and is equal to $v_1$ there (again $v_1$ is determined by the rest of the data). We write 
$$
\gamma\colon  \widetilde{\cR}_{1+2}\to \overline{\cR}_{1+2}
$$
for the projection. It is ind-pro-smooth. Finally, we have a natural map
$$
\begin{matrix}
\hfill \delta\colon   \overline{\cR}_{1+2} & \to & \cR_{12}^{12}  =  \{((x_1,x_2),\cE,f,v)\}/\sim\\
 ((x_1,x_2),\cE_1,\cE_2,f_1,g_1^{-1}f_2,v_1,v_2) & \mapsto & ((x_1,x_2),\cE_2,f_1g_1^{-1}f_2,v_2).\end{matrix}
$$
Note that $\delta$ factors as $\delta=\delta'\delta''$ where $\delta''\colon  \overline{\cR}_{1+2}\to\bullet$ is an ind-closed embedding of finite codimension and $\delta'\colon \bullet\to \cR_{12}^{12}$ is defined by the Cartesian square
$$
\begin{matrix}\bullet & \xrightarrow{\delta'} & \cR_{12}^{12}\\
\downarrow && \downarrow\\
G_{12}^1\frac{\times_{X_1\times X_2}}{^{G_{12}}}G_{12}^2/G_{12} & \xrightarrow{d} & G_{12}^{12}/G_{12}\end{matrix}
$$
where the bottom row is simply the top row for $\bN=0$, and the vertical maps forget $v_1,v_2,v$. It is well-known that $d$ is ind-projective; this fact shows up already in \cite{MV2} and essentially follows from~\ref{proper}. It follows that $\delta$ is also ind-projective, meaning that in each piece of the inductive structure, $\delta$ is Zariski-locally of the form
$$
Y\times \BA\xrightarrow{f\times \on{id}} Z\times \BA
$$
for $f\colon Y\to Z$ a projective map between schemes of finite type, and $\BA$ some affine space of countable dimension. In fact, $\delta$ is an isomorphism over $U$, while over the diagonal its fibers are products of closed subvarieties of affine Grassmannians. Furthermore, $\delta$ is $G_{12}$-equivariant.\\

\subsubsection{}The global convolution diagram is
$$
\cR_1^1\times\cR_2^2\xleftarrow{\alpha} \cR_{1+2}\xleftarrow{\beta} \widetilde{\cR}_{1+2}\xrightarrow{\gamma} \overline{\cR}_{1+2}\xrightarrow{\delta} \cR_{12}^{12}.
$$
As we have explained, $\alpha$, $\beta$ are $G_{12}\times_{ X_1\times X_2}G_{12}$-equivariant, $\gamma$ is the quotient map by the free action of the right-hand $G_{12}$, and $\delta$ is equivariant for the remaining copy of $G_{12}$. We have already explained how $\alpha$ defines a map
$$
\alpha^*\colon  H^{G_1\times G_2}_{*-2\dim \bN_1\times\bN_2}(\cR_1^1\times\cR_2^2)\to H^{G_{12}\times_{X_1\times X_2} G_{12}}_{*-2\dim \bN_{12}\times_{X_1\times X_2}\bN_{12}}(\cR_{1+2}).
$$
Everything else works out essentially as in the main paper, as we now indicate. First, recall the $G_{12}\times_{ X_1\times X_2}G_{12}$-equivariant Cartesian diagram \ref{beta}:
$$
\begin{matrix}\widetilde{\cR}_{1+2}&\xrightarrow{\beta}&\cR_{1+2}\\
\downarrow&&\downarrow\\
\widetilde{\cT}_{1+2}&\xrightarrow{b}&\cT_{12}^1\times_{X_1\times X_2}\cR_{12}^2\end{matrix}
$$
and recall that $\widetilde{\cT}_{1+2}$ is nothing other than $G_{12}^1\times_{X_1\times X_2}\cR_{12}^2$. Thus we may write
$$
\begin{matrix}b &=& \on{pr}_1b\times_{X_1\times X_2}\on{pr}_2b\\
\on{pr}_1b & = & \psi\phi \hfill \\
\on{pr}_2b & = & \on{pr}_2 \hfill \end{matrix}
$$
where we have factored $\on{pr}_1b$ as
$$
G_{12}^1\times_{X_1\times X_2}\cR_{12}^2\xrightarrow{\phi} G_{12}^1\times_{X_1\times X_2}\bN_{12}\xrightarrow{\psi} \cT_{12}^1.
$$
It follows that
\begin{align*}
b^*\DC_{\cT_{12}^1\times_{X_1\times X_2}\cR_{12}^2}[-2\dim \bN_{12}\times_{X_1\times X_2}\bN_{12}] \cong \DC_{\widetilde{\cT}_{1+2}}[-2\dim \bN_{12}\times_{X_1\times X_2}G_{12}]
\end{align*}
and hence by base change we have have a map
\begin{align}\label{beta2}
\beta^*\DC_{\cR_{1+2}}[-2\dim \bN_{12}\times_{X_1\times X_2}\bN_{12}] \to \DC_{\widetilde{\cR}_{1+2}}[-2\dim \bN_{12}\times_{X_1\times X_2}G_{12}].
\end{align}
This map is equivariant, and it therefore determines a `pullback with support' map:
$$
\beta^*\colon H^{G_{12}\times_{X_1\times X_2} G_{12}}_{*-2\dim \bN_{12}\times_{X_1\times X_2}\bN_{12}}(\cR_{1+2})\to H^{G_{12}\times_{X_1\times X_2} G_{12}}_{*-2\dim \bN_{12}\times_{X_1\times X_2}G_{12}}(\widetilde{\cR}_{1+2}).
$$
Since it is a $G_{12}$-torsor, $\gamma$ induces an isomorphism
$$
\gamma^*\colon H^{G_{12}}_{*-2\dim \bN_{12}}(\overline{\cR}_{1+2}) \xrightarrow{\sim} H^{G_{12}\times_{X_1\times X_2} G_{12}}_{*-2\dim \bN_{12}\times_{X_1\times X_2}G_{12}}(\widetilde{\cR}_{1+2}).
$$
Finally since it is ind-proper and equivariant, $\delta$ induces a map
$$
\delta_*\colon H^{G_{12}}_{*-2\dim \bN_{12}}(\overline{\cR}_{1+2})\to H^{G_{12}}_{*-2\dim \bN_{12}}(\cR_{12}^{12}).
$$

\subsubsection{}Recall that (dual) specialization maps commute with pullbacks along smooth maps and pushforwards along proper maps, and are compatible with equivariance with respect to smooth group schemes. Therefore, since every space in sight is a reasonably nice ind-scheme and the groups $G_S$ are pro-smooth over $X^S$, we have (dual) specialization maps to the diagonal $X_0\subset X_1\times X_2$:
$$
\begin{matrix}s_1\colon  & H^{G_1\times G_2}_{*-2\dim \bN_1\times\bN_2}(\cR_1^1\times\cR_2^2) & \to & H^{G_0\times_{X_0} G_0}_{*-2\dim \bN_0\times_{X_0}\bN_0}(\cR_0^0\times_{X_0}\cR_0^0)\\
s_2\colon  & H^{G_{12}\times_{X_1\times X_2} G_{12}}_{*-2\dim \bN_1\times\bN_2}(\cR_1^1\times\cR_2^2) & \to & H^{G_{0}\times_{X_0} G_{0}}_{*-2\dim \bN_0\times_{X_0}\bN_0}(\cR_0^0\times_{X_0}\cR_0^0)\\
s_3\colon  & H^{G_{12}\times_{X_1\times X_2} G_{12}}_{*-2\dim \bN_{12}\times_{X_1\times X_2}\bN_{12}}(\cR_{1+2}) & \to & H^{G_{0}\times_{X_0} G_{0}}_{*-2\dim \bN_{0}\times_{X_0}\bN_{0}}(\cR_{0}^0\times_{X_0}\cR_0^0)\\
s_4\colon  & H^{G_{12}\times_{X_1\times X_2} G_{12}}_{*-2\dim \bN_{12}\times_{X_1\times X_2}G_{12}}(\widetilde{\cR}_{1+2}) & \to & H^{G_{0}\times_{X_0} G_{0}}_{*-2\dim \bN_{0}\times_{X_0}G_{0}}(\widetilde{\cR}_{0})\\
s_5\colon  & H^{G_{12}}_{*-2\dim \bN_{12}}(\overline{\cR}_{1+2}) & \to & H^{G_{0}}_{*-2\dim \bN_{0}}(\overline{\cR}_{0})\\
s_6\colon  & H^{G_{12}}_{*-2\dim \bN_{12}}(\cR_{12}^{12}) & \to & H^{G_{0}}_{*-2\dim \bN_{0}}(\cR_{0}^{0}).\end{matrix}
$$
Here $\widetilde{\cR}_{0}$, $\overline{\cR}_{0}$ are respectively locally trivial
$p^{-1}(\cR\times\cR)$, $q(p^{-1}(\cR\times\cR))$-bundles over $X_0$ in the notations of
diagram~\eqref{eq:65}. In fact, the restriction of the convolution diagram to $X_0$ induces the following maps between the targets of the specialization maps:
$$
\begin{matrix}H^{G_0\times_{X_0} G_0}_{*-2\dim \bN_0\times_{X_0}\bN_0}(\cR_0^0\times_{X_0}\cR_0^0) & \xrightarrow{\on{id}} & H^{G_0\times_{X_0} G_0}_{*-2\dim \bN_0\times_{X_0}\bN_0}(\cR_0^0\times_{X_0}\cR_0^0)\\
& \xrightarrow{\on{id}} & H^{G_0\times_{X_0} G_0}_{*-2\dim \bN_0\times_{X_0}\bN_0}(\cR_0^0\times_{X_0}\cR_0^0)\\
& \xrightarrow{\beta_0^*} & H^{G_{0}\times_{X_0} G_{0}}_{*-2\dim \bN_{0}\times_{X_0}G_{0}}(\widetilde{\cR}_{0})\\
& \xrightarrow{(\gamma_0^*)^{-1}} & H^{G_{0}\times_{X_0} G_{0}}_{*-2\dim \bN_{0}\times_{X_0}G_{0}}(\widetilde{\cR}_{0})\\
& \xrightarrow{(\delta_0)_*} & H^{G_{0}}_{*-2\dim \bN_{0}}(\cR_{0}^{0})\end{matrix}
$$
I claim that the maps $\alpha^*,\beta^*,(\gamma^*)^{-1},\delta_*$ are intertwined with $\on{id},\beta_0^*,(\gamma_0^*)^{-1},(\delta_0)_*$ by the (dual) specialization maps. For $\alpha^*$, $(\gamma^*)^{-1}$ it is a consequence of ind-pro-smoothness of $\alpha,\gamma$ (and also pro-smoothness of $G_{12}$). For $\delta_*$ it is a consequence of ind-properness. For $\beta^*$, it is because the map
$$
(\beta_0)^*\DC_{\cR_0^0\times_{X_0}\cR_0^0}[-2\dim \bN_0\times_{X_0}\bN_0] \to \DC_{\widetilde{\cR}_{0}}[-2\dim \bN_{0}\times_{X_0}G_{0}]
$$
defined using the Cartesian square:
$$
\begin{matrix}\widetilde{\cR}_{0}&\xrightarrow{\beta_0}&\cR_{0}^0\\
\downarrow&&\downarrow\\
\widetilde{\cT}_{0}^0&\xrightarrow{b_0}&\cT_{0}^0\times_{X_0}\cR_{0}^0\end{matrix}
$$
obtained by restricting diagram \ref{beta} to $X_0$, factors as:
$$
\begin{matrix}
(\beta_0)^*\DC_{\cR_0^0\times_{X_0}\cR_0^0}[-2\dim \bN_0\times_{X_0}\bN_0] & \cong & (\beta_0)^*i_1^!\DC_{\cR_{1+2}}[-2\dim \bN_{12}\times_{X_1\times X_2}\bN_{12}+2]\\
& \xrightarrow{can} & i_2^!\beta^*\DC_{\cR_{1+2}}[-2\dim \bN_{12}\times_{X_1\times X_2}\bN_{12}+2]\\
& \xrightarrow{i_2^![2](\text{\ref{beta2}})} & i^!\DC_{\widetilde{\cR}_{1+2}}[-2\dim \bN_{12}\times_{X_1\times X_2}G_{12}+2]\\
& \cong & \DC_{\widetilde{\cR}_{0}}[-2\dim \bN_{0}\times_{X_0}G_{0}].\end{matrix}
$$
Here $can$ is the canonical map arising from the base change isomorphism,~(\ref{beta2}) denotes the map of equation~\ref{beta2}, and $i_1,i_2$ denote the appropriate inclusions of the diagonal subspaces. The consequence is the following formula:
$$
s_6\delta_*(\gamma^*)^{-1}\beta^*\alpha^*=(\delta_0)_*(\gamma_0^*)^{-1}\beta_0^*s_1\colon H^{G_1\times G_2}_{*-2\dim \bN_1\times\bN_2}(\cR_1^1\times\cR_2^2)\to H^{G_{0}}_{*-2\dim \bN_{0}}(\cR_{0}^{0}).
$$

\subsubsection{}Now each (dual) specialization map $s_n$ factors as $s_n'j_n^*$ where $j_n^*$ is the restriction map to the equivariant Borel-Moore homology of the part lying over $U$, and $s_n'$ is some other map. Furthermore, the restriction of the convolution diagram to $U$ induces the following maps between the targets of the restriction maps:
$$
\begin{matrix} & & H^{(G_1\times G_2)|_U}_{*-2\dim (\bN_1\times\bN_2)|_U}((\cR_1^1\times\cR_2^2)|_U) \hfill\\
& \xrightarrow{} & H^{(G_1\times G_2)|_U\times_{U} (G_1\times G_2)|_U}_{*-2\dim (\bN_1\times\bN_2)|_U}((\cR_1^1\times\cR_2^2)|_U) \hfill\\
& \xrightarrow{} & H^{(G_1\times G_2)|_U\times_{U} (G_1\times G_2)|_U}_{*-2\dim (\bN_1\times \bN_2)|_U\times_{U}(\bN_1\times \bN_2)|_U}((\cR_1^1\times \bN_2)|_U\times_{U} (\bN_1\times \cR_2^2)|_U) \hfill\\
& \xrightarrow{} & H^{((G_1\times_{X_1} G_1)\times (G_2\times_{X_2} G_2))|_U}_{*-2\dim ((\bN_1\times_{X_1}G_1)\times (G_2\times_{X_2}\bN_2))|_U}((\widetilde{\cR}_1\times (G_2\times_{X_2}\cR_2^2))|_U) \hfill\\
& \xrightarrow{} & H^{(G_{1}\times G_2)|_U}_{*-2\dim (\bN_{1}\times\bN_2)|_U}((\cR_1^1\times \cR_2^2)|_U) \hfill\\
& \xrightarrow{} & H^{(G_{1}\times G_2)|_U}_{*-2\dim (\bN_{1}\times\bN_2)|_U}((\cR_1^1\times \cR_2^2)|_U) \hfill\end{matrix}
$$
Let us explain what each map does:
\begin{enumerate}\item The first map views any $(G_1\times G_2)|_U$-equivariant class as also equivariant for the trivial actions of the left-hand copy of $G_2$, and the right-hand copy of $G_1$, in $(G_1\times G_2)|_U\times_{U} (G_1\times G_2)|_U$.\
\item The second map pulls this back along the $(\bN_2\times\bN_1)|_U$-bundle map (i.e. multiplies fiberwise by the equivariant fundamental class of $\bN(\cO)\times\bN(\cO)$).\
\item The third map starts by rewriting $(\cR_1^1\times \bN_2)|_U\times_{U} (\bN_1\times \cR_2^2)|_U$ as $((\cR_1^1\times_{X_1}\bN_1)\times(\bN_2\times_{X_2}\cR_2^2))|_U$, and rewriting the action of $(G_1\times G_2)|_U\times_{U} (G_1\times G_2)|_U$ as one of $((G_1\times_{X_1} G_1)\times (G_2\times_{X_2} G_2))|_U$. By definition, $\widetilde{\cR}_1$ is the locally trivial $p^{-1}(\cR\times\cR)$-bundle on $X_1$ given as
$$
\widetilde{\cR}_1=\bN_1\times_{\bN_1^1}(G_1^1\times_{X_1}\bN_1).
$$
The $G_1\times_{X_1}G_1$-equivariant map from here to $\cR_1^1\times_{X_1}\bN_1$ is given as the product (over $X_1$) of the quotient by the right-hand copy of $G_1$ with the projection to the right-hand copy of $\bN_1$. The `pullback with support' map
$$
H_{*-2\dim \bN_1\times_{X_1}\bN_1}^{G_1\times_{X_1}G_1}(\cR_1^1\times_{X_1}\bN_1)\to H_{*-2\dim \bN_1\times_{X_1}G_1}^{G_1\times_{X_1}G_1}(\widetilde{\cR}_1)
$$
corresponds to the composition of usual `pullback with support' (spread out over $X_1$) with multiplication by $H^*_{G_1}(X_1)$ under the identification 
$$
H_{*-2\dim \bN_1\times_{X_1}\bN_1}^{G_1\times_{X_1}G_1}(\cR_1^1\times_{X_1}\bN_1)=H_{*-2\dim \bN_1}^{G_1}(\cR_1^1)\otimes_{H^*(X_1)} H^*_{G_1}(X_1).
$$ Meanwhile, the `pullback with support' (actually, here no support is required) map
$$
H_{*-2\dim \bN_2\times_{X_2}\bN_2}^{G_2\times_{X_2}G_2}(\bN_2\times_{X_2}\cR_2^2)\to H_{*-2\dim G_2\times_{X_2}\bN_2}^{G_2\times_{X_2}G_2}(G_2\times_{X_2}\cR_2^2)
$$
is isomorphic simply to the multiplication map 
$$
H^*_{G_2}(X_2)\otimes_{H^*(X_2)}H_{*-2\dim \bN_2}^{G_2}(\cR_2^2)\to H_{*-2\dim \bN_2}^{G_2}(\cR_2^2).
$$\\
\item The fourth map is the isomorphism, and the fifth is the identity.\end{enumerate}
The result is that the composition of all these maps is the identity. On the other hand, since the restriction maps $j_n^*$ intertwine these maps with the corresponding maps on the $X_1\times X_2$ level, we have the following:
$$\begin{matrix}
(\delta_0)_*(\gamma_0^*)^{-1}\beta_0^*s_1 & = & s_6\delta_*(\gamma^*)^{-1}\beta^*\alpha^*\hfill\\
& = & s'_6j_6^*\delta_*(\gamma^*)^{-1}\beta^*\alpha^*\\
& = & s'_6j_1^*.\hfill\end{matrix}
$$

\subsubsection{}Finally, note that this last map $s'_6j_1^*$ is symmetric with respect to the automorphism $\tau$ of $H^{G_1\times G_2}_{*-2\dim \bN_1\times\bN_2}(\cR_1^1\times\cR_2^2)$ induced by the degree $2$ automorphisms of $G_1\times G_2$, $\cR_1^1\times\cR_2^2$ which switch the factors (and also exchange $1$ with $2$). Therefore, $(\delta_0)_*(\gamma_0^*)^{-1}\beta_0^*s_1$ has the same property. But, taking $X=\BC$, we identify the domain
$$
H^{G_1\times G_2}_{*-2\dim \bN_1\times\bN_2}(\cR_1^1\times\cR_2^2)=H^{G(\cO)}_{*-2\dim \bN(\cO)}(\cR)\otimes H^{G(\cO)}_{*-2\dim \bN(\cO)}(\cR)
$$
and the target 
$$
H^{G_{0}}_{*-2\dim \bN_{0}}(\cR_{0}^{0}) = H^{G(\cO)}_{*-2\dim \bN(\cO)}(\cR).
$$
The map $(\delta_0)_*(\gamma_0^*)^{-1}\beta_0^*s_1$ is the usual convolution ($s_1$ is an isomorphism) while $\tau$ is the standard twist. Therefore, the Coulomb branch is commutative as claimed.


\bibliographystyle{myamsalpha}
\bibliography{nakajima,mybib,coulomb}

\def\cprime{$'$} \def\cprime{$'$} \def\cprime{$'$} \def\cprime{$'$}
  \def\cprime{$'$}
  \providecommand{\noopsort}[1]{}\def\cftil#1{\ifmmode\setbox7\hbox{$\accent"5E#1$}\else
  \setbox7\hbox{\accent"5E#1}\penalty 10000\relax\fi\raise 1\ht7
  \hbox{\lower1.15ex\hbox to 1\wd7{\hss\accent"7E\hss}}\penalty 10000
  \hskip-1\wd7\penalty 10000\box7}
\providecommand{\bysame}{\leavevmode\hbox to3em{\hrulefill}\thinspace}
\providecommand{\MR}{\relax\ifhmode\unskip\space\fi MR }
\providecommand{\MRhref}[2]{%
  \href{http://www.ams.org/mathscinet-getitem?mr=#1}{#2}
}
\providecommand{\href}[2]{#2}
\begin{thebibliography}{KWWY14}

\bibitem[Part II]{main}
A.~{\noopsort{01}}~Braverman, M.~Finkelberg, and H.~Nakajima, \emph{{Towards a
  mathematical definition of Coulomb branches of $3$-dimensional $\mathcal N=4$
  gauge theories, II}}, Adv. Theor. Math. Phys. \textbf{22} (2018), no.~5,
  1071--1147, \href{http://arxiv.org/abs/1601.03586}{{\ttfamily
  arXiv:1601.03586 [math.RT]}}.

\bibitem[{\noopsort{a}}Quiver]{blowup}
\bysame, \emph{{Coulomb branches of $3d$ {$\mathcal N=4$} quiver gauge theories
  and slices in the affine {G}rassmannian}}, Adv. Theor. Math. Phys.
  \textbf{23} ({\noopsort{2019a}}2019), no.~1, 75--166, With two appendices by
  Braverman, Finkelberg, Joel Kamnitzer, Ryosuke Kodera, Nakajima, Ben Webster
  and Alex Weekes, \href{http://arxiv.org/abs/1604.03625}{{\ttfamily
  arXiv:1604.03625 [math.RT]}}. \MR{4020310}

\bibitem[ABG04]{MR2053952}
S.~Arkhipov, R.~Bezrukavnikov, and V.~Ginzburg, \emph{Quantum groups, the loop
  {G}rassmannian, and the {S}pringer resolution}, J. Amer. Math. Soc.
  \textbf{17} (2004), no.~3, 595--678. \MR{2053952 (2005g:16055)}

\bibitem[Bap15]{MR3366026}
A.~Bapat, \emph{Torus actions and tensor products of intersection cohomology},
  Pacific J. Math. \textbf{276} (2015), no.~1, 19--34. \MR{3366026}

\bibitem[BeiDr]{Beilinson-Drinfeld}
A.~Beilinson and V.~Drinfeld, \emph{Quantization of {H}itchin's integrable
  system and {H}ecke eigensheaves}, available at
  \url{http://www.math.uchicago.edu/~mitya/langlands.html}, 2000.

\bibitem[BDF16]{bdf}
A.~{Braverman}, G.~{Dobrovolska}, and M.~{Finkelberg}, \emph{{Gaiotto-Witten
  superpotential and Whittaker D-modules on monopoles}}, Adv. Math.
  \textbf{300} (2016), 451--472,
  \href{http://arxiv.org/abs/1406.6671}{{\ttfamily arXiv:1406.6671 [math.AG]}}.

\bibitem[BF08]{MR2422266}
R.~Bezrukavnikov and M.~Finkelberg, \emph{Equivariant {S}atake category and
  {K}ostant-{W}hittaker reduction}, Mosc. Math. J. \textbf{8} (2008), no.~1,
  39--72, 183. \MR{2422266 (2009d:19008)}

\bibitem[BF14]{bf14}
A.~Braverman and M.~Finkelberg, \emph{Semi-infinite {S}chubert varieties and
  quantum {K}-theory of flag manifolds}, J. Amer. Math. Soc. \textbf{27}
  (2014), no.~4, 1147--1168.

\bibitem[BK94]{brko}
R.~Brylinski and B.~Kostant, \emph{Nilpotent orbits, normality and
  {H}amiltonian group actions}, J. Amer. Math. Soc. \textbf{7} (1994),
  269--298.

\bibitem[BK05]{brion}
M.~Brion and S.~Kumar, \emph{Frobenius splitting methods in geometry and
  representation theory}, Progress in Mathematics, 231, Birkh\"auser, Boston,
  2005.

\bibitem[BL94]{BL}
J.~Bernstein and V.~Lunts, \emph{Equivariant sheaves and functors}, Lecture
  Notes in Mathematics, vol. 1578, Springer-Verlag, Berlin, 1994. \MR{1299527
  (95k:55012)}

\bibitem[BP08]{MR2443330}
R.~Bielawski and V.~Pidstrygach, \emph{Gelfand-{Z}eitlin actions and rational
  maps}, Math. Z. \textbf{260} (2008), no.~4, 779--803. \MR{2443330
  (2010j:53174)}

\bibitem[Bro93]{MR1223221}
B.~Broer, \emph{Line bundles on the cotangent bundle of the flag variety},
  Invent. Math. \textbf{113} (1993), no.~1, 1--20. \MR{1223221 (94g:14027)}

\bibitem[Bry89]{Brylinski}
R.~K. Brylinski, \emph{Limits of weight spaces, {L}usztig's q-analogs, and
  fiberings of adjoint orbits}, J. Amer. Math. Soc. \textbf{2} (1989),
  517--533.

\bibitem[BTX10]{Benini:2010uu}
F.~Benini, Y.~Tachikawa, and D.~Xie, \emph{{Mirrors of 3d Sicilian theories}},
  JHEP \textbf{1009} (2010), 063,
  \href{http://arxiv.org/abs/1007.0992}{{\ttfamily arXiv:1007.0992 [hep-th]}}.

\bibitem[CHLZ17]{2017arXiv170103825C}
G.~{Cheng}, A.~{Hanany}, Y.~{Li}, and Y.~{Zhao}, \emph{{Coulomb Branch for
  A-type Balanced Quivers in 3d $\mathcal{N}=4$ gauge theories}}, ArXiv
  e-prints (2017), \href{http://arxiv.org/abs/1701.03825}{{\ttfamily
  arXiv:1701.03825 [hep-th]}}.

\bibitem[CHMZ14a]{Cremonesi:2014kwa}
S.~Cremonesi, A.~Hanany, N.~Mekareeya, and A.~Zaffaroni, \emph{{Coulomb branch
  Hilbert series and Hall-Littlewood polynomials}}, JHEP \textbf{1409} (2014),
  178, \href{http://arxiv.org/abs/1403.0585}{{\ttfamily arXiv:1403.0585
  [hep-th]}}.

\bibitem[CHMZ14b]{Cremonesi:2014vla}
\bysame, \emph{{Coulomb branch Hilbert series and Three Dimensional Sicilian
  Theories}}, JHEP \textbf{1409} (2014), 185,
  \href{http://arxiv.org/abs/1403.2384}{{\ttfamily arXiv:1403.2384 [hep-th]}}.

\bibitem[Dri06]{drin}
V.~Drinfeld, \emph{Infinite-dimensional vector bundles in algebraic geometry},
  The unity of mathematics, Progr. Math., vol. 244, Birkh\"auser Boston,
  Boston, MA, 2006, pp.~263--304.

\bibitem[Gai01]{MR1826370}
D.~Gaitsgory, \emph{Construction of central elements in the affine {H}ecke
  algebra via nearby cycles}, Invent. Math. \textbf{144} (2001), no.~2,
  253--280. \MR{1826370 (2002d:14072)}

\bibitem[Gai12]{MR3006961}
D.~Gaiotto, \emph{{$N=2$} dualities}, J. High Energy Phys. (2012), no.~8, 034,
  front matter + 57, \href{http://arxiv.org/abs/0904.2715}{{\ttfamily
  arXiv:0904.2715 [hep-th]}}. \MR{3006961}

\bibitem[Gin91]{ginsburg}
V.~Ginsburg, \emph{Perverse sheaves and ${{\mathbb C}^*}$-actions}, J. Amer.
  Math. Soc. \textbf{4} (1991), no.~3, 483--490.

\bibitem[{Gin}95]{1995alg.geom.11007G}
V.~{Ginzburg}, \emph{{Perverse sheaves on a Loop group and Langlands'
  duality}}, ArXiv e-prints (1995),
  \href{http://arxiv.org/abs/alg-geom/9511007}{{\ttfamily
  arXiv:alg-geom/9511007 [alg-geom]}}.

\bibitem[GK]{GK}
V.~Ginzburg and D.~Kazhdan, \emph{Construction of symplectic varieties arising
  in `{S}icilian theories'}, in preparation.

\bibitem[Gor17]{268987}
V.~Gorin, \emph{A combinatorial expression of {H}all-{L}ittlewood polynomials},
  MathOverflow, \url{https://mathoverflow.net/q/268987} (version: 2017-05-04).

\bibitem[GR12]{2012JHEP...05..145G}
D.~{Gaiotto} and S.~S. {Razamat}, \emph{{Exceptional indices}}, Journal of High
  Energy Physics \textbf{5} (2012), 145,
  \href{http://arxiv.org/abs/1203.5517}{{\ttfamily arXiv:1203.5517 [hep-th]}}.

\bibitem[GW09]{MR2610576}
D.~Gaiotto and E.~Witten, \emph{{$S$}-duality of boundary conditions in {$\scr
  N=4$} super {Y}ang-{M}ills theory}, Adv. Theor. Math. Phys. \textbf{13}
  (2009), no.~3, 721--896. \MR{2610576 (2011j:81322)}

\bibitem[Hes80]{MR593631}
W.~H. Hesselink, \emph{Characters of the nullcone}, Math. Ann. \textbf{252}
  (1980), no.~3, 179--182. \MR{593631 (82c:17004)}

\bibitem[Jos97]{MR1446576}
A.~Joseph, \emph{On a {H}arish-{C}handra homomorphism}, C. R. Acad. Sci. Paris
  S\'er. I Math. \textbf{324} (1997), no.~7, 759--764. \MR{1446576}

\bibitem[Kac90]{Kac}
V.~G. Kac, \emph{Infinite-dimensional {L}ie algebras}, third ed., Cambridge
  University Press, Cambridge, 1990. \MR{MR1104219 (92k:17038)}

\bibitem[KN18]{2016arXiv160800875K}
R.~Kodera and H.~Nakajima, \emph{Quantized {C}oulomb branches of {J}ordan
  quiver gauge theories and cyclotomic rational {C}herednik algebras},
  String-{M}ath 2016, Proc. Sympos. Pure Math., vol.~98, Amer. Math. Soc.,
  Providence, RI, 2018, pp.~49--78,
  \href{http://arxiv.org/abs/1608.00875}{{\ttfamily arXiv:1608.00875
  [math.RT]}}. \MR{3821749}

\bibitem[KS90]{KaSha}
M.~Kashiwara and P.~Schapira, \emph{Sheaves on manifolds}, Grundlehren der
  Mathematischen Wissenschaften [Fundamental Principles of Mathematical
  Sciences], vol. 292, Springer-Verlag, Berlin, 1990, With a chapter in French
  by Christian Houzel. \MR{MR1074006 (92a:58132)}

\bibitem[KV04]{MR2102701}
M.~Kapranov and E.~Vasserot, \emph{Vertex algebras and the formal loop space},
  Publ. Math. Inst. Hautes \'Etudes Sci. (2004), no.~100, 209--269. \MR{2102701
  (2005k:14044)}

\bibitem[KWWY14]{kwy}
J.~Kamnitzer, B.~Webster, A.~Weekes, and O.~Yacobi, \emph{Yangians and
  quantizations of slices in the affine {G}rassmannian}, Algebra and Number
  Theory \textbf{8} (2014), no.~4, 857--893.

\bibitem[LO08]{lasols}
Y.~Laszlo and M.~Olsson, \emph{The six operations for sheaves on {A}rtin
  stacks. {II}. {A}dic coefficients}, Inst. Hautes \'Etudes Sci. Publ. Math.
  (2008), no.~107, 169--210.

\bibitem[Los06]{slice}
I.~Losev, \emph{Symplectic slices for reductive groups}, Sbornik Math.
  \textbf{197} (2006), no.~2, 213--224.

\bibitem[Lus81]{Lu-Green}
G.~Lusztig, \emph{Green polynomials and singularities of unipotent classes},
  Adv. in Math. \textbf{42} (1981), no.~2, 169--178. \MR{MR641425 (83c:20059)}

\bibitem[Lus83]{Lus-ast}
G.~Lusztig, \emph{Singularities, character formulas, and a q-analog of weight
  multiplicities}, Ast\'erisque \textbf{101-102} (1983), 208--229.

\bibitem[MOV05]{mov}
A.~Malkin, V.~Ostrik, and M.~Vybornov, \emph{The minimal degeneration
  singularities in the affine {G}rassmannians}, Duke Math. J. \textbf{126}
  (2005), no.~2, 233--249.

\bibitem[MT12]{MR2985331}
G.~W. Moore and Y.~Tachikawa, \emph{On 2d {TQFT}s whose values are holomorphic
  symplectic varieties}, String-{M}ath 2011, Proc. Sympos. Pure Math., vol.~85,
  Amer. Math. Soc., Providence, RI, 2012, pp.~191--207. \MR{2985331}

\bibitem[MV03]{MR1968260}
I.~Mirkovi{\'c} and M.~Vybornov, \emph{On quiver varieties and affine
  {G}rassmannians of type {$A$}}, C. R. Math. Acad. Sci. Paris \textbf{336}
  (2003), no.~3, 207--212. \MR{1968260 (2004c:14093)}

\bibitem[MV07]{MV2}
I.~Mirkovi{\'c} and K.~Vilonen, \emph{Geometric {L}anglands duality and
  representations of algebraic groups over commutative rings}, Ann. of Math.
  (2) \textbf{166} (2007), no.~1, 95--143. \MR{2342692 (2008m:22027)}

\bibitem[Nak17]{268865}
H.~Nakajima, \emph{A combinatorial expression of {H}all-{L}ittlewood
  polynomials}, MathOverflow, \url{https://mathoverflow.net/q/268865} (version:
  2017-05-06).

\bibitem[Nak16]{2015arXiv150303676N}
\bysame, \emph{{Towards a mathematical definition of Coulomb branches of
  $3$-dimensional $\mathcal N=4$ gauge theories, I}}, Adv. Theor. Math. Phys.
  \textbf{20} ({\noopsort{2015}}2016), no.~3, 595--669,
  \href{http://arxiv.org/abs/1503.03676}{{\ttfamily arXiv:1503.03676
  [math-ph]}}.

\bibitem[{Nam}16]{2016arXiv160306105N}
Y.~{Namikawa}, \emph{{A characterization of nilpotent orbit closures among
  symplectic singularities}}, ArXiv e-prints (2016),
  \href{http://arxiv.org/abs/1603.06105}{{\ttfamily arXiv:1603.06105
  [math.AG]}}.

\bibitem[Rub08]{MR2342006}
H.~Rubenthaler, \emph{The {$(A_2,G_2)$} duality in {$E_6$}, octonions and the
  triality principle}, Trans. Amer. Math. Soc. \textbf{360} (2008), no.~1,
  347--367. \MR{2342006}

\bibitem[Tac]{Tach-review}
Y.~Tachikawa, \emph{{A pseudo-mathematical pseudo-review on 4d $\mathcal N=2$
  supersymmetric quantum field theories}}, available at
  \url{http://member.ipmu.jp/yuji.tachikawa/not-on-arxiv.html}.

\bibitem[Tac17]{2017arXiv171209456T}
\bysame, \emph{{On 'categories' of quantum field theories}}, ArXiv e-prints
  (2017), \href{http://arxiv.org/abs/1712.09456}{{\ttfamily arXiv:1712.09456
  [math-ph]}}.

\end{thebibliography}


\begin{thebibliography}{Bie21}
\bibitem[Bie21]{Bie21} R.~Bielawski, \emph{On the Moore-Tachikawa varieties},
  arXiv e-prints (2021), {{\ttfamily arXiv:2104.0555 [math.SG]}}.
\end{thebibliography}

\section*{Erratum}

As pointed out by Bielawski~\cite{Bie21}, the statement of
Proposition~5.20 is not correct. The action of $G^\vee$ on the second
factor should be twisted by the automorphism of $G^\vee$ interchanging
conjugacy classes of $g$ and $g^{-1}$. We use the automorphism in the
proof of Proposition~5.20, as $\mathfrak C_{G^\vee}$ is induced from
that.
The same correction applies to the last sentence of Remark~5.22.

We also need to twist the first factor of the diagonal embedding
$\Gr_G\to(\Gr_G)^2$ in Proposition~5.23 by the automorphism induced by
$g\mapsto g^{-1}$, $G((z))\to G((z))$. The induced automorphism
$\mathcal C_G$ on $D_G(\Gr_G)$ corresponds to $\mathfrak C_{G^\vee}$
under the derived Satake equivalence $\Psi$. The latter appeared
in $5$ lines above from the end of the proof of Proposition~5.23.


\end{document}